\documentclass[12
pt]{amsproc}
\usepackage[cp1251]{inputenc}
\RequirePackage{srcltx}
\usepackage{amsmath}
\usepackage{tikz}
\usepackage{verbatim}

\usepackage{amssymb}
\usepackage{amsxtra}

\usepackage{latexsym}
\usepackage{ifthen}
\usepackage{mathrsfs}

\def\N{{{\Bbb N}}}
\def\Z{{{\Bbb Z}}}
\def\T{{{\Bbb T}}}
\def\R{{\Bbb R}}
\def\C{{\Bbb C}}
\def\l{{\lambda }}
\def\a{{\alpha }}
\def\D{{\Delta }}

\def\a{{\alpha}}
\def\b{{\beta}}
\def\d{{\delta}}
\def\e{{\varepsilon}}
\def\s{{\sigma}}
\def\vp{{\varphi}}
\def\t{{\theta }}
\def\g{{\gamma }}

\def\w{{\omega }}

\def\p{{\partial }}

\def\){\right)}
\def\({\left(}

\def\supp{\operatorname{supp}}

\def\sign{\operatorname{sign}}
\def\Re{\operatorname{Re}}

\def\Lip{\operatorname{Lip}}
\def\const{\operatorname{const}}

\numberwithin{equation}{section}
\newtheorem{corollary}{Corollary}[section]
\newtheorem{lemma}{Lemma}[section]
\newtheorem{example}{Example}[section]
\newtheorem{theorem}{Theorem}[section]
\newtheorem{proposition}{Proposition}[section]
\newtheorem{remark}{Remark}[section]

\newcommand{\sgn}{\text{sgn}}

\def\R{\Bbb R}

\def\XXint#1#2#3{{\setbox0=\hbox{$#1{#2#3}{\int}$}
     \vcenter{\hbox{$#2#3$}}\kern-.5\wd0}}

\par

\bigskip
\bigskip

\begin{document}

\title[Hardy--Littlewood and Ulyanov inequalities]{
Hardy--Littlewood and Ulyanov inequalities}

\author{Yurii
Kolomoitsev$^*$}
\address{
Yu.
Kolomoitsev,
Universit\"at zu L\"ubeck,
Institut f\"ur Mathematik,
Ratzeburger Allee 160,
23562 L\"ubeck, Germany
}
\email{kolomoitsev@math.uni-luebeck.de}

\author{Sergey Tikhonov}
\address{S.~Tikhonov, ICREA, Centre de Recerca Matem\`{a}tica, and UAB,
Apartat 50~08193 Bellaterra, Barce\-lona, Spain}
\email{stikhonov@crm.cat}

%\thanks{The first author was supported by the project AFFMA that has received funding from the European Union's Horizon 2020 research and innovation programme under the Marie Sklodowska-Curie grant agreement No~704030.
%The second author's research was partially supported by  MTM 2014-59174-P, 2014 SGR 289, and
%by the Alexander von Humboldt Foundation.}
\date{\today}

\thanks{$^*$Corresponding author}

\bigskip\bigskip\bigskip\bigskip

\bigskip
\begin{abstract}
We give the full solution of the following  problem: obtain sharp inequalities between the moduli of smoothness
 $\omega_\alpha(f,t)_q$ and $\omega_\beta(f,t)_p$ for $0<p<q\le \infty$. A similar problem for the generalized  $K$-functionals and their realizations between the couples $(L_p, W_p^\psi)$ and $(L_q, W_q^\varphi)$ is also solved.

 The main tool is the new Hardy--Littlewood--Nikol'skii inequalities. More precisely, we obtained the asymptotic behavior of the quantity
$$
\sup_{T_n}
\frac{\Vert \mathcal{D}(\psi)(T_n)\Vert_q}{\Vert
\mathcal{D}(\vp)(T_n)\Vert_p},\qquad 0<p<q\le \infty,
$$
where the supremum is taken over all nontrivial trigonometric polynomials $T_n$ of degree at most $n$
and $\mathcal{D}(\psi), \mathcal{D}(\vp)$ are  the Weyl-type differentiation operators.

 We also prove the Ulyanov and Kolyada-type inequalities in the Hardy spaces.
 Finally, we apply  the obtained estimates to derive new embedding theorems for the Lipschitz and Besov spaces.
\end{abstract}

\maketitle

\newpage

\bigskip\bigskip\bigskip\bigskip

\bigskip
\tableofcontents

 %$\widetilde{K}_{\psi}\to \mathcal{R}$

%${\bf W_p(\psi)}$

%${ H_a\to \mathcal{H}_\a}$

%${ \mathcal{D}(g)\to \mathcal{D}(\psi/\phi)}$

%$L_p  L_p$\qquad   \qquad

%$\int_0 \int\limits_0$
\newpage
\bigskip

\section*{Basic notation}

Let $\R^d$ be the $d$-dimensional Euclidean space with elements $x=(x_1,\dots,x_d)$, and $(x,y)=x_1y_1+\dots+x_dy_d$,
 $|x|=(x,x)^{1/2}$, $|x|_1=|x_1|+\dots+|x_d|$, $|x|_\infty=\max_{1\le j\le d}|x_j|$, and $\T^d=\R^d / 2\pi \Z^d$.

In what follows, $X=\R^d$ or $X=\T^d$. As usual, the space
 $L_p(X)$ consists of all measurable functions $f$ such that
for $0<p<\infty$
$$
\Vert f\Vert_{L_p(X)}=\bigg(\int_{X}|f(x)|^p
dx\bigg)^\frac1p<\infty
$$
and for $p=\infty$
$$
\Vert f\Vert_{L_\infty(X)}=\mathrel{\mathop{\mbox{\rm ess\,sup}}_{x\in
X}}|f(x)|<\infty.
$$
Note that  $\Vert f\Vert_{L_p(X)}$ for $0<p<1$ is a quasi-norm satisfying $\Vert f+g\Vert^p_{L_p(X)}\le \Vert f\Vert^p_{L_p(X)}+\Vert g\Vert^p_{L_p(X)}$.  In this paper, we mostly deal with the
$L_p(\T^d)$-setting. In this case, for simplicity, we write $\Vert f\Vert_p=\Vert f\Vert_{L_p(\T^d)}$.

We denote by  $\mathcal{T}_n$ the
space of all trigonometric polynomials of order at most $n$, i.e.,
$$
\mathcal{T}_{n}={\rm span}\left\{e^{i(k,x)}, %|k_j|\le n,
 \,\, k\in \Z^d,\,\,|k|_\infty\le n \right\}.
$$
Also, set
$$
\mathcal{T}=\bigcup_{n\ge 0}\mathcal{T}_n
$$
and
$$
\mathcal{T}'_n=\left\{T\in
\mathcal{T}_n\,:\,T(x)\not\equiv0\quad\text{and}\quad\int_{\T^d}T(x) dx=0 \right\}.
$$
The Fourier transform of $f\in L_1(\R^d)$ is defined by
$$
\mathcal{F}f(\xi)=\widehat{f}(\xi)=\frac1{(2\pi)^{d/2}}\int_{\R^d} f(x)e^{-i(x,\xi)}dx.
$$

Throughout the paper, we use the notation
$$
\, A \lesssim B,
$$
with $A,B\ge 0$, for the
estimate
$\, A \le C\, B,$ where $\, C$ is a positive constant independent of
the essential variables in $\, A$ and $\, B$ (usually, $f$, $\d$, and $n$). %\textbf{??? The same is for $ A \gtrsim B$. ???}
 If $\, A \lesssim B$
and $\, B \lesssim A$ simultaneously, we write $\, A \asymp B$ and say that $\, A$
is equivalent to $\, B$.
For two function spaces
$\, X$ and $\, Y,$ we will use the notation $$\, Y \hookrightarrow X$$ if
$\, Y \subset X$ and $\, \| f\|_X \lesssim \| f\|_Y$ for all $\, f \in
Y.$

%By $C$ we denote some positive constants depending on indicated parameters.

Let $1\le p\le\infty$ and $0<q\le\infty$. In what follows, $p'$ and $q_1$ are given by
%We also use $\frac{1}{p}+\frac{1}{p'}=1$ and
% for any $q\in (0,\infty]$ we define
$$
\frac{1}{p}+\frac{1}{p'}=1\quad\text{and}\quad q_1:=\left\{
      \begin{array}{ll}
        q, & \hbox{$q<\infty$;} \\
        1, & \hbox{$q=\infty$.}
      \end{array}
    \right.
$$

Finally, for any real number $a$,
$[a]$ be the floor function, that is,
$[a]$
is the largest integer not greater than $a$
and  let $$a_+=\left\{
                                            \begin{array}{ll}
                                              a, & \hbox{$a\ge 0$;} \\
                                              0, & \hbox{otherwise.}
                                            \end{array}
                                          \right.
$$

\newpage

\section{Introduction}

\subsection {Ulyanov inequalities for moduli of smoothness}
%We begin our discussion with the one dimensional case, followed by the higher dimensional case.
%%In what follows $\Vert \cdot\Vert_p=\Vert
%%\cdot\Vert_{L_p(\T^1)}.$

The problem of estimating the moduli of smoothness of a function in
$L_q$
in
terms of its moduli of smoothness in
$L_p$
 has a long history. The starting point was a study of embeddings of the Lipschitz spaces %classes
 \begin{equation*}
 {\textnormal {Lip}} (\alpha,p)
 =
\Big\{f\in L_p(\T)\,:\, \|f(x+h)-f(x)\|_{L_p(\T)}=\mathcal{O}(h^\alpha) \Big\}
\end{equation*}
  (E.~Titchmarsh \cite{tit}, G. H. Hardy
and J. E. Littlewood \cite{hardy-l}, and S. M. Nikol’skii \cite{nikol-book}; see also the references in \cite{besov1, KOLYADA1, KOLYADA2, uly}).
  %Recent papers (\cite{trebels, haroske-triebel, simonot-tikhonov, Kolyada-PerezLazaro}).
%Sharp inequalities between moduli of smoothness in different $L_p$ metrics play an important role in study of
% embedding theorems for smooth function spaces.
  In particular, the classical
Hardy--Littlewood  embedding  (see~\cite{hardy-l})
 \begin{equation*}
 %\label{ul-0-}
 {\textnormal {Lip}} (\alpha,p)\hookrightarrow {\textnormal {Lip}} (\alpha-\theta,q),\qquad
\end{equation*}
where
$$1\le p<q<\infty, \quad\theta =\frac{1}{p}- \frac{1}{q},\qquad \theta< \alpha\le 1,
$$
%and the Lipschitz spaces are given by
%\begin{equation*}
% {\textnormal {Lip}} (\alpha,p) =
%\Big\{f\in L_p(\T)\,:\, \|f(x+h)-f(x)\|_{L_p(\T)}=\mathcal{O}(h^\alpha) \Big\},
%\end{equation*}
%proved by Hardy and Littlewood  \cite{hardy-l}
can easily be  obtained from the $(L_p, L_q)$ inequality for moduli of smoothness
proved by Ulyanov \cite{uly}
\begin{equation}\label{ul-1}
\omega_k(f, \delta)_{L_q(\T)} \lesssim \left( \int_{0}^{\delta}
\Bigl(t^{-\theta} \omega_k(f,t)_{L_p(\T)} \Bigl)^{q_1} \frac{dt}{t}
\right)^{\frac{1}{q_1}},
\end{equation}
where
$$1
\le p<q\le \infty, \qquad \theta =\frac{1}{p}- \frac{1}{q}.
%\qquad q_1=\left\{
%                               \begin{array}{ll}
%                                 q, & \qquad\hbox{$q<\infty$;} \\
%                                 1, & \qquad \hbox{$q=\infty$,}
 %                              \end{array}
%                             \right.
$$
Nowadays, inequality (\ref{ul-1}) is
known as Ulyanov's inequality.

Here and in what follows,  the $\, k$-th order modulus of smoothness
$\, \omega_{k}(f,\delta)_{L_p(\T^d)}$ is defined in the standard way by
\begin{equation}\label{def-mod}
\omega_{k}(f,\delta)_{L_p(\T^d)}:=\sup_{|h|\le \delta }
\| \Delta_h^k f\|_{L_p(\T^d)},%\quad
\end{equation}
where
$$
\Delta_h f(x)= f(x+h) -f(x),\quad \Delta_h^k =
\Delta_h \Delta_h^{k -1},\quad h\in \R^d,\quad d\ge 1.
$$
%On the one hand,
%Ulyanov inequality (\ref{ul-1}) provides
% sharp results  for functions from certain smooth spaces, e.g., if
%$\omega_{k}(f,\delta)_p=\mathcal{O}(\delta^\alpha),$ $0<\alpha<k$.
%On the other hand,   the best possible estimate of $\omega_k(f, \delta)_q$,< for the fixed $k$, for $C^\infty$ functions  which can be obtained using %(\ref{ul-1})
%is $\omega_{k}(f,\delta)_q=\mathcal{O}(\delta^{k-\theta}),$ which is substantially different from the best possible estimate %$\omega_{k}(f,\delta)_q=\mathcal{O}(\delta^{k})$.

%It was already remarked a long time ago
It was known long ago that
 the straightforward application of the Ulyanov inequality (\ref{ul-1}) for $C^\infty$-functions gives only the estimate $\omega_{k}(f,\delta)_{L_q(\T)}=\mathcal{O}(\delta^{k-\theta}),$ which differs substantially from the best possible estimate $\omega_{k}(f,\delta)_{L_q(\T)}=\mathcal{O}(\delta^{k})$.

To overcome this shortcoming, recently (see~\cite{ST, Treb}) the sharp Ulyanov inequality was established using the fractional moduli of smoothness.
We recall the definition of the modulus of smoothness  $\omega_{\a} (f,\delta)_{L_p(\T^d)}$  of fractional order $\a >0$ of
a~function $\, f \in L_p(\mathbb{T}^d)$, $\,0< p\le\infty$, which appeared for the first time in 1970's (see \cite[p. 788]{But}, \cite{tabeR}):
\begin{equation}\label{def-mod+}
\omega_{\a} (f,\delta)_{L_p(\T^d)}
 :=\sup_{|h|\le \delta }
 \left\|
\Delta_{h}^{\a} f \right\|_{L_p(\T^d)}, %\qquad
\end{equation}
where
\begin{equation}\label{def-mod++}
\Delta_{h}^{\a} f(x) =\sum\limits_{\nu=0}^\infty(-1)^{\nu}
\binom{\a}{\nu} f\,\big(x+(\a-\nu) h\big),\quad h\in \R^d,
\end{equation}
and
$
\binom{\a}{\nu}=\frac{\a (\a-1)\dots (\a-\nu+1)}{\nu!}$,\quad
$\binom{\a}{0}=1.
$
It is clear  that for integer $\a$ we deal with the classical definition (\ref{def-mod}).
Also, in the case of $0<p<1$, since
$$
\sum_{\nu=0}^\infty \Big|\binom{\a}{\nu}\Big|^p<\infty\quad\text{for}\quad \a\in\N\cup \big((1/p-1),\infty\big),
$$
%for
%$\a\in\N\cup \big((1/p-1),\infty\big)$
it is natural to assume that
$
\a\in\N\cup \big((1/p-1)_+,\infty\big)
$
while defining the fractional modulus of smoothness in $L_p$.
Note that the moduli of smoothness of fractional order
satisfy  all usual properties of the classical moduli of smoothness (see Section~\ref{sec4}).

The sharp Ulyanov inequality for the  fractional moduli of smoothness for
$1<p<q<\infty$ reads as follows:
\begin{equation}\label{ul-11}
\omega_\a(f, \delta)_{L_q(\T^d)} \lesssim \left( \int_{0}^{\delta}
\Bigl(t^{-\theta} \omega_{\a+\t}(f,t)_{L_p(\T^d)} \Bigl)^{q_1} \frac{dt}{t}
\right)^{\frac{1}{q_1}},\qquad \t=d\Big(\frac1p-\frac1q\Big).
\end{equation}
Moreover, this inequality
is sharp over the  class
$$
{\textnormal{Lip}} \,\big(\omega(\cdot),\alpha+\theta,p\big)
 \,
=
\, \Big\{f\in L_p(\T)\,:\,\omega_{k+\theta}(f,\delta)_p= \mathcal{O} \(\omega(\delta)\)\Big\}
$$
(see~\cite{ST}). More precisely,  for any function $\omega \in \Omega_{\alpha  + \theta}$
(i.e., such that $\omega(\delta)$ is non-decreasing
and $\delta^{-\alpha-\theta}\omega(\delta)$ is non-increasing), there exists a function
$$
f_0(x)=f_0(x,p,\omega) \in {\textnormal{Lip}} \,\big(\omega(\cdot),\alpha+\theta,p\big)
$$ such that
for any $q\in (p, \infty)$
and  for any $\delta>0$
\begin{equation*}
%\label{gg}
\omega_{\alpha}(f_0, \delta)_{L_q(\T)} \geq C \left( \int_{0}^{\delta}
\Bigl(t^{- \theta} \omega(t) \Bigl)^q
\frac{dt}{t} \right)^{\frac{1}{q}},
\end{equation*}
 where a constant  $C$ is independent of $\delta$ and $\omega$.

Nowadays the Ulyanov inequalities are used in approximation theory,
function spaces, and interpolation theory (see, e.g., \cite{besov1, Cohen, devore,
goga, gol, hatr, kol, KOLYADA1, kolyada1,
KOLYADA2, KOLYADA3, netr, PST2016, ST, Ti,Treb}). Very recently
\cite{goga}, sharp Ulyanov inequalities were shown between the
Lorentz–Zygmund spaces $\, L_{p,r}(\log L)^{\alpha -\gamma},\, \a,
\gamma >0,$ and $\, L_{q,s}(\log L)^\alpha $  over $\, \mathbb{T}^d$
in the case  $\, 1<p \le q<\infty,$ and the corresponding embedding
results were established.
In the case of quasi-normed Lebesgue spaces
 $L_p$, $0<p<1$, an Ulyanov inequality  was established only in its not-sharp form given by (\ref{ul-1}) (see~\cite{diti, S} and~\cite[Ch.~3.7]{Cohen}).

It is known (see~\cite{Ti}) that in the limiting cases $p=1, q<\infty$ or $1<p, q=\infty$, inequality~(\ref{ul-11}) does not hold  in general.
The best possible inequality is given by
\begin{equation}\label{ul-1+}
\omega_\a(f, \delta)_{L_q(\T)} \lesssim \left( \int_{0}^{\delta}
\Bigl(t^{-\theta} \big(\ln 2/t\big)^{\max(1/p',1/q)} \omega_{\a+\t}(f,t)_{L_p(\T)} \Bigl)^{q_1} \frac{dt}{t}
\right)^{\frac{1}{q_1}}, %\qquad 1<p<q<\infty,
\end{equation}
where $\big(\ln 2/t\big)^{\max(1/p',1/q)}$ cannot be replaced by $o\big(\ln 2/t\big)^{\max(1/p',1/q)}.$

Interestingly, in the case  $p=1$ and $q=\infty$, inequality  (\ref{ul-11}) is valid in the one-dimensional case (see~\cite{Ti}). Note that in this case $\t=1$ and one can use the classical (non fractional) moduli of smoothness.

Another improvement of the Ulyanov inequality (\ref{ul-1}) was suggested by Kolyada~\cite{kol}:
%The known Kolyada inequality reads as follows:
$$%\begin{equation}\label{ul-kol}
\delta^{\a-\theta} \left( \int_{\delta}^1 \bigl(t^{\theta-\a} \omega_\a (f,t)_{L_q(\T)} \bigl)^p \frac{dt}{t} \right)^{
\frac{1}{p} } \lesssim \left( \int_0^{\delta}
\bigl(t^{-\theta} \omega_\a (f,t)_{L_p(\T)} \bigl)^q
\frac{dt}{t} \right)^{ \frac{1}{q} },
$$%\end{equation}
where
 $1<p<q<\infty$ and $\theta ={1}/{p}- {1}/{q}$. Unlike the one-dimensional case, in the multidimensional case ($d\ge 2$, $\t=d(1/p-1/q)$), the corresponding inequality also holds for $p=1<q<\infty$.

Roughly speaking, the sharp Ulyanov inequality refines
 inequality (\ref{ul-1}) with respect to
 the right-hand side %of   (\ref{ul-1})
 while the Kolyada inequality refines (\ref{ul-1})  with respect to  the left-hand side. Note that the Kolyada inequality is not comparable to the sharp Ulyanov inequality.

In this paper, one of our main goals is to study sharp $(L_p, L_q)$  inequalities of Ulyanov-type with variable parameters for moduli of smoothness, generalized $K$-functionals, and their realizations.
More precisely, we  solve the long-standing problem (see~\cite{KOLYADA1, uly}) of
finding the sharp inequalities between
the moduli of smoothness
 $$\omega_\alpha(f,t)_{L_q(\T^d)}\quad \mbox{and}\quad \omega_\beta(f,t)_{L_p(\T^d)}$$ in the general case: $0<p<q\le \infty$ and $d\ge 1$.

One of our main contributions is the full description of sharp Ulyanov inequalities given by the following result.
\begin{theorem}\label{th1--}
    Let $f\in L_p(\T^d)$, $0<p<q\le\infty$,  $\a \in \N\cup ((1-1/q)_+,\infty)$, and $\a+\g\in \N\cup ((1/p-1)_+,\infty)$.
    Then, for any $\d \in (0,1)$, we have
    \begin{equation}\label{eqlemMM1----------}
        \w_\a(f,\d)_{L_q(\T^d)}\lesssim \frac{\w_{\a+\g}(f,\d)_{L_p(\T^d)}}{\d^\g}\s\(\frac 1\d\)+\(\int_0^\d   \(  \frac{\w_{\a+\g}(f,t)_{L_p(\T^d)}}{t^{d(\frac1p-\frac1q)}} \)^{q_1}\frac {dt}{t}\)^\frac 1{q_1},
    \end{equation}
where % $\s$ is defined by the following equalities

{\rm (1)} if $0<p\le 1$ and $p<q\le\infty$, then
$$
\s(t)
:=\left\{
         \begin{array}{ll}
           t^{d(\frac1p-1)}, & \hbox{$\g> d\(1-\frac1q\)_+$}; \\
           t^{d(\frac1p-1)}, & \hbox{$\g=d\(1-\frac1q\)_+\ge 1$, $d\ge 2$, and $\a+\g\in \N$}; \\
           t^{d(\frac1p-1)}\ln^\frac1{q_1} (t+1), & \hbox{$\g=d\(1-\frac1q\)_+\ge 1$, $d\ge 2$,  and $\a+\g\not\in \N$}; \\
           t^{d(\frac1p-1)}\ln^\frac1{q} (t+1), & \hbox{$0<\g=d\(1-\frac1q\)_+=1$ and $d=1$}; \\
           t^{d(\frac1p-1)}\ln^\frac1{q} (t+1), & \hbox{$0<\g=d\(1-\frac1q\)_+<1$}; \\
           t^{d(\frac1p-\frac1q)-\g}, & \hbox{$0< \g<d\(1-\frac1q\)_+$};\\
           t^{d(\frac1p-\frac1q)}, & \hbox{$\g=0$,}
         \end{array}
       \right.
$$

{\rm (2)} if $1<p\le q\le\infty$, then
$$
\s(t):=
\left\{
         \begin{array}{ll}
           1, & \hbox{$\g\ge d(\frac1p-\frac1q),\quad q<\infty$}; \\
           1, & \hbox{$\g> \frac dp,\quad q=\infty$}; \\
           \ln^\frac1{p'} (t+1), & \hbox{$\g=\frac dp,\quad q=\infty$}; \\
           t^{d(\frac1p-\frac1q)-\g}, & \hbox{$0\le \g<d(\frac1p-\frac1q)$}.\\
         \end{array}
       \right.
$$
\end{theorem}
%It is important to notice that if $0<p\le 1$  and $q=\infty$
%the logarithmic term in the definition of $\s$ is missing
%unlike the case $0<p<q<\infty$ (also cf. (\ref{ul-1+})).

We make several comments on this result.
The result is sharp in the sense that, for given $p$ and $q$ such that $p<q$, there exists a non-trivial function $f_0\in L_q(\T^d)$ for which the left-hand side of~\eqref{eqlemMM1----------} is equivalent to the right-hand side (see Section~\ref{sec8}).

We see that the critical values of the parameter $\gamma$ in the definition of $\s$ in the cases $p>1$ and $p\le 1$ are different: $\g= d(1/p-1/q)$
and
$\g= d(1-1/q)$, correspondingly.

It is important to mention that the sharp form of the Ulyanov inequalities, that is, an optimal choice of the function $\sigma$,
essentially depends on the method  we choose to measure the smoothness or, in other words, how we define a modulus of smoothness.
For simplicity, we explain this in the one dimensional case. On the one hand, it is well known that the classical modulus of smoothness is equivalent to the $K$-functional (for $1\le p\le\infty$) and the realization functional (for $0<p\le\infty$) defined by means of the classical Weyl derivative. In this case, the sharp Ulyanov inequality is given by Theorem~\ref{th1--}. On the other hand, if instead of the Weyl derivative in the definition of the realization concept, we consider the Riesz derivative, then the corresponding modulus of smoothness of order $\a>0$ is given by (see~\cite{RS3})
\begin{equation}\label{eqModulRiesz+++}
    \w_{\langle\a\rangle}(f,\d)_{L_p(\T)}=\sup_{0<h\le \d}\bigg\Vert \sum_{\nu\in \Z\setminus \{0\}} \(-\frac{\b_{|\nu|}(\a)}{\b_0(\a)}\)
f(\cdot+\nu h)-f(\cdot) \bigg\Vert_{L_p(\T)},
\end{equation}
where
% in~\cite{RS3} (see also~\cite{RS5}). For $\b>0$ and $f\in L_p(\T)$ following \cite{RS3} we define where
$$
\b_m(\a)=\sum_{j=m}^\infty (-1)^{j+1}2^{-2j}\binom{\a/2}{j}\binom{2j}{j-m},\quad m\in \Z_+.
$$

For the modulus~\eqref{eqModulRiesz+++}, a sharp form of Ulyanov inequality in critical  cases, which is of most interest, is different that the one given in Theorem~\ref{th1--}. In particular, for the limiting parameters  $0<p\le 1$ and $q=\infty$ (in this case $\g=1$), the corresponding inequality
reads as follows (see Section~\ref{sec9})
    \begin{equation*}
    %\label{eqth1.1-----}
        \w_{\langle \a\rangle}(f,\d)_{L_q(\T)}\lesssim \frac{\w_{\langle \a+1\rangle}(f,\d)_{L_p(\T)}}{\d}\ln\(\frac 1\d+1\)+
            \int_0^\d
            \bigg(
            \frac{\w_{\langle \a+1\rangle}(f,t)_{L_p(\T)}}{t^{\frac1p-\frac1q}}
            \bigg)\frac{dt}{t},
    \end{equation*}
where $\a+1\in (2\N)\cup (1/p-1,\infty)$.
Note that for $\a\in 2\N$, the modulus~\eqref{eqModulRiesz+++} coincides with the classical modulus of smoothness $\w_{\a}(f,\d)_{L_p(\T)}$.

%
%
%\bigskip
%\bigskip
%
%
%
%\begin{theorem}
%  Let $f\in L_p(\T)$, $0<p\le 1$ и $p<q\le\infty$, $\a>0$, and $\g\ge 0$.
% Let also
%$\a\in (2\N) \cup ((\frac1q-1)_+,\infty)$ and
%$\a+\g\in (2\N) \cup ((\frac1p-1)_+,\infty)$.
% Then for any $\d\in (0,1]$  we have
%    \begin{equation}\label{eqth1.1-----}
%        \w_{\langle \a\rangle}(f,\d)_q\lesssim \frac{\w_{\langle \a+\g\rangle}(f,\d)_p}{\d^{\g}}\s\(\frac 1\d\)+
%            \left(\int_0^\d
%            \bigg(
%            \frac{\w_{\langle \a+\g\rangle}(f,t)_p}{t^{d(\frac1p-\frac1q)}}
%            \bigg)^{q_1}\frac{dt}{t}\right)^{\frac1{q_1}},
%    \end{equation}
%where
%$$
%\s(t)
%:=\left\{
%         \begin{array}{ll}
%           t^{d(\frac1p-1)}, & \hbox{$\g> d\(1-\frac1q\)_+$;} \\
%           t^{d(\frac1p-1)}\ln^\frac1{q_1} (t+1), & \hbox{$0<\g=d\(1-\frac1q\)_+$;} \\
%           t^{d(\frac1p-\frac1q)-\g}, & \hbox{$0< \g<d\(1-\frac1q\)_+$;}\\
%           t^{d(\frac1p-\frac1q)}, & \hbox{$\g=0$.}
%         \end{array}
%       \right.
%$$
%\end{theorem}
%
%Note that in the case of $1<p<q< \infty$, the definition of $\s(t)$ in (\ref{eqth1.1----}) coincides with $\s(t)$ in (\ref{eqth1.1-----}), see Theorem \ref{th1} and Corollary \ref{Corth1dMod} below.

In the multidimensional case, new interesting effects occur  not only when $q=\infty$. % but in the general case as well.
%Recall that in the multidimensional case, if
%$\, f \in L_p(\mathbb{T}^d)$, $\,0< p < \infty$, $d\ge 2$, the modulus of smoothness is given by (\ref{def-mod+}) with $\|\cdot\|_p=\|\cdot\|_{L_p(\mathbb{T}^d)}$, where the $\alpha$-th difference is defined by  (\ref{def-mod++}) with $h\in \R^d$.
In particular, the sharp Ulyanov inequality~\eqref{ul-11} holds not only for $1<p<q<\infty$ but also for $p=1, q\le \infty$, provided that
$\a+d(1-1/q)\in \N$ and $d\ge 2$ (see Corollary~\ref{important corollary}).

%The sharp Ulyanov inequality for the case $1<p<q<\infty$, that is,
%$$\omega_\a(f, \delta)_q \lesssim \left( \int_{0}^{\delta}
%\Bigl(t^{-\theta} \omega_{\a+\t}(f,t)_p \Bigl)^{q} \frac{dt}{t}
%\right)^{\frac{1}{q}},\qquad \t=d\Big(\frac1p-\frac1q\Big).
%$$
% was recently proved in \cite{ST, Treb}. Moreover, Trebels \cite{Treb} showed that the latter inequality holds for $p=1$ provided that
%$\t\ge 1$, $\alpha+\t-1\in\N$, and $d\ge 2.$
%
%The  sharp Ulyanov inequalities in the general case $0<p<q\le \infty$ will be proved in Section \ref{sec8}.

Finally, we would like to mention two remarks. First, likewise to $(L_p,L_q)$ inequalities for the moduli of smoothness, several authors have studied
a similar problem for  the $K$-functionals, see, e.g., \cite{bsh, DL, devore}. Note also that Ulyanov-type estimates are sometimes called {\it weak type inequalities} (following \cite{devore}, where a general theory based on inequalities of this type was developed).
Second,  inequalities for moduli of smoothness of various order in the same metrics, i.e.,
$\omega_\alpha(f,t)_p$ and $\omega_\beta(f,t)_p$ for $0<p\le \infty$,
have been extensively investigated starting from 1927, when Marchaud published his  paper~\cite{Marchaud}. Sharp inequalities were obtained, in particular, in~\cite{boman, dai-d, ddt, DL, marchaud, PST2016, timan}.

\subsection{Specifics  of  the case $0<p<1$}\label{sec1.2p<1}

One of the main contribution of this paper is a thorough  study of the case $0<p<1$.  We recall that dealing with smoothness properties in $L_p$, $p<1$, differs dramatically from the case $p\ge 1$.

First, let us illustrate this by discussing how the classical Bernstein inequality for the fractional derivatives of trigonometric polynomials (see~\cite{Liz})
\begin{equation}\label{cl-ber}
    \Vert T_n^{(\a)}\Vert_{L_p(\T)}\lesssim n^\a \Vert T_n\Vert_{L_p(\T)},\quad \a>0,\quad 1\le p<\infty,
\end{equation}
changes in $L_p(\T)$, $0<p<1$.

Recall that the fractional derivative of a polynomial $T_n(x)=\sum_{|k|\le n} c_k e^{ikx}\in\mathcal{T}_n$ in the sense of Weyl is given by
$$
T_n^{(\a)}(x)=\sum_{|k|\le n} (ik)^{\a}c_k e^{ikx},\quad (ik)^\a=|k|^\a
e^{\frac{i\pi\a}{2}\sign k}.
$$
%Пусть также при $\g>0$
%$$
%T_n^{(-\g)}(x)=%\sum_{k\ne 0}
%\mathop{{\sum}'}_k\frac1{(ik)^{\g}}c_k e^{ikx}.
%$$

The Bernstein inequality in $L_p(\T)$, $0<p<1$, reads as follows (see~\cite{BL} and~\cite{K15, RS} for the multidimensional case).
\begin{proposition}\label{lemBLf}
{\it Let $0<p<1$. Then
%\bigskip
\begin{equation*}
\sup_{\Vert T_n\Vert_{L_p(\T)}\le 1}\Vert T_n^{(\a)}\Vert_{L_p(\T)}\asymp \left\{%
\begin{array}{ll}
n^{\a}, & \hbox{$\a\in\mathbb{Z}_+$ or $\a\not\in\mathbb{Z}_+$ and $\a>\frac 1p-1$;} \\
n^{\frac1p-1}, & \hbox{$\a\not\in\mathbb{Z}_+$ and $\a<\frac 1p-1$;}  \\
n^{\frac1p-1}\log^\frac 1p n, & \hbox{$\a=\frac 1p-1\not\in\mathbb{Z}_+$.} \\
\end{array}%
\right.
\end{equation*}
%($\asymp$ --- двусторонние неравенства с положительными константами,
%зависящими только от $p$ и $\a$).
}
\end{proposition}
%
%The multidimensional analogue in case $\a>d(\frac 1p-1)$ was proved in \cite{RS}.

It is worth mentioning  that if the polynomial $T_n$ is such that ${\rm spec\,}(T_n)\subset [0,n]$, that is, $c_k=0$ for $k<0$, then the Bernstein inequality changes
drastically:  for any  $0<p\le \infty$, one has
$$
\Vert T_n^{(\a)}\Vert_{L_p(\T)}\lesssim n^\a \Vert T_n\Vert_{L_p(\T)},\quad \a>0;
$$
see Belinskii's paper ~\cite{B}.

\smallskip

Second, let us discuss the smoothness of functions in the sense of behaviour of their moduli of smoothness in $L_p$.
As an example, we consider the splines of maximum smoothness.
Denote by $h_m$ the following function sequence:
$h_1(x)=\frac{\pi}{2} \sgn (\cos x),$  $x\in \T, $
and %if $\gamma\in\N$, $h_\gamma$ can be defined by induction:
$$
h_m(x)=\int_0^x \Big(h_{m-1}(t)-\frac1{2\pi}\int_\T {h}_{m-1}(z)dz \Big)dt, \qquad m=2, 3,\dots.
$$
Note that up to a constant
$$
h_m(x) \sim \sum_{k\in \Z} \frac{ e^{i(2k+1)x}}{(2k+1)^m}.$$
It is well known  (see~\cite[p. 359]{DL}) that, for $0<p\le\infty$ and $l\in \N$, we have
\begin{equation}\label{eqModLp}
    \w_l(h_{m},\d)_{L_p(\T)}\asymp \left\{
                        \begin{array}{ll}
                          \d^{m-1+\frac1p}, & \hbox{$l\ge m$;} \\
                          \d^l, & \hbox{$l<m$.}
                        \end{array}
                      \right.
\end{equation}

Let us compare the classical and sharp Ulyanov inequalities for such functions  in the case   $0<p<1$.
First, we have  $\w_{k}(h_{r+k},\d)_{L_p(\T)}\asymp \d^{k}$, and the standard (not-sharp) Ulyanov inequality (\ref{ul-1}) implies  $\w_{k}(h_{r+k},\d)_{L_q(\T)} \lesssim \d^{k-(1/p-1/q)}$.
At the same time, the sharp Ulyanov inequality (\ref{eqlemMM1----------}) with $\g=r\in \N$ yields
\begin{equation*}
  \begin{split}
            \w_k(h_{r+k},\d)_{L_q(\T)}
       &\lesssim
       \frac{\w_{k+r}(h_{r+k},\d)_{L_p(\T)}}{\d^{r}} \d^{1-\frac1p}\\
       &\qquad\qquad\qquad+
        \left(\int_0^\d\bigg(\frac{\w_{r+k}(h_{r+k},t)_{L_p(\T)}}{t^{\frac1p-\frac1q}}\bigg)^q\frac{dt}{t}\right)^{\frac1q}
 \asymp %& \d^{k-\frac{1}{p}+1} \s\(\frac1\d\) \asymp
 \d^{k},
   \end{split}
\end{equation*}
which is the best possible estimate since $\w_k(h_{r+k},\d)_{L_q(\T)} \asymp \d^{k}$.
Therefore, the sharp Ulyanov inequality implies better estimates for $L_p$-smooth functions when $0<p<1$.

It is interesting to note that, for any $f\in C^\infty (\T)$ and for all $0<p \le \infty$, one has
$\omega_\gamma(f,\delta)_{L_p(\T)}\asymp \delta^\gamma$. In other words, taking into account~\eqref{eqModLp}, we see that spline functions in the case $0<p<1$ are smoother than
 $C^\infty$-functions
in the sense of the behavior of their moduli of smoothness.
This phenomena disappears if we deal with absolutely continuously functions. More precisely,
if $f^{(\a-1)}\in AC(\T)$, $\a\in \N$, and
$\w_\a(f,\d)_p=o(\d^\a)$ as $\d\to 0$,
then $f\equiv \const$ (see \cite{SKO} and Proposition~\ref{proposition}).

\subsection {Moduli of smoothness, $K$-functionals, and their realizations}
% One of the key concepts of this paper is the notion of moduli of
%smoothness.

The key approach to obtain the sharp Ulyanov inequalities is to use the realizations of the $K$-functionals.
For simplicity, we start with the one-dimensional case.
It is  known (see~\cite[p.~341]{BeSh}) that if $1\le p \le\infty$, then the classical
modulus of smoothness is equivalent to the $K$-functional given by
$$
 {K}_{\a}(f,\d)_{L_p(\T)}:=\inf_{g^{(\a)}\in L_p(\T)} \(\|f-g\|_{L_p(\T)}+ \delta^\alpha \| g^{(\a)}\|_{L_p(\T)}\),
%\Vert f \Vert_{\dot W_p^{\a}}=
%\Vert f^{(\a)} \Vert_{p}.
%\sum_{|\nu|=\a}\Vert D^\nu f\Vert_p.
$$
that is,
\begin{equation}\label{eq.th6.0+}
\omega_\a(f,\delta)_{L_p(\T)}%\widetilde{K}_\psi(f,\delta)_p
\asymp {K}_{\a}(f,\d)_{L_p(\T)},\quad 1\le p \le\infty,\quad \a>0.
\end{equation}

This equivalence fails for  $0<p<1$ since in this case $K_\a(f,\d)_{L_p(\T)}\equiv 0$ (see \cite{DHI}). A
suitable substitute for the $K$-functional for $p<1$ is the
realization concept given by
\begin{equation}\label{eqKf2}
\mathcal{R}_{\a}(f,\d)_{L_p(\T)}=\inf_{T\in\mathcal{T}_{[1/\d]}}\(\Vert
f-T\Vert_{L_p(\T)}+\d^{\a}\Vert T^{(\a)}\Vert_{L_p(\T)}\).
\end{equation}

The next result was proved in \cite{DHI} for  $\a\in\N$. For any positive $\a$,
the proof follows from Theorem \ref{lemNSB} below. %  in~\cite{K11} .

\begin{proposition}\label{lemKfp<1} {\sc(Realization result)}
%\begin{thmb}\label{lemKfp<1}
Let $f\in L_p(\T)$, $0<p\le \infty$, and $\a\in\N\cup \big((1/p-1)_+,\infty\big)$. Then, for any $\d\in (0,1)$, we have
\begin{equation}\label{eq.th6.0}
{\omega}_\a(f,\d)_{L_p(\T)}\asymp \mathcal{R}_\a(f,\d)_{L_p(\T)}.
\end{equation}
\end{proposition}

In the multidimensional case,
an analogue of equivalence (\ref{eq.th6.0+}) holds for the classical moduli of smoothness of integer order (see~\cite[p.~341]{BeSh}). More precisely,  if $1\le p \le\infty$
and  $\a\in\N$, we have for $f\in L_p(\T^d)$
\begin{equation*}
%\label{equivalence++}
\omega_\a(f,\delta)_{L_p(\T^d)}%\widetilde{K}_\psi(f,\delta)_p
\asymp \inf_{g\in \dot W_p^{\a}} \big(\|f-g\|_{L_p(\T^d)}+ \delta^\alpha \|
g\|_{ {\dot W_p^{\a}(\T^d)}}\big),
\end{equation*}
where
$$
\Vert f \Vert_{\dot W_p^{\a}(\T^d)}=\sum_{|\nu|_1=\a}\Vert D^\nu f\Vert_{L_p(\T^d)},\quad {D}^\nu={D}^{\nu_1,\dots, \nu_d}=\frac{\partial^{\nu_1+\dots+\nu_d}}{\partial x_1^{\nu_1}
\dots \partial x_d^{\nu_d}}.
$$
%and$$
%{D}^\nu={D}^{\nu_1,\dots, \nu_d}=\frac{\partial^{\nu_1+\dots+\nu_d}}{\partial x_1^{\nu_1}
%\dots \partial x_d^{\nu_d}}.
%$$

For the fractional moduli of smoothness ($\alpha>0$), Wilmes \cite{Wil} proved that,  for $1<p<\infty$,
\begin{equation}\label{RealKpge1-}
\omega_\a(f,\delta)_{L_p(\T^d)}%\widetilde{K}_\psi(f,\delta)_p
\asymp \inf_{(-\Delta)^{\a/2}g \in L_p} \big(\|f-g\|_{L_p(\T^d)}+ \delta^\alpha \|(-\Delta)^{\a/2}g\|_{L_p(\T^d)}\big).
\end{equation}
The corresponding realization result (see, e.g., \cite{goga}) is given by
\begin{equation*}
%\label{RealKpge1}
  \omega_\a(f,\delta)_{L_p(\T^d)}%\widetilde{K}_\psi(f,\delta)_p
\asymp \inf_{T\in\mathcal{T}_{[1/\d]}} \big(\|f-T\|_{L_p(\T^d)}+ \delta^\alpha \|(-\Delta)^{\a/2}T\|_{L_p(\T^d)}\big),
\end{equation*}
where
\begin{equation}\label{eqLaplacian}
  (-\D)^{\a/2}T(x)=\sum_{|k|_\infty\le [1/\d]} |k|^\a c_k e^{i(k,x)}.
\end{equation}
%$$
%(-\D)^{\a/2}T(t)=\sum_{|k|\le n} |k|^\a c_k e^{i(k,t)}.
%$$

We will show (see Theorem \ref{lemKfp<1d} below) that for any $0<p\le \infty$ and for any  $\a\in \N\cup \((1/p-1)_+,\infty\)$ the following realization result holds

%Let $f\in L_p(\T^d)$, $d\ge 1$, $0<p\le \infty$, and $\b\in\N\cup \big((1/p-1)_+,\infty\big)$. Then

\begin{equation}\label{eq.th6.0d++++}
{\omega}_\a(f,\d)_{L_p(\T^d)}\asymp
\inf_{T\in\mathcal{T}_{[1/\d]}}\left(\Vert
f-T\Vert_{L_p(\T^d)}+\d^{\a}\sup_{\xi\in \R^d,\,|\xi|=1}\bigg\Vert \(\frac{\partial}{\partial\xi}\)^{\a} T\bigg\Vert_{L_p(\T^d)}\right),
%{\widetilde{K}}_\b(f,\d)_{p},\quad \d\in (0,1),
\end{equation}
where $\(\frac{\partial}{\partial\xi}\)^{\a} T$ is
 the directional derivative of order $\a$, that is, % is defined by
$$
\(\frac{\partial}{\partial\xi}\)^{\a} T(x)=\sum_{|k|_\infty\le [1/\d]} {\(i(k,\xi)\)^\a} c_k e^{i(k,x)}.
$$
Equivalence (\ref{eq.th6.0d++++}) is  a crucial relation to obtain sharp Ulyanov-type inequalities in the multidimensional case for all $0<p\le q\le \infty$.

%The corresponding realization concepts in the multidimensional case will be given  in Section hhh below.
%{\bf multidimensional case it is known only in some cases}
%In the multidimensional case
%we have the following analogue of equivalence (\ref{equivalence++}):

\subsection{Main tool: Hardy--Littlewood--Nikol'skii polynomial inequalities}\label{subsec1.4}

The Hardy--Littlewood inequality for  fractional integrals states that
 for any $f\in L_p(\T^d)$ such that $\int_{\T^d} f(x)dx=0$, we have

% Let $1<p<q<\infty$ and $\t=d(\frac1p-\frac1q)$. Then

\begin{equation}\label{HLfrac}
  \Vert  (-\D)^{-\t/2} f \Vert_{L_q(\T^d)} \lesssim \Vert f \Vert_{L_p(\T^d)}, \quad
 1<p<q<\infty,\quad \t=d\(\frac1p-\frac1q\).
\end{equation}
Remark that~\eqref{HLfrac} does not hold outside the range $1<p<q<\infty$.
On the other hand, the following Nikol'skii inequality (sometimes called the reverse H\"{o}lder inequality)
holds
for any
trigonometric polynomial $T_n$ of degree at most $n$:
\begin{equation}\label{nik}
  \Vert
    T_n \Vert_{L_q(\T^d)} \lesssim  n^\theta \Vert T_n \Vert_{L_p(\T^d)}, \quad
 0<p<q\le \infty,\quad \t=d\(\frac1p-\frac1q\).
\end{equation}
Both inequalities are, in a way, the  limiting cases of the following general estimate:
\begin{equation}\label{nik-vsp}
{\Vert (-\D)^{-\g/2}
T_n\Vert_{L_q(\T^d)}}
\lesssim \s(n) {\Vert
T_n\Vert_{L_p(\T^d)}},
\qquad
 0<p<q\le \infty.
\end{equation}

The key step in our proof of Ulyanov-type inequalities, which is of its own interest, is to obtain a sharp asymptotic behavior of $\s(n)=\s(n, p, q, \g, d)$.
 %This problem is of great interest by itself.
Our result reads as follows (see Corollary \ref{corr++}).
%Let $0< p<q\le \infty$, $\a>0$, and $\g\ge 0$.
%We have

{\it
\textnormal{(1)} \quad If $0<p\le 1$ and $p<q\le\infty$, then
$$
\sup_{T_n\in \mathcal{T}'_{n}
}\frac
{\Vert (-\D)^{-\g/2}
T_n\Vert_{L_q(\T^d)}}
{\Vert T_n\Vert_{L_p(\T^d)}}
\asymp
\left\{
         \begin{array}{ll}
           n^{d(\frac1p-1)}, & \hbox{$\g>d\(1-\frac1q\)_+$}; \\
           n^{d(\frac1p-1)}\ln^\frac1{q_1} (t+1), & \hbox{$0<\g=d\(1-\frac1q\)_+$}; \\
           n^{d(\frac1p-\frac1q)-\g}, & \hbox{$0< \g<d\(1-\frac1q\)_+$};\\
           n^{d(\frac1p-\frac1q)}, & \hbox{$\g=0$.}\\

         \end{array}
       \right.
$$

\textnormal{(2)} \quad If $1<p<q\le \infty$, then
$$
\sup_{T_n\in \mathcal{T}'_{n}
}\frac
{\Vert (-\D)^{-\g/2}
T_n\Vert_{L_q(\T^d)}}
{\Vert T_n\Vert_{L_p(\T^d)}}
\asymp
\left\{
         \begin{array}{ll}
           1, & \hbox{$\g\ge d(\frac1p-\frac1q),\quad q<\infty$}; \\
           1, & \hbox{$\g> \frac dp,\quad q=\infty$}; \\
           \ln^\frac1{p'} (n+1), & \hbox{$\g=\frac dp,\quad q=\infty$}; \\
           n^{d(\frac1p-\frac1q)-\g}, & \hbox{$0\le \g<d(\frac1p-\frac1q)$}.\\
         \end{array}
       \right.
$$
}
In particular, this implies the following analogue of the Hardy--Littlewood inequality (\ref{HLfrac}) for the limiting cases:
 % when $p=1$ or/and $q=\infty$:
 {\it
 If $1\le p<q\le\infty$, $f\in L_p(\T^d)$, and $\int_{\T^d} f(x)dx=0$, then
$$
\Vert (-\D)^{-\g/2}f\Vert_{L_q(\T^d)}\lesssim \Vert f\Vert_{L_p(\T^d)}
$$
holds provided $\g>d(1/p-1/q)$, $p=1$ or/and $q=\infty$.
}

In fact, we will
study a more general problem than (\ref{nik-vsp}) by considering the Weyl-type derivatives defined by homogeneous functions in place of powers of Laplacians.

\subsection {Main goals and work organization} In this paper, our main goals are the following:

\begin{enumerate}
  \item[1.]  To prove Hardy--Littlewood--Nikol'skii polynomial inequalities for the generalized Weyl  derivatives with the whole range of parameters $0<p<q\le \infty$.
  \item[2.]
      To obtain sharp $(L_p, L_q)$  inequalities of Ulyanov-type with variable parameters for moduli of smoothness, generalized $K$-functionals, and their realizations.
  \item[3.]
 To apply Ulyanov inequalities to obtain new optimal embedding theorems of Lipschitz-type and Besov spaces.
   \end{enumerate}

%\begin{enumerate}
%  \item To prove Hardy--Littlewood--Nikol'skii polynomial inequalities ??? for different derivatives ??? the whole range of parameters $0<p<q\le \infty$.
%  \item To obtain sharp $(L_p, L_q)$  inequalities of Ulyanov-type with variable parameters for moduli of smoothness, generalized $K$-functionals, and their realizations.
%  \item To study
%\end{enumerate}

The paper is organized as follows.
In Section~\ref{sec2}, we collect auxiliary results on %polynomial inequalities of Nikol'skii--Stechkin--Boas type,
  embedding theorems for Besov and Triebel--Lizorkin spaces, Fourier multiplier theorems, and Hardy--Littlewood theorems on Fourier series with monotone coefficients. Moreover, we obtain a result which connects the $L_p$-norms of the Fourier transform of a continuous function and $L_p$-norm of polynomial generated by this function.

Section \ref{sec3} deals with Nikol'skii--Stechkin--Boas--type inequalities for the relationship between  norms of derivatives and differences of trigonometric polynomials written as follows
$$
\sup_{\xi\in \R^d,\,|\xi|=1}\bigg\Vert \(\frac{\partial}{\partial\xi}\)^{\a} T_n\bigg\Vert_{L_p(\T^d)}
\asymp \d^{-\a}\w_\a (T_n,\d)_{L_p(\T^d)},\quad T_n\in \mathcal{T}_n,
$$
where $0<\d\le \pi/n$, $0<p\le\infty$, and $\a>0$.
In particular, this relation plays a crucial role in obtaining the equivalence between moduli of smoothness  and realization concepts, see (\ref{eq.th6.0d++++}).

Basic properties of fractional multidimensional moduli of smoothness are given in Section~\ref{sec4}.

 %Section 3 is devoted to the study of Hardy--Littlewood--Nikol'skii inequality. This is the key estimate to prove sharp Ulyanov inequalities. In it is simple form, in the one dimension, the  problem  can be written as follows: to find
%$$\sup_{
%T\in\mathcal{T}_{n}}\frac{\|T^{(\a)}\|_{p}}{\|T^{(\a+\g)}\|_{q}}
% =: \eta(n,\a,\g,p,q)
%$$
%for all $0<p\le q\le \infty$.
%Note the if $\g=\frac1p-\frac1q$, then $\eta(n,\a,\g,p,q)\asymp 1$ for $1<p<q<\infty$ by the Hardy--Littlewood fractional integration theorem.
%Note the if $\g=0$, then $\eta(n,\a,\g,p,q)\asymp n^{\frac1p-\frac1q}$ for $0<p\le q\le \infty$ by the Nikol'skii inequality.
In Section  \ref{sec5}, we obtain sharp Hardy--Littlewood--Nikol'skii inequalities for the generalized Weyl derivatives of trigonometric polynomials in the  multidimensional case.
In more detail, we find the asymptotic behavior
of
$$
\sup_{{T_n\in \mathcal{T}'_{n}
}}
\frac{\Vert \mathcal{D}(\psi)T_n\Vert_{L_q(\T^d)}}{\Vert
\mathcal{D}(\vp)T_n\Vert_{L_p(\T^d)}},\qquad 0<p<q\le \infty,
$$
where $\mathcal{D}(\eta)$ is  the generalized Weyl-type differentiation operators, i.e.,
$\mathcal{D}(\eta)\,:\, \sum_{\nu}  c_\nu e^{i(\nu,x)}\to {\sum_{\nu\neq 0}} \eta(\nu) c_\nu e^{i(\nu,x)}$, and
$\eta$ is a homogeneous function of a certain degree.

In Section~\ref{sec6}, we study properties of  the generalized $K$-functionals and their realizations and
prove a general form of  sharp Ulyanov inequalities.

Section~\ref{sec7} provides an explicit formula of the sharp Ulyanov inequality for the realization concepts.

Section~\ref{sec8} deals with the sharp Ulyanov inequality for the  moduli of smoothness in one-dimensional and multi-dimensional cases, where we use results of Sections  \ref{sec6} and~\ref{sec7}.

Section ~\ref{sec9} is  devoted to the proof of sharp Ulyanov inequalities for moduli of smoothness and $K$-functional related to the Riesz derivatives.

Section~\ref{sec10} deals with the sharp Ulyanov inequalities, which are obtained using  Marchaud-type inequalities.

In Section~\ref{sec11},  we treat the Ulyanov and Kolyada-type  inequalities in the real and analytic Hardy spaces.

In Section  \ref{sec12}, we prove the sharp Ulyanov inequalities involving derivatives. In particular, we improve the known estimate
 (see~\cite{diti07})
$$
%\begin{eqnarray}\label{th-deri3}\qquad\qquad
 \omega_{\alpha}(f^{(r)},\delta)_{L_q(\T)}
 \lesssim
 \( \int_0^{\delta} \( t^{- r- (\frac1p-\frac1q)} \omega_{r +
\alpha}(f,t)_{L_p(\T)}\)^{q_1} \frac{dt}{t}
\)^{\frac{1}{q_1}}, \quad  0<p<q\le \infty, %\theta = \min(2,p)
$$%\end{eqnarray}
where $1/p-1/q<\alpha$.

Finally, in Section  \ref{sec13}, we discuss new
embedding theorems for smooth function spaces, which follow from inequalities for moduli of smoothness.

\bigskip
\bigskip

\bigskip
{\bf{Acknowledgements.}}
The first author was supported by the project AFFMA that has received funding from the European Union's Horizon 2020 research and innovation programme under the Marie Sklodowska-Curie grant agreement No~704030.
 The second author
  was partially
supported by MTM 2014-59174-P, 2014 SGR 289, % by the Alexander von Humboldt Foundation,
and by the CERCA Programme of the Generalitat de Catalunya.

\bigskip

\newpage

\section{Auxiliary results}\label{sec2}
\subsection {Besov and Triebel-Lizorkin spaces and their embeddings}

%First we recall some well-known definitions and results.
Let us recall the definition of the Besov space $B_{p,q}^s(\R^d)$ (see, e.g.,~\cite{TribF}). We  consider the Schwartz function
$\vp\in \mathscr{S}(\R^d)$ such that $\supp\vp\subset
\{\xi\in\R^d\,:\,1/2\le |\xi|\le 2\}$, $\vp(\xi)>0$ for $1/2< |\xi|<2$ and
\begin{equation*}
%\label{VV.eqRazbEd}
    \sum_{k=-\infty}^\infty \vp(2^{-k}\xi)=1\quad\text{if}\quad \xi\neq0.
\end{equation*}
We also introduce the functions $\vp_k$ and $\psi$ by means of the relations
$$
\mathscr{F}\vp_k(\xi)=\vp(2^{-k}\xi)
$$
and
$$
\mathscr{F}\psi(\xi)=1-\sum_{k=1}^\infty \vp(2^{-k}\xi).
$$

We  say that $f\in\mathscr{S}'(\R^d)$
belongs to the (non-homogeneous)  Besov space $B_{p,q}^s(\R^d)$, $s\in\R$, $0<p,q\le\infty$, if
$$
\Vert f\Vert_{B_{p,q}^s(\R^d)}=\Vert
\psi*f\Vert_{L_p(\R^d)}+\bigg(\sum_{k=1}^\infty 2^{sqk} \Vert
\vp_k* f\Vert_{L_p(\R^d)}^q\bigg)^\frac 1q<\infty
$$
(with the usual modification in the case $q=\infty$).
%, that is, $\Vert
%f\Vert_{B_{p,\infty}^s}=\Vert \psi* f\Vert_{L_p}+\sup_{k\ge 1}
%2^{sk}\Vert \vp_k*f\Vert_{L_p}$).

We will also deal with the Triebel-Lizorkin spaces. We  say that $f\in\mathscr{S}'(\R^d)$,
belongs to the (non-homogeneous) Triebel-Lizorkin space $F_{p,q}^s(\R^d)$, $s\in\R$, $0<p,q\le\infty$, if
$$
\Vert f\Vert_{F_{p,q}^s(\R^d)}=\Vert
\psi*f\Vert_{L_p(\R^d)}+\left\Vert\bigg(\sum_{k=1}^\infty
2^{sqk} |\vp_k* f(\cdot)|^q\bigg)^\frac 1q
\right\Vert_{L_p(\R^d)}<\infty
$$
(with the usual modification in the case $q=\infty$).

In order to define the homogeneous Besov spaces and Triebel-Lizorkin spaces, recall that

$$
\dot{\mathscr{S}}(\R^d)=\left\{\vp\in{\mathscr{S}}(\R^d)\,:\, (D^\nu
\widehat{\vp})(0)=0\,\,\,\text{for\ all}\,\,\,\nu\in\N^d\cup\{0\}\right\},
$$
where $\dot{\mathscr{S}}'(\R^d)$ is the space of all continuous
functionals on $\dot{\mathscr{S}}(\R^d)$. We  say that $f\in
\dot{\mathscr{S}}'(\R^d)$ belongs to the homogeneous Besov space
$\dot B_{p,q}^s(\R^d)$ if
$$
\Vert f\Vert_{\dot B_{p,q}^s(\R^d)}=\bigg(\sum_{k=-\infty}^\infty
2^{sqk} \Vert \vp_k* f\Vert_{L_p(\R^d)}^q\bigg)^\frac 1q<\infty.
$$
Similarly, $f\in \dot{\mathscr{S}}'(\R^d)$ belongs to the
homogeneous Triebel-Lizorkin space $\dot F_{p,q}^s(\R^d)$ if
$$
\Vert f\Vert_{\dot F_{p,q}^s(\R^d)}
=
\left\Vert\bigg(\sum_{k=-\infty}^\infty 2^{sqk}
|\vp_k* f(\cdot)|^q\bigg)^\frac 1q \right\Vert_{L_p(\R^d)}<\infty.
$$

Recall also the following relations between the real Hardy spaces and the
homogeneous Triebel-Lizorkin spaces. We have  $H_p(\R^d)=\dot F_{p,2}^0(\R^d)$ and $\Vert (-\D)^{s/2}
f\Vert_{H_p(\R^d)}=\Vert f\Vert_{\dot F_{p,2}^s(\R^d)}$ for
$0<p<\infty$ (see~\cite[Ch.~5]{TribF}) .

We will use the following embeddings.

\begin{lemma}\label{lemmapluss}
For $0<p<q<\infty$ and $\t=d(1/p-1/q)<\a$, we have
  \begin{equation}\label{eqKolya1}
    B_{p,q}^\t(\R^d)\hookrightarrow F_{q,2}^0(\R^d)\quad \text{and}\quad \Vert f\Vert_{H_q(\R^d)} \lesssim \Vert f\Vert_{\dot B_{p,q}^\t(\R^d)},
\end{equation}
and
\begin{equation*}
%\label{eqKolya2}
    F_{p,2}^\a(\R^d)\hookrightarrow B_{q,p}^{\a-\t}(\R^d)\quad \text{and}\quad \Vert f\Vert_{\dot B_{q,p}^{\a-\t}(\R^d)}
\lesssim \Vert f\Vert_{\dot F_{p,2}^\a(\R^d)}.
\end{equation*}
\end{lemma}

\begin{proof}
The proofs of these embeddings one can find in~\cite{Fra} and~\cite{Ja} (see also, e.g., \cite{TribF}).
Let us only  note that~\eqref{eqKolya1} was proved in~\cite{Fra} only for non-homogeneous spaces. The proof for homogeneous spaces can be easily obtained by using triangle inequality, the Nikol'skii-type inequality, and the fact that $\Vert \vp_k*f\Vert_{H_q(\R^d)}\asymp  \Vert \vp_k*f\Vert_{L_q(\R^d)}$ with the constants in $\asymp$ independent of $f$ and $k$ (see, e.g.,~\cite[p.~239]{Pe}).
\end{proof}

The following result was proved in \cite{K14} for the case $p=1$.
The case $0<p<2$ can be treated similarly (see \cite{K12}).
\begin{lemma}\label{lemBes}
%\label{th4}
{\it Let $0<p<2$, $1< q, r<\infty$, $s>d (1/p-1+1/r)$, and $f\in \dot{\mathscr{S}}'(\R^d)$. Let us suppose that $q=r=2$ or
\begin{equation*}
%\label{eq.th.4.-1}
\frac {1-\theta}{q}+\frac {\theta} {r}>\frac 12\,,\quad
\theta=\frac{d}{s}\Big(\frac 1p-\frac12\Big).
\end{equation*}
If, in addition, $f\in L_q(\Bbb R^d)$ and $(-\D)^{s/2} f\in L_r(\Bbb R^d)$, then $f\in \dot
B_{2,p}^{d(\frac1p-\frac12)}(\Bbb R^d)$.}
\end{lemma}

The next result easily follows from the definition of the Besov spaces (see, e.g.,
\cite[3.4.1, p.~206]{TribF}; mind the typo  in \cite[5.3.1, Remark~4, p.~239]{TribF}.
\begin{lemma}\label{besov}
We have
$$
\|f(\lambda \cdot)\|_{\dot{B}_{p,q}^{s}(\R^d)}\lesssim \lambda^{s-\frac{d}{p}}
\|f( \cdot)\|_{\dot{B}_{p,q}^{s}(\R^d)}.
$$
\end{lemma}

\subsection{Fourier multipliers} In this paper, we will use Fourier multipliers only for periodic functions.
Recall that a function $m: \R^d\to \C$ is called a Fourier multiplier (we will write $ m\in M_p$) if there exists a constant $C=C(m,p,d)$
such that for any $\e> 0 $ and $ f\in L_p (\T^d)$, $1\le p\le \infty$, one has
$$
\bigg\Vert \sum_{k\in \Z^d} m (\e k) \widehat{f}_k e^{i(k,x)} \bigg\Vert_p\le C\bigg\Vert \sum_{k\in \Z^d} \widehat{f}_k e^{i(k,x)}\bigg\Vert_p,
$$
where
$$
\widehat{f}_k=\frac1{(2\pi)^d}\int_{\T^d} f(x)e^{-i(k,x)}dx,\quad k\in \Z^d,
$$
are the Fourier coefficients of the function $f$.

We need the following properties of Fourier multipliers, which can be found, for example, in~\cite{Mikhlin}, \cite[p. 119]{Pe} or \cite[Ch. IV, 3.2]{SW}.

\begin{lemma}\label{lemMult}
\textnormal{(i) (Mikhlin--H\"{o}rmander's theorem)} Let $1 <p <\infty $. If
$$
\sup_{\xi\in\R^d} |\xi|^{|\nu|_1} | {D}^\nu m (\xi)|<\infty \quad\text{for all}\quad 0\le |\nu|_1\le [n/2]+1,\quad \nu\in \Z_+^d,
$$
then $ m\in M_p $.

\textnormal{(ii) (Peetre's theorem)} If  $m\in \dot B_{2,1}^{d/2}(\R^d)$, then $m\in M_1$ or, equivalently, $ m\in M_\infty$.

\end{lemma}

\subsection {Estimates of $L_p$-norm of an integrable function in terms of its Fourier coefficients}

The following result allows us to estimate a (quasi-)norm of an integrable function
via its Fourier coefficients. In the case $p=1$, this lemma was proved
in~\cite{Trigub15}.

\begin{lemma}\label{lemFLpBes}
Let $0<p\le 1$, $\Phi \in L_1(\T^d)$, and
    $$
     \Phi(x)\sim \sum_{k\in\Z^d} \vp(k)e^{i(k,x)}.
    $$
Let  $\vp_c$ be  such that
 $\vp_c(k)=\vp(k)$ for any $k\in \Z^d$ and $\vp_c\in
\dot{B}_{2,p}^{d(\frac1p-\frac12)} \cap \dot{B}_{2,1}^{\frac
d2}(\R^d)$.
    Then
    \begin{equation*}
    %\label{eqlemFLpBes1}
        \Vert \Phi \Vert_{L_p(\T^d)}\le C \Vert \vp_c
        \Vert_{\dot{B}_{2,p}^{d(\frac1p-\frac12)}(\R^d)},
    \end{equation*}
where  a constant $C$ is independent of $\Phi$.
\end{lemma}

\begin{proof}For any $\vp_c\in \dot{B}_{2,1}^{\frac
d2}(\R^d)$, there exists $g\in L_1(\R^d)$ such that
\begin{equation*}
    \vp_c(x)=%\frac{1}
    {(2\pi)^{-d/2}}\int_{\R^d} g(t)e^{i(t,x)}dt
\end{equation*}
(see, for example, \cite[p. 119]{Pe}).
For $k\in \Z^d$, we get
\begin{equation*}
    \begin{split}
\vp(k)=\vp_c(k)&= {(2\pi)^{-d/2}}
\int_{\R^d} g(t)e^{-i(k,t)}dt\\
&= {(2\pi)^{-d/2}}\sum_{\mu\in\Z^d}\int_{\T^d+2\pi \mu} g(t)e^{-i(k,t)}dt
\\
&= {(2\pi)^{-d/2}}\sum_{\mu\in\Z^d} \int_{\T^d}g(t+2\pi \mu)e^{-i(k,t)}dt\\
&= {(2\pi)^{-d/2}}\int_{\T^d}
g_T(t)e^{-i(k,t)}dt,
     \end{split}
\end{equation*}
where
$$
g_T(t)\sim \sum_{\mu\in\Z^d}g(t+2\pi \mu),\qquad t\in \T^d.
$$
Using Beppo Levi's theorem, we have $g_T\in L_1(\T^d)$. Since
$$
g_T(t)\sim \sum_{k\in \Z^d}\vp(k)e^{i(k,t)},
$$
then %for any continuation of то для любого продолжения
$$
\Vert\Phi\Vert_{L_p(\T^d)} =\Vert g_T\Vert_{L_p(\T^d)}\le \Vert
g\Vert_{L_p(\R^d)}\le C \Vert
\vp_c\Vert_{\dot{B}_{2,p}^{d(\frac1p-\frac12)}(\R^d)},
$$
where in the last estimate  we have used the Bernstein-type inequality
$$
\Vert f\Vert_{L_p(\R^d)}\le C\Vert
\widehat{f}\Vert_{\dot{B}_{2,p}^{d(\frac1p-\frac12)}(\R^d)}
$$
(see, e.g., \cite[p. 119]{Pe}).
\end{proof}

\subsection {Estimates of $L_p$-norms of polynomials related to Lebesgue constants}
We will widely use the following anisotropic result which connects the norms of the Fourier transform of a function and the polynomial generated
by this function (see~\cite[p. 106]{TB}).
Note that for $p=1$, this question is closely related to a study of Lebesgue constants of approximation methods (see~\cite{Bel}, \cite{L}, and~\cite[Ch. 4 and Ch. 9]{TB}).

\begin{theorem}\label{lemBL}
Let $0<p\le 1$. Let $\varphi\in C(\R^d)$ have a compact support and let
$\widehat{\varphi}\in L_p(\R^d)$.
Then
$$
\sup_{\e_j\neq 0,\,j=1,\dots,d} \(\prod_{j=1}^d |\e_j|\)^{1-\frac1p}\Vert
\Phi_\e\Vert_{L_p(\T^d)}=(2\pi)^{d/2} \Vert\widehat{\varphi}\Vert_{L_p(\R^d)},
$$
where
$$
\Phi_\e(x)=\sum_{k\in \Z^d}\varphi(\e_1 k_1,\dots, \e_d k_d)  e^{i(k,x)},\quad \e=(\e_1,\dots,\e_d).
$$
\end{theorem}
Note that this theorem was proved
 in the case $\e_1=\dots=\e_d$ in \cite{Bel} for $p=1$ and in \cite[p. 106]{TB} for $0<p\le 1$. In applications below, the anisotropic case $(\e_j\neq \e_i)$ plays an essential role.

\begin{proof}
We follow the proof of  Theorem~4.4.1 from~\cite{TB}.
First, let us verify the estimate from above.
By the Poisson summation formula (see~\cite[Ch. VII]{SW}), we get % and the standard properties of Fourier transform, we obtain
\begin{equation*}
  \begin{split}
     &\sum_{k\in \Z^d}\widehat{\vp}\(\frac{x_1+2\pi k_1}{\e_1},\dots,\frac{x_d+2\pi k_d}{\e_d}\)\\
     &=(2\pi)^{-d/2}\prod_{j=1}^d|\e_j| \sum_{k\in \Z^d} \vp\(-\e_1 k_1,\dots,-\e_d k_d\) e^{i(k,x)}\\
     &=(2\pi)^{-d/2}\prod_{j=1}^d|\e_j| \sum_{k\in \Z^d} \vp\(\e_1 k_1,\dots,\e_d k_d\) e^{-i(k,x)}.
   \end{split}
\end{equation*}
Further, simple calculations imply the following relations:
\begin{equation*}
  \begin{split}
  &\(\prod_{j=1}^d |\e_j|\)^{p-1}\Vert
\Phi_\e\Vert_{L_p(\T^d)}^p\\
&=(2\pi)^{pd/2} \prod_{j=1}^d |\e_j|^{-1} \int_{\T^d} \bigg| \sum_{k\in \Z^d} \widehat{\vp}\(\frac{x_1+2\pi k_1}{\e_1},\dots,\frac{x_d+2\pi k_d}{\e_d}\)\bigg|^p dx\\
&\le (2\pi)^{pd/2} \prod_{j=1}^d |\e_j|^{-1} \int_{\T^d}\sum_{k\in \Z^d}  \bigg|\widehat{\vp}\(\frac{x_1+2\pi k_1}{\e_1},\dots,\frac{x_d+2\pi k_d}{\e_d}\)\bigg|^p dx\\
&=(2\pi)^{pd/2} \prod_{j=1}^d |\e_j|^{-1} \sum_{k\in \Z^d} \int_{\T^d+2\pi k}\bigg| \widehat{\vp}\(\frac{x_1}{\e_1},\dots,\frac{x_d}{\e_d}\)\bigg|^p dx\\&=(2\pi)^{pd/2} \Vert \widehat{\vp}\Vert^p_{L_p(\R^d)}.
  \end{split}
\end{equation*}

To obtain the estimate from below, we  use the substitution $x_j\to \e_j x_j$, $j=1,\dots,d$, which gives
\begin{equation*}
\begin{split}
  &\(\prod_{j=1}^d |\e_j|\)^{p-1}\Vert
\Phi_\e\Vert_{L_p(\T^d)}^p\\
&=\(\prod_{j=1}^d |\e_j|\)^p\int_{0}^{\frac{2\pi}{\e_1}}\dots \int_{0}^{\frac{2\pi}{\e_d}} \bigg|\sum_{k\in \Z^d}\vp\(\e_1 k_1,\dots,\e_d k_d\) e^{-i(\sum_{j=1}^d\e_j k_j x_j)}\bigg|^p dx.
  \end{split}
\end{equation*}
 Assuming that the limit inferior  of this value as $|\e|=(\e_1^2+\cdots+\e_d^2)^{1/2}\to 0$ is finite and equals $M$, we will show
that
$$(2\pi)^{pd/2} \Vert\widehat{\varphi}\Vert_{L_p(\R^d)}^p\le M.$$ Given arbitrary $N>0$ and $\d>0$, we obtain for sufficiently small $|\e|$
\begin{eqnarray}\label{riemann}
  %\begin{split}
     &&\int_{N \T^d} \bigg|\sum_{k\in \Z^d}\vp\(\e_1 k_1,\dots,\e_d k_d\) e^{-i(\sum_{j=1}^d\e_j k_j x_j)}\prod_{j=1}^d \e_j\bigg|^p dx\\
     \nonumber
     &\le& \(\prod_{j=1}^d |\e_j|\)^p\int_{0}^{\frac{2\pi}{\e_1}}\dots \int_{0}^{\frac{2\pi}{\e_d}} \bigg|\sum_{k\in \Z^d}\vp\(\e_1 k_1,\dots,\e_d k_d\) e^{-i(\sum_{j=1}^d\e_j k_j x_j)}\bigg|^p dx\le M+\d.
   %\end{split}
\end{eqnarray}
Noting that we deal with the  Riemann integral sum  in (\ref{riemann}), we pass to the limit as $|\e|\to 0$ to obtain
\begin{equation*}
  \begin{split}
     (2\pi)^{pd/2}\int_{N \T^d}|\widehat{\vp}(x)|^pdx=\int_{N\T^d}\bigg|\int_{\R^d}\vp(y)e^{-i(x,y)}dy\bigg|^p dx\le M+\d.
   \end{split}
\end{equation*}
Letting $N\to\infty$ and, consequently, $\d\to 0$, we get the desired estimate from below.
\end{proof}

\begin{corollary}\label{lemBLX}
Let $0<p\le 1$ and $\vp\in C^\infty(\R^d)$ have a compact support. Then
$$
\sup_{\e_j\neq 0,\,j=1,\dots,d} \(\prod_{j=1}^d |\e_j|\)^{1-\frac1p}\bigg\Vert
\sum_{k\in \Z^d}\varphi(\e_1 k_1,\dots, \e_d k_d)  e^{i(k,x)}\bigg\Vert_{L_p(\T^d)}<\infty.
$$
\end{corollary}

\subsection {
Hardy--Littlewood type theorems for Fourier series with monotone coefficients}
It is well known that trigonometric series with monotone  coefficients have several important properties.
%\begin{lemma}\label{monot}
%\cite{Z}
In particular,
Hardy and Littlewood  \cite{Z} studied the series
$$
f(x)=\sum_{n\in \Z} a_n \cos nx, \qquad a_{n}\ge a_{n+1},
$$
and proved that
 a necessary and sufficient condition
 for $f\in L_p(\T)$, $1<p<\infty$, is % to be in  $L_p$, $1<p<\infty$ is % if and only if
 $\,\sum_{n\in \Z} a_n^p n^{p-2}<\infty.$ Moreover,

\begin{eqnarray}\label{monot}
\|f\|_p\asymp \(\sum_{n\in\Z} a_n^p n^{p-2}\)^{1/p}.
\end{eqnarray}
%\end{lemma}
It turns out that in applications one needs a more general condition on Fourier coefficients than just monotonicity.
We will give a higher-dimensional generalization of the  Hardy-Littlewood theorem for trigonometric series with general monotone coefficients.

 A non-negative
sequence $a=\{a_{{m}}\}, { {m}}=(m_1,\dots,m_d)\in \mathbb{N}^d,$ $d\ge 1$, satisfies {\it
the $GM^d(\beta)$-condition},  written $a\in GM^d(\beta)$, $\beta=\{\beta_k\}$, ${{k}}=(k_1,\dots,k_d)\in \mathbb{N}^d$
(\cite{dy}), if
$$
a_{{m}}\longrightarrow  0 \,\,\qquad \mbox{as}\,\,\qquad |{{m}}|_1\to
\infty
$$
and
$$%\begin{multline*}
\sum_{m_1=k_1}^{\infty}\cdots \sum_{m_d=k_d}^{\infty}
|\triangle^{(d)}a_{{m}}|\,\, \le \,C\beta_{{k}},\qquad k=(k_1, \dots, k_d)\in \N^d,
$$%% \,\,\left( a_{{k}} +
%\sum\limits_{i=1}^{n} \sum\limits_{m_i=k_i+1}^{\infty}
%\frac{a_{k_1,\cdots, k_{i-1},m_i,k_{i+1},\cdots,k_n}}{m_i} \right.
%\\ \qquad
%+ \sum\limits_{1\le i<j\le n} \sum\limits_{m_i=k_i+1}^{\infty}
%\sum\limits_{m_j=k_j+1}^{\infty} \frac{a_{k_1,\cdots,m_i,\cdots,m_j
%\cdots k_n}}{m_i m_j}  +
% \cdots\\+
% \left.
%\sum\limits_{{\bf{m}}={\bf{k+1}}}^{\bf{\infty}} \frac{a_{m_1,\cdots, m_n}}{m_1\cdots m_n}\right),
%\end{multline*}
 where
%%$\beta=\{\beta_{{m}}\}, { {m}}=(m_1,\cdots,m_d)\in \mathbb{N}^d,$
% $\sum\limits_{{{m}}={{k}}}^{{{\infty}}}=
% \sum\limits_{m_1=k_1}^{\infty}\cdots \sum\limits_{m_d=k_d}^{\infty}$ and the operator
%$\triangle^{(d)}$ is defined as
$$
\triangle^{(d)}\equiv\prod\limits_{j=1}^d \triangle^{j} \quad \mbox{and}\quad
\triangle^{j}a_{{m}}=a_{{m}}-a_{m_1,\dots,m_{j-1},m_j+1,m_{j+1},\dots, m_d}.
$$

\begin{lemma}\label{monot+}(See~\cite{dy, dya}.)
Let  $1<p<\infty$, $d\ge 1$, and let the Fourier series of $f$ be  of the  following type
 \begin{eqnarray}\label{multipleser}
f(x)\sim\sum_{m_1=1}^{\infty}\cdots \sum_{m_d=1}^{\infty}
 a_{{m}} \prod\limits_{j\in B} \cos m_j x_j \prod\limits_{j\in N\backslash B} \sin m_j x_j,
\end{eqnarray}
where $N=\{1,2,\cdots,d\}$ and $B \subseteq N$.
\\
{\rm (i)} Let
$a\in GM^d(\beta)$. Then
$$
\|f\|_{L_p(\T^d)}\lesssim \(
\sum_{m_1=1}^{\infty}\cdots \sum_{m_d=1}^{\infty} \beta_{{{m}}}^p \,\bigg(\prod_{j=1}^d m_j\bigg)^{p-2}
\)^{1/p}.
$$
{\rm (ii)} Let
$a\in GM^d(\beta)$ with
$$
\beta_{{k}} =
\sum_{m_1=k_1/c}^{\infty}\cdots \sum_{m_d=k_d/c}^{\infty} \frac{a_{m_1,\dots, m_d}}{m_1\cdots m_d},\qquad c>1, \qquad a_{m_1,\dots, m_d}\ge 0.$$
Then
 \begin{eqnarray}\label{multipleser+}
\|f\|_{L_p(\T^d)}\asymp
 \(
\sum_{m_1=1}^{\infty}\cdots \sum_{m_d=1}^{\infty} a_{{{m}}}^p \,\bigg(\prod_{j=1}^d m_j\bigg)^{p-2}
\)^{1/p}.
\end{eqnarray}
%$$
%f\in L_p[0,2\pi]^n
% \qquad \mbox{iff}\qquad
%\sum\limits_{{{m}}={{1}}}^{\bf{\infty}} a_{{\bf{m}}}^p \,\Bigl(\prod_{j=1}^n m_j\Bigr)^{p-2}
%\,\,<\,\,\infty
%$$
In particular  {\textnormal{(see also \cite{mo})}}, equivalence (\ref{multipleser+}) holds for the series (\ref{multipleser}) satisfying %such that
the condition $\triangle^{(d)}a_{{m}}\ge 0$ for any $m\in \N^d$.
\end{lemma}

\subsection{Weighted Hardy's inequality for averages}

We will frequently use the following weighted Hardy inequality
(see, e.g., \cite{persson}).

\begin{lemma}\label{har} Let $1<\l<\infty.$ Suppose
$u(x), v(x)\ge 0$ on $(0,\d)$. Then

\begin{eqnarray*}
%\label{har0}
\int_0^\d u(x) \left[\int_0^x\psi(t)
\,dt\right]^\l\,dx
\lesssim
\int_0^\d
\psi^\l(x) v(x) dx
\end{eqnarray*}
holds for each $\psi(x)\ge0$ if and only if
\begin{eqnarray}\label{har0co}
\sup\limits_{0<s<\d}\left(\int_s^\d
u(x) dx\right)^{1/\l}\left(\int_0^s v(x)^{1-\l'}dx\right)^{1/\l'}<\infty.
\end{eqnarray}
\end{lemma}

\bigskip

\newpage

\section{Polynomial inequalities of Nikol'skii--Stechkin--Boas--types
%Hardy--Littlewood--Nikol'skii inequalities for trigonometric polynomials
}\label{sec3}

In its simplest form, the Nikol'skii--Stechkin--type inequality for a polynomial $T_n(x)=\sum_{|k|\le n} c_k e^{i kx}$,
$x\in\T$,
states that,  for $r\in\N$ and $1\le p\le \infty$,
$$
\| T_n^{(r)}  \|_{L_p(\T)}\le
\left(\frac{n}{2\sin\frac{n\delta}{2}} \right)^r \| \triangle_\delta^r  T_n\|_{L_p(\T)},\quad 0<\delta\le\frac\pi n,
$$
extending the classical Bernstein inequality
(\ref{cl-ber}).
%:
%for any $T_n(x)=\sum_{|k|_\infty\le n} c_k e^{i(k,x)}$,
%one has
%$$
% \Vert T_n\Vert_{\dot W_p^r(\T^d)}\asymp \d^{-r}\w_r (T_n,\d)_{L_p(\T^d)},\quad 0<\d\le\frac\pi n.
%$$
Boas derived that
$$
\left(\frac{n}{2\sin\frac{nh}{2}} \right)^r\|\triangle_h^r  T_n  \|_{L_p(\T)}\le
\left(\frac{n}{2\sin\frac{n\delta}{2}} \right)^r \| \triangle_\delta^r  T_n \|_{L_p(\T)},\quad 0<h\le\delta\le\frac\pi n.
$$

In this section, we prove analogues of these results in the multidimensional case for all $0<p\le\infty$.
We will use the following notion of the directional derivative. Let
$$
f(x)\sim\sum_{k\in \Z^d} \widehat{f}_k e^{i(k,x)},
$$
we denote the directional derivative of $f$ of order $\a>0$ along a vector $\xi\in \R^d$ by
$$
\(\frac{\partial}{\partial\xi}\)^{\a} f(x)=\sum_{k\in \Z^d} {\(i(k,\xi)\)^\a} \widehat{f}_k e^{i(k,x)}
$$
(see, e.g.,~\cite{Wil}).

\begin{theorem}\label{lemNSB}
Let $0<p\le\infty$, $\a>0$, $n\in\mathbb{N}$, and $\xi\in \R^d$,
$0<|\xi|<{\pi}/{n}$. Then, for any
$$
T_n(x)=\sum_{|k|_\infty\le n}c_ke^{i(k,x)}\in \mathcal{T}_n,
$$
we have
%\bigskip
\begin{equation}\label{ineqNS3}
\bigg\Vert \(\frac{\partial}{\partial\xi}\)^{\a} T_n\bigg\Vert_{p}
\asymp\Vert\Delta_\xi^\a T_n\Vert_{p},
\end{equation}
where the constants in this equivalence depend only on $p$, $\a$, and $d$.
%(двустороннее неравенство с константами, зависящими только от $p, d$, and
%$\b$)
\end{theorem}

As we have already mentioned, inequalities of this type have a long and rich history starting from the  1940s.
When $d=1$ and $1\le p \le\infty$, this is the result of Nikol'skii \cite{nikoL}, Stechkin \cite{stechkiN}, and Boas~\cite{Boas} (for integer $\a$, see also \cite[p.~214 and p.~251]{timan}), and Taberski~\cite{tabeR} and Trigub~\cite{Trigub05} (for any positive $\a$).
When $d=1$ and $0<p<1$, (\ref{ineqNS3}) follows from \cite{DHI} for $\a\in\N$
and from \cite{K07} for $\a>0$.
When $d\in\N$ and $1\le p \le\infty$, the part $"\lesssim"$ was proved in~\cite[Theorem~3]{Wil} for any $\a>0$.

%Since it is difficult to find the last paper, we present here the proof of the above lemma in the case $0<p<1$.

\begin{proof}
 First, let us prove the estimate from above in (\ref{ineqNS3}). Let the function $v\in C^\infty(\mathbb{R})$, $v(s)=1$ for $|s|\le\pi$ and
$v(s)=0$ for $|s|\ge{3\pi}/{2}$.
We
define
\begin{equation*}
K_\xi(x)=\sum_{k\in \Z^d}\frac{\(i(k,\xi)\)^\a}{\left(2i\sin\frac{(k,\xi)}{2}\right)^\a}
v\(\sqrt{
(k_1\xi_1)^2+\cdots+(k_d\xi_d)^2}
\)e^{i(k,x)}
\end{equation*}
(in what follows we assume that $\frac 00=1$). Note that $K_\xi$ is a trigonometric polynomial of degree at most  $C/|\xi_j|$ in variable $x_j$. We can assume that $\xi_j\neq 0$ for any $j=1,\dots,d$, otherwise we consider the same problem in the space $\R^{d-k}$ with some $k>0$.

%Let $T_n(x)=\sum_{|k|\le n}a_ke^{i(k,x)}$. Then
Taking into account that
\begin{equation*}
\Delta_\xi^\a
T_n(x)=\sum_{|k|_\infty\le n}\left(2i\sin\frac{(k,\xi)}{2}\right)^\a
c_ke^{i(k,x+\frac{\xi\a}{2})}
\end{equation*}
and $\Big((k_1\xi_1)^2+\cdots+(k_d\xi_d)^2\Big)^{\frac12}\le |k||\xi|\le \pi$,
one has that
\begin{equation}\label{ineqNS4}
\(\frac{\partial}{\partial\xi}\)^{\a} T_n(x)=(K_\xi*\Delta_\xi^\a
T_n)\left(x-\frac{\xi\a}{2}\right),
\end{equation}
where $*$ stands for the convolution of periodic functions.

Let first $0<p<1.$
We will show that
\begin{equation}\label{ineqNS5}
\bigg\Vert \(\frac{\partial}{\partial\xi}\)^{\a} T_n\bigg\Vert_{p}\lesssim \(\prod_{j=1}^d |\xi_j|\)^{1-\frac 1p}\Vert
K_\xi\Vert_p\Vert \Delta_\xi^\a T_n\Vert_p.
\end{equation}
For this we need the following Nikol'skii's inequality (\cite{niko51}, see also  \cite{remez})
$$
\Vert U_{n_1,\dots,n_d}\Vert_1\lesssim \(\prod_{j=1}^d n_j\)^{\frac1p-1}\Vert U_{n_1,\dots,n_d}\Vert_p,
$$
where
$$
U_{n_1,\dots,n_d}({{{x}}})=
\sum_{|k_1|\le {n_1}}
\cdots
\sum_{|k_d|\le {n_d}}
a_{{k}} e^{i({{k}},{{{x}}})},
\quad {{k}}\in \mathbb{Z}^d.
$$
By using  (\ref{ineqNS4}) and the above Nikol'skii inequality
for the polynomial $K_\xi(x)\Delta_\xi^\a
T_n(x-\xi\a/2-t)$ of degree at most $C/|\xi_j|$ in $x_j$,
 we obtain
\begin{equation*}
\begin{split}
\bigg|\(\frac{\partial}{\partial\xi}\)^{\a} T_n(x)\bigg|^p&\le\frac{1}{(2\pi)^{p d}}
\left(\int_{\T^d}|K_\xi(t)\Delta_\xi^\a
T_n(x-\xi\a/2-t)|dt\right)^p\le\\
&\lesssim \(\prod_{j=1}^d |\xi_j|\)^{p-1}\int_{\T^d}\left|K_\xi(t)\Delta_\xi^\a
T_n(x-\xi\a/2-t)\right|^pdt.
\end{split}
\end{equation*}
Integrating the above inequality over $x$, we get
\begin{equation*}
\int_{\T^d}\bigg|\(\frac{\partial}{\partial\xi}\)^{\a} T_n(x)\bigg|^pdx\lesssim \(\prod_{j=1}^d |\xi_j|\)^{p-1}\int_{\T^d}\int_{\T^d}\left|K_\xi(t)\Delta_\xi^\a
T_n(x-\xi\a/2-t)\right|^pdtdx.
\end{equation*}
%Таким образом, из последнего соотношения и теоремы Фубини сразу
%следует неравенство (\ref{ineqNS5}).
Thus, applying Fubini's theorem derives (\ref{ineqNS5}).

%Далее нам понадобится лемма 4.1.1 из \cite{TrigBook}.

%\begin{lemma}\label{lem331}
%{\it Пусть функция $g\in C(\mathbb{R})$, имеет компактный носитель
%и $\widehat{g}\in L_p(\mathbb{R})$ при некотором $p\in(0,1]$.
%Тогда
%\begin{equation*}
%\sup_{\delta>0}\delta^{1-\frac
%1p}\left(\int_{-\pi}^\pi\bigg|\sum_kg(\delta
%k)e^{ikx}\bigg|^p\right)^\frac 1p=\sqrt{2\pi}\Vert
%\widehat{g}\Vert_{L_p(\mathbb{R})}.
%\end{equation*}}
%\end{lemma}

We now show that $\(\prod_{j=1}^d |\xi_j|\)^{1-1/p}\Vert
K_\xi\Vert_p$ is bounded from above. Set
\begin{equation*}
g_\xi(t)=\left(\frac{(t,\xi)}{2\sin\frac {(t,\xi)}2}\right)^\a %v((t,\xi)).
v\(\sqrt{
(t_1\xi_1)^2+\cdots+(t_d\xi_d)^2}
\).
\end{equation*}
Since $g_{\textbf{1}}\in C^\infty(\mathbb{R})$, $\textbf{1}=(1,\dots,1)$, and
$g_{\textbf{1}}$ has a compact support, we obtain that, for any $r\ge 1$, the following holds
\begin{equation*}
\widehat{g_{\textbf{1}}}(y)=O\left(\frac{1}{|y|^r}\right)\quad\text{as}\quad
|y|\to\infty.
\end{equation*}
The latter together with $|\widehat{g_{\textbf{1}}}(y)|\le C$ for $|y|\le 1$
gives that $\widehat{g_{\textbf{1}}}\in L_p(\R^d)$.
% $$
% \Vert
%\widehat{g_{\textbf{1}}}\Vert_{L_p(\mathbb{R}^d)}\le C<\infty,
%$$
%where a constant $C$  depends
% only of $p$ and $\a$.

Now, by using Theorem~\ref{lemBL},  we obtain
\begin{eqnarray*}
\(\prod_{j=1}^d |\xi_j|\)^{1-\frac 1p}\Vert
K_\xi\Vert_p&=&\(\prod_{j=1}^d |\xi_j|\)^{1-\frac
1p}\left(\int_{\T^d}\bigg|\sum_{k\in \Z^d}g_\xi(k)e^{i(k,x)}\bigg|^pdx\right)^\frac 1p\\
&=&\(\prod_{j=1}^d |\xi_j|\)^{1-\frac
1p}\left(\int_{\T^d}\bigg|\sum_{k\in \Z^d}g_{\textbf{1}}(\xi_1k_1,\dots,\xi_dk_d)e^{i(k,x)}\bigg|^pdx\right)^\frac 1p
\\&\le&
(2\pi)^{d/2}
\Vert
\widehat{g_{\textbf{1}}}\Vert_{L_p(\mathbb{R}^d)}.
\end{eqnarray*}
Thus, from the last inequality and
(\ref{ineqNS5}) we obtain the part $"\lesssim"$  in
(\ref{ineqNS3}).

To prove the estimate  $"\gtrsim"$, we consider the polynomial
\begin{equation*}
 \widetilde{K}_\xi(x)=\sum_{k\in \Z^d}\frac{\left(2i\sin\frac{(k,\xi)}{2}\right)^\a}{\(i(k,\xi)\)^\a}
v\(\sqrt{
(k_1\xi_1)^2+\cdots+(k_d\xi_d)^2}
\)
e^{i(k,x)},
\end{equation*}
the equality
\begin{equation*}
\Delta_\xi^\a
T_n(x)=\(\widetilde{K}_\xi* \(\frac{\partial}{\partial\xi}\)^{\a} T_n\)\left(x+\frac{\xi\a}{2}\right)
\end{equation*}
and repeat the above proof.

Let us consider the case $1\le p\le\infty$. By~\eqref{ineqNS4} and Young's inequality, we have
\begin{equation*}
\bigg\Vert \(\frac{\partial}{\partial\xi}\)^{\a} T_n\bigg\Vert_p\lesssim \Vert
K_\xi\Vert_1\Vert \Delta_\xi^\a T_n\Vert_p.
\end{equation*}
As above, we derive that $\Vert K_\xi\Vert_1\lesssim 1$, which implies the inequality $"\lesssim"$ in~\eqref{ineqNS3}. The converse inequality can be proved similarly.
\end{proof}

Using now the fact that
$$
\w_\a (T_n,\d)_{p}=
\sup_{\xi\in \R^d,\,|\xi|=1}\big\Vert
\Delta_{\xi\delta}^\a
T_n
\big\Vert_{p},
$$
we obtain the following result.
\begin{corollary}\label{corNSB}
  Let $0<p\le\infty$, $\a>0$, $n\in\mathbb{N}$, and
$0<\d\le \pi/ n$. Then, for any
$T_n\in\mathcal{T}_n$, we have
%\bigskip
\begin{equation}\label{ineqNS3cor}
\sup_{\xi\in \R^d,\,|\xi|=1}\bigg\Vert \(\frac{\partial}{\partial\xi}\)^{\a} T_n\bigg\Vert_{p}
\asymp \d^{-\a}\w_\a (T_n,\d)_{p}.
\end{equation}
\end{corollary}
%Boas~\cite{Boas}, see also \cite[p. 214 and p. 251]{timan}, proved this for $d=1$ and $\a\in \N$.

Further, using \eqref{ineqNS3cor} we prove the following Bernstein type inequality for fractional directional derivative of trigonometric polynomials.

\begin{corollary}\label{corNSBBEr}
  Let $0<p\le\infty$, $\a \in \N \cup ((1/p-1)_+,\infty)$, and $n\in\mathbb{N}$. Then, for any
$T_n\in\mathcal{T}_n$, we have
%\bigskip
\begin{equation}\label{ineqNS3corBEr}
\sup_{\xi\in \R^d,\,|\xi|=1}\bigg\Vert \(\frac{\partial}{\partial\xi}\)^{\a} T_n\bigg\Vert_{p}
\lesssim n^\a \Vert T_n \Vert_{p}.
\end{equation}
\end{corollary}
Inequality~\eqref{ineqNS3corBEr} is the classical Bernstein inequality for $d=1$ (see, e.g., \cite[Ch. 4]{devore}).
For the fractional $\a$ and $d=1$ see~\cite{Liz} and \cite{BL},
cf.~\cite{RS}. For $1\le p\le\infty$, $\a>0$, and $d\in \N$, inequality~\eqref{ineqNS3corBEr} can be obtained from \cite[Theorem 3]{Wil}.

Note also that~\eqref{ineqNS3corBEr} does not hold for $\a \not\in \N \cup ((1/p-1)_+,\infty)$ (see Proposition~\ref{lemBLf}).

\begin{proof}
To prove~\eqref{ineqNS3corBEr},  one should take $\d=1/n$  in~\eqref{ineqNS3cor} and  use the simple inequality
$$
\left\|
\Delta_{h}^{\a} f \right\|_{p}\le \left(\sum\limits_{\nu=0}^\infty \Big| \binom{\a}{\nu}\Big|^{\min(1,p)}\right)^{1/{\min(1,p)}}\left\|f \right\|_{p}\le C(\a,p)
\left\|f \right\|_{p}
$$
(cf. property (c) in Section \ref{sec4}).
\end{proof}

For moduli of smoothness of integer order, we have the following
Nikol'skii--Stechkin--Boas result, which is new in the case $0<p<1$, $d\ge 2$.

\begin{theorem}\label{eq++}
  Let $0< p\le\infty$, $r\in \N$, and $0<\d\le\pi/ n$. Then, for any $T\in \mathcal{T}_n$, $n\in\mathbb{N}$, one has
\begin{equation}\label{eq+++}
 \Vert T_n\Vert_{\dot W_p^r}\asymp \d^{-r}\w_r (T_n,\d)_{p},
\end{equation}
where the constants in this equivalence are independent of $T_n$ and $\d$.

\end{theorem}

\begin{proof}
Let us first show that
\begin{equation}\label{eq+++-}
 \Vert T_n\Vert_{\dot W_p^r}\lesssim n^r \omega_r\(T_n,\frac\pi n\)_p.
\end{equation}
Since
$$
\Vert T_n\Vert_{\dot W_p^r}=\sum_{|k|_1=r}\bigg\Vert \frac{\partial^{r} T_n}{\partial x_1^{k_1}\dots\partial x_d^{k_d}}\bigg\Vert_p,
$$
we have to verify that for any $k\in \Z_+^d$, $|k|_1=r$, one has
$$
\bigg\Vert \frac{\partial^{r} T_n}{\partial x_1^{k_1}\dots\partial x_d^{k_d}}\bigg\Vert_p\lesssim n^r\w_{r}\(T_n,\frac\pi n\)_p.
$$
Indeed, denoting $U_n(x)=\frac{\partial^{r-k_1} }{\partial x_2^{k_2}\dots\partial x_d^{k_d}}T_n(x)$ and applying Theorem~\ref{lemNSB} to $U_n$ in the one-dimensional case, we obtain
\begin{equation*}
  \begin{split}
\bigg\Vert \frac{\partial^{r} T_n}{\partial x_1^{k_1}\dots\partial x_d^{k_d}}\bigg\Vert_{p}=\bigg\Vert \frac{\partial^{k_1}}{\partial x_1^{k_1}}U_n\bigg\Vert_{p}&\lesssim n^{k_1}\Vert \D_{\frac\pi n; x_1}^{k_1} U_n\Vert_p\\
&=n^{k_1}\bigg\Vert \frac{\partial^{r-k_1}}{\partial x_2^{k_2}\dots\partial x_d^{k_d}}\D_{\frac\pi n; x_1}^{k_1} T_n\bigg\Vert_{p},
   \end{split}
\end{equation*}
where  $\D_{\frac\pi n; x_1}^{k_1} U_n$ is the $k_1$-th difference with respect to $x_1$.
Repeating this procedure $d-1$ times, we get
$$
\bigg\Vert \frac{\partial^{r} T_n}{\partial x_1^{k_1}\dots\partial x_d^{k_d}}\bigg\Vert_{p}\lesssim n^r \Vert \D_{\frac\pi n; x_1}^{k_1}\dots \D_{\frac\pi n; x_d}^{k_d} T_n\Vert_p\lesssim n^r\w_{k_1,\dots,k_d}\(T_n,\frac\pi n,\dots,\frac\pi n\)_p,
$$
where $\w_{k_1,\dots,k_d}\(T_n,\pi/n,\dots,\pi/n\)_p$ is the mixed modulus of smoothness (see, e.g.,~\cite{TiMod}). It remains only to apply the following equivalence result
$$
\sum_{|k|_1=r}\w_{k_1,\dots,k_d}\(T_n,\frac\pi n,\dots,\frac\pi n\)_p\asymp \w_r\(T_n,\frac\pi n\)_p.
$$
In the case $1\le p\le \infty$, the proof of this equivalence is given in~\cite{TiMod} and is based on the representation of the total difference of a function via the sum of the mixed differences.
Hence, repeating the proof of Theorem~7 from~\cite{TiMod}, one can easily verify that the above equivalence holds in the case $0<p<1$ too.
Thus, we have proved (\ref{eq+++-}). Noting that Corollary~\ref{corNSB} implies the equivalence
$
n^r \w_r(T_n,\pi/n)_p\asymp \d^{-r}\w_r(T_n,\d)_p$, $0<\d<\pi/n,
$
we conclude the proof of the part $"\lesssim"$ in (\ref{eq+++}).

The reverse part follows  from  Corollary \ref{corNSB}.
\end{proof}

We will also need the following result on equivalence between the Riesz derivatives (see~\eqref{eqLaplacian}) and the directional derivatives.
\begin{corollary}\label{lemRAZZZ}
  Let $1<p<\infty$ and $\a>0$. Then, for any $T\in \mathcal{T}$, one has
\begin{equation}\label{eqRAZZZ}
  \sup_{\xi\in \R^d,\,|\xi|=1}\bigg\Vert \(\frac{\partial}{\partial\xi}\)^{\a} T\bigg\Vert_{p}\asymp \Vert (-\D)^{\a/2}T\Vert_{p}.
\end{equation}
Moreover, if $0< p\le \infty$ and $r\in \N$, then
\begin{equation}\label{eqRAZZZ+}
  \sup_{\xi\in \R^d,\,|\xi|=1}\bigg\Vert \(\frac{\partial}{\partial\xi}\)^{r} T\bigg\Vert_{p}\asymp \Vert T\Vert_{\dot W_p^r}.
\end{equation}
\end{corollary}

\begin{proof}
The first part follows from~\eqref{ineqNS3cor} and the equivalence $\d^{-\a}\w_\a(T,\d)_p\asymp \Vert (-\D)^{\a/2} T\Vert_p$, where $0<\d\le 1/\deg T$ (see~\cite{Wil}).
Equivalence~(\ref{eqRAZZZ+}) follows from Theorem~\ref{eq++} and Corollary~\ref{corNSB}.
\end{proof}

We conclude this section by the following new description of the classical Sobolev seminorm in terms of the norm defined by the directional derivatives.
In the case $1\le p\le \infty$, this extends~\eqref{eqRAZZZ+} for any function $f\in  W_p^r$.

\begin{theorem}\label{eq++--}
  Let $1\le p\le\infty$ and $r\in \N$.
Then, for any $f\in  W_p^r$, one has
$$
 \Vert f\Vert_{\dot W_p^r}\asymp \sup_{\xi\in \R^d,\,|\xi|=1}\bigg\Vert \(\frac{\partial}{\partial\xi}\)^{r} f\bigg\Vert_{p}.
$$
%where the constant depends only on $p$ and $r$.
\end{theorem}

\begin{proof}
The estimate
\begin{equation}\label{multTTT}
  \sup_{\xi\in \R^d,\,|\xi|=1}\bigg\Vert \(\frac{\partial}{\partial\xi}\)^{r} f\bigg\Vert_{p} \lesssim \Vert f\Vert_{\dot W_p^r}
\end{equation}
 easily follows from the multinomial theorem given by
$$
(x_1+x_2+\dots+x_d)^r=\sum_{|k|_1=r}\frac{r!}{k_1!k_2!\dots k_d!}x_1^{k_1}x_2^{k_2}\dots x_d^{k_d}
$$
and the triangle inequality.

Let us prove the upper estimate for $\Vert f\Vert_{\dot W_p^r}$. Let $T_n\in\mathcal{T}_n$ be such that
$$
\Vert f-T_n\Vert_{\dot W_p^r}\to 0\quad\text{as}\quad n\to\infty.
$$
Then, by~\eqref{eqRAZZZ+} and~\eqref{multTTT}, we obtain
\begin{equation*}
  \begin{split}
  \Vert f\Vert_{\dot W_p^r}&\le \Vert T_n\Vert_{\dot W_p^r}+\Vert f-T_n\Vert_{\dot W_p^r}\lesssim \sup_{\xi\in \R^d,\,|\xi|=1}\bigg\Vert \(\frac{\partial}{\partial\xi}\)^{r} T_n\bigg\Vert_{p}+\Vert f-T_n\Vert_{\dot W_p^r}\\
  &\lesssim \sup_{\xi\in \R^d,\,|\xi|=1}\bigg\Vert \(\frac{\partial}{\partial\xi}\)^{r} f\bigg\Vert_{p}+\sup_{\xi\in \R^d,\,|\xi|=1}\bigg\Vert \(\frac{\partial}{\partial\xi}\)^{r} (f-T_n)\bigg\Vert_{p}+\Vert f-T_n\Vert_{\dot W_p^r}\\
  &\lesssim \sup_{\xi\in \R^d,\,|\xi|=1}\bigg\Vert \(\frac{\partial}{\partial\xi}\)^{r} f\bigg\Vert_{p}+\Vert f-T_n\Vert_{\dot W_p^r}.
   \end{split}
\end{equation*}
It remains only to pass to the limit as $n\to\infty$.
\end{proof}

\newpage

\section{Basic properties of fractional moduli of smoothness}\label{sec4}

 Let us recall several basic properties of moduli of smoothness (\cite{ But, RS3, samko,  tabeR}).
For  $f,f_1,f_2\in L_p(\T^d)$, $0< p\le \infty$, and $\alpha, \b\in \N\cup ((1/p-1)_+,\infty)$, we have %(\cite{butz}, \cite[\S 20]{samko}, \cite{tab}, \cite{real})
\begin{itemize}
\item[{\rm (a)}]
$ \omega_\a(f,\delta)_p$ is a non-negative non-decreasing function of $\delta$ such that
$\lim\limits_{\delta\to 0+} \omega_\a(f,\delta)_p=0;$ %({\bf nado li???})
\item[{\rm (b)}]
$              \omega_\a(f_1+f_2,\delta)_p\le 2^{(\frac1p-1)_+}
\big(     \omega_\a(f_1,\delta)_p+\omega_\a(f_2,\delta)_p\big);
$
%where $$C_1(p)=\left\{
%                 \begin{array}{ll}
%                   1, & \hbox{$1\le p\le \infty$;} \\
%                   2^{1/p-1}, & \hbox{$0<p<1$.}
%                 \end{array}
%               \right.
%$$
 \item[{\rm (c)}]
$              \omega_{\alpha+\b}(f,\delta)_p\le
\left(\sum\limits_{\nu=0}^\infty \left| \binom{\b}{\nu}\right|^{\min(1,p)}\right)^{1/{\min(1,p)}} \omega_{\a}(f,\delta)_p$;
%\qquad
%$$
%C_2(\alpha,p)=\left\{
%                                                               \begin{array}{ll}
%                                                                 \sum\limits_{\nu=0}^\infty \Big| \binom{\alpha}{\nu}\Big|, & \hbox{$1\le p\le \infty$;} \\
%                                                                 \left(\sum\limits_{\nu=0}^\infty \Big| \binom{\alpha}{\nu}\Big|^p\right)^\frac1p, & \hbox{$0<p<1$.}
%                                                               \end{array}
%                                                             \right.
%$$
%where
%$$\left(\sum\limits_{\nu=0}^\infty \Big| \binom{\alpha}{\nu}\Big|^p\right)^\frac1p;$$

\item[{\rm (d)}]
for $\lambda\ge 1$,
        $$\omega_{\a}(f,\lambda \delta)_p\le C(\a,p,d) \lambda^{\a+d(\frac1p-1)_+}  \omega_{\a}(f,\delta)_p;$$

\smallskip

\item[{\rm (e)}]
for $ 0<t\le\delta$,
      $$\frac{\omega_\a(f,\delta)_p}{\delta^{\a+d(\frac1p-1)_+}}\le C(\a,p,d)
\frac{\omega_\a(f,t)_p}{t^{\a+d(\frac1p-1)_+}}.$$
%\smallskip
%     \item[{\rm (f)}]
%for $1\le p\le \infty$,
%$$              \omega_{\alpha+\a}(f,\delta)_p\le
%C(\alpha) \delta^\alpha \omega_{\a}(
% f^{(\alpha)},\delta)_p,$$
%where   $f^{(\alpha)}$ is the Weyl derivative of $f$.

%\item[{\rm (g)}]
%for any $f \in L_p$ there exists $\omega \in  \Omega_\alpha$ such that
%$$\omega(t)\asymp\omega_\alpha(f,t)_p  \,\,\,(0<t<\infty);$$
%\item[{\rm (h)}]
%for any $ \omega \in \Omega_\alpha$ there exists $f \in
%L_p$ and $t_1>0$ such that
%$$\omega(t)\asymp\omega_\alpha(f,t)_p  \,\,\,(0<t<t_1).$$
    \end{itemize}
%Recall that $\Omega_\alpha\,\,(\alpha>0)$  is the collection of non-decreasing functions
% $\omega(\cdot)$ on $[0,1]$  such that
% $\omega(\delta)\to 0\quad\mbox{as}\quad \delta\to 0$,
%and $\delta^{-\alpha}\omega(\delta)$ is non-increasing.

\smallskip

Statement (a) follows from
$
|\w_\a(f_1,\d)_p-\w_\a(f_2,\d)_p|\lesssim \Vert f_1-f_2\Vert_p,
$
inequalities in (b) and (c) are clear, and (d) and (e) will be proved later in this section.

Now, we prove several basic results in approximation theory.
In what follows, $E_n(f)_{p}$ is the best approximation of $f\in L_p(\T^d)$ by trigonometric polynomials $T\in\mathcal{T}_n$, i.e.,
$$
E_n(f)_{p}=\inf_{T\in \mathcal{T}_n}\Vert f-T\Vert_{p}.
$$

We start with the direct inequality.

\begin{proposition}\label{converseMod+}
  Let $0<p\le\infty$, $\a \in \N \cup ((1/p-1)_+,\infty)$, and $n\in\mathbb{N}$. Then
%\bigskip
\begin{equation}\label{JacksonSO}
E_n(f)_p\lesssim \w_\a\(f,\frac1n\)_p.
\end{equation}
\end{proposition}
Inequality~\eqref{JacksonSO} follows immediately from the Jackson-type inequality for the moduli of smoothness of integer order \cite{SO}. Indeed,
for $\a+r\in \N$ and $r>(1/p-1)_+$, property (c) implies
$$
  E_n(f)_p\lesssim \w_{\a+r}\(f,\frac1n\)_p\lesssim \w_{\a}\(f,\frac1n\)_p.
$$%\end{equation}
Note that for $1<p<\infty$ the technique from \cite{ddt} allows us to obtain a sharper version of (\ref{JacksonSO})
for the fractional moduli of smoothness:% given by
$$
\frac1{n^\a}\bigg(\sum_{k=0}^n (k+1)^{\a \rho-1}E_k(f)_p^{\rho}\bigg)^{\frac1\rho}
\lesssim \w_\a\(f,\frac1n\)_p,
$$
where
$$
\rho=\rho(p)=\max(p,2).
%\left\{
%       \begin{array}{ll}
%         \max(p,2), & \hbox{$p<\infty$;} \\
%         1, & \hbox{$p=\infty$.}
%       \end{array}
%     \right.
$$

The inverse  inequality is given as follows.
\begin{proposition}\label{converseMod}
  Let $0<p\le\infty$, $\a \in \N \cup ((1/p-1)_+,\infty)$, and $n\in\mathbb{N}$. Then  we have
%\bigskip

\begin{equation}\label{eqconverseMod}
\w_\a\(f,\frac1n\)_p\lesssim \frac1{n^\a}\bigg(\sum_{k=0}^n (k+1)^{\a \tau-1}E_k(f)_p^{\tau}\bigg)^{\frac1\tau},
\end{equation}
where
\begin{equation}\label{tau+}
\tau=\tau(p)=\left\{
       \begin{array}{ll}
         \min(p,2), & \hbox{$p<\infty$;} \\
         1, & \hbox{$p=\infty$.}
       \end{array}
     \right.
\end{equation}

\end{proposition}

In the case $1\le p\le \infty$, $\a\in\N$, this result is well known (see, e.g.,~\cite[Ch. 7]{DL},  \cite{dai-d}, \cite{timan} and the reference therein).
To the best of our knowledge,
the case  $ p=1,\infty$ and $\a>0$ %was not considered earlier to the best of our knowledge.
 does not seem to have been considered before for this moduli of smoothness (cf.~\cite{dit-acta}).
In the case $0<p<1$, $d>1$ this result is new, while the case $0<p<1$ and $d=1$ was recently considered in \cite{RS3}.
Moreover, the
corresponding results in terms of different realizations of the $K$-functional (in particular, using the Laplace operator)
were obtained in  the papers \cite{run, RS3}.

\begin{proof}
The proof follows the standard argument. We give the proof only for the case $0<p\le 1$ or $p=\infty$. The case $1<p<\infty$ can be handled using
\cite{dai-d}.

Let $T_n\in\mathcal{T}_n$ be such that $E_n(f)_p=\Vert f-T_n\Vert_p$. Then for any $m\in\N$ we have with $p_1=\min(p,1)$
\begin{equation}\label{coco1}
  \begin{split}
     \w_\a\(f,\frac1n\)_p^{p_1}&\le \w_\a\(f-T_{2^{m+1}},\frac1n\)_p^{p_1}+\w_\a\(T_{2^{m+1}},\frac1n\)_p^{p_1}\\
     &\lesssim E_{2^{m+1}}(f)_p^{p_1}+\w_\a\(T_{2^{m+1}},\frac1n\)_p^{p_1}.
   \end{split}
\end{equation}
Next, by~\eqref{ineqNS3cor} and~\eqref{ineqNS3corBEr}, choosing $m$ such that $2^m<n\le 2^{m+1}$, we obtain
\begin{equation}\label{coco2}
  \begin{split}
    \w_\a\(T_{2^{m+1}},\frac1n\)_p^{p_1}&\le \w_\a\(T_{1}-T_0,\frac1n\)_p^{p_1}+\sum_{\nu=0}^m \w_\a\(T_{2^{\nu+1}}-T_{2^{\nu}},\frac1n\)_p^{p_1}\\
    &\lesssim \frac1{n^{\a p_1}}\(E_0(f)_p^{p_1}+\sum_{\nu=0}^m 2^{(\nu+1)\a p_1} E_{2^\nu}(f)_p^{p_1}\)\\
    &\lesssim \frac1{n^{\a p_1}}\sum_{k=0}^{2^m} (k+1)^{\a p_1-1} E_{k}(f)_p^{p_1}.
   \end{split}
\end{equation}
Combining~\eqref{coco1} and~\eqref{coco2}, we proved the desired result.
\end{proof}

Now, we will prove a realization result, which will be crucial in our further study.
\begin{theorem}\label{lemKfp<1d}
%\begin{thmb}\label{lemKfp<1}
Let $f\in L_p(\T^d)$, $0<p\le \infty$, and $\a\in\N\cup \big((1/p-1)_+,\infty\big)$. Then, for any $\d\in (0,1)$, we have
\begin{equation}\label{eq.th6.0d}
{\omega}_\a(f,\d)_p\asymp \mathcal{R}_{\a}(f,\d)_{p},
\end{equation}
where $\mathcal{R}_{\a}(f,\d)_{p}$ is the realization of the $K$-functional, i.e.,
$$
\mathcal{R}_{\a}(f,\d)_{p}:=\inf_{T\in\mathcal{T}_{[1/\d]}}\left\{\Vert
f-T\Vert_p+\d^{\a}\sup_{\xi\in \R^d,\,|\xi|=1}\bigg\Vert \(\frac{\partial}{\partial\xi}\)^{\a} T\bigg\Vert_{p}\right\}.
$$
Moreover, if $T\in\mathcal{T}_n$, $n=[1/\d]$, is such that
$
\Vert f-T\Vert_p\lesssim E_n(f)_p,
$
then
\begin{equation}\label{eqvwithbest}
  {\omega}_\a(f,\d)_p\asymp \Vert
f-T\Vert_p+\d^{\a}\sup_{\xi\in \R^d,\,|\xi|=1}\bigg\Vert \(\frac{\partial}{\partial\xi}\)^{\a} T\bigg\Vert_{p}.
\end{equation}
\end{theorem}

This theorem is known for $d=1$, $0<p\le\infty$, and $\a\in\N$ (see \cite{DHI});
for $d=1$, $0<p<1$, and $\a>0$ see \cite{K11};
for $d=1$, $1\le p\le \infty$, and $\a>0$ see \cite{simonov-sb}.
In the case
$d\ge 1$, $0<p\le\infty$, and $\a\in\N$,~\eqref{eqvwithbest} was stated in \cite{diti}   without the proof.
\begin{proof}

First, let us prove the estimate from above in~\eqref{eq.th6.0d}. Let $n=[1/\d]$ and $T_n\in\mathcal{T}_n$.
By Corollary~\ref{corNSB}, we estimate
\begin{equation*}
  \begin{split}
      \w_\a(f,\d)_p&\lesssim \Vert f-T_n\Vert_p+\w_\a(T_n,\d)_p\\
      &\lesssim \Vert f-T_n\Vert_p+\d^{\a}\sup_{\xi\in \R^d,\,|\xi|=1}\bigg\Vert \(\frac{\partial}{\partial\xi}\)^{\a} T_n\bigg\Vert_{p}.
   \end{split}
\end{equation*}
It remains to take infimum over all  $T_n\in\mathcal{T}_n$.

Let us prove the estimate from below.
Let $T\in \mathcal{T}_n$ be such that $\Vert f-T\Vert_p\lesssim E_n(f)_p$. Then by~\eqref{JacksonSO} and~\eqref{ineqNS3cor}, we obtain
\begin{equation*}
  \begin{split}
     \mathcal{R}_\a\(f,\frac1n\)_p&\lesssim \Vert f-T\Vert_p+n^{-\a}\sup_{\xi\in \R^d,\,|\xi|=1}\bigg\Vert \(\frac{\partial}{\partial\xi}\)^{\a} T\bigg\Vert_{p}\\
     &\lesssim \w_{\a}\(f,\frac1n\)_p+\w_{\a}\(T,\frac1n\)_p\lesssim \w_{\a}\(f,\frac1n\)_p,
   \end{split}
\end{equation*}
that is, \eqref{eq.th6.0d} follows.
% which implies the estimate from below in~\eqref{eq.th6.0d}.
%
%
%Now we are going to prove~\eqref{eqvwithbest}. The estimate from above is clear.
%Let us prove the estimate from below, that is the inequality $"\gtrsim"$ in~\eqref{eqvwithbest}.
%By Bernstein's inequality~\eqref{ineqNS3corBEr} we have for any $T^*\in\mathcal{T}_n$,
%\begin{equation*}
%  \begin{split}
%     &n^{-\b}\sup_{\xi\in \R^d,\,|\xi|=1}\bigg\Vert \(\frac{\partial}{\partial\xi}\)^{\a} T\bigg\Vert_{p}\\
%&\lesssim n^{-\a}\sup_{\xi\in \R^d,\,|\xi|=1}\bigg\Vert \(\frac{\partial}{\partial\xi}\)^{\a} (T-T^*)\bigg\Vert_{p}+n^{-\a}\sup_{\xi\in \R^d,\,|\xi|=1}\bigg\Vert \(\frac{\partial}{\partial\xi}\)^{\a} T^*\bigg\Vert_{p}\\
%&\lesssim \Vert T-T^*\Vert_{p}+n^{-\a}\sup_{\xi\in \R^d,\,|\xi|=1}\bigg\Vert \(\frac{\partial}{\partial\xi}\)^{\a} T^*\bigg\Vert_{p}\\
%&\lesssim \Vert f-T^*\Vert_{p}+\Vert f-T\Vert_{p}+n^{-\a}\sup_{\xi\in \R^d,\,|\xi|=1}\bigg\Vert \(\frac{\partial}{\partial\xi}\)^{\a} T^*\bigg\Vert_{p}.\\
%   \end{split}
%\end{equation*}
%This and the inequality $\Vert f-T\Vert_p\le \Vert f-T^*\Vert_p$  give the estimate from below in~\eqref{eqvwithbest}.

In fact, we have proved that
\begin{equation*}
  \begin{split}
     \w_{\a}\(f,\frac1n\)_p&\lesssim\mathcal{R}_\a\(f,\frac1n\)_p\\
&\lesssim \Vert f-T\Vert_p+n^{-\a}\sup_{\xi\in \R^d,\,|\xi|=1}\bigg\Vert \(\frac{\partial}{\partial\xi}\)^{\a} T\bigg\Vert_{p}\lesssim \w_{\a}\(f,\frac1n\)_p,
   \end{split}
\end{equation*}
which implies~\eqref{eqvwithbest}.
\end{proof}

Similarly, taking into account Theorem~\ref{eq++}, we can prove the following result.
\begin{theorem}\label{lemKfp<1d+}
%\begin{thmb}\label{lemKfp<1}
Let $f\in L_p(\T^d)$, $0<p\le\infty$, and $r\in\N$. Then, for any $\d\in (0,1)$, we have
\begin{equation*}
{\omega}_r(f,\d)_p\asymp \mathcal{R}^\sharp_r(f,\d)_{p},
\end{equation*}
where
$$
\mathcal{R}^\sharp_r(f,\d)_{p}:=\inf_{T\in\mathcal{T}_{[1/\d]}}\left\{\Vert
f-T\Vert_p+\d^{r}
\Vert T\Vert_{\dot W_p^r}
\right\}.
$$
Moreover, if $T\in\mathcal{T}_n$, $n=[1/\d]$, is such that
$
\Vert f-T\Vert_p\lesssim E_n(f)_p,
$
then
\begin{equation*}
  {\omega}_r(f,\d)_p\asymp \Vert
f-T\Vert_p+\d^{r}\Vert T\Vert_{\dot W_p^r}.
\end{equation*}
\end{theorem}
In the case $d=1$ and $0< p\le \infty$,  this was proved in \cite{DHI}; for
$1<p<\infty$ (see \cite{dit-acta}). The result is new in the case $d>1$, $0<p\le 1$, and $p=\infty$.

% the classical result, see, e.g., \cite{BeSh}.

Next we prove the following important property of the moduli of smoothness.

\begin{theorem}\label{lemBasicMod}
Let $f\in L_p(\T^d)$, $0<p\le \infty$, $\a\in \N\cup ((1/p-1)_+,\infty)$, $\l>0$, and $t>0$. Then
\begin{equation}\label{Run0}
  \w_\a(f,\l t)_p\lesssim  (1+\l)^{\a+d(\frac1p-1)_+}\w_\a(f,t)_p.
\end{equation}
In particular, for any $0<h<t$, one has
\begin{equation}\label{eqMonMod}
 \frac{\w_\a(f,t)_p}{t^{\a+d(\frac1p-1)_+}}\lesssim \frac{\w_\a(f,h)_p}{h^{\a+d(\frac1p-1)_+}}.
\end{equation}
\end{theorem}

When $d=1$ and $0< p \le 1$, this was obtained in \cite{oswald, rad} (for integer $\a$) and \cite{RS3} (for positive $\a$).
For the case $d\in \N$, $1<p<\infty$, and $\a>0$ see \cite[Theorem 7]{Wil}.

\begin{proof}
We follow the idea from~\cite[p. 194]{run}.
We will consider only the case $0<p\le 1$. The case $p=\infty$ can be proved by the same way. The case $1<p<\infty$ follows from
(\ref{eq.th6.0d}) and (\ref{RealKpge1-}).

Denote $\d=(\d_1,\dots,\d_d)\in \R^d$, $\d_j>0$, $j=1,\dots,d$, and
$$
K_\d(x)=\sum_{k\in\Z^d} \vp_\d(k)e^{i(k,x)},
$$
where $\vp_\d(t)=v\((\d,t)\)$, $v\in C^\infty(\R)$, $v(s)=1$ for $|s|\le 1/2$ and $v(s)=0$ for $|s|>1$.

We can assume that $\l>1$.
Suppose also that $|\d|\le h$ and $T\in \mathcal{T}_{1/h}$.

Using Theorem~\ref{lemNSB}, we obtain
\begin{equation}\label{Run1}
\begin{split}
    \Vert \D_{\l \d}^\a f\Vert_p^p&\lesssim \Vert f-K_{\l \d}*T\Vert_p^p+\Vert \D_{\l \d}^\a (K_{\l \d}*T)\Vert_p^p\\
    &\lesssim \Vert f-T\Vert_p^p+\Vert T-K_{\l \d}*T\Vert_p^p+\bigg\Vert \(\frac\p{\p (\l\d)}\)^\a (K_{\l \d}*T)\bigg\Vert_p^p\\
    &= \Vert f-T\Vert_p^p+\Vert T-K_{\l \d}*T\Vert_p^p+\l^{\a p}\bigg\Vert  K_{\l \d}*\(\frac\p{\p \d}\)^\a T   \bigg\Vert_p^p.
\end{split}
\end{equation}

Now, we will show that
\begin{equation}\label{Run2}
\Vert T-K_{\l \d}*T\Vert_p \lesssim \l^{\a+d(\frac1p-1)} \bigg\Vert  \(\frac\p{\p \d}\)^\a T \bigg\Vert_p
\end{equation}
and
\begin{equation}\label{Run3}
\bigg\Vert  K_{\l \d}*\(\frac\p{\p \d}\)^\a T   \bigg\Vert_p \lesssim \l^{d(\frac1p-1)} \bigg\Vert  \(\frac\p{\p \d}\)^\a T \bigg\Vert_p.
\end{equation}

We will prove only the first inequality, the second one can be obtained similarly.

It is easy to see that \eqref{Run2} is equivalent to the following inequality
\begin{equation*}
%\label{Run4}
\Vert A_{\l \d}*T\Vert_p \lesssim \l^{\a+d(\frac1p-1)} \Vert T \Vert_p,
\end{equation*}
where
$$
A_\d(x)=\sum_{k\in\Z^d} \psi_\d(k)e^{i(k,x)}
$$
and
$$
\psi_\d(t)=\frac{(1-\vp_\d(t))%\vp_\d(t/2)
}
{(it,\d)^\a}
v\(\frac{\sqrt{{
(t_1\delta_1)^2+\cdots+(t_d\delta_d)^2}}}{2}
\).
$$
The same argument as in the proof of (\ref{ineqNS5}) implies
\begin{equation*}
%\label{Run5}
  \begin{split}
      \Vert A_{\l \d}*T\Vert_p&\lesssim \(\prod_{j=1}^d \d_j\)^{1-\frac1p}\Vert A_{\l \d}\Vert_p \Vert T\Vert_p
\lesssim \Vert \widehat{\psi_{\l \textbf{1}}}\Vert_{L_p(\R^d)} \Vert T\Vert_p\\
&\lesssim \l^{d(\frac1p-1)}\Vert \widehat{\psi_{\textbf{1}}}\Vert_{L_p(\R^d)} \Vert T\Vert_p\lesssim \l^{d(\frac1p-1)}\Vert T\Vert_p,
  \end{split}
\end{equation*}
that is, \eqref{Run2} is verified.

Now, combining (\ref{Run1})--(\ref{Run3}), we obtain
\begin{equation*}
%\label{Run6}
  \begin{split}
    \Vert \D_{\l \d}^\a f\Vert_p&\lesssim \Vert f-T\Vert_p+\l^{\a+d(\frac1p-1)} \bigg\Vert  \(\frac\p{\p \d}\)^\a T \bigg\Vert_p\\
    &\lesssim \l^{\a+d(\frac1p-1)} \(\Vert f-T\Vert_p+ |\d|^\a \sup_{|\xi|=1,\, \xi\in \R^d}\bigg\Vert \(\frac\p{\p \xi}\)^\a T \bigg\Vert_p\).
    \end{split}
\end{equation*}
Finally,
taking infimum over all $T\in \mathcal{T}_{1/h}$ and using~\eqref{eq.th6.0d}, we get
 $$
\Vert \D_{\l \d}^\a f\Vert_p\lesssim \l^{\a+d(\frac1p-1)} \mathcal{R}_\a(f,h)_p \lesssim \l^{\a+d(\frac1p-1)} \w_\a(f,h)_p,
  $$
which implies (\ref{Run0}).
\end{proof}

We finish this section by
 the Marchaud inequality for the fractional moduli of smoothness.
\begin{theorem}\label{lemMarchaudMod}
  Let $f\in L_p(\T^d)$, $0<p\le \infty$, $\a\in \N\cup
    ((1/p-1)_+,\infty)$, and $\g>0 $ be  such that $\a+\g\in \N\cup
    ((1/p-1)_+,\infty)$. Then, for any $\d\in (0,1)$, one has
\begin{equation}\label{eq.lemMarchaudMod}
  \w_\a(f,\d)_p\lesssim \d^\a \(\int_\d^1 \(\frac{\w_{\g+\a}(f,t)_p}{t^{\a}}\)^\tau\frac{dt}{t}\)^\frac1\tau,
\end{equation}
where $\tau$ is given by (\ref{tau+}).
\end{theorem}

The Marchaud inequality for the moduli of smoothness of integer order can be found, e.g., in \cite[p. 48]{DL} and \cite{marchaud}.
The case $1<p<\infty$ for fractional moduli was handled   in~\cite[Theorem 2.1]{Treb}.

 %The case $0<q\le 1$ and $q=\infty$ can be obtained by using the standard scheme, that is by using the converse theorem~\ref{converseMod} and Jackson type theorem~\ref{???} ????, see for instance~\cite[p. 48]{DL}, see also~\cite{Di-Marchaudples1} \textbf{(??? Ditzian kakoinibud'). }

\begin{proof}
The proof is a combination of the direct and inverse estimates
\eqref{JacksonSO}
and
(\ref{eqconverseMod}) as well as the monotonicity property of modulus of smoothness $\w_\a(f,\d)_p\asymp \w_\a(f,2\d)_p$ (see (\ref{Run0})).
\end{proof}

%Let us now
%discuss the higher dimensional case.  In what follows we assume that $\Vert \cdot\Vert_p=\Vert
%\cdot\Vert_{L_p(\T^d)}$, $d\ge 1$.
%
%Then
%a Stechkin--Nikol'skii type lemma % in multidimensional case
%  reads as follows (for
%the case $1\le p\le\infty$  see also~\cite[(3.8)]{Wil}).
%\begin{lemma}\label{lemNSd}
%    Let $0<p\le \infty$, $d\ge 1$, $\b\in \N\cup ((1/p-1)_+,\infty)$, and $\d\in (0,\pi/n)$. Then for any $T_n\in\mathcal{T}_n$
%    \begin{equation}\label{eqlemNSd1dd}
%        \Vert \mathcal{D}^\b T_n\Vert_{L_p(\T^d)} \lesssim \d^{-\b}\sum_{j=1}^d \Vert \D_{\d e_j}^\b T_n\Vert_{L_p(\T^d)}
%\lesssim \d^{-\b}\w_\b(T_n,\d)_{L_p(\T^d)},
%    \end{equation}
%    where
%    $$
%    \mathcal{D}^\b T_n (x)=\sum_{j=1}^d \(\frac{\partial}{\partial x_j}\)^\b T_n(x).
%    $$
%\end{lemma}
%
%
%\begin{proof}
%    The case $d=1$ and $0<p<1$ of inequality (\ref{eqlemNSd1dd}) see in Lemma~\ref{lemNSB}.
%    The case $d>1$ and $0<p<1$ can be proved by using the one-dimensional case. Indeed, we have for any $j\in \{1,2,\dots,d\}$ (let, for example, $j=1$) that
%    \begin{equation*}
%        \begin{split}
%          \left\Vert \frac{\partial^\b}{\partial x_1^\b} T_n\right\Vert_p^p&=\int_{\T^{d-1}}\int_{\T^1}
%    \bigg|\frac{\partial^\b}{\partial x_1^\b} T_n(x)\bigg|^p \, dx_1 dx_2\dots dx_d
%    \\
%             &\lesssim \int_{\T^{d-1}}\d^{-\b p}\int_{\T^1}|\D_{\d e_1}^\b T_n(x)|^p \, dx_1 dx_2\dots dx_d
%             \\
%             &\lesssim
%             \d^{-\b p}\w_\b
%             (T_n,\d)_p^p,
%         \end{split}
%         \end{equation*}
%which implies \eqref{eqlemNSd1dd}.
%
%
%\end{proof}

\medskip

%Для формулировки леммы
%нам понадобятся вспомогательная функция $\varphi$ такая, что $\varphi\in
%C^\infty(\R)$, $\varphi(x)=0$ при $|x|>1$ и $\varphi(0)=1$. Будем еще
%дополнительно предполагать, что $\varphi$ является четной и монотонно
%убывает при $x\ge 0$. Используя функцию $\varphi$ введем
%последовательность полиномов
%$$\Varphi_n(x)=\sum_k \varphi\left(\frac kn\right)e^{ikx}. $$

%Нам понадобится также следующая лемма (см.~\cite{S}).
%\begin{lemma}\label{lemS}
%    Пусть $f\in L_p(\T)$ и $0<p<q\le \infty$. Тогда
%\begin{equation}\label{eqlemS1}
%E_n(f)_q\le
%C\(n^{{q_1}(\frac1p-\frac1q)}E_n(f)_p^{q_1}+\sum_{k=n+1}^\infty
%k^{\frac {q_1}p-2}E_k(f)_p^{q_1}\)^\frac1{q_1},
%\end{equation}
%где
%$$
%q_1=\left\{
%      \begin{array}{ll}
%        q, & \hbox{$q<\infty$;} \\
%        1, & \hbox{$q=\infty$,}
%      \end{array}
%    \right.
%$$
%а $C$ -- некоторая константа, не зависящая от $f$ и $n$.
%\end{lemma}

\bigskip

\newpage

\section{Hardy--Littlewood--Nikol'skii inequalities for trigonometric polynomials}\label{hardy-inequality-section}\label{sec5}

%Будем говорить, что функция $\psi\in H_\a$, $\a>0$, если $\psi$
%непрерывна на $\R^d$, $\psi\in C^\infty(\R^d\setminus \{0\})$,
%$\psi(-x)=\overline{\psi(x)}$ для $x\in\R^d$  и
%$\psi$   a homogeneous function of degree $\a$, i.e.,
%$$\psi(\tau x)=\tau^\a \psi(x),\quad \tau>0,\quad x\in\R^d.$$

We say that a continuous on $\R^d\setminus \{0\}$ function $\psi(\xi)$ belongs to the class $\mathcal{H}_\a$, $\a\in\R$,
if  %$\psi\in C^\infty(\R^d\setminus \{0\})$,
%$\psi(-\xi)=\overline{\psi(\xi)}$ for $\xi\in\R^d$, and
$\psi$ is  a homogeneous function of degree $\a$, i.e.,
$$
\psi(\tau \xi)=\tau^\a \psi(\xi),\quad \tau>0,\quad \xi\in\R^d.
$$
In addition, if $\a\le 0$, we assume that $\psi \in C^\infty(\R^d\setminus \{0\})$.

Any function $\psi\in \mathcal{H}_\a$ generates the Weyl-type differentiation operator as follows:
$$
\mathcal{D}(\psi)\,:\,
\sum_{\nu\in \Z^d}  c_\nu e^{i(\nu,x)}\to {\sum_{\nu\in \Z^d}}' \psi(\nu) c_\nu e^{i(\nu,x)},
%\nu\in\Z^d;
$$
where ${\sum_\nu^{'}}a_\nu$ means $\sum_{|\nu|>0}a_\nu$.
Let us give several important  examples of the Weyl-type  operators:

\begin{enumerate}
  \item[{(1)}]\; the linear differential operator
$$
P_m({D})f=
\sum_{{}_{\quad k\in \Z_+^d}^{ k_1+\cdots+k_d=m}}a_k {D}^k f,\qquad
{D}^k={D}^{k_1,\cdots, k_d}=\frac{\partial^{k_1+\cdots+k_d}}{\partial x_1^{k_1}
\cdots \partial x_d^{k_d}},
$$
$\phantom{ahhhn}$with
$$
\psi(\xi)=\sum_{{}_{\quad k\in \Z_+^d}^{ k_1+\cdots+k_d=m}}a_k (i\xi_1)^{k_1}\dots(i\xi_d)^{k_d};
$$
  \item [{(2)}]\; the fractional Laplacian  $(-\Delta)^{\alpha/2}f$ (here $\psi(\xi)=|\xi|^\alpha$, $\xi\in\R^d$);
  \item [{(3)}] \;
  the classical
Weyl derivative $f^{(\alpha)}$ (here $\psi(\xi)=(i \xi)^\alpha$, $\xi\in\R$).
\end{enumerate}

In this section, we study the sharp Hardy--Littlewood--Nikol'skii inequality given by
$$
\Vert \mathcal{D}(\psi)T\Vert_q
\lesssim
\eta(n){\Vert
\mathcal{D}(\vp)T\Vert_p},\qquad T\in \mathcal{T}'_{n},
$$
where $0<p\le q\le \infty$, $\psi\in \mathcal{H}_\a$, and $\vp\in \mathcal{H}_{\a+\g}$.
%More precisely, we will find the exact order of the
%following quantity
%$$
%\eta(n)
%:=\sup_{T\in \mathcal{T}'_{n},
%}\frac{\Vert \mathcal{D}(\psi)T\Vert_q}{\Vert
%\mathcal{D}(\vp)T\Vert_p}.
%$$

Note that if $\psi(\xi)=|\xi|^\a$, $\vp\equiv 1$, $1<p<q<\infty$, and $\a=d(1/p-1/q)$, then $\eta(n)=1$ by the Hardy--Littlewood theorem on fractional integration. On the other hand, if $\psi=\vp\equiv1$, then $\eta(n)=n^{d(1/p-1/q)}$, which corresponds to the Nikol'skii inequality for trigonometric polynomials (see Subsection~\ref{subsec1.4}).

The main result of this section is the following theorem.

\begin{theorem}\label{th-hardy-l-n}
 Let $0< p<q\le \infty$, $\a>0$, $\g\ge 0$, $\psi\in
\mathcal{H}_\a$, and $\vp\in \mathcal{H}_{\a+\g}$.
Let also
$\frac{\psi}{\varphi}\in C^\infty(\R^d\setminus \{0\})$ and
$\frac{\varphi}{\psi}\in C^\infty(\R^d\setminus \{0\})$.
We have
\begin{equation}\label{NNNeqlemHL1d++}
\sup_{T\in \mathcal{T}'_{n}
}\frac{\Vert \mathcal{D}(\psi)T\Vert_q}{\Vert
\mathcal{D}(\vp)T\Vert_p}
\asymp\s(n),
\end{equation}
where $\s(\cdot)$ is given as follows:

\textnormal{(1)} \quad if $0<p\le 1$ and $p<q<\infty$, then
$$
\s(t):=\left\{
         \begin{array}{ll}
           t^{d(\frac1p-1)}, & \hbox{$\g>d\(1-\frac1q\)_+$}; \\
           t^{d(\frac1p-1)}\ln^\frac1{q} (t+1), & \hbox{$0<\g=d\(1-\frac1q\)_+$}; \\
           t^{d(\frac1p-\frac1q)-\g}, & \hbox{$0< \g<d\(1-\frac1q\)_+$};\\
           t^{d(\frac1p-\frac1q)}, & \hbox{$\g=0$, $0<p\le 1 <q<\infty$};\\
             t^{d(\frac1p-\frac1q)}, & \hbox{$\g=0$, $0<q\le 1$, and $\frac{\psi(\xi)}{\vp(\xi)}\equiv \const$};\\
             t^{d(\frac1p-1)}, & \hbox{$\g=0$, $d=1$, $0<q<1$, {and} $\frac{\psi(\xi)}{\vp(\xi)}\not\equiv \const$}; \\
             t^{d(\frac1p-1)}\ln (t+1), & \hbox{$\g=0$, $d=1$, $q=1$, {and} $\frac{\psi(\xi)}{\vp(\xi)}\not\equiv \const$},
         \end{array}
       \right.
$$
%\quad
%\l_{q,d}(t):=\left\{
%               \begin{array}{ll}
%                 \ln^\frac1q 2t, & \hbox{$1<q<\infty, d\ge 1$;} \\
%                 \ln 2t, & \hbox{$q=\infty, d=1$;} \\
%                 t^{d-2}, & \hbox{$q=\infty, d\ge 2$.}
%               \end{array}
%             \right.

\textnormal{(2)} \quad if $0<p\le 1$ and $q=\infty$, then
$$
\s(t):=\left\{
         \begin{array}{ll}
           t^{d(\frac1p-1)}, & \hbox{$\g>d$}; \\
           t^{d(\frac1p-1)}, & \hbox{$\g=d=1$ {and}\, $\frac{\psi(\xi)}{\vp(\xi)}=A{|\xi|^{-\g}}\sign \xi$ for some $A\in \C \setminus \{0\}$}; \\
           t^{d(\frac1p-1)}\ln (t+1), & \hbox{$\g=d=1$ {and}\,  $\frac{\psi(\xi)}{\vp(\xi)}\ne A{|\xi|^{-\g}}\sign \xi$ for any $A\in \C \setminus \{0\}$}; \\
           t^{d(\frac1p-1)}\ln (t+1), & \hbox{$\g=d\ge 2$}; \\
           t^{d(\frac1p-\g)}, & \hbox{$0\le  \g<d$},
                    \end{array}
       \right.
$$

\textnormal{(3)} \quad if $1<p<q\le \infty$, then
$$
\s(t):=
\left\{
         \begin{array}{ll}
           1, & \hbox{$\g\ge d(\frac1p-\frac1q),\quad q<\infty$}; \\
           1, & \hbox{$\g> \frac dp,\quad q=\infty$}; \\
           \ln^\frac1{p'} (t+1), & \hbox{$\g=\frac dp,\quad q=\infty$}; \\
           t^{d(\frac1p-\frac1q)-\g}, & \hbox{$0\le \g<d(\frac1p-\frac1q)$}.\\
         \end{array}
       \right.
$$
\end{theorem}

\begin{proof}
The proof of Theorem \ref{th-hardy-l-n} can be obtained by combining Lemmas~\ref{lemHLd}--\ref{lemHLd++} from the following two Subsections~\ref{Section 5.1} and~\ref{Subsection 5.2}.
In more detail,
  in the case (1) the proof follows from Lemmas~\ref{lemgamma0} and~\ref{lemgamma0+} (the case $\g=0$) and
 from  Lemmas~\ref{lemHLd}, \ref{v+}, \ref{v-}, and~\ref{v} (the case $\g>0$).
 In the case (2), the proof follows from Lemmas~\ref{v++} and~\ref{lemgamma0}.
   %in the case $\g=0$, the proof of (1) follows from Lemmas~\ref{lemgamma0} and~\ref{lemgamma0+}. In the case $\g>0$, it can be obtained from %Lemmas~\ref{lemHLd}, \ref{v+}, \ref{v-}, and~\ref{v}. The case (2) follows from Lemmas~\ref{v++} and~\ref{lemgamma0}.
  Concerning the case (3), see Lemma~\ref{lemHLd++}.
\end{proof}

In what follows, the de la Vall\'{e}e Poussin type kernel is defined by
\begin{equation}\label{vallee}
    V_n(x):=\sum_{k\in\Z^d} \left(
    \prod_{j=1}^d v\(\frac{k_j}{n}\)\right)e^{i(k,x)},
\end{equation}
where $v\in C^\infty(\R)$ is monotonic for $t\ge 0$, $v(t)=v(-t)$,
$v(t)=1$ for $|t|\le 1$ and $v(t)=0$ for $|t|\ge 2$.

\begin{remark}\textnormal{ We do not know the sharp growth behavior
  of $\sup_{T\in \mathcal{T}'_{n}
}\frac{\Vert \mathcal{D}(\psi)T\Vert_q}{\Vert
\mathcal{D}(\vp)T\Vert_p}$ in the case $0<p<q\le 1$, $\g=0$, $d\ge 2$, and $\frac{\psi(\xi)}{\vp(\xi)}\not\equiv \const$. However, we can show that
\begin{equation*}
%\label{dobav0}
  n^{d(\frac1p-1)} \Vert
\mathcal{D}(\psi/\vp)\mathcal{{V}}_n\Vert_q \lesssim \sup_{T\in \mathcal{T}'_{n}}\frac{\Vert \mathcal{D}(\psi)T\Vert_q}{\Vert
\mathcal{D}(\vp)T\Vert_p}\lesssim  n^{d(\frac1p-1)} \Vert
\mathcal{D}(\psi/\vp)V_n\Vert_q,
\end{equation*}
where $\mathcal{{V}}_n(x)=e^{ix}V_{n/4}(2x)$ (see Lemma~\ref{lemHLd} and the proof of Lemma~\ref{lemgamma0+}).}
\end{remark}

\begin{remark}\label{voobshedopremark}
\textnormal{Let again $0<p<q\le1$, $\g=0$, $d\ge 2$, and $\frac{\psi(\xi)}{\vp(\xi)}\not\equiv \const$.  Let either $\psi$ be a polynomial
or $d/{(d+\a)}< q\le 1$ and let either $\vp$ be a polynomial
or $d/{(d+\a)}< p<1$. Then, we can show the following estimate from above
\begin{equation}\label{dobav0}
\sup_{T\in \mathcal{T}'_{n}}\frac{\Vert \mathcal{D}(\psi)T\Vert_q}{\Vert
\mathcal{D}(\vp)T\Vert_p}\lesssim  n^{d(\frac1p-1)} \left\{
                                                      \begin{array}{ll}
                                                        1, & \hbox{$0<q<1$;} \\
                                                        \ln (n+1), & \hbox{$q=1$.}
                                                      \end{array}
                                                    \right.
\end{equation}
Since the proof of~\eqref{dobav0} is based on Theorem~\ref{thRealKUM}$'$, we will give it in Section~\ref{sec10}.
}

\textnormal{Regarding the estimate from below,  by Theorem~\ref{lemBL} and Lemma~\ref{lemRunFour} (ii), we have that $\sup_n n^{d(\frac1q-1)}\Vert
\mathcal{D}(\psi/\vp)\mathcal{{V}}_n\Vert_q=\infty$. Thus, by~\eqref{dobav0}, there exists a sequence $\{\e_n\}$ such that $\e_n\to\infty$ and
\begin{equation*}
  n^{d(\frac1p-\frac1q)}\e_n\lesssim \sup_{T\in \mathcal{T}'_{n}}\frac{\Vert \mathcal{D}(\psi)T\Vert_q}{\Vert
\mathcal{D}(\vp)T\Vert_p},
\end{equation*}
which provides optimality of~\eqref{dobav0} in some sense.
}
\end{remark}

\begin{remark}\label{remark-HLN}
\textnormal{Without the assumption $\frac{\varphi}{\psi}\in C^\infty(\R^d\setminus \{0\})$
in Theorem  \ref{th-hardy-l-n}, the result remains true if \eqref{NNNeqlemHL1d++} is replaced  by
$$
\sup_{T\in \mathcal{T}'_{n}}\frac{\Vert \mathcal{D}(\psi)T\Vert_q}{\Vert
\mathcal{D}(\vp)T\Vert_p}
\lesssim\s(n).
$$
}
\end{remark}

\begin{remark}\label{remark+X}
\textnormal{To illustrate a different behavior of $\s(\cdot)$ in {\rm (2)} for $\g=d=1$, we consider the following example.
Taking  $\varphi(\xi)=(i\xi)$ and $\psi(\xi)=|\xi|^2$, we get
    \begin{eqnarray*}
 \s(n)
 %=
%\sup_{T\in \mathcal{T}'_{n}}\frac{\|T_n'\|_q}{\|T_n'\|_p}
\asymp
n^{\frac1p-1} \bigg\Vert \mathcal{D}\(\frac{i\xi}{|\xi|^2}\)V_{n}\bigg\Vert_\infty
&\asymp&
n^{\frac1p-1} \bigg\Vert \sum_{0<|k|<n}\frac{e^{ikx}}{ik}
%\mathcal{D}\big(\frac{1}{ix}\big)
\bigg\Vert_\infty
\\&\asymp&
n^{\frac1p-1} \bigg\Vert \sum_{0<k<n}\frac{\sin kx}{k}
%\mathcal{D}\big(\frac{1}{ix}\big)
\bigg\Vert_\infty \\&\asymp&
  n^{\frac1p-1}.
    \end{eqnarray*}
On the other hand,
considering   $\varphi(\xi)=|\xi|$ and $\psi(\xi)=|\xi|^2$ gives
    \begin{eqnarray*}
     \s(n)
 %=
%\sup_{T\in \mathcal{T}'_{n}}\frac{\|T_n'\|_q}{\|T_n'\|_p}
\asymp
n^{\frac1p-1} \bigg\Vert \mathcal{D}\(\frac{1}{|\xi|}\)
V_{n}\bigg\Vert_\infty
&\asymp&
n^{\frac1p-1} \bigg\Vert \sum_{0<|k|<n}\frac{e^{ikx}}{|k|}
%\mathcal{D}\big(\frac{1}{ix}\big)
\bigg\Vert_\infty
\\
&\asymp&
n^{\frac1p-1} \bigg\Vert \sum_{0<k<n}\frac{\cos kx}{k}
%\mathcal{D}\big(\frac{1}{ix}\big)
\bigg\Vert_\infty \\
&\asymp&
  n^{\frac1p-1} \ln n.
\end{eqnarray*}
Moreover, similar examples can be constructed to illustrate different behavior with respect to
$\psi$ and $\vp$ in {\rm (1)} when $0<q\le 1$.
}
\end{remark}

Now, let us present two important cases of Theorem~\ref{th-hardy-l-n}.

\begin{corollary}\label{corr++--}
Let $d=1$, $0< p<q\le \infty$, $\a>0$, and $\g\ge 0$.
We have
\begin{equation*}
%\label{NNNeqlemHL1d++DD}
\sup_{T\in \mathcal{T}'_{n}
}\frac{\Vert T^{(\a)}\Vert_q}{\Vert
T^{(\a+\g)}\Vert_p}
\asymp\s(n),
\end{equation*}
where $\s(\cdot)$ is given as follows:

\textnormal{(1)} \quad if $0<p\le 1$ and $p<q\le\infty$, then
$$
\s%_{q,\g}
(t)
:=\left\{
         \begin{array}{ll}
           t^{\frac1p-1}, & \hbox{$\g>\(1-\frac1q\)_+$}; \\
           t^{\frac1p-1}\ln^\frac1q (t+1), & \hbox{$0<\g=\(1-\frac1q\)_+$}; \\
           t^{\frac1p-\frac1q-\g}, & \hbox{$0< \g<\(1-\frac1q\)_+$};\\
           t^{\frac1p-\frac1q}, & \hbox{$\g=0$},
         \end{array}
       \right.
$$

\textnormal{(2)} \quad if $1<p\le q\le\infty$, then
$$
\s(t):=\left\{
         \begin{array}{ll}
           1, & \hbox{$\g\ge \frac1p-\frac1q,\quad q<\infty$}; \\
           1, & \hbox{$\g> \frac1p,\quad q=\infty$}; \\
           \ln^\frac1{p'} (t+1), & \hbox{$\g=\frac1p,\quad q=\infty$}; \\
           t^{(\frac1p-\frac1q)-\g}, & \hbox{$0\le \g<\frac1p-\frac1q$}.\\
         \end{array}
       \right.
$$
\end{corollary}

Now, let us consider the corresponding result in the case of Riesz derivatives.

\begin{corollary}\label{corr++}
Let $0< p<q\le \infty$, $\a>0$, and $\g\ge 0$.
We have
\begin{equation*}
%\label{NNNeqlemHL1d++DD}
\sup_{T\in \mathcal{T}'_{n}
}\frac{\Vert (-\D)^{\a/2}T\Vert_q}{\Vert
(-\D)^{(\a+\g)/2}T\Vert_p}
\asymp\s_\D(n),
\end{equation*}
where $\s_\D(\cdot)$ is given as follows:

\textnormal{(1)} \quad if $0<p\le 1$ and $p<q\le\infty$, then
$$
\s_\D(t):=\left\{
         \begin{array}{ll}
           t^{d(\frac1p-1)}, & \hbox{$\g>d\(1-\frac1q\)_+$}; \\
           t^{d(\frac1p-1)}\ln^\frac1{q_1} (t+1), & \hbox{$0<\g=d\(1-\frac1q\)_+$}; \\
           t^{d(\frac1p-\frac1q)-\g}, & \hbox{$0< \g<d\(1-\frac1q\)_+$};\\
           t^{d(\frac1p-\frac1q)}, & \hbox{$\g=0$},\\

         \end{array}
       \right.
$$

\textnormal{(2)} \quad if $1<p<q\le \infty$, then
$$
\s_\D(t):=
\left\{
         \begin{array}{ll}
           1, & \hbox{$\g\ge d(\frac1p-\frac1q),\quad q<\infty$}; \\
           1, & \hbox{$\g> \frac dp,\quad q=\infty$}; \\
           \ln^\frac1{p'} (t+1), & \hbox{$\g=\frac dp,\quad q=\infty$}; \\
           t^{d(\frac1p-\frac1q)-\g}, & \hbox{$0\le \g<d(\frac1p-\frac1q)$}.\\
         \end{array}
       \right.
$$
\end{corollary}

%\textbf{!!! Даже можно написать если воспользоваться неравенством Бернштейна}
%\begin{equation*}
%  n^{d(\frac1p-3)}\e_n\lesssim \sup_{T\in \mathcal{T}'_{n}}\frac{\Vert \mathcal{D}(\psi)T\Vert_q}{\Vert
%\mathcal{D}(\vp)T\Vert_p}.
%\end{equation*}

%At the same time, at least in the case $q=1$, the above inequalities can be estimate by using the following result (see~\cite[p.414]{TB}):
%let $f$ be a bounded measurable functions on $\R^d$ with compact support. Then for any $n\in \N$, we have
%\begin{equation}\label{FormulaTBp414}
%  \Vert \mathcal{D}(f)V_n\Vert_1=(2\pi)^{d/2}\(\int_{n\T^d}|\widehat{f}(x)|dx\)+\t
%\(\int_{n\T^d}|x|n^{-1}|\widehat{f}(x)|dx\)+\max_{u,v,\dots,w\in (2\pi n)^{-1}\T^d}
%\end{equation}

In particular, this and Theorem \ref{th-hardy-l-n} imply the following two analogues of
 the Hardy–Littlewood fractional integration theorem.

\begin{corollary}
  Let $1\le p<q\le\infty$, $f\in L_p(\T^d)$, and $\int_{\T^d} f(x)dx=0$. Then
$$
\Vert (-\D)^{-\g/2}f\Vert_q\lesssim \Vert f\Vert_p
$$
holds provided $\g>d(1/p-1/q)$, $p=1$ or/and $q=\infty$.
\end{corollary}

\begin{corollary}
  Let $d=1$, $0<p\le 1$, $f\in AC(\T)$, and $\int_{\T} f(x)dx=0$. Then
$$
\Vert f\Vert_\infty\lesssim \Vert f'\Vert_p.
$$
%holds provided $\g>d(\frac1p-\frac1q)$, $p=1$ or/and $q=\infty$.
\end{corollary}

%\medskip

\subsection{Hardy--Littlewood--Nikol'skii $(L_p,L_q)$ inequalities for $0<p\le 1$}\label{Section 5.1}
%Рассмотрим отдельно случай $0<p\le 1$. Для этого нам понадобится следующая вспомогательная лемма.

Recall that $V_n$ is given by (\ref{vallee}).

\begin{lemma}\label{lemHLd}
{\it Let $0<p\le 1$, $p<q\le \infty$, $\a>0$, $\g\ge 0$, $\psi\in
\mathcal{H}_\a$, and $\vp\in \mathcal{H}_{\a+\g}$.
%Define
%$$
%\eta(n)
%:=\sup_{T\in \mathcal{T}'_{n}
%}\frac{\Vert \mathcal{D}(\psi)T\Vert_q}{\Vert
%\mathcal{D}(\vp)T\Vert_p}.
%% \asymp n^{d(\frac1p-1)}\Vert \mathcal{D}(\psi/\vp) V_n    \Vert_q.
%$$%

{\rm (i)} We have
\begin{equation}\label{eqlemHL1d}
\sup_{T\in \mathcal{T}'_{n}
}\frac{\Vert \mathcal{D}(\psi)T\Vert_q}{\Vert
\mathcal{D}(\vp)T\Vert_p}\lesssim n^{d(\frac1p-1)} \Vert
\mathcal{D}(\psi/\vp)V_n\Vert_q.
\end{equation}

{\rm (ii)}
Assuming that $\g> 0$ and  $\vp(x)=0$ if and only if $x=0$, we have
\begin{equation*}
%\label{eqlemHL1d}
n^{d(\frac1p-1)}
\Vert
\mathcal{D}(\psi/\vp)V_{n/4}\Vert_q
\lesssim
\sup_{T\in \mathcal{T}'_{n}
}\frac{\Vert \mathcal{D}(\psi)T\Vert_q}{\Vert
\mathcal{D}(\vp)T\Vert_p}.
\end{equation*}

{\rm (iii)}
Assuming that $\g> 0$ and
$\frac{\psi}{\varphi}, \frac{\varphi}{\psi}\in C^\infty(\R^d\setminus \{0\})$, we have
\begin{equation}\label{eqlemHL1d+}
\sup_{T\in \mathcal{T}'_{n}
}\frac{\Vert \mathcal{D}(\psi)T\Vert_q}{\Vert
\mathcal{D}(\vp)T\Vert_p}
%:=\sup_{T\in \mathcal{T}'_{n}}\frac{\Vert \mathcal{D}(\psi)T\Vert_q}{\Vert\mathcal{D}(\vp)T\Vert_p}
 \asymp
 n^{d(\frac1p-1)}\Vert \mathcal{D}(\psi/\vp) V_n\Vert_q.
\end{equation}

}
\end{lemma}

%\begin{remark}
%The condition on the unique zero of  $\vp$  can be dropped if one assumes that
%$\frac{\psi}{\varphi}\in C^\infty(\R^d\setminus \{0\})$ and
%$\frac{\varphi}{\psi}\in C^\infty(\R^d\setminus \{0\})$.
%\end{remark}

\begin{proof}
(i) First, we consider the case $0<q\le 1$.
We have
\begin{equation*}
%\label{zv}
    \mathcal{D}(\psi)T_n(x)=\frac{1}{(2\pi)^d}\int_{\T^d}
    \big(\mathcal{D}(\psi/\vp)V_n(t)\big)\big(\mathcal{D}(\vp)T_n(x-t)\big)dt.
\end{equation*}
Considering the product $\big(\mathcal{D}(\psi/\vp)V_n(t)\big)\big(\mathcal{D}(\vp)T_n(x-t)\big)$
as a trigonometric polynomial of degree $3n$ in each variable $t_j$, $j=1,\dots,d$, and applying ($L_1, L_q$) Nikol'skii's inequality, we obtain
\begin{equation*}
    |\mathcal{D}(\psi)T_n(x)|^q\lesssim
    n^{d(1-q)}\int_{\T^d}
    \left|
     \big(\mathcal{D}(\psi/\vp)V_n(t)\big)\big(\mathcal{D}(\vp)T_n(x-t)\big)\right|^q dt.
\end{equation*}
Integrating this inequality with respect to $x$ and applying the Fubini theorem and the $(L_q, L_p)$ Nikol'skii inequality, we get
\begin{equation*}
    \begin{split}
\Vert \mathcal{D}(\psi)T_n\Vert_q&\lesssim n^{d(\frac1q-1)}\Vert
\mathcal{D}(\psi/\vp)V_n\Vert_q \Vert \mathcal{D}(\vp)
T_n\Vert_q
\\
&\lesssim n^{d(\frac1p-1)}\Vert \mathcal{D}(\psi/\vp)V_n\Vert_q \Vert
\mathcal{D}(\vp) T_n\Vert_p.
     \end{split}
\end{equation*}
Now, let $1<q\le \infty$.
Then the Young convolution inequality and the Nikol'skii inequality between $(L_p, L_1)$ imply
%$$
%\Vert \mathcal{D}(\psi)T_n\Vert_q\lesssim \Vert
%\mathcal{D}(\psi/\vp)V_n\Vert_q \Vert \mathcal{D}(\vp) T_n\Vert_1
%%\lesssim n^{d(1/p-1)}\Vert
%%\mathcal{D}(\psi/\vp)V_n\Vert_q \Vert \mathcal{D}(\vp) T_n\Vert_p
%\lesssim n^{d(1/p-1)}
%\Vert
%\mathcal{D}(\psi/\vp)V_n\Vert_q \Vert \mathcal{D}(\vp) T_n\Vert_p.
%$$

\begin{equation*}
  \begin{split}
     \Vert \mathcal{D}(\psi)T_n\Vert_q&\lesssim \Vert
\mathcal{D}(\psi/\vp)V_n\Vert_q \Vert \mathcal{D}(\vp) T_n\Vert_1\\
&\lesssim n^{d(1/p-1)}
\Vert\mathcal{D}(\psi/\vp)V_n\Vert_q \Vert \mathcal{D}(\vp) T_n\Vert_p.
   \end{split}
\end{equation*}
Thus, the proof of~\eqref{eqlemHL1d} is complete.

(ii)
Denote
$$
V_n^\vp (x):=\mathcal{D}(1/\vp)\big(V_n(x)-V_n(2x)\big).
$$
Note that  $V_n^\vp\in \mathcal{T}'_{4n}$.
Then
\begin{equation}\label{eqlemHL2d}
    \Vert \mathcal{D}(\vp) V_n^\vp\Vert_p=\Vert
    V_n(\cdot)-V_n(2\cdot)\Vert_p\le 2^{1/p}\Vert V_n\Vert_p.
    %\le C    n^{d(1-\frac1p)}.
\end{equation}
Set
$$
{B}(\xi):=
    \prod_{j=1}^d v\(\xi_j\)
$$
Thus, taking into account that $B\in C^\infty(\R^d)$ and it has a compact support, by Corollary~\ref{lemBLX} and~\eqref{eqlemHL2d}, we obtain
\begin{equation}\label{*Vn}
  \Vert \mathcal{D}(\vp) V_n^\vp\Vert_{L_p(\T^d)}\lesssim\Vert V_n\Vert_{L_p(\T^d)}
\lesssim n^{-d(1/p-1)}\Vert \widehat{B} \Vert_{L_p(\R^d)} \lesssim n^{-d(1/p-1)}.
\end{equation}

Note also that
\begin{equation*}
    \begin{split}
    \mathcal{D}(\psi)V_n^\vp
    (x)&=
    \sum_{k\ne 0}\frac{\psi(k)}{\varphi(k)}B\(\frac kn\)e^{i(k,x)}-
        \sum_{k\ne 0}\frac{\psi(2k)}{\varphi(2k)}B\(\frac kn\)e^{i(2k,x)}
\\&=    \mathcal{D}(\psi/\vp)V_n(x)-\frac1{2^\g}
    (\mathcal{D}(\psi/\vp)V_n)(2x)
     \end{split}
\end{equation*}
and, therefore,
\begin{equation}\label{eqlemHL3d}
    \Vert \mathcal{D}(\psi)V_n^\vp\Vert_q\ge C \Vert
    \mathcal{D}(\psi/\vp)V_n\Vert_q,
\end{equation}
where $C=(1-2^{-\gamma \tilde{q}})^{1/\tilde{q}}$, $\tilde{q}=\min(1,q)$.

Thus, combining inequalities (\ref{*Vn}) and
(\ref{eqlemHL3d}), we derive
\begin{equation*}
\sup_{T\in \mathcal{T}'_{4n}}\frac{\Vert \mathcal{D}(\psi)T\Vert_q}{\Vert
\mathcal{D}(\vp)T\Vert_p}\ge \frac{\Vert
\mathcal{D}(\psi)V_n^\vp\Vert_q}{\Vert
\mathcal{D}(\vp)V_n^\vp\Vert_p}\gtrsim n^{d(\frac1p-1)}\Vert
    \mathcal{D}(\psi/\vp)V_n\Vert_q.
\end{equation*}

(iii) Equivalence (\ref{eqlemHL1d+}) follows from
Lemma~\ref{v+} in the case $0<q\le 1$, Lemmas~\ref{v-} and~\ref{v} in the case $1<q<\infty$,
and Lemma~\ref{v++} in the case $q=\infty$.

\end{proof}

%{\bf{Computation of $\s(n)$ for $0<p\le 1$.}}

\begin{lemma}\label{v+}
Let $0<q\le 1$, $\gamma>0$,
$\psi\in \mathcal{H}_\a$, $\vp\in \mathcal{H}_{\a+\g}$, and
$\frac{\psi}{\varphi}\in C^\infty(\R^d\setminus \{0\})$, then
\begin{equation*}
%\label{zvezda}
    \|
\mathcal{D}(\psi/\vp)
V_n \|_q
\asymp 1. % \|(-\Delta)^{-\gamma/2}V_n \|_q.
\end{equation*}
\end{lemma}

\begin{remark}\label{mu}
\textnormal{If $\g=0$, the statement of Lemma~\ref{v+} does not hold in the case $\vp(\xi)=C\psi(\xi)$, $C\ne 0$,  since in this case one has
\begin{equation}\label{mu+}
\|
\mathcal{D}(\psi/\vp)
V_n \|_q\asymp n^{d(1-1/q)}.
\end{equation}
Here, the part \;$"\lesssim"$ follows from Corollary~\ref{lemBLX} and the part \;$"\gtrsim"$ follows from Nikol'skii's inequality and the fact that $\|V_n\|_1 \asymp 1.$
}
%\textbf{????????? как-нибудь переписать или убрать !!!!}
%
% However, it is interesting to remark that (\ref{zvezda}) holds if
% $\vp$ and $\psi$ are such that  $\vp(x)\ne C\psi(x)$ for some constant $C$ and for all $x\in \R^d$, at least if $d=1$. This can be show by the same way as we prove (\ref{hom}) in the case $d=1$ below.
%
%\textbf{?????????}

\end{remark}

\begin{proof}
Let us first show that $\Vert
\mathcal{D}(\psi/\vp) V_{n}\Vert_q\lesssim 1.$
We set
$$f_{\psi,\vp}(x)=\sum_{|k|\ne 0}\frac{\psi(k)}{\varphi(k)}e^{i(k,x)}.
$$
Then
\begin{equation}\label{vsdg}
  \begin{split}
     \Vert \mathcal{D}(\psi/\vp) V_{n}\Vert_q
 &\lesssim
 \Vert \mathcal{D}(\psi/\vp) V_{n}\Vert_1=\|V_n*f_{\psi,\vp}\|_1\\
 &\lesssim \|V_n\|_1 \|f_{\psi,\vp}\|_1.
  \end{split}
\end{equation}
%By Lemma \ref{lemBL},
%we have $\|V_n\|_1\lesssim 1$.
We have by Remark \ref{remarkVn} below that
 \begin{eqnarray*}
\|f_{\psi,\vp}\|_1\lesssim \|f_{\psi,\vp}\|_r \lesssim
\Vert f_\gamma
%\sum_{|k|\ne 0}\frac{e^{i(k,x)}}{|k|^\gamma}
\Vert_r\quad\text{for some}\quad r>1,
\end{eqnarray*}
 where
$$%\begin{eqnarray}\label{wainger}
f_\gamma(x)=\sum_{|k|\ne 0}\frac{e^{i(k,x)}}{|k|^\gamma}.
$$% \end{eqnarray}
The fact that $f_\gamma\in L_r(\T^d)$ follows from
\begin{eqnarray}\label{wainger}
f_\gamma(x)\asymp |x|^{\gamma-d}\quad \mbox{as} \quad x\to 0, \quad 0<\gamma<d
 \end{eqnarray}
 (see \cite{W}).
Therefore, by (\ref{vsdg}),
$$
\Vert \mathcal{D}(\psi/\vp) V_{n}\Vert_q\lesssim \Vert V_n\Vert_1\lesssim 1,
$$
where the last estimate follows from Corollary~\ref{lemBL}.

To show
\begin{equation*}
%\label{eqth1.12Kd}
\Vert \mathcal{D}(\psi/\vp)V_n\Vert_q
\gtrsim 1,
\end{equation*}
we again use the function $f_{\psi,\vp}$:
$$
\Vert f_{\psi,\vp}-f_{\psi,\vp}*V_n\Vert_q \lesssim
\Vert f_{\psi,\vp}-f_{\psi,\vp}*V_n\Vert_1\to 0\quad \text{as}\quad n\to \infty,
$$
since $V_n$ is an approximate identity and $f_{\psi,\vp}\in L_1(\T^d)$. This implies that
$$
 \Vert \mathcal{D}(\psi/\vp)V_n\Vert_q^q\ge \Vert f_{\psi,\vp}
 \Vert_q^q - \Vert f_{\psi,\vp}-f_{\psi,\vp}*V_n\Vert_q^q\ge C(\psi,\vp),
$$
 since $f_{\psi,\vp}\in L_1(\T^d)$.
% Следовательно, (\ref{eqth1.12Kd}) holds.
\end{proof}

Now, we study the case $0<q\le 1$ and $\g=0$.

\begin{lemma}\label{v++q<1}
  Let $0<q\le 1$, $\a>0$, and $\psi,\,\vp\in \mathcal{H}_\a$. Let also
$\frac{\psi}{\varphi}\in C^\infty(\R^d\setminus \{0\})$ and
$\frac{\varphi}{\psi}\in C^\infty(\R^d\setminus \{0\})$. Then
\begin{equation*}
  \Vert \mathcal{D}(\psi/\vp)V_n\Vert_q\asymp \left\{
                                     \begin{array}{ll}
                                       1, & \hbox{$d=1$, $0<q<1$,\; {and} \; $\frac{\psi(\xi)}{\vp(\xi)}\not\equiv \const$}; \\
                                       \ln (n+1), & \hbox{$d=1$, $q=1$,\; {and} \; $\frac{\psi(\xi)}{\vp(\xi)}\not\equiv \const$}; \\
                                       n^{d(1-\frac1q)}, & \hbox{$d\ge 1$ and $\frac{\psi(\xi)}{\vp(\xi)}\equiv \const$.}
                                     \end{array}
                                   \right.
\end{equation*}
\end{lemma}

\begin{proof}
The simplest case is when $\frac{\psi(\xi)}{\vp(\xi)}\equiv \const$. In this case, the estimate from above follows from Corollary~\ref{lemBLX}. At the same time, by Nikol'skii's inequality, we have
$$
\Vert \mathcal{D}(\psi/\vp)V_n\Vert_q\gtrsim n^{d(1-\frac1q)}\Vert \mathcal{D}(\psi/\vp)V_n\Vert_1\gtrsim n^{d(1-\frac1q)},
$$
completing the proof. %which implies the estimate from below in the considered case.

To prove the lemma in the case $d=1$ and $\frac{\psi(\xi)}{\vp(\xi)}\not\equiv \const$, we need the following two auxiliary results.

\begin{lemma}\label{lemTrigub+} \textnormal{(See~\cite[p.~119]{TB}.)}
  Let $f\,:\, \R\to \C$ be a function of bounded variation and let $\lim_{|\xi|\to \infty} f(\xi)=0$. Then, for any $\e>0$, we have
  \begin{equation*}
  %\label{eqTB+1}
    \sup_{0<|x|\le \pi}\left|\e \sum_{k=-\infty}^\infty f(\e k) e^{ikx}-\int_{-\infty}^\infty f(u) e^{iu\frac{x}{\e}} du \right|\le 2\e V_{-\infty}^\infty (f),
  \end{equation*}
  where $V_{-\infty}^\infty (f)$ is the total variation of $f$ on $\R$.
\end{lemma}

\begin{lemma}\label{lemRunFour} \textnormal{(\cite{RS})}
{\it Let $f_\b\in C^\infty (\R^d\setminus \{0\})$ be a homogeneous function of order $\b\ge 0$ and not a polynomial and let
$\eta\in C^\infty(\R^d)$  have a compact support.

{\rm (i)} We have
  $$
  |\widehat{f_\b\eta}(\xi)|\le C_1(1+|\xi|)^{-\b-d},\quad \xi\in\R^d.
  $$

{\rm (ii)}  There exist $\rho>1$, $\theta>0$, and $u_0\in
  {\mathbb{S}}^{d-1}$ such that
\begin{equation}\label{eqRunovskii}
  |\widehat{f_\b\eta}(\xi)|\ge C_2|\xi|^{-\b-d},\quad \xi\in\Omega,
\end{equation}
%  $$
%     |\widehat{f_\b\eta}(x)|\ge C_2|x|^{-\b-d},\quad x\in\Omega,
%  $$
where
$$
\Omega\equiv \Omega(\rho,\theta,u_0)=\{\xi=ru\,:\, r\ge\rho,\,
u\in{\mathbb{S}}^{d-1},\, \cos \theta\le (u,u_0)\le 1\},
$$
${\mathbb{S}}^{d-1}$ is the unit sphere in  $\R^d$ and $C_1$ and $C_2$ are some positive constants.}
\end{lemma}

Denote $f(y)=\frac{\psi(y)}{\vp(y)}v(y)$. By Lemmas~\ref{lemTrigub+} and~\ref{lemRunFour}, we obtain
\begin{equation*}
  \begin{split}
    \Vert \mathcal{D}(f)V_n\Vert_q&\lesssim n \(\int_\T \left|\int_\R f(y) e^{i y n x}dy\right|^q dx \)^\frac1q+V_{-\infty}^\infty(f)\\
    &\lesssim n \(\int_\T \frac{dx}{(1+|nx|)^q} \)^\frac1q+V_{-\infty}^\infty(f)\\
    &\lesssim \left\{
                \begin{array}{ll}
                  1, & \hbox{$0<q<1$;} \\
                  \ln (n+1), & \hbox{$q=1$.}
                \end{array}
              \right.
   \end{split}
\end{equation*}

Similarly, one can prove the estimate from below. The proof of Lemma \ref{v++q<1} is now complete.

\end{proof}

The next two lemmas deal with the case $1<q<\infty$.

\begin{lemma}\label{v-}
Let $1<q<\infty$ and $\g\in \R$. Let also $\mu\in \mathcal{H}_{-\g}$ be such that  $\mu, 1/\mu \in C^\infty(\R^d\setminus \{0\})$.
Then for any
$f\in L_q(\T^d)$ such that $(-\Delta)^{-\gamma/2}f \in L_q(\T^d)$ the following two-sided inequalities hold
\begin{equation}\label{delta-calc+}
\| \mathcal{D}(\mu)f \|_q \asymp \|(-\Delta)^{-\gamma/2}f \|_q,\quad \g\neq 0,
\end{equation}
and
\begin{equation}\label{delta-calc+++}
\| \mathcal{D}(\mu)f \|_q \asymp \left\|f-\frac1{(2\pi)^d}\int_{\T^d}f(x)dx \right\|_q,\quad \g=0.
\end{equation}

\end{lemma}
\begin{remark}\label{remarkVn}
\textnormal{If we do not assume
$\mu\in C^\infty(\R^d\setminus \{0\})$ in Lemma \ref{v-}, then we only have
$$\|
\mathcal{D}(\mu)
f \|_q
\lesssim \|(-\Delta)^{-\gamma/2}f \|_q,\quad \g\neq 0.
$$}
%2)  It is clear that if  $\g\in \R_+$, then $f\in L_q(\T^d)$ implies $(-\Delta)^{-\gamma/2}f \in L_q(\T^d)$.

%If $f$ is some trigonometric polynomials, then assertion of Lemma~\ref{v-} holds for any $\g\in \R$.

\end{remark}

\begin{proof}
To obtain (\ref{delta-calc+}), we consider the function
$$
u(\xi)=|\xi|^\g \mu(\xi).
$$
By properties of homogeneous functions, we have that
the function $D^\nu u$  is  a homogeneous function of order $-|\nu|_1$ and it belongs to $C^\infty(\R^d \backslash \{0\})$ for any multi-index  $\nu \in \Z_+^d$. Hence,
%$D^\nu u \in H_{-|\nu|_1}$ и, следовательно,
 \begin{eqnarray*}
 %\label{4zvzv}
\sup_{\xi\in \R^d}|\xi|^{|\nu|_1}|D^\nu u(\xi)|<\infty.
 \end{eqnarray*}
Thus, by Lemma~\ref{lemMult}, we get that for any function  $f\in L_q(\T^d)$,
$1<q<\infty$,
 \begin{eqnarray*}
 %\label{mikh}
\Vert \mathcal{D}(u)f\Vert_q \lesssim \Vert f\Vert_q.
 \end{eqnarray*}
 The proof of
the reverse inequality is similar using
${1}/{\mu}\in C^\infty(\R^d\setminus \{0\})$.

Inequality~\eqref{delta-calc+++} can be obtained similarly.
\end{proof}

\begin{lemma}\label{v}
Let $1<q<\infty$ and $\g\ge 0$.
 We have, for any $n\in \N$,
\begin{equation}\label{delta-calc}
  \|
 (-\Delta)^{-\gamma/2}V_n
 \|_q\asymp
 %$$
%\s_{q,\g,d}(t):=
\left\{
         \begin{array}{ll}
           1, & \hbox{$\g>d\(1-\frac1q\)$}; \\
           \ln^\frac1{q} (n+1), & \hbox{$\g=d\(1-\frac1q\)$}; \\
           n^{d(1-\frac1q)-\g}, & \hbox{$ 0\le \g< d\(1-\frac1q\)$.}
%           \\n^{d(1-\frac1q)}, & \hbox{$\g=0$,}
         \end{array}
       \right.
\end{equation}

\end{lemma}

\begin{proof}
Let us prove~\eqref{delta-calc} only for $d=2$. For $d=1$ or $d>2$ the proof is similar.

In what follows, $\xi,\eta\in\R$,
$$
g_n(\xi,\eta):=\left\{
           \begin{array}{ll}
             \displaystyle \frac1{(\xi^2+\eta^2)^{\g/2}}, & \hbox{$|\xi|\le 2n,\quad |\eta|\le 2n,\quad |\xi|+|\eta|\neq 0$;} \\
             \displaystyle 0, & \hbox{otherwise,}
           \end{array}
         \right.
$$
and
$$
a_{k,l}=a_{k,l}^{(n)}=g_n(k,l)v\(\frac{|k|}{n}\)v\(\frac{|l|}{n}\), \qquad  a_{0,0}=0.
$$
We have
\begin{equation*}
\begin{split}
(-\D)^{-\g/2} V_n(x,y)&=\sum_{k\in \Z}\sum_{l\in\Z} a_{k,l}
e^{i(kx+ly)}
\\&=2\sum_{k=1}^{2n} a_{k,0}\cos
kx+2\sum_{l=1}^{2n} a_{0,l}\cos
ly+4\sum_{k=1}^{2n}\sum_{l=1}^{2n} a_{k,l}\cos kx\cos ly.
\end{split}
\end{equation*}
First, we estimate $\Vert (-\D)^{-\g/2} V_n\Vert_q$ from below as follows
\begin{equation*}
%    \label{2zv-v}
\Vert (-\D)^{-\g/2} V_n \Vert_{L_q(\T^2)}
\ge
4\bigg\Vert
\sum_{k=1}^{2n}\sum_{l=1}^{2n} a_{k,l}\cos kx\cos
ly\bigg\Vert_{L_q(\T^2)}
-
4(2\pi)^{1/q}\bigg\Vert \sum_{k=1}^{2n} a_{k,0}\cos
kx\bigg\Vert_{L_q(\T^1)}.
\end{equation*}
Using the Hardy--Littlewood theorem for series with monotone coefficients (see (\ref{monot})), we get
\begin{equation}
    \label{2zv-v}
\bigg\Vert \sum_{k=1}^{2n} a_{k,0}e^{ikx}\bigg\Vert_{L_q(\T^1)}
 \asymp \left(
\sum_{k=1}^{2n} \frac{k^{q-2}}{k^{\g q}}v\(\frac kn\)
\right)^{1/q}\asymp \s_{1}(n).
\end{equation}
Here we set
$$
\s_{d}(n):=\left\{
         \begin{array}{ll}
           1, & \hbox{$\g>d\(1-\frac1q\)$;} \\
           \ln^\frac1{q} (n+1), & \hbox{$\g=d\(1-\frac1q\)$;} \\
           n^{d(1-\frac1q)-\g}, & \hbox{$\g<d\(1-\frac1q\)$,}
           %\\
           %t^{d(1-\frac1q)}, & \hbox{$\g=0$,}
         \end{array}
       \right.
$$
where $d$ stands for the dimension.

To estimate the norm of the double sum, we use Lemma~\ref{monot+}, which is
a multidimensional analog of the Hardy-Littlewood theorem for series with monotone coefficients $a_{n,m}$ in the sense of Hardy, i.e., $\D^{(2)} a_{n,m}\ge 0$.
Recall that for a sequence $\{a_{n,m}\}_{n,m\in\N}$ the differences
$\Delta^{(2)} a_{n,m}$ is defined  as follows:
$$
\Delta^{1,0} a_{n,m}=a_{n,m}- a_{n+1,m},\qquad
\Delta^{0,1} a_{n,m}=a_{n,m}- a_{n,m+1},% \qquad  \Delta^{1,1}a_{n,m}=\Delta^{0,1} (\Delta^{1,0} a_{n,m}).
$$
$$
\Delta^{(2)}
a_{n,m}=\Delta^{0,1} (\Delta^{1,0} a_{n,m}).
$$
Since for $1\le k\le 2n-1$ and $1\le l\le 2n-1$, we have
\begin{equation}\label{4zvedy}
  \D^{(2)}
g_n(k,l)=\int_{k}^{k+1}\int_{l}^{l+1}\(\frac{\partial^2}{\partial \xi
\partial \eta}\(g_n(\xi,\eta)v\(\frac{\xi}{n}\)v\(\frac{\eta}{n}\)\)\)d\xi d\eta\ge 0,
\end{equation}
%$$
%\D^{(2)}
%g_n(k,l)=\int_{k}^{k+1}\int_{l}^{l+1}\(\frac{\partial^2}{\partial x
%\partial y}\(g_n(x,y)v\(\frac{x}{n}\)v\(\frac{y}{n}\)\)\)dx dy\ge 0,
%$$
then Lemma~\ref{monot+} implies
\begin{equation}\label{2zv-v-1}
  \begin{split}
\bigg\Vert \sum_{k=1}^{2n}\sum_{l=1}^{2n} a_{k,l}^{(n)}\cos
kx\cos ly\bigg\Vert
_{L_q(\T^2)}
&\asymp
  \left(
 \sum_{k=1}^{2n}\sum_{l=1}^{2n}
\frac{{k^{q-2}}{l^{q-2}}}{(k^2+l^2)^{q\g/2}}
\(v\(\frac{k}{n}\)v\(\frac{l}{n}\)\)^q\right)^{1/q}
\\
&\asymp
 \left(
 \sum_{k=1}^{n}\sum_{l=1}^{n}
\frac{{k^{q-2}}{l^{q-2}}}{(k^2+l^2)^{q\g/2}}
\right)^{1/q}
\\
&\asymp
 \left(
 \int_{1}^{n} \int_{1}^{n}
\frac{{(\xi\eta)^{q-2}}}{(\xi^2+\eta^2)^{q\g/2}}d\xi d\eta
\right)^{1/q}
\asymp
\s_{2}(n).
  \end{split}
\end{equation}

Assume first that $\g\le d(1-1/q)$. To complete the proof in this case, it only remains to use the fact that for sufficiently large $n$  and
$\g\le d(1-1/q)$ one has
$$
\s_{2}(n)- (2\pi)^{1/q} \s_{1}(n)\gtrsim \s_{2}(n),
$$
which follows from (\ref{2zv-v}) and (\ref{2zv-v-1}).

Let now $\g> d(1-1/q)$.
To estimate $\Vert (-\D)^{-\g/2} V_n\Vert_q$ from above, we use the inequality
$$
\Vert (-\D)^{-\g/2} V_n\Vert_{L_q(\T^2)}
 \lesssim
 \bigg\Vert
\sum_{k=1}^{2n}\sum_{l=1}^{2n} a_{k,l}\cos kx\cos
ly\bigg\Vert_{L_q(\T^2)}
+\bigg\Vert \sum_{k=1}^{2n} a_{k,0}\cos
kx\bigg\Vert_{L_q(\T^1)}  \lesssim
1,
$$
where we again used
 (\ref{2zv-v}) and (\ref{2zv-v-1}).
The estimate from below is given as follows
\begin{eqnarray*}
\begin{split}
\Vert (-\D)^{-\g/2} V_n\Vert_q &\gtrsim\Vert (-\D)^{-\g/2} V_n\Vert_1
\\&\gtrsim
\int_{\T^2} \Big(\sum_{k,l\in \Z} a_{k,l} e^{i(k-1)x +i(l-1)y}\Big) dx dy
\\&\gtrsim
%\gtrsim
a_{1,1}>0.
\end{split}
\end{eqnarray*}
%Hence, (\ref{delta-calc}) is shown.

\end{proof}

For the case $q=\infty$, the quantity
$\|
\mathcal{D}(\psi/\vp)
V_n \|_q
$
may have different growth properties depending on given $\psi$ and $\vp$ (see also Remark~\ref{remark+X}).
However, let us  show that under some natural conditions it may happen only in the one-dimensional case.

\begin{lemma}\label{v++}
Let $\a>0$, $\g\ge 0$,
$\psi\in \mathcal{H}_\a$, $\vp\in \mathcal{H}_{\a+\g}$, and
$\frac{\psi}{\varphi}\in C^\infty(\R^d\setminus \{0\})$.

{\rm (i)} If $d=1$, then

\begin{equation}\label{eqDV2nVn.0}
  \left\Vert \mathcal{D}\(\psi/\vp\)V_n\right\Vert_\infty\asymp \left\{
  \begin{array}{ll}
  1, & \hbox{$\g>1$}; \\
  1, & \hbox{$\g=1$ {and}\,  $\frac{\psi(\xi)}{\vp(\xi)}=A{|\xi|^{-\g}}\sign \xi$  for some  $A\in \C \setminus \{0\}$}; \\
  \ln (n+1), & \hbox{$\g=1$ {and}\,  $\frac{\psi(\xi)}{\vp(\xi)}\ne A{|\xi|^{-\g}}\sign \xi$  for any  $A\in \C \setminus \{0\}$}; \\
  n^{1-\g}, & \hbox{$0\le \g<1$,}
  \end{array}
  \right.
\end{equation}

{\rm (ii)} if $d\ge 2$ and  $\frac{\varphi}{\psi}\in C^\infty(\R^d\setminus \{0\})$, then
\begin{equation}\label{eqDV2nVn.0d}
  \left\Vert \mathcal{D}\(\psi/\vp\)V_n\right\Vert_\infty\asymp \left\{
                                                                  \begin{array}{ll}
                                                                    1, & \hbox{$\g>d$}; \\
                                                                    \ln n, & \hbox{$\g=d$}; \\
                                                                    n^{d-\g}, & \hbox{$0\le \g<d$}.
                                                                  \end{array}
                                                                \right.
\end{equation}
\end{lemma}

\begin{remark}
\textnormal{Note that without the condition $\frac{\varphi}{\psi}\in C^\infty(\R^d\setminus \{0\})$,
the right-hand side estimate in~\eqref{eqDV2nVn.0d}, i.e. $"\lesssim"$,  still holds. %in  (\ref{eqDV2nVn.0d}).
%in the case $d\ge 2$ and $0\le \g\le d$.
}
\end{remark}

\begin{proof}

(i) For $d=1$, we write
$$
\frac{\psi(\xi)}{\vp(\xi)}=\frac{1}{|\xi|^\g}\Phi\(u\),
$$
where $u=\frac{\xi}{|\xi|}\in\{-1;1\}$
and the function $\Phi(u)\neq 0$ is such that one of the following conditions hold:
\smallskip
\begin{enumerate}
  \item $\Phi(1)=\Phi(-1)=A$;
  \item $\Phi(1)=A$, $\Phi(-1)=-A$;
  \item $\Phi(1)=A$, $\Phi(-1)=B$, and $|A|\neq |B|$,
\end{enumerate}
where $A, B\in \C$.
\smallskip

In the first case, we have
$$
\|\mathcal{D}(\psi/\vp)
V_n\|_\infty=2|A|\bigg\|\sum_{k=1}^{2n}v\(\frac kn\)\frac1{k^\g}\cos
kx\bigg\|_\infty \asymp \left\{
                                                                  \begin{array}{ll}
                                                                    1, & \hbox{$\g>1$;} \\
                                                                    \ln n, & \hbox{$\g=1$;} \\
                                                                    n^{1-\g}, & \hbox{$0\le \g<1$.}
                                                                  \end{array}
                                                                \right.
$$
In the second case, if $\g\le 1$, by Bernstein's inequality, we get
\begin{equation*}
    \begin{split}
       \|\mathcal{D}(\psi/\vp) V_n\|_\infty&=2|A|\bigg\|\sum_{k=1}^{2n}v\(\frac kn\)\frac1{k^\g}\sin kx\bigg\|_\infty\\
       &\ge \frac{2|A|}{n} \bigg\|\sum_{k=1}^{2n}v\(\frac kn\)\frac1{k^{\g-1}}\cos kx\bigg\|_\infty
       \\
       &\gtrsim n^{1-\g}.
%\left\{
%                                                                  \begin{array}{ll}
%                                                                    1, & \hbox{$\g\ge 1$;} \\
%                                                                    n^{1-\g}, & \hbox{$0\le \g<1$.}
%                                                                  \end{array}
%                                                                \right.
     \end{split}
\end{equation*}
If $\g>1$, then
\begin{equation*}
  \begin{split}
     \|\mathcal{D}(\psi/\vp) V_n\|_\infty&\ge \frac{|A|}{\pi}\bigg\|\sum_{k=1}^{2n}v\(\frac kn\)\frac1{k^\g}\sin kx\bigg\|_1 \\
&\ge \frac{|A|}{\pi}\int_0^{2\pi}\(\sum_{k=1}^{2n}v\(\frac kn\)\frac1{k^\g}\sin kx\)e^{-ix} dx\gtrsim 1.
   \end{split}
\end{equation*}
At the same time, by using the boundedness of $\|V_n\|_1$, we obtain
\begin{equation*}
    \begin{split}
       \|\mathcal{D}(\psi/\vp) V_n\|_\infty
       =
            2 \Big\|V_n*\sum_{\nu=1}^{2n}\frac{\sin \nu x}{\nu^\gamma}\Big\|_\infty
    &\le 2
         \big\|V_n\big\|_1 \Big\|\sum_{\nu=1}^{2n}\frac{\sin \nu x}{\nu^\gamma}\Big\|_\infty
           \\  &\lesssim \left\{
                                                                  \begin{array}{ll}
                                                                    1, & \hbox{$\g\ge 1$;} \\
                                                                    n^{1-\g}, & \hbox{$0\le \g<1$}
                                                                  \end{array}
                                                                \right.
     \end{split}
\end{equation*}
(see, e.g., \cite[Ch. 5, \S 2]{Z}).

In the third case, assuming for definiteness that $|A|>|B|$, we obtain
\begin{equation}\label{zvezda+}
    \begin{split}
       \|\mathcal{D}(\psi/\vp) V_n\|_\infty&=\bigg\|A
       \mathop{{\sum}'}_{k=-2n}^{2n}v\(\frac kn\)\frac1{k^\g}e^{ikx}+(B-A){\sum_{k=-2n}^{-1}}v\(\frac kn\)\frac1{k^\g}e^{ikx}\bigg\|_\infty
       \\
       &\gtrsim 2|A|\sum_{k=1}^{2n}\frac{1}{|k|^\g}v\(\frac kn\)-|B-A|\sum_{k=-2n}^{-1}\frac{1}{|k|^\g}v\(\frac kn\)
       \\
       &\gtrsim \left\{
                                                                  \begin{array}{ll}
                                                                    1 , & \hbox{$\g\ge 1$;} \\
                                                                    \ln (n+1), & \hbox{$\g= 1$;} \\
                                                                    n^{1-\g}, & \hbox{$0\le \g<1$.}
                                                                  \end{array}
                                                                \right.
     \end{split}
\end{equation}
Estimate from above is similar to~\eqref{zvezda+}:
$$
\|\mathcal{D}(\psi/\vp) V_n\|_\infty\lesssim \left\{
                                                                  \begin{array}{ll}
                                                                    1 , & \hbox{$\g\ge 1$;} \\
                                                                    \ln (n+1), & \hbox{$\g= 1$;} \\
                                                                    n^{1-\g}, & \hbox{$0\le \g<1$.}
                                                                  \end{array}
                                                                \right.
$$

%Note also that in all cases we have the following estimate from above
%\begin{equation*}
%    \begin{split}
%       \|\mathcal{D}(\psi/\vp) V_n\|_\infty
%       \lesssim \left\{
%                                                                  \begin{array}{ll}
%                                                                    1 , & \hbox{$\g> 1$;} \\
%                                                                    \ln n, & \hbox{$\g= 1$;} \\
%                                                                    n^{1-\g}, & \hbox{$0\le \g<1$.}
%                                                                  \end{array}
%                                                                \right.
%     \end{split}
%\end{equation*}

Thus, summarizing the above estimates and noting that the case (2) corresponds to the case $\frac{\psi(\xi)}{\vp(\xi)}=A|\xi|^{-\g}\sign \xi$, we get~(\ref{eqDV2nVn.0}) for $d=1$.

(ii) Let now $d\ge 2$.
As above, we have
$
\frac{\psi(\xi)}{\vp(\xi)}=\frac{1}{|\xi|^\g}\Phi\(\frac{\xi}{|\xi|}\),
$
where $\Phi\in C^\infty ({\mathbb{S}}^{d-1})$.
Since $\frac\vp\psi \in C^\infty(\R^d\setminus \{0\})$ and $\frac{\psi}{\vp} \in C^\infty(\R^d\setminus \{0\})$,  we obtain that $\Phi(u)>0$ or $\Phi(u)<0$ for all $u\in {\mathbb{S}}^{d-1}$. Moreover, one can find a constant $C>0$ such that $|\Phi(u)|>C$ for all $u\in {\mathbb{S}}^{d-1}$.

Thus,
\begin{equation*}
    \begin{split}
\|\mathcal{D}(\psi/\vp) V_n\|_\infty
&\ge |\mathcal{D}(\psi/\vp) V_n(0)|
\\&\ge
C{\mathop{{\sum}'}_{k_1=-n}^n}\dots
{\mathop{{\sum}'}_{k_d=-n}^n}\frac1{|k|^\g}\gtrsim \left\{
                                                                  \begin{array}{ll}
                                                                    1, & \hbox{$\g>d$}; \\
                                                                    \ln (n+1), & \hbox{$\g=d$}; \\
                                                                    n^{d-\g}, & \hbox{$0\le \g<d$}.
                                                                  \end{array}
                                                                \right.
     \end{split}
\end{equation*}
On the other hand, $$
\|\mathcal{D}(\psi/\vp) V_n\|_\infty\le C{\mathop{{\sum}'}_{k_1=-2n}^{2n}}\dots
{\mathop{{\sum}'}_{k_d=-2n}^{2n}}\frac1{|k|^\g}\lesssim  \left\{
                                                                  \begin{array}{ll}
                                                                    1, & \hbox{$\g>d$}; \\
                                                                    \ln (n+1), & \hbox{$\g=d$}; \\
                                                                    n^{d-\g}, & \hbox{$0\le \g<d$},
                                                                  \end{array}
                                                                \right.
$$
which implies~\eqref{eqDV2nVn.0d}.
\end{proof}

We are now in a position to obtain an explicit form of the Hardy-Littlewood-Nikol'skii inequality in the case $\g=0$: for
 $p\le 1<q$,  see Lemma~\ref{lemgamma0}  and
 for $p<q\le 1$,
 Lemma~\ref{lemgamma0+}.

\begin{lemma}\label{lemgamma0}
  Let $0<p\le 1<q\le \infty$, $\a>0$, and $\psi,\,\vp\in \mathcal{H}_\a$. Let also
$\frac{\psi}{\varphi}\in C^\infty(\R^d\setminus \{0\})$ and
$\frac{\varphi}{\psi}\in C^\infty(\R^d\setminus \{0\})$. Then
\begin{equation}\label{abovebelow1}
  \sup_{T\in \mathcal{T}'_{n}
}\frac{\Vert \mathcal{D}(\psi)T\Vert_q}{\Vert
\mathcal{D}(\vp)T\Vert_p} \asymp n^{d(\frac1p-\frac1q)}.
\end{equation}
\end{lemma}

\begin{proof}
The estimate from above in~\eqref{abovebelow1} follows from Lemmas~\ref{lemHLd} and~\ref{v-} and Lemma~\ref{v} (the case $1<q<\infty$)
and Lemma~\ref{v++} (the case $q=\infty$).

Let us prove the estimate from below. For this, we denote
$$
V_n^\vp(x)=\mathcal{D}\(\frac1\vp\)(V_{Nn}(x)-V_n(x)),
$$
where  $N\in \N$ will be chosen later.

By~\eqref{*Vn}, we have
\begin{equation}\label{lele1}
   \Vert \mathcal{D}(\vp) V_n^\vp\Vert_{p}\lesssim n^{-d(\frac1p-1)}.
\end{equation}

At the same time, using Lemma~\ref{v-} and Lemma~\ref{v} (the case $1<q<\infty$)
and Lemma~\ref{v++} (the case $q=\infty$), we derive that
\begin{equation}\label{lele2}
\begin{split}
  \Vert \mathcal{D}(\psi) V_n^\vp \Vert_q&\ge \Vert \mathcal{D}(\psi/\vp) V_{Nn}^\vp \Vert_q-\Vert \mathcal{D}(\psi/\vp) V_{n}^\vp \Vert_q\\
  &\ge C_1 (Nn)^{d(1-\frac1q)}-C_2n^{d(1-\frac1q)}\ge n^{d(1-\frac1q)},
\end{split}
\end{equation}
where $N^{d(1-1/q)}\ge (C_2+1)/C_1$.
Thus, combining the inequality
\begin{equation*}
%\label{typetakoy}
  \sup_{T\in \mathcal{T}'_{2Nn}
}\frac{\Vert \mathcal{D}(\psi)T\Vert_q}{\Vert
\mathcal{D}(\vp)T\Vert_p}\ge \frac{\Vert \mathcal{D}(\psi)V_n^\vp\Vert_q}{\Vert \mathcal{D}(\vp)V_n^\vp\Vert_p}
\end{equation*}
with \eqref{lele1} and \eqref{lele2},
we obtain the estimate from below in~\eqref{abovebelow1}.
\end{proof}

\begin{lemma}\label{lemgamma0+}
  Let $0<p<q\le 1$, $\a>0$, and $\psi,\,\vp\in \mathcal{H}_\a$. Let also
$\frac{\psi}{\varphi}\in C^\infty(\R^d\setminus \{0\})$ and
$\frac{\varphi}{\psi}\in C^\infty(\R^d\setminus \{0\})$. Then
\begin{equation}\label{abovebelow1+}
\begin{split}
    \sup_{T\in \mathcal{T}'_{n}
}\frac{\Vert \mathcal{D}(\psi)T\Vert_q}{\Vert
\mathcal{D}(\vp)T\Vert_p}
%\\
%&\qquad\qquad
\asymp \left\{
                                     \begin{array}{ll}
                                       n^{d(\frac1p-1)}, & \hbox{$d=1$, $0<q<1$,\; {and} \; $\frac{\psi(\xi)}{\vp(\xi)}\not\equiv \const$;} \\
                                       n^{d(\frac1p-1)}\ln (n+1), & \hbox{$d=1$, $q=1$,\; {and} \; $\frac{\psi(\xi)}{\vp(\xi)}\not\equiv \const$;} \\
                                       n^{d(\frac1p-\frac1q)}, & \hbox{$d\ge 1$ and $\frac{\psi(\xi)}{\vp(\xi)}\equiv \const$.}
                                     \end{array}
                                   \right.
\end{split}
\end{equation}
\end{lemma}

\begin{proof}
The estimate from above in all cases follows from Lemmas~\ref{lemHLd} and~\ref{v++q<1}.

The estimate from below in~\eqref{abovebelow1+} in the case $\frac{\psi(\xi)}{\vp(\xi)}\equiv \const$ can be proved  the same way as in the proof of the corresponding estimate in Lemma~\ref{lemgamma0}. We only note that we use Lemma~\ref{v++q<1} instead of Lemmas~\ref{v-}, ~\ref{v}, and~\ref{v++}.

To prove the estimate from below in the case $d=1$ and $\frac{\psi(\xi)}{\vp(\xi)}\not\equiv \const$, we put
$$
V_n^\vp(x)=\mathcal{D}\(\frac1\vp\)\(\sum_{k\in \Z} v\(\frac kn\)e^{i(2k+1)x}\).
$$
Then,  by Corollary~\ref{lemBLX}, we  have
\begin{equation}\label{lele1+}
   \Vert \mathcal{D}(\vp) V_n^\vp\Vert_{p}=\Vert V_n (2 \cdot)\Vert_p\lesssim n^{-(\frac1p-1)}.
\end{equation}
Next, using the fact that $\psi/\vp$ is a homogeneous function of order zero, we obtain
\begin{equation}\label{dostalo}
  \Vert \mathcal{D}(\psi) V_n^\vp\Vert_{q}=\bigg\Vert \sum_{k\in \Z}\frac{\psi(2k+1)}{\vp(2k+1)}v\(\frac kn\)e^{i(2k+1)x}\bigg\Vert_q
  =\Vert  \widetilde{V}_n\Vert_{q},
\end{equation}
where
$$
\widetilde{V}_n(x)=\sum_{k\in \Z} f_n\(\frac kn\)v\(\frac kn\) e^{ikx},\quad f_n(y)=\frac{\psi(2y+\frac1n)}{\vp(2y+\frac1n)}.
$$
By properties of the Fourier transform,
$$
\widehat{f_n v}(\xi)={e^{i\frac{\xi}{2n}}} \widehat{f_\infty v_n}\(\xi\),\quad v_n(y)=v\(y-\frac1{2n}\).
$$
Thus, from~\eqref{eqRunovskii} (see also the proof of (4.7) in~\cite{RS}), it is easy to see that for sufficiently large $n$,
\begin{equation}\label{eq4++++}
  |\widehat{f_n v}(\xi)|\ge \frac{C}{|\xi|},\quad |\xi|>\rho,
\end{equation}
where $C$ and $\rho$ do not depend on $n$.

Now, using Lemma~\ref{lemTrigub+} and~\eqref{eq4++++}, by analogy with the proof of Lemma~\ref{v++q<1}, we derive
\begin{equation}\label{eqfinal++++}
  \Vert\widetilde{V}_n\Vert_q \gtrsim \left\{
                                                                      \begin{array}{ll}
                                                                        1, & \hbox{$0<q<1$;} \\
                                                                        \ln n, & \hbox{$q=1$.}
                                                                      \end{array}
                                                                    \right.
\end{equation}
It remains to combine inequalities~\eqref{lele1+}, \eqref{dostalo}, and~\eqref{eqfinal++++} with the inequality
\begin{equation*}
  \sup_{T\in \mathcal{T}'_{2n}
}\frac{\Vert \mathcal{D}(\psi)T\Vert_q}{\Vert
\mathcal{D}(\vp)T\Vert_p}\ge \frac{\Vert \mathcal{D}(\psi)V_n^\vp\Vert_q}{\Vert \mathcal{D}(\vp)V_n^\vp\Vert_p}.
\end{equation*}
\end{proof}

\subsection{Hardy--Littlewood--Nikol'skii $(L_p,L_q)$ inequalities for $1< p<q\le \infty$.}\label{Subsection 5.2}

\begin{lemma}\label{lemHLd++}
{\it Let $1< p<q\le \infty$, $\a>0$, $\g\ge 0$, $\psi\in
\mathcal{H}_\a$, and $\vp\in \mathcal{H}_{\a+\g}$.
Let also
$\frac{\psi}{\varphi}\in C^\infty(\R^d\setminus \{0\})$ and
$\frac{\varphi}{\psi}\in C^\infty(\R^d\setminus \{0\})$.
We have
\begin{equation}\label{eqlemHL1d++}
\eta(n)
:=\sup_{T\in \mathcal{T}'_{n}
}\frac{\Vert \mathcal{D}(\psi)T\Vert_q}{\Vert
\mathcal{D}(\vp)T\Vert_p} \asymp
\left\{
         \begin{array}{ll}
           1, & \hbox{$\g\ge d(\frac1p-\frac1q),\quad q<\infty$}; \\
           1, & \hbox{$\g> \frac dp,\quad q=\infty$}; \\
           \ln^\frac1{p'} (n+1), & \hbox{$\g=\frac dp,\quad q=\infty$}; \\
           n^{d(\frac1p-\frac1q)-\g}, & \hbox{$0\le \g<d(\frac1p-\frac1q)$}.\\
         \end{array}
       \right.
\end{equation}

}
\end{lemma}

\begin{proof}
 To estimate $\eta(n)$ from above in the case $q<\infty$,
we use the fact that
$$
\eta(n)\asymp \sup_{T\in \mathcal{T}'_{n}
}\frac{\Vert (-\D)^{-\g/2}T\Vert_q}{\Vert T\Vert_p}
$$
(see Lemma~\ref{v-}).

%we consider the inequality
%$$
%\Vert (-\D)^{-\g/2} T_n\Vert_q \le A(p,q,d,n) \Vert T_n\Vert_p,\qquad T_n\in\mathcal{T}_n'.
%$$
%As above, since by Lemma \ref{v-} we have $\eta(n)\lesssim A(p,q,d,n)$,
%it suffices to estimate  $A(p,q,d,n)$.

Let $\g\ge d(1/p-1/q)$ and $q<\infty$. Using the Hardy-Littlewood inequality for fractional integrals, see~\eqref{HLfrac}, and taking into account the fact that by Lemma~\ref{lemMult} the function
$$
u(\xi)=\frac{1-v(2|\xi|)}{|\xi|^{\g-d(\frac1p-\frac1q)}}
$$
is a Fourier multiplier in $L_p(\T^d)$, we arrive at
$$
\Vert (-\D)^{-\g/2} T_n\Vert_q\lesssim\Vert \mathcal{D}(u)
T_n\Vert_p\lesssim \Vert T_n\Vert_p.
$$

If $0<\g<d(1/p-1/q)$ and $q<\infty$, the required estimate follows from the Hardy--Littlewood  inequality for fractional integrals and the Bernstein inequality  (see, e.g.,~\cite{RS}):
\begin{eqnarray*}
\Vert (-\D)^{-\g/2} T_n\Vert_q&=&\Vert
(-\D)^{-d(\frac1p-\frac1q)/2}(-\D)^{(d(\frac1p-\frac1q)-\g)/2}
T_n\Vert_q
\\
&\lesssim& \Vert (-\D)^{(d(\frac1p-\frac1q)-\g)/2} T_n\Vert_p
\\
&\lesssim& n^{d(\frac1p-\frac1q)-\g}\Vert T_n\Vert_p.
\end{eqnarray*}

If $\gamma=0$ and $q<\infty$, the part $"\lesssim"$ in (\ref{eqlemHL1d++}) follows from
${\Vert \mathcal{D}(\psi)T\Vert_p}\lesssim{\Vert
\mathcal{D}(\vp)T\Vert_p}$ (see Lemma~\ref{v-}) and
the classical Nikol'skii inequality:
$$
\eta(n)\lesssim
\sup_{T\in \mathcal{T}'_{n}}
\frac{\Vert \mathcal{D}(\psi)T\Vert_q}{\Vert \mathcal{D}(\psi)T\Vert_p}
%\sup_{T\in \mathcal{T}'_{n}}\frac{\Vert T\Vert_q}{\VertT\Vert_p}
\lesssim
n^{d(\frac1p-\frac1q)}.
$$

Let now $\g\ge 0$ and $q=\infty$.
In this case, the proof follows from the inequality
\begin{eqnarray*}
\Vert \mathcal{D}({\psi}/{\varphi})T_n\Vert_\infty
&=&
\bigg\Vert (2\pi)^{-d} \int_{\T^d}
\Big(
\mathcal{D}({\psi}/{\varphi})V_n(t)\Big) T_n(x-t)dt \bigg\Vert_\infty
\\
&\lesssim& \Vert
\mathcal{D}({\psi}/{\varphi})V_n\Vert_{p'}\Vert T_n\Vert_p
\end{eqnarray*}
and Lemmas \ref{v-}  and \ref{v}.

Now, we study the estimate of $\eta(n)$ from below, i.e., the part $"\gtrsim"$ in
equivalence~\eqref{eqlemHL1d++}.

 If  $\g=0$ and $q<\infty$, then
 $\eta(n)\asymp
\sup_{T\in \mathcal{T}'_{n}
}\frac{\Vert T\Vert_q}{\Vert
T\Vert_p}$ and, therefore, \eqref{eqlemHL1d++} is the classical Nikol'skii's  inequality.
 The sharpness follows from
the  Jackson-type kernel example, see \cite[\S~4.9]{timan}:
$$T(x)=\left(\frac{\sin \frac{nt}2}{n \sin\frac{t}2}\right)^{2r}, \qquad r\in \N.$$
% for $1\le p \le \infty$; in the case of $0<p<1$ it is enough to take $r$ large enough ($r>\frac1{2p}$).

 If  $0< \g< d(1/p-1/q)$ and $q<\infty$, using Lemmas \ref{v-}  and \ref{v}, we estimate
$$
\eta(n)
\ge
\frac{
\Vert (-\D)^{-\g/2} ( V_{n/2}-1)\Vert_q}{\Vert V_{n/2}-1\Vert_p}
\ge
\frac{
\Vert (-\D)^{-\g/2} V_{n/2}\Vert_q}{\Vert V_{n/2}\Vert_p+ (2\pi)^{d/p}}
\gtrsim
           n^{d(\frac1p-\frac1q)-\g},\quad n\ge 2.
$$

If $\g\ge d(1/p-1/q)$ and $q<\infty$, we have
$$
\eta(n)
\ge
\frac{\Vert \mathcal{D}(\psi)T^*\Vert_q}{\Vert
\mathcal{D}(\vp)T^*\Vert_p} \asymp 1,\qquad T^*(x)=\cos x_1,
$$
and the desired result follows.

Assume now that $q=\infty$.
Set
$$
T_{n,\l}(x)=
\mathop{{\sum}'}_{|k|_\infty\le n}
\frac{e^{i(k,x)}}{|k|^\l},\quad \l>0.
$$
Note that Lemma \ref{v-}  implies that
\begin{equation}\label{eqlemHL1d+++}
\eta(n)
=
\sup_{T\in \mathcal{T}'_{n}
}\frac{\Vert
%\mathcal{D}(\psi)(T_{n,d})
T
\Vert_\infty}{\Vert
\mathcal{D}(\vp/\psi)(T)\Vert_p}
\gtrsim
\sup_{T\in \mathcal{T}'_{n}
}
\frac{\Vert
T
\Vert_\infty}{
\Vert (-\D)^{\g/2} T\Vert_p
}
.
\end{equation}
We divide the rest of the proof into three cases.

1) Let first $\g=d/p$.
Since
$$
\Vert
T_{n,d}
\Vert_\infty
\gtrsim
\mathop{{\sum}'}_{|k|_\infty\le n}
\frac{1}{|k|^d}
\gtrsim
\ln (n+1),
$$
then,
by \eqref{eqlemHL1d+++} and Lemma~\ref{v}, we have
\begin{eqnarray*}
   \eta(n)
\gtrsim
\frac{\Vert
T_{n,d}
\Vert_\infty}{
\Vert (-\D)^{\g/2} T_{n,d}\Vert_p
}
&\gtrsim&
\frac{\ln (n+1)}
{\Big\|{\mathop{{\sum}'}_{|k|_\infty\le n}
\frac{e^{i(k,x)}}{|k|^{d(1-1/p)}}}
\Big\|_p}
%{\ln (n+1)}
%\Big\|{{\mathop{{\sum}'}}_{|k|\le n}
%\frac{e^{i(k,x)}}{|k|^{d(1-\frac1p)}}}
%\Big\|_p^{-1}
 \\
&\gtrsim&
\frac{
\ln (n+1)
}{
\|
 (-\Delta)^{-d(1-1/p)/2}V_n
 \|_p}
 \\
&\gtrsim&
 \ln^\frac1{p'} (n+1).
\end{eqnarray*}

2) Let now $\g>d/p$. Considering
$T_{n,d+\g}$ and using again (\ref{eqlemHL1d+++}) and Lemma~\ref{v}, we get
\begin{eqnarray*}
   \eta(n)
\gtrsim
\frac{\Vert
T_{n,d+\g}
\Vert_\infty}{
\Vert (-\D)^{\g/2} T_{n,d+\g}\Vert_p
}
\gtrsim
\frac{\Vert
T_{n,d+\g}
\Vert_\infty}{
  \|
 (-\Delta)^{-d/2}V_n
 \|_p
}
\gtrsim
1.
%&\gtrsim&
%{\ln (n+1)}
%\Big\|{{\mathop{{\sum}'}}_{|k|\le n}
%\frac{e^{i(k,x)}}{|k|^{d(1-\frac1p)}}}
%\Big\|^{-1}
%\gtrsim
%\frac{
%\ln (n+1)
%}{
%\|
% (-\Delta)^{-\(d(1-\frac1p)\)/2}V_n
% \|_p}
% \\
%&\gtrsim&
%\frac{
%\ln (n+1)
%}{\ln^\frac1{p} (n+1)}
% =
%           \ln^\frac1{p'} (n+1).
\end{eqnarray*}

3) Finally, if  $0\le\g<d/p$, considering
$T_{n,\l}$ with $\l=\g+(d-d/p)/2$ and using the same procedure, we arrive at
\begin{eqnarray*}
   \eta(n)
&\gtrsim&
\frac
{\Vert
T_{n,\l}
\Vert_\infty}
{
\Vert (-\D)^{\g/2} T_{n,\l}\Vert_p
}
\\&\gtrsim&
\frac
{\Vert
T_{n,\l}
\Vert_\infty}
{
  \|
 (-\Delta)^{-(\l-\g)/2}V_n
 \|_p
}
\\&\asymp&
\frac{n^{d-\l}}{n^{d(1-\frac1p)-(\l-\g)}} =
n^{\frac dp-\g}.
\end{eqnarray*}
Here, we take into account that
$0<\l-\g<
d(1-1/p)$,
$0<{\l}<d$, and
$$
\Vert
T_{n,\l}
\Vert_\infty
\gtrsim
\mathop{{\sum}'}_{|k|_\infty\le n}
\frac{1}{|k|^\l}
\gtrsim
n^{d-\l}.
$$
\end{proof}

%\subsection{Hardy--Littlewood--Nikol'skii $(L_p,L_q)$ inequalities. The general case.}

\subsection{Hardy--Littlewood--Nikol'skii $(L_p,L_q)$ inequalities for directional derivatives}

%In this subsection we deal with the Hardy--Littlewood--Nikol'skii inequalities for the derivatives of the type $\(\frac{\partial}{\partial\xi}\)^{\a} T_n$.

For applications, in particular, to obtain the sharp Ulyanov inequality for moduli of smoothness, it is important to use the  Hardy-Littlewood-Nikol'skii inequalities for directional derivatives. Here the interesting case is $p\le 1$, otherwise see equivalence~\eqref{eqRAZZZ} in Corollary~\ref{lemRAZZZ}.

\begin{lemma}\label{lemma+}
Let $0<p\le 1$, $1<q<\infty$, $\a>0$, $\g>0$, and $\a+\g \neq 2k+1$, $k\in \Z_+$. Then
\begin{equation}\label{will+}
  \sup_{T_n\in\mathcal{T}_n'}\frac{\sup\limits_{|\xi|=1,\,\xi\in\R^d}\bigg\Vert \(\frac{\partial}{\partial\xi}\)^{\a} T_n\bigg\Vert_{q}}{\sup\limits_{|\xi|=1,\,\xi\in\R^d}\bigg\Vert \(\frac{\partial}{\partial\xi}\)^{\a+\g} T_n\bigg\Vert_p}\lesssim \s(n),
\end{equation}
where $\s(\cdot)$ is given as follows:

{\rm (1)} if $0<p\le 1$ and $1<q<\infty$, then
$$
\s(t)
:=\left\{
         \begin{array}{ll}
           t^{d(\frac1p-1)}, & \hbox{$\g> d\(1-\frac1q\)$}; \\
%           t^{d(\frac1p-1)}, & \hbox{$\g=d\(1-\frac1q\)\ge 1$ and $\a+\g\in \N$;} \\
%           t^{d(\frac1p-1)}\ln^\frac1{q} 2t, & \hbox{$\g=d\(1-\frac1q\)\ge 1$  and $\a+\g\not\in \N$;} \\
           t^{d(\frac1p-1)}\ln^\frac1{q} (t+1), & \hbox{$0<\g=d\(1-\frac1q\)$}; \\
           t^{d(\frac1p-\frac1q)-\g}, & \hbox{$0\le \g<d\(1-\frac1q\)$},
         %  t^{d(\frac1p-\frac1q)}, & \hbox{$\g=0$.
%           }
         \end{array}
       \right.
$$

{\rm (2)} if $1<p\le q<\infty$, then
$$
\s(t):=
\left\{
         \begin{array}{ll}
           1, & \hbox{$\g\ge d(\frac1p-\frac1q),\quad q<\infty$}; \\
           %1, & \hbox{$\g> \frac dp,\quad q=\infty$}; \\
           %\ln^\frac1{p'} 2t, & \hbox{$\g=\frac dp,\quad q=\infty$}; \\
           t^{d(\frac1p-\frac1q)-\g}, & \hbox{$0\le \g<d(\frac1p-\frac1q)$}.\\
         \end{array}
       \right.
$$
\end{lemma}

 \begin{proof}
First, we note that
$$
\Vert \mathcal{D}^{\a+\g} T_n\Vert_p\lesssim
\sum_{j=1}^d \Big\Vert \(\frac{\partial}{\partial x_j}\)^{\a+\g} T_n\Big\Vert_p
\lesssim
\sup_{|\xi|=1,\,\xi\in\R^d}\bigg\Vert \(\frac{\partial}{\partial\xi}\)^{\a+\g} T_n\bigg\Vert_p,
$$
where
$$
\mathcal{D}^{\a+\g} T_n (x):=\sum_{j=1}^d \(\frac{\partial}{\partial x_j}\)^{\a+\g} T_n(x)=\sum_{|k|_\infty\le n} \vp(k) \widehat{(T_n)}_k e^{i(k,x)},
$$
 $\widehat{(T_n)}_k$ is the $k$-th Fourier coefficient of $T_n$, and
$\vp(y):=(iy_1)^{\a+\g}+\dots+(iy_d)^{\a+\g}$.
Thus, using Corollary~\ref{lemRAZZZ}, we obtain
\begin{equation}\label{+1}
  \sup_{T_n\in\mathcal{T}_n'}\frac{\sup\limits_{|\xi|=1,\,\xi\in\R^d}\bigg\Vert \(\frac{\partial}{\partial\xi}\)^{\a} T_n\bigg\Vert_{q}}{\sup\limits_{|\xi|=1,\,\xi\in\R^d}\bigg\Vert \(\frac{\partial}{\partial\xi}\)^{\a+\g} T_n\bigg\Vert_p}\lesssim
\sup_{{T\in\mathcal{T}_{n}'}}\frac{\Vert (-\D)^{\a/2} T_n\Vert_q}{\Vert\mathcal{D}^{\a+\g}T_n\Vert_p}.
\end{equation}

Note that
\begin{multline}\label{eqMM4}
    \vp(y)=\cos\frac{({\a+\g})\pi}{2}\(|y_1|^{\a+\g}+\dots+|y_d|^{\a+\g}\)\\
    +i\(\sin\frac{({\a+\g}) \pi \sign y_1}{2}|y_1|^{\a+\g}+\dots+\sin\frac{({\a+\g}) \pi \sign y_d}{2}|y_d|^{\a+\g}\).
\end{multline}
Hence, it is easy to see that $\Re \vp(y)\neq 0$ for $y\neq 0$ and ${\a+\g} \neq 2k+1$, $k\in \Z_+$.

Next, by using Lemma~\ref{lemHLd}, we obtain for
$0<p\le 1$ and $1<q<\infty$
\begin{equation}\label{eqMM5}
    \sup_{T\in \mathcal{T}_n'}\frac{\Vert (-\D)^{\a/2} T_n\Vert_q}{\Vert\mathcal{D}^{\a+\g}T_n\Vert_p}\lesssim
    n^{d(1/p-1)}\Vert \mathcal{D}(\psi/\vp) V_n\Vert_q,
\end{equation}
where
$$
\frac{\psi(y)}{\vp(y)}=\frac{|y|^\a}{(iy_1)^{\a+\g}+\dots+(iy_d)^{\a+\g}}.
$$
Let us show that
\begin{equation}\label{eqMM6}
\Vert \mathcal{D}(\psi/\vp) V_n\Vert_q \lesssim \Vert (-\D)^{-\g/2} V_n \Vert_q.
\end{equation}
For this, it is sufficient to verify that the function
$$
h(y)=\frac{|y|^{\a+\g}}{(iy_1)^{\a+\g}+\dots+(iy_d)^{\a+\g}}
$$
is a Fourier multiplier in $L_q(\T^d)$. This easily follows from (\ref{eqMM4}) and  Lemma~\ref{lemMult}.
Thus, combining (\ref{+1}), (\ref{eqMM5}), and (\ref{eqMM6})
and using the bounds for $\Vert (-\D)^{-\g/2} V_n \Vert_q$
from Lemma~\ref{v},
 we complete the proof.
\end{proof}

%Let us first prove two additional lemmas.
Recall that the homogeneous Sobolev norm is given by
$
\Vert f \Vert_{\dot W_p^{r}}=\sum_{|\nu|_1=r}\Vert D^\nu f\Vert_p.
$

\begin{lemma}\label{lemVspom} Let $d\ge 1$.

{\textnormal {(i)}} We have, for $1/{q^*}=(d-1)/d$,
\begin{equation}\label{eqStein}
  \Vert T\Vert_{q^*}\le \frac1d
  \sum_{j=1}^d \left\Vert \frac{\partial T}{\partial x_j}\right\Vert_1,\quad T\in \mathcal{T}'.
\end{equation}

{\textnormal {(ii)}} We have %for  $0<q\le\infty$
\begin{equation}\label{eqVspom}
\Vert T\Vert_\infty \lesssim \Vert T\Vert_{\dot W_1^d}\,,\quad T\in
\mathcal{T}'.
\end{equation}

\end{lemma}

\begin{proof}
(i) To prove \eqref{eqStein} see~\cite[pp.~129--130]{Stein}.

(ii) First, we use the limiting case of the Sobolev embedding theorem (see, e.g., \cite[Ch.~10]{besov1}, \cite[Theorem~33]{pel}) given by
$$
\Vert T\Vert_\infty \lesssim \Vert T\Vert_{\dot W_1^d}+\|T\|_1.%\,,\quad T\in jjj
$$
Then by (i), we estimate
$$
\|T\|_1\lesssim \|T\|_{q^*}\lesssim \sum_{j=1}^d \left\Vert \frac{\partial T}{\partial x_j}\right\Vert_1.
%\Vert f\Vert_\infty \lesssim \Vert f\Vert_{\dot W_1^d}+\|f\|_1.%\,,\quad T\in jjj
$$
Thus, (\ref{eqVspom}) follows from
\begin{equation}\label{eqVspom'}
  \Big\Vert\frac{\partial T}{\partial x_j}\Big\Vert_1\lesssim \Big\Vert\frac{\partial^2 T}{\partial^2 x_j}\Big\Vert_1\lesssim\dots\lesssim   \Big\Vert\frac{\partial^d T}{\partial x_j^d}\Big\Vert_1, \quad T\in\mathcal{T}',\quad j=1,\dots,d.
\end{equation}

It remains to show~\eqref{eqVspom'}. We write
% We have for any $j=1,\dots,d$ and $z_j\in \T$ that
\begin{equation*}
\begin{split}
      \frac{\partial T(x)}{\partial x_j}=\int_{z_j}^{x_j}\frac{\partial^2}{\partial x_j^2}
    &T(x_1,\dots,x_{j-1},t_j,x_{j+1},\dots,x_d)dt_j\\
&+\frac{\partial T(x)}{\partial x_j}(x_1,\dots,x_{j-1},z_j,x_{j+1},\dots,x_d).
\end{split}
\end{equation*}
Then integrating over $(0,2\pi)$ by $z_j$ and using the fact that
$\frac{\partial T(x)}{\partial x_j}$ as a polynomial of $x_j$ belongs to $ \mathcal{T}'$, we obtain
$$%
    2\pi \frac{\partial T(x)}{\partial x_j}=\int_{0}^{2\pi}
    \int_{z_j}^{x_j}\frac{\partial^2 T}{\partial x_j^2} T(x_1,\dots,x_{j-1},t_j,x_{j+1},\dots,x_d)\, dt_j
    dz_j,
$$%\end{equation}
which gives
\begin{equation*}
%\label{eqRavenstvoT}
\Big\Vert\frac{\partial T}{\partial x_j}\Big\Vert_1\lesssim \Big\Vert\frac{\partial^2 T}{\partial x_j^2}\Big\Vert_1.
\end{equation*}
%, \qquad T\in\mathcal{T}',\qquad
Repeating this procedure $d-1$ times, we complete the proof of~\eqref{eqVspom'}.
% {\bf Yura, Change} To prove the lemma we only need to use
%the equality
%\begin{equation}\label{eqRavenstvo}
%  T(x)=\int_0^{x_1}\dots\int_{0}^{x_d} \frac{\partial^d}{\partial x_1  \dots      \partial x_d} T(t_1,\dots,t_d) dx_1\dots dx_d, \quad T\in \mathcal{T}'.
%\end{equation}

\end{proof}

Now, we proceed with the case $\a+\g\in \N$ and $\g\ge 1$. It turns out that, in this case, an interesting effect in the Hardy-Littlewood inequality occurs. More precisely, in the limiting case $\g=d(1-1/q)$, we can obtain an improved version of Lemma~\ref{lemma+}.

\begin{lemma}\label{lemPolSob}
Let $0<p\le 1$, $1<q\le\infty$, $d\ge 2$, $\a>0$, $\g\ge 1$, and
$\a+\g\in \N$. Suppose also that $\a\in\N$ if $q=\infty$.
Then
\begin{equation}\label{eqlemMM2Sob}
  \sup_{T_n\in\mathcal{T}_n'}\frac{\sup\limits_{|\xi|=1,\,\xi\in\R^d}\bigg\Vert \(\frac{\partial}{\partial\xi}\)^{\a} T_n\bigg\Vert_{q}}{\sup\limits_{|\xi|=1,\,\xi\in\R^d}\bigg\Vert \(\frac{\partial}{\partial\xi}\)^{\a+\g} T_n\bigg\Vert_p}\lesssim n^{d(\frac1p-1)+\left(d(1-\frac1q)-\g\right)_+}.
\end{equation}
\end{lemma}

\begin{remark} \textnormal{ %Note that  inequality (\ref{eqlemMM2Sob})
%and (\ref{eqlemMM2Sob_inft}) can be viewed as analogues of  Hardy-Littlewood-Nikol'skii's  inequalities (see section \ref{hardy-inequality-section})
%but in the homogeneous Sobolev spaces. In particular,
It is worth mentioning that inequality (\ref{eqlemMM2Sob}) in the case of $\g=d(1-1/q)$ gives shaper bound than both inequality (\ref{will+}) and the following relation
\begin{equation}\label{zvezda v kruge}
   \sup_{T_n\in \mathcal{T}_n'}
\frac{\Vert \mathcal{D}(\psi)T_n \Vert_{q}}{\Vert \mathcal{D}(\vp)T_n \Vert_{q}}\asymp
n^{d(\frac1p-1)}\ln^{\frac 1q}(n+1)\quad \text{if}\quad
0<p\le 1<q<\infty,
\end{equation}
which follows from Theorem~\ref{th-hardy-l-n}. We assume in~\eqref{zvezda v kruge} that
$\psi\in \mathcal{H}_\a$, $\vp\in \mathcal{H}_{\a+\g}$, and
$\frac{\psi}{\varphi}, \frac{\varphi}{\psi}\in C^\infty(\R^d\setminus \{0\})$.
%(see Lemmas \ref{lemHLd}, \ref{v-}, and \ref{v})
Note that the left-hand sides of~\eqref{eqlemMM2Sob} and~\eqref{zvezda v kruge} are not equivalent.
}

%Moreover, inequality (\ref{eqlemMM2Sob_inft})
%with  $0<p\le 1$ and $q=\infty$
% provides  an explicit  bound unlike the following  estimate
%$$ \sup_{T_n\in \mathcal{T}_n'}
%\frac{\Vert T_n \Vert_{\dot W_q(\psi)}}{\Vert T_n \Vert_
%{\dot W_p(\varphi)}}\asymp
%n^{d(\frac1p-1)}\Big\Vert \mathcal{D}\(\frac{\psi}{\varphi}\) V_{n}
%\Big\Vert_{q}, $$
%which depends on choice of $\varphi$ and $\psi$ (see Lemma \ref{lemHLd}).

\end{remark}

\begin{proof}
By Corollary~\ref{lemRAZZZ}
 and Nikol'skii's inequality, we have
\begin{equation*}
  \Vert T_n\Vert_{\dot W_1^{\a+\g}}\lesssim \sup_{|\xi|=1,\,\xi\in\R^d}\bigg\Vert \(\frac{\partial}{\partial\xi}\)^{\a+\g} T_n\bigg\Vert_1
\lesssim n^{d(\frac1p-1)} \sup_{|\xi|=1,\,\xi\in\R^d}\bigg\Vert \(\frac{\partial}{\partial\xi}\)^{\a+\g} T_n\bigg\Vert_p
\end{equation*}
%Thus, taking again into account Corollary~\ref{lemRAZZZ}, we obtain
and
\begin{equation*}
%\label{eqlemMM2Sob+}
  \frac{\sup\limits_{|\xi|=1,\,\xi\in\R^d}\bigg\Vert \(\frac{\partial}{\partial\xi}\)^{\a} T_n\bigg\Vert_{q}}{\sup\limits_{|\xi|=1,\,\xi\in\R^d}\bigg\Vert \(\frac{\partial}{\partial\xi}\)^{\a+\g} T_n\bigg\Vert_p}
  \lesssim
  {n^{d(\frac1p-1)}}
  I(T_n),
\end{equation*}
where
$$I(T_n):=
\left\{
  \begin{array}{ll}
    \displaystyle\frac{\Vert (-\D)^{\a/2}T_n \Vert_{q}}{\Vert T_n \Vert_{\dot W_1^{\a+\g}}}, & \hbox{$1<q<\infty$;} \\
    \displaystyle\frac{\Vert T_n \Vert_{\dot W_q^\a}}{\Vert T_n \Vert_{\dot W_1^{\a+\g}}}, & \hbox{$q=\infty$.}
  \end{array}
\right.
$$
Let us first consider the case $1<q<\infty$.  The proof of~\eqref{eqlemMM2Sob} is in five steps.

1) Suppose that $\g=d\(1-1/q\)\ge
1$. Choose $q^*\in (1,q]$ such that
$$
\g-1=d\(\frac1{q^*}-\frac1q\).
$$
Then,  using the Hardy-Littlewood inequality for fractional integrals and  Lemma~\ref{lemVspom} (i),
we have
\begin{equation}\label{equsingStein}
  \begin{split}
\Vert (-\D)^{\a/2}T_n \Vert_{q} &\lesssim \Vert (-\D)^{(\a+\g-1)/2}T_n \Vert_{q^*}\lesssim \Vert T_n\Vert_{\dot W_{q^*}^{\a+\g-1}}\\
&\lesssim \Vert T_n\Vert_{\dot W_{1}^{\a+\g}},
   \end{split}
\end{equation}
%The last inequality implies the estimation from above in (\ref{eqlemMM2Sob}).
that is, $I(T_n)\lesssim 1$.

2) If $1\le \g<d(1-1/q)$, then we can choose $\widetilde{q}\in
(1,q)$ such that
$$
\g=d\(1-\frac1{\widetilde{q}}\).
$$
Then,  using Nikol'skii's inequality and \eqref{equsingStein}, we
get
\begin{equation*}
    \Vert (-\D)^{\a/2} T_n\Vert_q\lesssim n^{d(\frac1{\widetilde{q}}-\frac1{q})}\Vert (-\D)^{\a/2}
T_n\Vert_{\widetilde{q}}\lesssim
n^{d(1-\frac1q)-\g}\Vert T_n\Vert_{\dot W_{1}^{\a+\g}},
\end{equation*}
i.e., $I(T_n)\lesssim n^{d(1-1/q)-\g}$.

3) If $d\(1-1/q\)<\g<d$, then we can choose $q_*>q$ such that
$$
\g=d\(1-\frac1{q_*}\)>d\(1-\frac1q\).
$$
Then, using H\"older's inequality and  (\ref{equsingStein}), we
have
\begin{equation*}
  \Vert (-\D)^{\a/2}T_n \Vert_{q} \lesssim \Vert (-\D)^{\a/2}T_n \Vert_{q_*}\lesssim \Vert T_n\Vert_{\dot W_{1}^{\a+\g}},
\end{equation*}
i.e., $I(T_n)\lesssim 1$.

4) If $\g>d$, then we have
$$
(-\D)^{\a/2} T_n(x)=(-\D)^{(\a+\g-d)/2}T_n * f_{\g-d}(x),
$$
where
$$
f_{\g-d}(x)=\sum_{k\neq 0} \frac{e^{i(k,x)}}{|k|^{\g-d}}.
$$
Note that $f_{\g-d}\in L_q(\T^d)$ (see (\ref{wainger})). Thus, from the above and inequality (\ref{eqVspom}) we obtain
\begin{equation*}
%\label{equsingStein+}
  \begin{split}
\Vert (-\D)^{\a/2}T_n \Vert_{q} &\lesssim \Vert (-\D)^{(\a+\g-d)/2}T_n \Vert_{q} \Vert f_{\g-d}\Vert_1 \\
&\lesssim \Vert T_n \Vert_{\dot W_q^{\a+\g-d}}\lesssim \Vert T_n \Vert_{\dot W_\infty^{\a+\g-d}}\lesssim \Vert T_n
\Vert_{\dot W_1^{\a+\g}},
   \end{split}
\end{equation*}
which gives $I(T_n)\lesssim 1$.

5) Assume that  $\g=d$. Noting that $\a\in \N$,  Lemma~\ref{lemVspom} (ii)  gives
 \begin{equation*}
 %\label{equsingStein++++}
  \begin{split}
\Vert (-\D)^{\a/2}T_n \Vert_{q}\lesssim
\Vert T_n \Vert_{\dot W_q^{\a}}\lesssim
\Vert T_n \Vert_{\dot W_1^{\a+d}},
   \end{split}
\end{equation*}
i.e., $I(T_n)\lesssim 1$.

Finally, let us consider the case $q=\infty$.
If $\g\ge d$, then applying Lemma~\ref{lemVspom} (ii), and using~\eqref{eqVspom'}
 in the case $\g> d$, we obtain
$$
\Vert T_n\Vert_{\dot W_\infty^\a}\lesssim \Vert T_n\Vert_{\dot
W_1^{\a+d}}\lesssim \Vert T_n\Vert_{\dot W_1^{\a+\g}}.
$$
If $1\le \g<d$, then we choose $\widetilde{q}>1$ such that
$$
\g=d\(1-\frac1{\widetilde{q}}\).
$$
Now, applying the Nikol'skii inequality and~\eqref{equsingStein}, we get
$$
\Vert T_n\Vert_{\dot W_\infty^\a}\lesssim n^{\frac d{\widetilde{q}}}\Vert
T_n\Vert_{\dot W_{\widetilde{q}}^\a}=n^{d-\g}\Vert T_n\Vert_{\dot
W_{\widetilde{q}}^\a}\lesssim n^{d-\g}\Vert (-\D)^{\a/2}T_n\Vert_{\widetilde{q}}\lesssim n^{d-\g}\Vert T_n\Vert_{\dot
W_1^{\a+\g}}.
$$
Thus, we have proved (\ref{eqlemMM2Sob})
for all cases.

%The lemma is proved.

\end{proof}

\bigskip

\newpage

\section{General form of the Ulyanov inequality for moduli of smoothness, $K$-functionals, and their realizations %: main properties
%Sharp Ulyanov inequalities for the generalized $K$-functionals and realizations
}
\label{sec6}

Let $\psi\in \mathcal{H}_\a$, $\a>0$, and let ${W_p(\psi)}$ be the space of $\psi$-smooth functions in $L_p$, that is,
$$
W_p(\psi)=\left\{g\in L_p(\T^d)\,:\, \mathcal{D}(\psi)g\in L_p(\T^d)\right\}.
$$
% with$$\Vert g\Vert_{X_p(\psi)}=....$$
We define the generalized $K$-functional by
$$
K_{\psi}(f,\d)_p=\inf_{g\in  W_p(\psi)}\left\{\Vert
f-g\Vert_p+\d^{\a}\Vert \mathcal{D}(\psi)g\Vert_p\right\}.
$$
%where
%$$
%\mathcal{T}_n=\left\{T\,:\,T(t)=\sum_{|k|\le n} c_k e^{i(k,t)},\qquad t\in \T^d \right\}.
%$$
%We will also use the following notation
%$$ \mathcal{T}'_n=\left\{T\in
%\mathcal{T}_n\,:\,\int_{\T^d}T(t) dt=0 \right\}. $$

It is known that ${K}_{\psi}(f,\d)_p=0$ if $0<p<1$. This was shown for $d=1$, $\psi(\xi)=(i\xi)^r$, $r\in \N$, in \cite{DHI} and in the general case in \cite{run}. The suitable substitution of the $K$-functional is given by the realization concept
%Its realization is given by % of this $K$-functional
$$
\mathcal{R}_{\psi}(f,\d)_p=\inf_{T\in \mathcal{T}_{[1/\d]}}
\left\{\Vert
f-T\Vert_p+\d^{\a}\Vert \mathcal{D}(\psi)T\Vert_p\right\}.
$$

We list below several important properties of the $K$-functionals and their realizations.

\begin{lemma}\label{lemma4.1} {\sc(\cite[Theorem 4.21]{run})}  Let $\psi\in \mathcal{H}_\a$, $\a>0$, and $1\le p \le\infty$.
We have $$\mathcal{R}_{\psi}(f,\d)_p\asymp {K}_{\psi}(f,\d)_p.$$
\end{lemma}

The next lemma is a simple corollary of  Lemma~\ref{lemma4.1}.

\begin{lemma}\label{rrrun} {\sc(\cite[Lemma 4.10]{run})}
Let $\psi\in \mathcal{H}_\a$, $\a>0$, and $1< p<\infty$.
We have
$$
\mathcal{R}_\psi(f,\delta)_p\asymp \inf_{(-\Delta)^{\a/2}g \in L_p} \big\{\|f-g\|_p+ \delta^\alpha \|(-\Delta)^{\a/2}g\|_p\}.
$$
\end{lemma}

The following result was proved in \cite[Theorem 4.24]{run} (with the remark that if $\psi$ is a polynomial one can consider any $p>0$).
%We will use the following natation
%$$d_\a=\left\{
%       \begin{array}{ll}
%         \frac d{d+\a}, & \hbox{$\a\not\in\N$;} \\
%         0, & \hbox{$\a\in\N$.}
%       \end{array}
%     \right.
%$$

\begin{lemma} \label{real}
Let $\psi\in \mathcal{H}_\a$, $\a>0$, and either $\psi$ be a polynomial and $0< p\le\infty$ or
 $d/{(d+\a)}< p\le\infty$.
Then
$$
\mathcal{R}_{\psi}(f,\d)_p\asymp
\Vert
f-T\Vert_p+\d^{\a}\Vert \mathcal{D}(\psi)T\Vert_p,
$$
where $T\in \mathcal{T}_n, \,\,n=[1/\delta],\,$ is such that
$$\Vert
f-T\Vert_p\lesssim E_n(f)_p.$$
\end{lemma}
The proof is standard using the Bernstein inequality
\begin{equation}\label{BerRunSch}
  \Vert
\mathcal{D}(\psi)T\Vert_p \lesssim n^{\a}\Vert T\Vert_p
\end{equation}
%$$
%\Vert
%\mathcal{D}(\psi)T\Vert_p \lesssim n^{\a}\Vert T\Vert_p,
%$$
(see  \cite{RS}).

The next lemma easily follows from the definition of the generalized
$K$-functional.
\begin{lemma}\label{rrun}
Let $\psi\in \mathcal{H}_\a$, $\a>0$, and $1\le p \le\infty$.
We have
$${K}_\psi(f,\delta)_{p}
\lesssim
 \delta^{\a}\Vert
\mathcal{D}(\psi)f\Vert_{p}.
$$
\end{lemma}

\begin{lemma}\label{monoto} {\sc (\cite[Theorem 4.22]{run})}
Let $\psi\in \mathcal{H}_\a$, $\a>0$, and $0< p\le \infty$. We have
$$
\mathcal{R}_\psi(f,\delta)_p\le\mathcal{R}_\psi(f,n\delta)_p\lesssim
n^{\alpha+d\big(\frac{1}{p}-1\big)\!_+}
\mathcal{R}_\psi(f,\delta)_p, \quad n\in\N.
$$
\end{lemma}
 % {\bf \cite{run} Disser? Eur j m ?}

From Lemma~\ref{real} and inequality~\eqref{BerRunSch}, it is easy to obtain the following result.

\begin{lemma}\label{newdopsvoistvo}
Let $\a_1\ge \a_2>0$, $\psi_1\in \mathcal{H}_{\a_1}$, $\psi_2\in \mathcal{H}_{\a_2}$,  and $0< p\le \infty$. If $\psi_1/\psi_2\in C^\infty(\R^d\setminus\{0\})$ and  either $\psi_1/\psi_2$ is a polynomial or $d/{(d+\a_1-\a_2)}< p\le\infty$. Then
$$
\mathcal{R}_{\psi_1}(f,\d)_p\lesssim \mathcal{R}_{\psi_2}(f,\d)_p.
$$
\end{lemma}

The next result provides the general form of the sharp Ulyanov inequality for realizations. It will play the key role in our further study.
\begin{theorem}\label{mainlemma} %\begin{thma}
%\begin{theorem}
\label{th1dK0}
 Let $f\in L_p(\T^d)$, $0<p<q\le\infty$, $\a>0$, $\g\ge
0$, $\psi\in \mathcal{H}_\a$, and $\vp\in \mathcal{H}_{\a+\g}$.
\\
%{\bf S: Случай $\gamma=0$! }\\
   {\rm (A)} For any  $\d\in
(0,1)$,  we have
    \begin{equation}\label{eqth1.1Kd0}
        \mathcal{R}_{\psi}(f,\d)_q
        \lesssim
        \frac{\mathcal{R}_{\vp}(f,\d)_p}{\d^{\g}}\eta\(\frac
        1\d\)+
        \left(\int_0^\d
        \bigg(
        \frac{\mathcal{R}_{\vp}(f,t)_p}{t^{d(\frac1p-\frac1q)}}
        \bigg)^{q_1}\frac{dt}{t}
        \right)^{\frac1{q_1}},
    \end{equation}
where
$$
\eta(t)
:=
\eta(t;\psi,\vp,p,q,d)
%\s_{q,p,d}^{(\psi,\vp)}(t)
:=\sup_{T\in \mathcal{T}'_{[t]}
}\frac{\Vert \mathcal{D}(\psi)T\Vert_q}{\Vert
\mathcal{D}(\vp)T\Vert_p}.
$$
\\
{\rm (B)} Let
either $\psi$ be a polynomial
or $d/{(d+\a)}< q\le\infty$.
Let also
% Пусть дополнительно $d_{\a+\g}<p<\infty$ и $d_\a<q\le\infty$,
 $\a+\g>d(1/p-1/q)$, and
$$
n^{d(1/p-1/q)-\g} \lesssim \eta(n)\quad\text{as}\quad n\to\infty.$$
Then there exists a sequence of nontrivial polynomials $T_n\in \mathcal{T}_n$, $n\in \N$, such that
\begin{equation}\label{eqth1.B1}
            \mathcal{R}_{\psi}(T_n,\delta)_q\asymp
        \frac{\mathcal{R}_{\vp}(T_n,\delta)_p}{\delta^{\g}}
        \eta\(\frac
        1\d\)+
        \left(\int_0^{\delta}
        \bigg(
        \frac{\mathcal{R}_{\vp}(T_n,t)_p}{t^{d(\frac1p-\frac1q)}}
        \bigg)^{q_1}\frac{dt}{t}
        \right)^{\frac1{q_1}},
\end{equation}
where
$n=[1/\delta].$
%\end{theorem}
\end{theorem}

\begin{remark}
{\textnormal{
(i) The condition
 $\a+\g>d(1/p-1/q)$ is natural since
 by Lemma~\ref{monoto}  it is clear that $t^{\a+\g} \mathcal{R}_{\varphi}(f,1)_p\lesssim \mathcal{R}_{\varphi}(f,t)_p$, $t\in (0,1)$,
  $1\le p\le \infty$.
   Therefore, if
  $$\int_0^1
        \bigg(
        \frac{\mathcal{R}_{\vp}(f,t)_p}{t^{d(\frac1p-\frac1q)}}
        \bigg)^{q_1}\frac{dt}{t}<\infty,
$$
then $\a+\g>d(1/p-1/q)$.
\\
(ii) The condition $n^{d(1/p-1/q)-\g} \lesssim \eta(n)$ as $n\to\infty$ is also natural since
 it always holds if $\frac{\psi}{\varphi}\in C^\infty(\R^d\setminus \{0\})$ (see Theorem~\ref{th1dK}).
}}
\end{remark}

\begin{proof}
(A) Let $T_n\in \mathcal{T}_n$, $n\in \N$,  be such that
\begin{equation}\label{eqth1.1Kd0--}
\Vert f-T_n\Vert_p+n^{-(\a+\g)}\Vert \mathcal{D}(\vp)T_n\Vert_p\le 2
\mathcal{R}_{\vp}(f,n^{-1})_p.
\end{equation}
From  the Nikol'skii inequality~\eqref{nik},
%\begin{equation}\label{eqth1Nikd0}
%\Vert T\Vert_q\le C n^{d(\frac1p-\frac1q)}\Vert T\Vert_p,\quad
%0<p<q<\infty,\quad T\in\mathcal{T}_n,
%\end{equation}
it follows that (see \cite[Lemma 4.2]{diti})
\begin{equation}\label{eqth++}\begin{split}
\Vert f-T_{2^n}\Vert_q
&\lesssim
 \(\sum_{\nu=n}^\infty{2^{\nu q_1
d(\frac1p-\frac1q)}}\Vert
f-T_{2^{\nu}}\Vert_p^{q_1}\)^{\frac1{q_1}}\\
&\lesssim
 \(\sum_{\nu=n}^\infty{2^{\nu q_1
d(\frac1p-\frac1q)}}\mathcal{R}_{\vp}(f,2^{-\nu})_p^{q_1}\)^{\frac1{q_1}}.
\end{split}
\end{equation}
By the definition of $\eta (t)$ and~\eqref{eqth1.1Kd0--}, we get
\begin{equation}\label{eqth1.3Kd0}
\begin{split}
2^{-\a n}\Vert \mathcal{D}(\psi)T_{2^n}\Vert_q&\le \eta(2^{n})2^{-\a
n}\Vert
\mathcal{D}(\vp)T_{2^n}\Vert_p\\
%&\le C\s_{q,\g,d}^{(\psi,\vp)}(2^{n})2^{(-\a +d(\frac1p-1))n}\Vert
%\mathcal{D}(\vp)T_{2^n}\Vert_p\le\\
&\le 2 \eta(2^{n})2^{\g n}\mathcal{R}_{\vp}(f,2^{-n})_p.
\end{split}
\end{equation}
Thus, taking into account (\ref{eqth1.1Kd0--}), (\ref{eqth++}), and
(\ref{eqth1.3Kd0}), we obtain
\begin{equation*}
\begin{split}
\mathcal{R}_{\psi}(f,2^{-n})_q
        &\le \Vert f-T_{2^{n}}\Vert_q+2^{-\a n}\Vert
        \mathcal{D}(\psi)T_{2^{n}}\Vert_q\\
        &\lesssim  2^{\g n} \mathcal{R}_{\vp}(f,2^{-n})_p \eta(2^{n})+ \(\sum_{\nu=n}^\infty \({2^{\nu
d(\frac1p-\frac1q)}}\mathcal{R}_{\vp}(f,2^{-\nu})_p\)^{q_1}\)^{\frac1{q_1}}.
\end{split}
 \end{equation*}
From the last inequality,  (\ref{eqth1.1Kd0}) follows immediately.

(B) Let us prove the second part of the theorem. We choose a sequence of polynomials
$T_n$, $n\in \N$, such that
\begin{equation}\label{eqth1.B2}
    \eta(n)\lesssim \frac{\Vert \mathcal{D}(\psi)T_n\Vert_q}{\Vert
\mathcal{D}(\vp)T_n\Vert_p}.
\end{equation}
Further,
by Lemma \ref{real}, we have
\begin{equation}\label{eqth1.B3}
\mathcal{R}_\psi(T_n,1/n)_q\gtrsim n^{-\a}\Vert
\mathcal{D}(\psi)T_n\Vert_q.
\end{equation}
Using the definition of the realization of $K$-functional, we get
\begin{equation}\label{eqth1.B4}
\begin{split}
B_n:&=\mathcal{R}_\vp(T_n,1/n)_p n^\g \eta(n)
+
\left(\int_0^{1/n}
        \bigg(
        \frac{\mathcal{R}_{\vp}(T_n,t)_p}{t^{d(\frac1p-\frac1q)}}
        \bigg)^{q_1}\frac{dt}{t}
        \right)^{\frac1{q_1}}
        \\
&\lesssim
\Vert \mathcal{D}(\vp)T_n\Vert_p n^{-\a}\eta(n)
+
\Vert
\mathcal{D}(\vp)T_n\Vert_p
\left(\int_0^{1/n}
        \bigg(
        \frac{t^{\alpha+\gamma}}{t^{d(\frac1p-\frac1q)}}
        \bigg)^{q_1}\frac{dt}{t}
        \right)^{\frac1{q_1}}
%n^{-\a-\g+d(1/p-1/q)}
\\
&\lesssim  \Vert
\mathcal{D}(\vp)T_n\Vert_p n^{-\a}\eta(n).
\end{split}
\end{equation}
Thus, from (\ref{eqth1.B2}), (\ref{eqth1.B3}), and
(\ref{eqth1.B4}) we derive
\begin{equation*}
\mathcal{R}_\psi(T_n,1/n)_q \gtrsim   n^{-\a}\Vert \mathcal{D}(\psi)T_n\Vert_q\gtrsim
n^{-\a}\eta(n) \Vert
\mathcal{D}(\vp)T_n\Vert_p
%\frac{\Vert \mathcal{D}(\psi)(T_n)\Vert_q}{\Vert \mathcal{D}(\vp)(T_n)\Vert_p \s(n)}
\gtrsim
{B_n}.
\end{equation*}
This yields  (\ref{eqth1.B1}) by the monotonicity property given in Lemma~\ref{monoto}.
%\begin{equation*}
%    \frac{A_n}{B_n}\gtrsim \frac{\Vert \mathcal{D}(\psi)(T_n)\Vert_q}{\Vert
%\mathcal{D}(\vp)(T_n)\Vert_p \s(n)}\gtrsim 1.
%\end{equation*}
%The lemma is proved.
\end{proof}

Taking into account Lemma~\ref{lemma4.1}, we obtain the following result.

\begin{corollary}\label{CORCORCOR}
  Under all conditions of Theorem~\ref{mainlemma} if $p\ge 1$,
    we have
    \begin{equation*}\label{eqth1.1Kd0CORCORCOR}
        K_{\psi}(f,\d)_q
        \lesssim
        \frac{K_{\vp}(f,\d)_p}{\d^{\g}}\eta\(\frac
        1\d\)+
        \left(\int_0^\d
        \bigg(
        \frac{K_{\vp}(f,t)_p}{t^{d(\frac1p-\frac1q)}}
        \bigg)^{q_1}\frac{dt}{t}
        \right)^{\frac1{q_1}}.
    \end{equation*}
\end{corollary}

The next corollary easily follows from~\eqref{eqth1.1Kd0} and Nikol'skii's inequality~\eqref{nik}.

\begin{corollary}\label{CORCORCOR22}
  Under all conditions of Theorem~\ref{mainlemma}, if $\psi(x)=\vp(x)$,
   we have
    \begin{equation}\label{eqth1.1Kd0CORCORCOR22}
        \mathcal{R}_{\vp}(f,\d)_q
        \lesssim
        \left(\int_0^\d
        \bigg(
        \frac{\mathcal{R}_{\vp}(f,t)_p}{t^{d(\frac1p-\frac1q)}}
        \bigg)^{q_1}\frac{dt}{t}
        \right)^{\frac1{q_1}}.
    \end{equation}
\end{corollary}

Following the proof of Theorem~\ref{mainlemma} and taking into account Lemmas~\ref{real} and~\ref{newdopsvoistvo}, it is easy to prove a slightly more general form of Theorem~\ref{mainlemma}.

\medskip

{\sc Theorem~\ref{mainlemma}$'$.} {\it Suppose that under conditions of Theorem~\ref{mainlemma}, the function $\vp$ is either a polynomial or $d/(d+\a+\g)<p\le\infty$. Let also $m>0$ and the function $\phi\in \mathcal{H}_{\a+m}$ be either a polynomial or $d/(d+\a+m)<p\le\infty$. Then
    \begin{equation*}
    %\label{eqth1.1Kd0+++}
        \mathcal{R}_{\psi}(f,\d)_q
        \lesssim
        \frac{\mathcal{R}_{\vp}(f,\d)_p}{\d^{\g}}\eta\(\frac
        1\d\)+
        \left(\int_0^\d
        \bigg(
        \frac{\mathcal{R}_{\phi}(f,t)_p}{t^{d(\frac1p-\frac1q)}}
        \bigg)^{q_1}\frac{dt}{t}
        \right)^{\frac1{q_1}}.
    \end{equation*}}

An analogue of Theorem~\ref{mainlemma}$'$ for the moduli of smoothness reads as follows.

\begin{theorem}\label{lemMainMod}
    Let $f\in L_p(\T^d)$, $d\ge 1$, $0<p<q\le \infty$,
    $\a>0$, $\g,m\ge 0,$ $\a\in \N\cup ((1/q-1)_+,\infty)$, and $\a+\g, \a+m\in \N\cup ((1/p-1)_+,\infty)$.
    Then,  for any $\d \in (0,1)$, we have
    \begin{equation*}
    %\label{eqlemMM1}
        \w_\a(f,\d)_p\lesssim \w_{\a+\g}(f,\d)_p\eta\(\frac 1\d\)+\(\int_0^\d   \(  \frac{\w_{\a+m}(f,t)_p}{t^{d(\frac1p-\frac1q)}} \)^{q_1}\frac {dt}{t}\)^\frac 1{q_1},
    \end{equation*}
    where
    \begin{equation*}
    %\label{eqlemMM2}
    \eta(t):=\sup_{{T\in\mathcal{T}_{[t]}'}}\frac{\w_\a(T,1/t)_q}{\w_{\a+\g}(T,1/t)_p}.
    \end{equation*}
\end{theorem}

%\begin{proof}
The proof goes along the same line as in Theorem~\ref{mainlemma}
 for the realization of the $K$-functionals.
%\end{proof}

\newpage

\section{Sharp Ulyanov inequalities for  $K$-functionals and realizations
%\section{Sharp Ulyanov inequality for realizations of the K-functionals
}
\label{sec7}
The goal of this section is to prove the following theorem, which provides an explicit form of the sharp Ulyanov inequality.
Here, we assume that $\psi/\varphi$ is a smooth function (cf.~Theorem~\ref{mainlemma}).
%satisfies some additional conditions.

\begin{theorem}\label{th1dK}
  {Let  $0<p<q\le\infty$, $\a>0$, $\g\ge 0$, $\psi\in \mathcal{H}_\a$, $\vp\in \mathcal{H}_{\a+\g}$, and $\frac \psi\vp\in C^\infty(\R^d \backslash \{0\})$.
\\
  {\rm (A)}  Let $f\in L_p(\T^d)$. Then, for any  $\d\in (0,1)$, we have
    \begin{equation}\label{eqth1.1}
        \mathcal{R}_{\psi}(f,\d)_q
     \lesssim \frac{\mathcal{R}_{\vp}(f,\d)_p}{\d^{\g}}\s\(\frac1\d\)+
        \left(\int_0^\d
        \bigg(
        \frac{\mathcal{R}_{\vp}(f,t)_p}{t^{d(\frac1p-\frac1q)}}
        \bigg)^{q_1}\frac{dt}{t}
        \right)^{\frac1{q_1}},
    \end{equation}
%где $C$ -- константа, не зависящая от $\d$ и $f$,
where
%$\s=\s_{q,p,d}^{(g)}$ определяется следующим образом:

\textnormal{(A1)} \quad if $0<p\le 1$ and $p<q<\infty$, then
$$
\s(t):=\left\{
         \begin{array}{ll}
           t^{d(\frac1p-1)}, & \hbox{$\g>d\(1-\frac1q\)_+$}; \\
           t^{d(\frac1p-1)}\ln^\frac1{q} (t+1), & \hbox{$0<\g=d\(1-\frac1q\)_+$}; \\
           t^{d(\frac1p-\frac1q)-\g}, & \hbox{$0< \g<d\(1-\frac1q\)_+$};\\
           t^{d(\frac1p-\frac1q)}, & \hbox{$\g=0$, $0<p\le 1 <q<\infty$};\\
             t^{d(\frac1p-\frac1q)}, & \hbox{$\g=0$, $0<q\le 1$, and $\frac{\psi(\xi)}{\vp(\xi)}\equiv \const$};\\
             t^{d(\frac1p-1)}, & \hbox{$\g=0$, $0<q<1$, {and} $\frac{\psi(\xi)}{\vp(\xi)}\not\equiv \const$}; \\
             t^{d(\frac1p-1)}\ln (t+1), & \hbox{$\g=0$,  $q=1$, {and}  $\frac{\psi(\xi)}{\vp(\xi)}\not\equiv \const$},
         \end{array}
       \right.
$$
%\quad
%\l_{q,d}(t):=\left\{
%               \begin{array}{ll}
%                 \ln^\frac1q 2t, & \hbox{$1<q<\infty, d\ge 1$;} \\
%                 \ln 2t, & \hbox{$q=\infty, d=1$;} \\
%                 t^{d-2}, & \hbox{$q=\infty, d\ge 2$.}
%               \end{array}
%             \right.

\textnormal{(A2)} \quad if $0<p\le 1$ and $q=\infty$, then
$$
\s(t):=\left\{
         \begin{array}{ll}
           t^{d(\frac1p-1)}, & \hbox{$\g>d$}; \\
           t^{d(\frac1p-1)}, & \hbox{$\g=d=1$ {and}  $\frac{\psi(\xi)}{\vp(\xi)}=A{|\xi|^{-\g}}\sign \xi$ for some $A\in \C \setminus \{0\}$}; \\
           t^{d(\frac1p-1)}\ln (t+1), & \hbox{$\g=d=1$ {and} $\frac{\psi(\xi)}{\vp(\xi)}\ne A{|\xi|^{-\g}}\sign \xi$ for any $A\in \C \setminus \{0\}$}; \\
           t^{d(\frac1p-1)}\ln (t+1), & \hbox{$\g=d\ge 2$}; \\
           t^{d(\frac1p-\g)}, & \hbox{$0\le  \g<d$},
                    \end{array}
       \right.
$$

\textnormal{(A3)}  \quad if $1<p<q\le \infty$, then
$$
\s(t):=
\left\{
         \begin{array}{ll}
           1, & \hbox{$\g\ge d(\frac1p-\frac1q),\quad q<\infty$}; \\
           1, & \hbox{$\g> \frac dp,\quad q=\infty$}; \\
           \ln^\frac1{p'} (t+1), & \hbox{$\g=\frac dp,\quad q=\infty$}; \\
           t^{d(\frac1p-\frac1q)-\g}, & \hbox{$0\le \g<d(\frac1p-\frac1q)$}.\\
         \end{array}
       \right.
$$
%(отметим что в 3) мы не предполагаем, что $g=\frac \psi\vp\in C^\infty(\R^d \backslash \{0\})$ ) --- ???.
%\\
{\rm (B)}
Inequality  \eqref{eqth1.1} % при $0<p\le 1$ и $p<q\le \infty$ или $1<p<q<\infty$
  is sharp in the following sense.
Let $\psi$ be either a polynomial or  $d/{(d+\a)}<q\le \infty$. Let % $\psi(x)=0$ iff $x=0$ and
$\frac \vp\psi\in C^\infty(\R^d \backslash \{0\})$ and
$\a+\g>d(1-1/q)_+$.
 Let also $\psi(x)=C\vp(x)$ for some constant $C$ in the case $\g=0$ and $0<p<q\le 1$
or
$\g=0$ and $0<p\le 1$, $q=\infty$.
Then there exists a function $f_0\in L_q(\T^d)$, $f_0\not\equiv \const$,
such that
    \begin{equation}\label{eqth1.2Kd}
        \mathcal{R}_{\psi}(f_0,\d)_q
        \asymp  \frac{\mathcal{R}_{\vp}(f_0,\d)_p}{\d^\g}\s\(\frac1\d\)+
        \left(\int_0^\d
        \bigg(
        \frac{\mathcal{R}_{\vp}(f_0,t)_p}{t^{d(\frac1p-\frac1q)}}
        \bigg)^{q_1}\frac{dt}{t}
        \right)^{\frac1{q_1}}
    \end{equation}
as $\d\to 0$.}
\end{theorem}

In light of Lemma~\ref{lemma4.1}, Theorem \ref{th1dK} implies  the following result.

\begin{corollary}\label{CORCORCORpart}
  Under all conditions of Theorem~\ref{th1dK} if $p\ge 1$,
  we have
    \begin{equation*}\label{eqth1.1Kd0CORCORCORpart}
        K_{\psi}(f,\d)_q
        \lesssim
        \frac{K_{\vp}(f,\d)_p}{\d^{\g}}\s\(\frac1\d\)+
        \left(\int_0^\d
        \bigg(
        \frac{K_{\vp}(f,t)_p}{t^{d(\frac1p-\frac1q)}}
        \bigg)^{q_1}\frac{dt}{t}
        \right)^{\frac1{q_1}}.
    \end{equation*}
\end{corollary}

\begin{remark}\label{remark+}
\textnormal{(i)} Note that in the proof of Theorem \ref{th1dK},
in \textnormal{(A1)} with $\g=0$, $0<p<q\le 1$ and in \textnormal{(A2)} we may not assume that $\frac \psi\vp\in C^\infty(\R^d \backslash \{0\})$.

\textnormal{(ii)} %The construction of the  sharpness in the Main Lemma
It is important to remark that in part (B) of  Theorem~\ref{mainlemma},
we show sharpness of the corresponding Ulyanov type inequality by constructing a  sequence of functions
which depends on $\delta$, while in Theorem
\ref{th1dK} we construct a function $f_0$ which is independent of~$\d$.

\end{remark}

Similarly to  Theorem~\ref{mainlemma}$'$, one can obtain the following more general analogue of Theorem~\ref{th1dK}.

\medskip

{\sc Theorem~\ref{th1dK}$'$}. {\it Suppose that under conditions of Theorem~\ref{th1dK}, the function $\vp$ is either a polynomial or $d/(d+\a+\g)<p\le\infty$. Let also $m>0$ and the function $\phi\in \mathcal{H}_{\a+m}$ be either a polynomial or $d/(d+\a+m)<p\le\infty$. Then
    \begin{equation}\label{eqth1.1Kd0+++}
        \mathcal{R}_{\psi}(f,\d)_q
        \lesssim
        \frac{\mathcal{R}_{\vp}(f,\d)_p}{\d^{\g}}\s\(\frac
        1\d\)+
        \left(\int_0^\d
        \bigg(
        \frac{\mathcal{R}_{\phi}(f,t)_p}{t^{d(\frac1p-\frac1q)}}
        \bigg)^{q_1}\frac{dt}{t}
        \right)^{\frac1{q_1}}.
    \end{equation}
}

\begin{proof}[Proof of Theorem~\ref{th1dK}]
(A) Taking into account Theorem~\ref{mainlemma},
 it is enough to estimate
$$
\eta(t)=\sup_{T\in \mathcal{T}'_{[t]}
}\frac{\Vert \mathcal{D}(\psi)T\Vert_q}{\Vert
\mathcal{D}(\vp)T\Vert_p}.
$$
Since $\frac \psi\vp\in C^\infty(\R^d \backslash \{0\})$,
the inequality $$\eta(t)\lesssim \s(t)$$ follows from  Theorem~\ref{th-hardy-l-n}  and Remark~\ref{remark-HLN} in all cases except the case $0<p<q\le 1$, $\g=0$, $d\ge 2$, and $\frac{\psi(\xi)}{\vp(\xi)}\not\equiv \const$. In the latter case, the proof of this inequality follows from Corollary~\ref{corKg0} stated below.

\medskip

To prove {(B)}, we will construct a nontrivial function  $f_0\in L_q(\T^d)$ such that equivalence (\ref{eqth1.2Kd}) holds as $\d\to 0$.
%We split the proof into three part.

\bigskip

\noindent
\underline{Part (B1): $0<p\le 1$ and $p<q\le\infty$.}
We set
\begin{equation}\label{eqfun0ples1}
    f_0(x)=F_{\vp}(x)-\frac1{N^{\a+\g}}F_{\vp}(Nx),
\end{equation}
where
\begin{equation}\label{eqhvpd}
    F_{\vp}(x)\sim\sum_{k\neq 0}\frac{e^{i(k,x)}}{\vp(k)}
\end{equation}
and $N$ is a fixed sufficiently large integer that will be chosen later.

Let us show that  $F_{\vp} \in L_{q^*}(\T^d)$, where $q^*=\max(q,1)$. For this, we
consider the function
$$
\mathfrak{F}_{\alpha+\g}(x)\sim \sum_{k\neq 0}\frac{e^{i(k,x)}}{|k|^{\a+\g}}.
$$
It is known that $\mathfrak{F}_{\alpha+\g}\in L_{q^*}(\T^d)$ for $\a+\g>d(1-1/{q^*})$
(see~\eqref{wainger}).
At the same time, the function
$$
u(y)=\frac{|y|^{\a+\g-\e}}{\vp(y)}\(1-v(2|y|)\),\quad 0<\e<\a+\g,
$$
is a Fourier multiplier in  $L_r(\T^d)$ for any $1\le r\le\infty$.
This follows from  Lemma~\ref{lemMult} (see also Lemma~\ref{lemBes}) and the fact that $\varphi(y)=0$ if and only if  $y=0$.
Take $\e$ such that $\a+\g-\e>d(1-1/{q^*})$. We have that
$$
{\sum_{k\neq 0}}\frac{e^{i(k,x)}}{|k|^{\a+\g-\e}}\in L_{q^*}(\T^d).
$$
Thus, taking into account that
$$
F_\vp(x)\sim {\sum_{k\neq 0}}\frac{u(k)}{|k|^{\a+\g-\e}}e^{i(k,x)},
$$
we obtain that
 $F_{\vp} \in L_{q^*}(\T^d)$ and, therefore, $f_0 \in L_{q^*}(\T^d)$.

Next,
\begin{equation}\label{eqth1.4Kd}
    \mathcal{R}_\vp(f_0,n^{-1})_p\lesssim
    \Vert
    f_0-f_0*V_n\Vert_p+n^{-\a-\g}\Vert \mathcal{D}(\vp) (f_0*V_n)\Vert_p.
\end{equation}
Recall that $V_n$ is given by~\eqref{vallee}.

Let us consider the second summand in~(\ref{eqth1.4Kd}). In light of
\begin{equation*}
\begin{split}
 (F_\varphi(N\cdot)*V_n)(x)
 &=
 \sum_{k\neq 0} \prod_{j=1}^d v\(\frac{N|k_j|}n\)\frac{e^{i(Nk,x)}}{\varphi(k)}
\end{split}
\end{equation*}
and since $\vp$ is homogeneous of order $\a+\g$, we obtain
\begin{equation*}
\begin{split}
\mathcal{D}(\vp) (f_0*V_n)(x)&=\sum_{k\neq 0} \prod_{j=1}^d
v\(\frac{|k_j|}n\)e^{i(k,x)}-\sum_{k\neq 0} \prod_{j=1}^d
v\(\frac{N|k_j|}n\)e^{i(Nk,x)}
\\
&=V_n(x)-V_{n/N}(Nx).
\end{split}
\end{equation*}
Using the last equalities and Corollary~\ref{lemBLX}, we derive that
\begin{equation}\label{eqth1.5Kd}
\Vert \mathcal{D}(\vp)\big( f_0*V_n\big)
\Vert_p
\lesssim
\Vert V_n\Vert_p
+
\Vert V_{n/N}\Vert_p
\lesssim
n^{d(1-\frac1p)}.
\end{equation}

Now, let us consider the first summand in~(\ref{eqth1.4Kd}). We estimate
\begin{equation}\label{eqth1.6Kd}
    \Vert f_0-f_0*V_n\Vert_p^p\le \Vert F_{\vp}-F_{\vp} *
    V_n\Vert_p^p+N^{-p(\a+\g)}\Vert F_{\vp}(N\cdot)-F_{\vp}(N\cdot)*
    V_n\Vert_p^p.
\end{equation}
It is enough to estimate the first term in~(\ref{eqth1.6Kd}).
We have
\begin{eqnarray*}
\Vert F_{\vp}-F_{\vp} *
    V_n\Vert_p&=&\bigg\Vert
    \sum_{k\in\Z^d}\bigg(1-\prod_{j=1}^d v\bigg(\frac{|k_j|}n\bigg)\bigg)\frac{e^{i(k,x)}}{\vp(k)}\bigg\Vert_p\\&=&\frac1{n^{\a+\g}}\bigg\Vert
    \sum_{k\in\Z^d} \xi\(\frac kn\)e^{i(k,x)}\bigg\Vert_p,
\end{eqnarray*}
where
$$
\xi(y)=\frac{1-\prod_{j=1}^d v(|y_j|)}{\vp(y)}.
$$

By Lemma~\ref{lemFLpBes}, using the fact that  $\xi\in\dot{B}_{2,p}^{d(\frac1p-\frac12)} \cap
\dot{B}_{2,1}^{\frac d2}(\R^d)$ (see, e.g., Lemma~\ref{lemBes}), and the
homogeneity property of the Besov (quasi-)norm (see Lemma~\ref{besov}), %$\dot{B}_{2,p}^{d(\frac1p-\frac12)}(\R^d)$ (см.,напр.,~\cite[c.339]{TribF}),
we get
\begin{equation}\label{eqNNN}
\begin{split}
\Vert F_{\vp}-F_{\vp} *
    V_n\Vert_p&\lesssim
     n^{-\a-\g}\left\Vert \xi\(\frac{y}{n}\)\right\Vert_{\dot{B}_{2,p}^{d(\frac1p-\frac12)}(\R^d)}
    \\&\lesssim
    n^{d(1-\frac1p)-\a-\g}\Vert
    \xi\Vert_{\dot{B}_{2,p}^{d(\frac1p-\frac12)}(\R^d)}\lesssim
    n^{d(1-\frac1p)-\a-\g}.
\end{split}
\end{equation}
Hence, from (\ref{eqth1.6Kd}) we have
\begin{equation}\label{eqth1.7Kd}
\Vert
f_0-f_0*V_n
\Vert_p\lesssim n^{d(1-\frac1p)-\a-\g}.
\end{equation}

Thus, combining  (\ref{eqth1.4Kd}), (\ref{eqth1.5Kd}), and
(\ref{eqth1.7Kd}), we arrive at
\begin{equation}\label{eqth1.8Kd}
    \mathcal{R}_\vp(f_0,n^{-1})_p\lesssim n^{d(1-\frac1p)-\a-\g}.
\end{equation}

Let us assume that
\begin{equation}\label{eqth1.9Kd}
\mathcal{R}_\psi(f_0,1/n)_q\gtrsim  n^{-\a}
\Vert
\mathcal{D}(\psi/\vp) V_{ n}\Vert_q,\quad n\in \N, %,\qquad 0<q\le\infty.
\end{equation}
and
 \begin{equation}\label{sigma+}
 t^{d(1/p-1/q)-\g}=\mathcal{O}(\eta(t))\qquad\mbox{as}\qquad t\to\infty,
\end{equation}
where $\eta$ is given by
$$\eta(t)
=
%\eta(t;\psi,\vp,p,q,d):=
\sup_{T\in \mathcal{T}'_{[t]}}\frac{\Vert \mathcal{D}(\psi)T\Vert_q}{\Vert
\mathcal{D}(\vp)T\Vert_p}.
$$
We will prove~\eqref{eqth1.9Kd} and~\eqref{sigma+} later.

Now, comparing estimates (\ref{eqth1.8Kd}) and (\ref{eqth1.9Kd}), we derive
 the following inequalities for $\d\to 0$
    \begin{eqnarray}\label{eqth1.2Kd2}\nonumber
        \mathcal{R}_{\psi}(f_0,\d)_q
        &\gtrsim&  \d^{\a}\Vert \mathcal{D}(\psi/\vp) V_{1/\d} \Vert_q
                \qquad\qquad\qquad\qquad\qquad\qquad \qquad\quad\,(\mbox{by }\quad (\ref{eqth1.9Kd})
)
        \nonumber
\\
        &=&
      {\d^{\a+d(\frac1p-1)}}
        \frac{1}{\d^{d(\frac1p-1)}}\Vert \mathcal{D}(\psi/\vp) V_{1/\d} \Vert_q
        \nonumber
\\
        &\gtrsim&   {\d^{\a+d(\frac1p-1)}}\eta ({1/\d})\qquad\qquad\qquad\qquad\qquad\qquad \qquad\qquad(\mbox{by Lemma
}\;\; \ref{lemHLd})
        \nonumber
        \\
        &\gtrsim&   {\d^{\a+d(\frac1p-1)}}\s ({1/\d})+{\d^{\a+d(\frac1p-1)}}\eta ({1/\d})\qquad\qquad\qquad\qquad\!\!(\mbox{by Theorem}\quad \ref{th-hardy-l-n})
        \nonumber
\\
        &\gtrsim&  \frac{\mathcal{R}_{\vp}(f_0,\d)_p}{\d^\g} %\d^{\a+d(\frac1p-1)}}
        \s ({1/\d})+ {\d^{\a+d(\frac1p-1)}}\eta ({1/\d})\quad\qquad \qquad\qquad(\mbox{by }\quad (\ref{eqth1.8Kd}))
        \nonumber
\\
        &\gtrsim&   \frac{\mathcal{R}_{\vp}(f_0,\d)_p}{\d^\g} %\d^{\a+d(\frac1p-1)}}
        \s ({1/\d}) + \d^{
        \a+\g+d(\frac1q-1) }
         \qquad\qquad\qquad\qquad\quad(\mbox{by }\quad (\ref{sigma+}))
        \nonumber
\\
        &\gtrsim&  \frac{\mathcal{R}_{\vp}(f_0,\d)_p}{\d^\g} %\d^{\a+d(\frac1p-1)}}
        \s ({1/\d})
        +
        \left(\int_0^\d
        \bigg(
        \frac{\mathcal{R}_{\vp}(f_0,t)_p}{t^{d(\frac1p-\frac1q)}}
        \bigg)^{q_1}\frac{dt}{t}
        \right)^{\frac1{q_1}}\quad\,\,\,\,(\mbox{by }\quad (\ref{eqth1.8Kd})),
    \end{eqnarray}
    where in the last inequality we used the fact that $\a+\g>d(1/p-1/q)\ge d(1-1/q)$.
Therefore,  (\ref{eqth1.2Kd}) follows. To complete the proof of (B1), we need to verify (\ref{eqth1.9Kd}) and (\ref{sigma+}).
\bigskip

%\begin{proof}[Proof of estimate (\ref{eqth1.9Kd}) for $\g>0$]

{\sc{ Proof of estimate (\ref{eqth1.9Kd})}.}
We divide  the proof of this fact
into three steps.

\smallskip
\smallskip
\underline{Step 1.} Assume that $0<q\le 1$ and $\g>0$.
%Let us first prove this result  for $0<q\le 1$.
Denoting
$$
H_{\vp,n}(x):=F_{\vp}*V_n(x)-\frac1{N^{\a+\g}}F_{\vp}(N\cdot)* V_{Nn}(x),
$$
we estimate
\begin{equation}\label{eqth1.10Kd}
\mathcal{R}_\psi(f_0,1/n)_q\ge
C(q)\Big(\mathcal{R}_\psi(H_{\vp,n},1/n)_q-\mathcal{R}_\psi(f_0-H_{\vp,n},1/n)_q\Big).
\end{equation}
Since $H_{\vp,n}$ is a trigonometric polynomial of degree $2N n$,  using Lemma~\ref{real}, we get
$$
\mathcal{R}_\psi(H_{\vp,n},{1}/{n})_q\gtrsim
n^{-\a}\Vert
\mathcal{D}(\psi) H_{\vp,n}\Vert_q.
$$
Using the representation
$$\mathcal{D}(\psi) H_{\vp,n}(x)
=
\sum_{k\neq 0} \prod_{j=1}^d v\(\frac{|k_j|}n\) \frac{\psi(k)}{\varphi(k)}e^{i(k,x)}-
N^{-\gamma}
\sum_{k\neq 0} \prod_{j=1}^d v\(\frac{|k_j|}n\)\frac{\psi(k)}{\varphi(k)}e^{i(Nk,x)},
$$
we obtain
\begin{equation}\label{eqth1.11Kd}
\mathcal{R}_{\psi}(H_{\vp,n},{1}/{n})_q
%\geC n^{-\a}\Vert \mathcal{D}(\psi) H_{\vp,n}\Vert_q \ge
\gtrsim
\(1-\frac1{N^{\g q}}\)^{1/q}n^{-\a}\Vert
\mathcal{D}(\psi/\vp)V_n\Vert_q.
\end{equation}
Further, applying Holder's inequality and Lemma~\ref{rrun} yields
\begin{eqnarray*}\label{eqth1.13Kd}
\mathcal{R}_{\psi}(f_0-H_{\vp,n},{1}/{n})_q&\lesssim&
\mathcal{R}_{\psi}(f_0-H_{\vp,n},{1}/{n})_{1}\lesssim
{K}_\psi(f_0-H_{\vp,n},{1}/{n})_{1}
\\&\lesssim&
 n^{-\a}\Vert
\mathcal{D}(\psi)(f_0-H_{\vp,n})\Vert_{1}\\
&\lesssim&  n^{-\a}\(\Vert F_{\vp/ \psi}-F_{\vp/
\psi}*V_n\Vert_{1}+\Vert F_{\vp/ \psi}(N\cdot)-F_{\vp/
\psi}*V_{Nn}(N\cdot)\Vert_{1}\),
\end{eqnarray*}
where
$F_{\vp/ \psi}(x)=\sum_{k\neq 0}\frac{\psi(k)}{\vp(k)}e^{i(k,x)}$ (see (\ref{eqhvpd})).

Since
$\|F_{\vp/ \psi}\|_1\lesssim
\|F_{\vp/ \psi}\|_{1+\varepsilon}
\lesssim
\|\mathfrak{F}_\g\|_{1+\varepsilon}
\lesssim
1$ (see~\eqref{wainger}), we have that $F_{\vp/ \psi}\in L_1(\T^d)$ and, therefore,
\begin{equation}\label{eqth1.14Kd}
\begin{split}
\mathcal{R}_{\psi}(f_0-H_{\vp,n},n^{-1})_q=o(n^{-\a})\quad\text{as}\quad n\to \infty.
\end{split}
\end{equation}

Thus, combining  (\ref{eqth1.10Kd})--(\ref{eqth1.14Kd}),
we immediately obtain
$$
\mathcal{R}_{\psi}(f_0,1/n)_q\gtrsim
  n^{-\a}\Vert
\mathcal{D}(\psi/\vp) V_{ n}\Vert_q-o(n^{-\a}).
$$
Note that by Lemma \ref{v+}, there exists a constant $C$ independent of $n$ such that
\begin{equation*}
%\label{eqth1.12Kd}
\Vert \mathcal{D}(\psi/\vp)V_n\Vert_q
\ge C.
\end{equation*}
Hence,
$$
\mathcal{R}_{\psi}(f_0,1/n)_q\gtrsim
  n^{-\a}\Vert
\mathcal{D}(\psi/\vp) V_{ n}\Vert_q,
$$
which implies (\ref{eqth1.9Kd}) for  $0<q\le 1$ and $\g>0$.

%Понятно, что далее достаточно рассматривать только случай $q<\infty$. {\bf S4:  ссылка на формулу.}

\smallskip
\smallskip
\underline{Step 2.} Let $0<q\le 1$ and $\g=0$.
%\underline{Part (B4): $\g=0$ and $0<p<q\le 1$.}
In this case, we assume that $\vp(y)=C\psi(y)$ with some constant $C$.
%Let $f_0$ be the function defined in (\ref{eqfun0ples1}).
By Lemma~\ref{monoto},
one has that
$$%\begin{equation}\label{eqKsharpqles1}
    \mathcal{R}_{\psi}(f_0,\d)_q\gtrsim \mathcal{R}_{\psi}(f_0,1)_q \d^{\a+d(1/q-1)}.
$$%\end{equation}
On the other hand,
\eqref{mu+} in Remark~\ref{mu} implies that
  $$%\s(t)=
% t^{d(\frac1p-1)}
\Vert \mathcal{D}\({\psi}/{\varphi}\) V_{1/\delta}
\Vert_{q}
\asymp \d^{d(1-1/q)}
$$
 and hence
$$  \mathcal{R}_{\psi}(f_0,\d)_q\gtrsim
\d^{\a}
\Vert \mathcal{D}\({\psi}/{\varphi}\) V_{1/\delta}
\Vert_{q},
$$
which is (\ref{eqth1.9Kd}).  %Then (\ref{eqth1.2Kd}) follows from (\ref{eqth1.8Kd}) and  (\ref{eqKsharpqles1}).

\smallskip
\smallskip
\underline{Step 3.} Assume that %$0<q\le 1$ and $\g=0$.
 $1<q\le \infty$.
Taking into account that  $f_0* V_n$ is the near best approximant of $f_0$ in $L_q$, that is,
$\Vert f_0* V_n-f_0\Vert_q \le C E_n(f_0)_q$, we have that
Lemma~\ref{real} implies
\begin{eqnarray}\label{estim1}
\begin{split}
\mathcal{R}_{\psi}(f_0,1/n)_q&\gtrsim n^{-\a}\Vert \mathcal{D}(\psi)
(f_0* V_n)\Vert_q
\\
&\gtrsim n^{-\a}
\big(\Vert \mathcal{D}(\psi/\vp)
V_n\Vert_q-\frac1{N^\g}\Vert \mathcal{D}(\psi/\vp) V_{n/N}\Vert_q\big).
\end{split}
\end{eqnarray}
%Остается только воспользоваться полученными ниже оценками сверху и
%снизу для $\Vert \mathcal{D}(\psi/\vp) V_n\Vert_q$ и выбрать
%подходящим образом $N$.

If $1<q<\infty$ and $\g\ge 0$, Theorem~\ref{th-hardy-l-n}  gives the exact growth order of
$
\Vert \mathcal{D}(\psi/\vp)
V_n\Vert_q
$
as $n\to \infty$, which yields
\begin{eqnarray}\label{estim2}
\begin{split}\Vert \mathcal{D}(\psi/\vp)
V_n\Vert_q-\frac1{N^\g}\Vert \mathcal{D}(\psi/\vp) V_{n/N}\Vert_q
\gtrsim
\Vert \mathcal{D}(\psi/\vp)
V_n\Vert_q
\end{split}
\end{eqnarray}
for sufficiently large $N$. Therefore,
 (\ref{eqth1.9Kd}) follows.

If  $q=\infty$ and $\g>0$, we take  take into account that
\begin{equation*}
  \begin{split}
\mathcal{D}(\psi/\varphi) V_{n/N}(x)&=\sum_{k\neq 0}\frac{\psi(k)}{\varphi(k)}
    \prod_{j=1}^d v\(\frac{N|k_j|}{n}\)e^{i(k,x)}\\
    &=\sum_{k\neq 0}\frac{\psi(k)}{\varphi(k)}
    \prod_{j=1}^d v\(\frac{N|k_j|}{n}\)
     v\(\frac{|k_j|}{n}\)e^{i(k,x)}
   \end{split}
\end{equation*}
and that the function $ B(y)=\prod_{j=1}^d v\(y_j\)$ is a Fourier multiplier in $L_\infty(\T^d)$ (see Lemmas~\ref{lemMult} and~\ref{lemBes}). Hence, we have
\begin{equation*}
  \begin{split}
\Vert
\mathcal{D}(\psi/\varphi) V_{n/N}\Vert_\infty &\lesssim \Vert B(N y)\Vert_{L_\infty\to L_\infty}\Vert
\mathcal{D}(\psi/\varphi) V_{n}\Vert_\infty\\
&=\Vert B\Vert_{L_\infty\to L_\infty}\Vert
\mathcal{D}(\psi/\varphi) V_{n}\Vert_\infty\lesssim \Vert
\mathcal{D}(\psi/\varphi) V_{n}\Vert_\infty,
   \end{split}
\end{equation*}
which yields~\eqref{estim2} for sufficiently large $N$.

Finally, suppose that  $q=\infty$ and $\g=0$.
In this case, we assume that $\vp(y)=C\psi(y)$ with some constant $C$. Therefore,
$$
\Vert \mathcal{D}(\psi/\vp)
V_n\Vert_\infty
\asymp
\Vert
V_n\Vert_\infty
\asymp
n^d.
$$
Then, using (\ref{estim1})
and (\ref{estim2})
for sufficiently large $N$, we arrive at (\ref{eqth1.9Kd}).

Thus, the proof of (\ref{eqth1.9Kd}) is complete. % for $\g>0$

\bigskip

{\sc{ Proof of estimate (\ref{sigma+}).}}
First, assume that $\g>0$. By Lemma~\ref{lemHLd} (iii),
$\eta(n) \asymp  n^{d(1/p-1)}\Vert \mathcal{D}(\psi/\vp) V_n    \Vert_q.$

If $1<q<\infty$, then Lemmas~\ref{v-} and~\ref{v}  describe the sharp growth order of
$\Vert \mathcal{D}(\psi/\vp) V_n    \Vert_q$ and, in particular,
$$\Vert \mathcal{D}(\psi/\vp) V_n    \Vert_q\gtrsim n^{d(1-\frac1q)-\g}.$$ This gives (\ref{sigma+}).

If $p<q\le 1$, Lemma \ref{v+} yields
$$\eta(n)
 \gtrsim
 n^{d(\frac1p-1)} \gtrsim n^{d(\frac1p-\frac1q)-\g}.
$$

Finally, let $q=\infty$.  In this case, by Lemma~\ref{v++}, we have
\begin{equation*}
%\label{hom}
\Vert \mathcal{D}(\psi/\vp) V_n    \Vert_\infty\gtrsim n^{d-\g},
\end{equation*}
which implies the desired result.

Second, assume that $\g=0$ and $\vp(\xi)=C\psi(\xi)$. Then (\ref{sigma+}) is equivalent to
$$
n^{d(\frac1p-\frac1q)}\lesssim \sup_{T_n\in \mathcal{T}'_{n}}\frac{\Vert T\Vert_q}{\Vert
T\Vert_p}.
$$
It is enough to consider extremizers for the classical Nikol'skii inequality, for example, the Jackson-type kernel (see~\cite[\S~4.9]{timan}).

Hence, the proof of inequality (\ref{sigma+}) and the part (B1) are complete.
%\end{proof}
\bigskip

 \noindent
\underline{Part (B2): $1<p<q<\infty$.}
%Случай $1<p<\infty$ и $q\neq\infty$ можно доказать следующим образом.
In this case, if $\g\ge d(1/p-1/q)$, then we can take as $f_0$ any non-trivial
function from $C^\infty(\T^d)$.
Indeed, by Lemma~\ref{rrrun} and~\eqref{RealKpge1-}, for any $1<r<\infty$, we have
\begin{equation}\label{kkk}
    \begin{split}
\mathcal{R}_{\psi}(f_0,\delta)_r&\asymp \inf_g (\|f_0-g\|_r+ \delta^\alpha \|(-\Delta)^{\a/2}g\|_r)\asymp
\omega_\a(f_0,\delta)_r\asymp \delta^\a.
\end{split}
\end{equation}
Then,
\begin{equation}\label{kkkk}
    \begin{split}
\mathcal{R}_{\psi}(f_0,\delta)_q\gtrsim \delta^\a
&\gtrsim \frac{\delta^{\a+\gamma}}{\delta^\g}
\\&\gtrsim
 \frac{\mathcal{R}_{\vp}(f_0,\d)_p}{\d^{\g}}
 +
        \left(\int_0^\d
        \bigg(
        \frac{\mathcal{R}_{\vp}(f_0,t)_p}{t^{d(\frac1p-\frac1q)}}
        \bigg)^{q_1}\frac{dt}{t}
        \right)^{\frac1{q_1}}.
\end{split}
\end{equation}

If
$0\le \g<d(1/p-1/q)$, then we take
$$
f_0(x)=F_{\vp}(x)\sim \sum_{k\ne 0} \frac{e^{i(k,x)}}{\varphi(k)}.
$$

Let us first estimate $\mathcal{R}_{\vp}(f_0,\delta)_p$ from above.
We have
\begin{equation*}
    \mathcal{R}_{\vp}(f_0,2^{-n})_p\lesssim
    \Vert
    f_0-f_0*V_{2^{n}}\Vert_p+2^{-n(\a+\g)}\Vert \mathcal{D}(\vp) (f_0*V_{2^{n}})\Vert_p,
\end{equation*}
 where $V_{2^n}$ is given by (\ref{vallee}).

Note that by (\ref{eqNNN}),
$\Vert
    f_0-f_0*V_{2^{n}}\Vert_1\lesssim 2^{-(\a+\g)n}$. Then, using Lemma 4.2 from~\cite{diti} and taking into account that $\a+\g >d(1-{1}/{q})>d(1-{1}/{p})$, we obtain
\begin{eqnarray}\label{kl+}
\Vert
    f_0-f_0*V_{2^{n}}\Vert_p&\lesssim&\left(\sum_{\nu=n}^\infty2^{\nu pd(1-\frac{1}{p})}
    \big\Vert
    f_0-f_0*V_{2^\nu}\big\Vert_1^p
    \right)^{1/p}
    \nonumber
    \\&\lesssim& 2^{-(\a+\g +d(\frac{1}{p}-1))n}.
\end{eqnarray}
Further, using the multidimensional Hardy-Littlewood theorem for series with monotone coefficients (see Lemma~\ref{monot+}), we derive that
\begin{eqnarray}\label{kl++}
\Vert \mathcal{D}(\vp) (f_0*V_{2^{n}})\Vert_p
&\asymp&
\|V_{2^{n}}-1\|_p
\nonumber
\\
&\asymp&
\bigg\|{\mathop{{\sum}'}\limits_{{}^{|k|_\infty\le 2^{n}}}}
e^{i(k,x)}\bigg\|_p
\nonumber
\\
&\asymp&
\bigg(\sum_{\nu=1}^{2^n} \nu^{p-2}\bigg)^{d/p}\asymp 2^{dn(1-\frac{1}{p})}.
\end{eqnarray}
Therefore,
\begin{equation}\label{kkkk+}
\mathcal{R}_{\vp}(f_0,\delta)_p\lesssim\delta^{\a+\g +d(\frac{1}{p}-1)}.
\end{equation}

Let us estimate $\mathcal{R}_{\psi}(f_0,1/n)_q$ from below.
Taking into account that  $f_0* V_n$ is the near best approximant of $f_0$ in $L_q$ and
using Lemmas~\ref{v} and~\ref{real} with $0\le \g<d(1/p-1/q)<d(1-1/q)$, we get
\begin{equation}\label{kkkk++}
    \begin{split}
\mathcal{R}_{\psi}(f_0,1/n)_q&\gtrsim n^{-\a}\Vert \mathcal{D}(\psi)
(f_0* V_n)\Vert_q
\\
&\gtrsim n^{-\a}\Vert (-\D)^{-\g/2}V_n \Vert_q\gtrsim
n^{d(1-\frac1q)-(\a+\g)}.
\end{split}
\end{equation}
Therefore, combining  (\ref{kkkk+}) and (\ref{kkkk++}), we finally obtain
 \begin{eqnarray*}\nonumber
 &&\!\!\!\!\!\!\!\!\!\!\!\! \frac{\mathcal{R}_{\vp}(f_0,\d)_p}{\d^\g} %\d^{\a+d(\frac1p-1)}}
        \s ({1/\d})
        +
        \left(\int_0^\d
        \bigg(
        \frac{\mathcal{R}_{\vp}(f_0,t)_p}{t^{d(\frac1p-\frac1q)}}
        \bigg)^{q}\frac{dt}{t}
        \right)^{\frac1{q}}
        \\
 &\lesssim&
  \d^{\a+d(\frac1p-1)}
        \s ({1/\d})
        +
        \left(\int_0^\d
        t^{q(\a+\g+d(\frac1q-1))}
        \frac{dt}{t}
        \right)^{\frac1{q}}
        \quad\quad\,\,\, (\mbox{by }\quad (\ref{kkkk+}))
        \\
 &\lesssim&
  \d^{\a+d(\frac1p-1)}
        \s ({1/\d})
        +
        \d^{\a+\g+d(\frac1q-1)}
        \quad\quad\qquad \qquad\qquad\,(\mbox{since }\quad \a+\g>d(1-1/q))
        \\
        &\lesssim&
        \d^{\a+\g+d(\frac1q-1)}
        %\lesssim         \mathcal{R}_{\psi}(f_0,\d)_q
        \qquad\qquad\qquad\qquad\qquad\qquad \qquad\quad\quad\,\,(\mbox{by the definition of}\quad \s)
        \\
        &\lesssim&
        \mathcal{R}_{\psi}(f_0,\d)_q
        \quad\qquad\quad\qquad\quad\qquad\quad\qquad \qquad\quad\,\,\,\,\,\quad(\mbox{by }\quad (\ref{kkkk++})),
    \end{eqnarray*}
which proves (\ref{eqth1.2Kd}).

\bigskip
 \noindent
\underline{Part (B3): $1<p<\infty$ and $q=\infty$.}
First, let $\g>{d}/{p}$. We take a non-trivial function $f_0\in C^\infty(\T^d)$.
Then, by Lemma~\ref{rrrun}, for any $1<r<\infty$, we have
$$
\mathcal{R}_{\psi}(f_0,\delta)_\infty
\gtrsim\mathcal{R}_{\psi}(f_0,\delta)_r
\gtrsim \d^{\a}\Vert (-\D)^{-\a/2}f_0 \Vert_r
\gtrsim \delta^\alpha.
$$
By (\ref{kkk}),
 $\mathcal{R}_\vp(f_0,\delta)_p\asymp \delta^{\alpha+\g}$
and (\ref{eqth1.2Kd}) follows since in this case $\s(t)=1$ (see (\ref{kkkk})).

Now, we assume that $0\le \g<{d}/{p}$
and $\s(t)=t^{{d}/{p}-\g}$.
We consider the function
$$
f_0(x)={\mathop{{\sum}'}_{k\in \Z^d}} \frac{e^{i(k,x)}}{|k|^{\g}\psi(k)}.
$$
Then, by Lemma \ref{real}, we have
 \begin{eqnarray}
\label{kkkk++--}
\mathcal{R}_{\psi}(f_0,2^{-n})_\infty
&\gtrsim&
 2^{-\a n}\Vert \mathcal{D}(\psi)
(f_0* V_{2^n})\Vert_\infty
\nonumber
\\&=&
 2^{-\a n}\Vert
 (-\Delta)^{-\g/2}
V_{2^n}\Vert_\infty
\gtrsim
2^{n (d -\a-\g)}.
\end{eqnarray}
On the other hand, as above, by (\ref{kl+}) and (\ref{kl++}), we obtain that
    $$
\Vert
    f_0-f_0*V_{2^{n}}\Vert_p\lesssim 2^{-(\a+\g +d(\frac{1}{p}-1))n}
    $$
    and
$$\Vert \mathcal{D}(\vp) (f_0*V_{2^{n}})\Vert_p
\lesssim 2^{dn(1-\frac{1}{p})}.
$$
Hence,
\begin{equation}\label{kkkk+--}
\mathcal{R}_{\vp}(f_0, {2^{-n}})_p\lesssim
2^{-(\a+\g +d(\frac{1}{p}-1))n}.
\end{equation}
Thus, (\ref{kkkk++--}) and (\ref{kkkk+--}) give  (\ref{eqth1.2Kd}).

Finally, let $\g={d}/{p}$ and $\s(t)=
           \ln^{1/{p'}} (t+1).$ In this case the proof is more technical.
Consider
$$
f_0(x)={\mathop{{\sum}'}_{k\in \mathbb{Z}^d}} \frac{1}{\psi(k)|k|^d}e^{i(k,x)}.
$$
First, by Lemma~\ref{real},
\begin{equation}\label{est-general++}
\mathcal{R}_{\psi}(f_0,1/n)_\infty \gtrsim n^{-\a}\Vert \mathcal{D}(\psi)
(f_0* V_n)\Vert_\infty\gtrsim
n^{-\a}
{\mathop{{\sum}'}_{|k|_\infty\le n}}
\frac{1}{|k|^d}
\gtrsim n^{-\a}\ln n.
\end{equation}
Note that
$$
\mathcal{R}_{\vp}(f_0,\delta)_p\asymp \mathcal{R}_{\vp}(f_0^*,\delta)_p,
$$
where
 $$
f^*_0(x)={\mathop{{\sum}'}_{k\in \mathbb{Z}^d}} \frac{1}{|k|^{\alpha+d}}e^{i(k,x)}.
$$
Denoting by $D_n$ the cubic Dirichlet kernel, i.e.,
$$D_n(x)=\sum_{{}^{|k|_\infty\le n}}
 e^{i(k,x)},$$
we have
\begin{equation}\label{est-general}
\begin{split}
\mathcal{R}_{\vp}(f_0,1/n)_p&\lesssim
    \mathcal{R}_{\vp}(f_0^*,1/n)_p\\
    &\lesssim
    \Vert
    f_0^*-f_0^**D_n\Vert_p+n^{-\a-\g}\Vert \mathcal{D}(\vp) (f_0^**D_n)\Vert_p
    \\
&\lesssim
    \Big\|
\sum_{{}^{|k|_\infty\ge n}}  \frac{e^{i(k,x)}}{|k|^{\alpha+d}}\Big\|_p
+
n^{-\a-\g}\Big\|
\mathop{{\sum}'}_{{}^{|k|_\infty\le n}}
  \frac{e^{i(k,x)}}{|k|^{d-\g}}\Big\|_p.
\end{split}
    \end{equation}
By Lemma \ref{v}, we get
$$
n^{-\a-\g}\bigg\|
\mathop{{\sum}'}_{{}^{|k|_\infty\le n}}
  \frac{e^{i(k,x)}}{|k|^{d-\g}}\bigg\|_p\lesssim
  n^{-\a-\g}
  \ln^\frac1{p} (n+1), \qquad
  \g=\frac{d}{p}.
$$
Let us prove that
\begin{equation}\label{est-second}
    \bigg\Vert
\mathop{{\sum}'}_{{}^{|k|_\infty\ge n}}  \frac{e^{i(k,x)}}{|k|^{\alpha+d}}\bigg\Vert_p
\lesssim
  n^{-\a-\g}.
\end{equation}
For simplicity,  we consider only the case $d=2$.
Taking into account the  Hardy--Littlewood theorem for multiple series (Lemma~\ref{monot+}), we get
\begin{equation*}
  \begin{split}
&\Big\|
\sum_{k_1= n}^{\infty} \sum_{k_2= n}^{\infty} \frac{\cos k_1 x_1 \cos k_2 x_2} {|k|^{\alpha+d}}\Big\|_p
\\
&\qquad\qquad\lesssim
    \left(
\sum_{k_1= 1}^{\infty} k_1^{p-2}
\sum_{k_2= 1}^{\infty} k_2^{p-2}
\Big(
\sum_{\nu_1= k_1}^{\infty}
\sum_{\nu_2= k_2}^{\infty}
\big|
\Delta^{(2)}
a_{\nu_1,\nu_2}
 \big|\Big)^p
 \right)^{1/p}=:D,
   \end{split}
\end{equation*}
 where
$$
a_{\nu_1,\nu_2}=\left\{
                  \begin{array}{ll}
                    {|\nu|^{-(\alpha+d)}}, & \hbox{$|\nu|_\infty\ge n$;} \\
                    0, & \hbox{$|\nu|_\infty< n$.}
                  \end{array}
                \right.
$$
We divide this series into four parts ($D\lesssim D_1+D_2+D_3+D_4$) and estimate each of them separately. We have
 \begin{align*}
D_1&:=
    \left(
\sum_{k_1= 1}^{n-2}
\sum_{k_2= 1}^{n-2}
 k_1^{p-2}
 k_2^{p-2}
\Big(
\sum_{\nu_1= n-1}^{\infty}
\sum_{\nu_2= n-1}^{\infty}
\big|
\Delta^{(2)}
a_{\nu_1,\nu_2}
 \big|\Big)^p
 \right)^{1/p}
 \\&\lesssim
{n^{2(p-1)/p}}\sum_{\nu_1= n-1}^{\infty}
\sum_{\nu_2= n-1}^{\infty}
\big|
\Delta^{(2)}
a_{\nu_1,\nu_2}
 \big|
 \\&\lesssim
{n^{2(p-1)/p}}
a_{n-1,n-1}
 \lesssim
\frac{n^{2(1-\frac1p)}}{n^{\a+d}},
\end{align*}
since $p>1$ and $\Delta^{(2)} a_{\nu_1,\nu_2}\ge 0$ (see~\eqref{4zvedy}).

Further,
 \begin{align*}
D_2&:=
    \left(
\sum_{k_1= 1}^{n-2}
\sum_{k_2= n-1}^{\infty}
 k_1^{p-2}
 k_2^{p-2}
\Big(
\sum_{\nu_1= n-1}^{\infty}
\sum_{\nu_2= k_2}^{\infty}
\big|
\Delta^{(2)}
a_{\nu_1,\nu_2}
 \big|\Big)^p
 \right)^{1/p}
\\& \lesssim
    \left(
\sum_{k_1= 1}^{n-2}
k_1^{p-2}
 \sum_{k_2= n-1}^{\infty}
 k_2^{p-2}
a_{n-1,k_2}^p
 \right)^{1/p}
\\& \lesssim
{n^{1-\frac1p}}
    \left(
\sum_{k_2= n-1}^{\infty}
 \frac {k_2^{p-2}}{(n^2+k_2^2)^{\frac{(\a+d)p}{2}}}
 \right)^{1/p}
\\& \lesssim
{n^{1-\frac1p}}
    \left(
\sum_{k_2= n-1}^{\infty}
 \frac {1}{k_2^{{(\a+d)p-p+2}}}
 \right)^{1/p}
.
\end{align*}
In light of $\a+\g=\a+{d}/{p}>d$, we have $(\a+d)p-p+1>0$ and, therefore,
$$
D_2  \lesssim
\frac{n^{2(1-\frac1p)}}{n^{\a+d}}.
$$
Similarly, we derive that
 \begin{align*}
D_3&:=
    \left(
\sum_{k_1= n-1}^{\infty}
\sum_{k_2= 1}^{n-2}
 k_1^{p-2}
 k_2^{p-2}
\Big(
\sum_{\nu_1= k_1}^{\infty}
\sum_{\nu_2= n-1}^{\infty}
\big|
\Delta^{(2)}
a_{\nu_1,\nu_2}
 \big|\Big)^p
 \right)^{1/p}
 \lesssim
\frac{n^{2(1-\frac1p)}}{n^{\a+d}}.
\end{align*}
Finally,
 \begin{align*}
D_4&:=
    \left(
\sum_{k_1= n-1}^{\infty}
\sum_{k_2= n-1}^{\infty}
 k_1^{p-2}
 k_2^{p-2}
\Big(
\sum_{\nu_1= k_1-1}^{\infty}
\sum_{\nu_2= k_2-1}^{\infty}
\big|
\Delta^{(2)}
a_{\nu_1,\nu_2}
 \big|\Big)^p
 \right)^{1/p}
\\& \lesssim
\left(
\sum_{k_1= n-1}^{\infty}
k_1^{p-2}
 \sum_{k_2= n-1}^{\infty}
\frac {k_2^{p-2}}{(k_1^2+k_2^2)^{\frac{(\a+d)p}{2}}}
 \right)^{1/p}
\\& \lesssim
\left(
\sum_{k_1= n-1}^{\infty}
k_1^{p-2}
 \sum_{k_2= n-1}^{k_1}
\frac {k_2^{p-2}}{(k_1^2+k_2^2)^{\frac{(\a+d)p}{2}}}
 \right)^{1/p}
+
\left(
\sum_{k_1= n-1}^{\infty}
k_1^{p-2}
 \sum_{k_2= k_1}^{\infty}
\frac {k_2^{p-2}}{(k_1^2+k_2^2)^{\frac{(\a+d)p}{2}}}
 \right)^{1/p}
\\& =:
D_4^*+D_4^{**}.
\end{align*}
We estimate each term separately. First,
 \begin{align*}
D_4^*&\lesssim
\left(
\sum_{k_1= n-1}^{\infty}
\frac{k_1^{p-2}}{k_1^{(\a+d)p}}
 \sum_{k_2= n-1}^{k_1}
 {k_2^{p-2}}
 \right)^{1/p}
 \\&
  \lesssim
  \left(
\sum_{k_1= n-1}^{\infty}
\frac{k_1^{2p-3}}{k_1^{(\a+d)p}}
 \right)^{1/p}.
\end{align*}
Using again the assumption  $\a+{d}/{p}>d$, it is easy to see  that  $$(\a+d)p-2p+2>0,$$ which implies that % and therefore
$$
D_4^*
  \lesssim
  \frac{n^{2(1-\frac1p)}}{n^{\a+d}}.
$$
Further,
 \begin{align*}
D_4^{**}&\lesssim
\left(
\sum_{k_1= n-1}^{\infty}
k_1^{p-2}
 \sum_{k_2= k_1}^{\infty}
\frac {k_2^{p-2}}{k_2^{{(\a+d)p}}}
 \right)^{1/p}
 \\&
  \lesssim
  \left(
\sum_{k_1= n-1}^{\infty}
\frac{k_1^{2p-3}}{k_1^{(\a+d)p}}
 \right)^{1/p}
   \\&
  \lesssim
  \frac{n^{2(1-\frac1p)}}{n^{\a+d}}.
\end{align*}
Combining these estimates, we arrive at inequality (\ref{est-second}).
Thus, by (\ref{est-general}) and (\ref{est-second}), we obtain that
\begin{equation}\label{est-general+}
\begin{split}
\mathcal{R}_{\vp}(f_0,1/n)_p&\lesssim   n^{-\a-\g}
  \ln^\frac1{p} (n+1)+ \frac{n^{2(1-\frac1p)}}{n^{\a+d}}\lesssim
   n^{-\a-\frac{d}{p}}
  \ln^\frac1{p} (n+1).
\end{split}
    \end{equation}
Hence,
 \begin{eqnarray*}\nonumber
  && \!\!\!\!\!\!\!\!\!\!\!\!\frac{\mathcal{R}_{\vp}(f_0,\d)_p}{\d^\g} %\d^{\a+d(\frac1p-1)}}
        \s ({1/\d})
        +
        \int_0^\d
        \frac{\mathcal{R}_{\vp}(f_0,t)_p}{t^{\frac dp}}
        \frac{dt}{t}
        %^{\frac1{q}}
        \\
 &\asymp&
 \frac{\mathcal{R}_{\vp}(f_0,\d)_p}{\d^{d/p}} %\d^{\a+d(\frac1p-1)}}
        \ln^{1/p'} ({1/\d})
        +
        \int_0^\d
        \frac{\mathcal{R}_{\vp}(f_0,t)_p}{t^{\frac dp}}
        \frac{dt}{t}
        %^{\frac1{q}}
        \\
         &\lesssim&
  \d^{\a}  \ln({1/\d})
        +
        \int_0^\d
        t^{\a}
                 \ln^\frac1{p} \(\frac1t\)
        \frac{dt}{t}
        %^{\frac1{q}}
        \quad\qquad \qquad\qquad\,\,\,(\mbox{by }\quad(\ref{est-general+}))
        \\
         &\lesssim&
  \d^{\a}  \ln({1/\d})
\lesssim
        \mathcal{R}_{\psi}(f_0,\d)_\infty
        \quad\qquad \qquad\qquad\qquad\quad(\mbox{by }\quad (\ref{est-general++}))
    \end{eqnarray*}
and     (\ref{eqth1.2Kd}) follows.

\end{proof}

\newpage

\section{Sharp Ulyanov inequalities for moduli of smoothness}
\label{sec8}
\subsection{One-dimensional results
%The case of Weyl derivatives in one dimension.
%The case of Weyl derivatives. Sharp Ulyanov inequalities for moduli of smoothness in one dimensional case
}
For $d=1$,
by using Theorem~\ref{th1dK} with   $\psi(\xi)=(i\xi)^\a$ and $\vp(\xi)=(i\xi)^{\a+\g}$, which correspond to the  fractional Weyl derivatives,
we obtain the following sharp inequality.

\begin{theorem}\label{th1}
  {\it  Let $0<p<q\le \infty$, $\a\in\N\cup
    ((1/q-1)_+,\infty)$ and $\g,m\ge 0$ be such that $\a+\g, \a+m\in \N\cup
    ((1/p-1)_+,\infty)$.

{\rm (A)} Let $f\in L_p(\T)$. Then, for any  $\d\in (0,1)$,  we have
    \begin{equation}\label{eqth1.1--}
        \w_\a(f,\d)_q\lesssim\frac{\w_{\a+\g}(f,\d)_p}{\d^{\g}} \s\(\frac1\d\)+
        \left(\int_0^\d\bigg(\frac{\w_{\a+m}(f,t)_p}{t^{\frac1p-\frac1q}}\bigg)^{q_1}\frac{dt}{t}\right)^{\frac1{q_1}},
    \end{equation}
where

{\rm (A1)} if $0<p\le 1$ and $p<q\le\infty$, then
$$
\s%_{q,\g}
(t)
:=\left\{
         \begin{array}{ll}
           t^{\frac1p-1}, & \hbox{$\g>\(1-\frac1q\)_+$}; \\
           t^{\frac1p-1}\ln^\frac1q (t+1), & \hbox{$0<\g=\(1-\frac1q\)_+$}; \\
           t^{\frac1p-\frac1q-\g}, & \hbox{$0< \g<\(1-\frac1q\)_+$};\\
           t^{\frac1p-\frac1q}, & \hbox{$\g=0$},
         \end{array}
       \right.
$$

{\rm (A2)} if $1<p\le q\le\infty$, then
$$
\s(t):=\left\{
         \begin{array}{ll}
           1, & \hbox{$\g\ge \frac1p-\frac1q,\quad q<\infty$}; \\
           1, & \hbox{$\g> \frac1p,\quad q=\infty$}; \\
           \ln^\frac1{p'} (t+1), & \hbox{$\g=\frac1p,\quad q=\infty$}; \\
           t^{(\frac1p-\frac1q)-\g}, & \hbox{$0\le \g<\frac1p-\frac1q$}.\\
         \end{array}
       \right.
$$
%\quad
%\l_{q,d}(t):=\left\{
%               \begin{array}{ll}
%                 \ln^\frac1q 2t, & \hbox{$1<q<\infty, d\ge 1$;} \\
%                 \ln 2t, & \hbox{$q=\infty, d=1$;} \\
%                 t^{d-2}, & \hbox{$q=\infty, d\ge 2$.}
%               \end{array}
%             \right.
%%%           t^{d(\frac1p-1)}
%%%\Vert \mathcal{D}\(\frac{\psi}{\varphi}\) V_{t}
%%%\Vert_{q}, & \hbox{$\g=0$, $0<p<q\le 1$;}

\medskip

{\rm (B)}  Inequality \eqref{eqth1.1--} is sharp in the following sense. Let $\a+\g>(1-1/q)_+$ and $m-\g\in \Z_+ \cup ((1/p-1)_+,\infty)$. There exists a function $f_0\in L_q(\T)$, $f_0\not\equiv \const$,
such that the following equivalence holds
\begin{equation*}
%\label{eqth1.2}
    {\w_\a(f_0,\d)_q}
    \asymp
    \frac{\w_{\a+\g}(f_0,\d)_p}{\d^{\g}} \s\(\frac1\d\)+
        \left(\int_0^{\d}\bigg(\frac{\w_{\a+m}(f_0,t)_p}{t^{\frac1p-\frac1q}}\bigg)^{q_1} \frac{dt}t\right)^{\frac1{q_1}}
\end{equation*}
as $\d\to 0$.

}
\end{theorem}

\begin{corollary}\label{remar}
Under the conditions of Theorem~\ref{th1},
inequality (\ref{eqth1.1--}) with $m=\g$ implies the following inequality
\begin{equation}\label{eqth1.1-1}
        \w_\a(f,\d)_q
 \lesssim\left(\int_0^\d\bigg(
                \frac{\w_{\a+\g}(f,t)_p}{t^{\g}} \s\(\frac1t\)
\bigg)^{q_1}\frac{dt}{t}\right)^{\frac1{q_1}}.
    \end{equation}
    Moreover,  if $\a+\g>(1-1/q)_+$, then
 there exists a nontrivial  function $f_0\in L_q(\T)$ such that %$\d\to 0$
$$%\begin{equation}\label{eqth1.2}
        \w_\a(f_0,\d)_q
\asymp \left(\int_0^\d\bigg(
                \frac{\w_{\a+\g}(f_0,t)_p}{t^{\g}} \s\(\frac1t\)
\bigg)^{q_1}\frac{dt}{t}\right)^{\frac1{q_1}}.
$$%\end{equation}

\end{corollary}

 \begin{proof}[Proof of Theorem \ref{th1}]
 Part (A) follows from the realization result (\ref{eq.th6.0}) and Theorem~\ref{th1dK}
 with
 $\psi(\xi)=(i\xi)^\a$ and $\vp(\xi)=(i\xi)^{\a+\g}$.

The appearance of $m$ in the integral in the right-hand side of (\ref{eqth1.1--})  follows from the fact that
inequality (\ref{eqth++}) can be written as

 \begin{equation}\label{bsp}\begin{split}
\Vert f-T_{2^n}\Vert_q &\lesssim
 \(\sum_{\nu=n}^\infty{2^{\nu q_1 (\frac1p-\frac1q)}}\Vert f-T_{2^{\nu}}\Vert_p^{q_1}\)^{\frac1{q_1}}
\\
&\lesssim
 \(\sum_{\nu=n}^\infty\({2^{\nu q_1 (\frac1p-\frac1q)}}{\w_{\a+m}(f_0,2^{-\nu})_p}\)^{q_1}\)^{\frac1{q_1}}
\\
&\lesssim
\left(\int_0^{2^{-n}}
\bigg(\frac{\w_{\a+m}(f,t)_p}{t^{\frac1p-\frac1q}}\bigg)^{q_1}\frac{dt}{t}\right)^{\frac1{q_1}},
\end{split}
\end{equation}
and, therefore, (\ref{eqth1.1--}) holds.

Part (B) follows from  Theorem~\ref{th1dK} (B) noting that in our case (\ref{eqth1.2Kd2}) is given by
     \begin{eqnarray}\nonumber
     \w_\a(f_0,\d)_q
        &\gtrsim&  \frac{
       \w_{\a+\g}(f_0,\d)_p
}{\d^\g} %\d^{\a+d(\frac1p-1)}}
        \s ({1/\d})
        +
        \left(\int_0^\d
        \bigg(
        \frac{
        \w_{\a+\g}(f_0,t)_p}{t^{\frac1p-\frac1q}}
        \bigg)^{q_1}\frac{dt}{t}
        \right)^{\frac1{q_1}}.
    \end{eqnarray}
    Using the  inequality $\w_{\a+\g}(f_0,t)_p\gtrsim \w_{\a+m}(f_0,\d)_p$ for $m\ge\g$, we arrive at the statement of part (B).
 \end{proof}

\bigskip

\subsection{Comparison between the sharp  and classical Ulyanov inequalities
}
Let us compare the obtained
inequality (\ref{eqth1.1--}) and the classical Ulyanov inequality given by
\smallskip
\begin{equation}\label{eqUl}
        \w_\a(f,\d)_q\lesssim\left(\int_0^\d\bigg(\frac{\w_{\a}(f,t)_p}{t^{\frac1p-\frac1q}}\bigg)^{q_1}\frac{dt}{t}\right)^{\frac1{q_1}},\qquad
        0<p<q\le\infty.
\end{equation}

First, if $1<p<q<\infty$,  the sharp version of this inequality, i.e.,
\begin{equation}\label{eqUl-1}
        \w_\a(f,\d)_q\lesssim\left(\int_0^\d\bigg(\frac{\w_{\a+\theta}(f,t)_p}{t^{\frac1p-\frac1q}}\bigg)^{q}\frac{dt}{t}\right)^{\frac1{q}},
        \qquad \theta=\frac{1}{p}-\frac{1}{q},
\end{equation}
which coincides with (\ref{eqth1.1--}), clearly gives a better estimate than (\ref{eqUl}). Moreover, (\ref{eqUl-1})
is sharp over the class
 %  We will show that estimate (\ref{ul-4}) is sharp over the class
$$
{\textnormal{Lip}} \,\big(\omega(\cdot),\a + \theta,p\big)\,= \,
\Big\{f\in L_p(\T):\,\omega_{\a+\theta}(f,\delta)_p=\mathcal{O}(\omega(\delta))\Big\}
$$
(see~\cite{ST}).
In more detail,
for any function $\omega \in \Omega_{\alpha +  \theta}$, there exists a function
$
f_0(x)=f_0(x,p,\omega) \in {\textnormal{Lip}}
\,\big(\omega(\cdot),\alpha+\theta,p\big)
$ such that
for any $q\in (p, \infty)$ and  for any $\delta>0$
\begin{equation*}
%\label{gg}
\omega_{\alpha}(f_0, \delta)_q \geq C \left(
\int_{0}^{\delta} \Bigl(t^{- \theta} \omega(t) \Bigl)^q
\frac{dt}{t} \right)^{\frac{1}{q}},
\end{equation*}
 where a constant  $C$ is independent of $\delta$ and $\omega$.
\smallskip

Let us give several examples showing essential differences between~\eqref{eqth1.1--} and~\eqref{eqUl}.

\begin{example}
\textnormal{ Let $1 < p < q < \infty$ and $\alpha > \theta$. Define
$$
f_0(x) \sim \sum\limits_{m = 1}^{\infty}a_m \cos mx, \quad a_m=\frac{1}{m^{\alpha+\varepsilon+1-\frac{1}{q}}},\quad \varepsilon>0.
$$
The Hardy--Littlewood theorem (see~\eqref{monot}) implies that  $f_0 \in
L_q(\T)$.
Moreover, realization~\eqref{eq.th6.0} and Theorem 6.1 from~\cite{GT}  give
$$
\omega_{\xi}(f_0,1/n)_\nu \asymp
n^{-\xi}
\left(
\sum_{k=1}^n a_k^\nu k^{\xi \nu+\nu-2}\right)^{1/\nu}
+
\left(
\sum_{k=n+1}^\infty a_k^\nu k^{\nu-2}\right)^{1/\nu}, \qquad\xi>0,\quad p\le\nu\le q.
$$
Then it is easy to see that
\begin{eqnarray*}
\w_\a(f_0,\d)_q\asymp\d^\a,
\end{eqnarray*}
\begin{eqnarray*}
\left(\int_0^\d\bigg(\frac{\w_{\a+\t}(f_0,t)_p}{t^{\frac1p-\frac1q}}\bigg)^{q}\frac{dt}{t}\right)^{\frac1{q}}
\asymp
\left\{
               \begin{array}{ll}
                \d^{\a} \ln^\frac1p \frac 1\d, & \hbox{$\varepsilon=0$;} \\
                 \d^{\a}, & \hbox{$\varepsilon>0$,}
               \end{array}
             \right.
\end{eqnarray*}
and
\begin{eqnarray*}
\left(\int_0^\d\bigg(\frac{\w_{\a}(f_0,t)_p}{t^{\frac1p-\frac1q}}\bigg)^{q}\frac{dt}{t}\right)^{\frac1{q}}
\asymp\d^{\a-\t},
\end{eqnarray*}
that is,  (\ref{eqUl-1}) is essentially sharper than (\ref{eqUl}).
}
\end{example}

Second, if $0<p<1$, then an $L_p$-function may have certain pathological behavior in the sense of its smoothness properties.
This
 phenomenon was observed earlier
 (see, e.g., \cite{DHI}, \cite{Krot}, and~\cite{peetre}).
Let us now consider functions, which are  smooth in  $L_p$, $0<p<1$ (in the sense of behaviour of their moduli of smoothness), and show that
(\ref{eqth1.1--}), unlike (\ref{eqUl}), provides the best possible estimate for such functions.

\smallskip
\begin{example}
\textnormal{Let $0<p<1$, $p<q\le \infty$, $\a\in\N\cup
    ((1/q-1)_+,\infty)$, and let $\g> 0$ be such that $\a+\g\in \N\cup ((1/p-1)_+,\infty)$ and $\a+\g>(1-1/q)_+$.
Let, for example, $f=f_0$, where $f_0$ is given by~\eqref{eqfun0ples1} with $\vp(\xi)=(i\xi)^{\a+\g}$. Then, by~\eqref{eq.th6.0}, \eqref{eqth1.8Kd}, and Lemma~\ref{monoto}, we have
$$
\w_{\a+\g}(f,\d)_p\asymp \d^{\a+\g+\frac1p-1}.
$$
At the same time, \eqref{eq.th6.0}  and~\eqref{eqth1.9Kd} with $\psi(x)=(ix)^\a$ imply
$$
\w_\a(f,\d)_q \gtrsim \d^\a \Vert V_{1/\d}^{(-\g)}\Vert_q \asymp  \d^\a  \s\(\frac1\d\),
$$
where $V_{1/\d}$ is given by~\eqref{vallee} and $\s$ is defined in Theorem~\ref{th1}.
By using the above formulas for the moduli of smoothness, it is easy to calculate that
\begin{eqnarray*}
       \w_\a(f,\d)_q
       &\asymp&
       \frac{\w_{\a+\g}(f,\d)_p}{\d^{\g}} \s\(\frac1\d\)
       +
        \left(\int_0^\d\bigg(\frac{\w_{\a+\g}(f,t)_p}{t^{\frac1p-\frac1q}}\bigg)^q\frac{dt}{t}\right)^{\frac1q}
\\
& \asymp& \d^{\a} \s\(\frac1\d\),
\end{eqnarray*}
i.e., (\ref{eqth1.1--}) is sharp for the function $f_0$.
On the other hand, inequality (\ref{eqUl}) for $f_0$ implies only the following (non-sharp) estimate:
%, в отличие от неравенства (\ref{eqth1.1--}) дает другую менее точную оценку:
\begin{eqnarray*}
%\begin{equation*}
\d^{\a}
\s\(\frac1\d\)\asymp \w_\a(f,\d)_q&\lesssim&
        \left(\int_0^\d\bigg(\frac{\w_{\a}(f,t)_p}{t^{\frac1p-\frac1q}}\bigg)^q\frac{dt}{t}\right)^{\frac1q}\\&\asymp& \d^{\a-(\frac1p-\frac1q)}.
\end{eqnarray*}
}
\end{example}
%Таким образом, мы показали, что для функций более гладких в смысле
%пространства $L_p$, $0<p<1$, неравенство (\ref{eqth1.1--}) является
%более точным чем неравенство (\ref{eqUl}). Для достаточно гладких
%функций в обычным смысле неравенство (\ref{eqth1.1--}) также является
%более точным чем неравенство (\ref{eqUl}), однако только в случае
%$q>1$.
%\smallskip
%Легко видеть, что неравенство \eqref{eqth1.1--} при $\g=0$ совпадает с
%(\ref{eqUl}). Покажем, что за счет выбора параметра $\g\ge 0$
%неравенство \eqref{eqth1.1--} действительно является улучшением
%неравенства (\ref{eqUl}).
We conclude this subsection by the following example dealing with the case $p=1$.
%\medskip

\begin{example}
%Let $0<p\le 1$.
\textnormal{Let $1=p< q < \infty$.
We define $f_0$ such that
$$
f_0(x)={\sum_{\nu\in \Z}}' \frac{a_{|\nu|}}{
(i\nu)^{\alpha+1-1/q}
} e^{i \nu x},
$$
where $\{a_\nu\}$ is a convex sequence of positive numbers, which can tend to zero very slowly.
Then, by \cite[Ch.~V, (1.5), p.~183]{Z}, one gets $f_0^{(\alpha+1-1/q)}\in L_1(\T)$.
By the realization result, we then have that
$
\omega_{\alpha  + 1 - {1}/{q}}(f_0,t)_1\,
\lesssim \,t^{\alpha + 1- {1}/{q}}.
$
Hence, using the Marchaud inequality, we obtain that
$
\omega_{\alpha}(f_0,t)_1\,
\lesssim\,t^{\alpha}. $
Then the classical inequality (\ref{eqUl}) yields
$$
\omega_{\alpha}(f_0,t)_q\,
\lesssim\,t^{\alpha-(1- {1}/{q})}. $$
On the other hand, sharp Ulyanov inequality (\ref{eqUl-1}) gives a much better estimate:
 $$
\omega_{\alpha}(f_0,t)_q\,
\lesssim\,t^{\alpha} \ln^{{1}/{q}} \frac 1t. $$
Moreover, this estimate is sharp in the sense that
$
\omega_{\alpha}(f_0,t)_q\,
\gtrsim\,t^{\alpha} \ln^{{1}/{q}} \frac 1t a_{[1/t]}. $
}
 \end{example}

%Это же можно показать и для $\a\in\N$, а $\g\not\in\N$. Для этого
%нужно будет применять разные теоремы, связывающие дробные
%производные в периодическом и непериодическом случае (см.,
%напр.,~\cite[p.268 and p. 141 ]{SKM}).

%\bigskip

\subsection{Multidimensional results}

Since for $1<p<\infty$ one has
  \begin{equation}\label{laplas}  \w_\a(f,\d)_p \asymp
\mathcal{R}_{\psi}(f,\delta)_p,\qquad  \psi(\xi) = |\xi|^\a
 \end{equation}
(see Lemma~\ref{rrrun}),  the sharp Ulyanov inequalities for moduli of smoothness in  the case $1<p<q<\infty$ immediately  follows from Theorem \ref{th1dK}.
Note that inequality (\ref{laplas}) is not true anymore if $p=1$ or $p=\infty$. Therefore, since Theorems~\ref{mainlemma} and~\ref{th1dK} cannot be applied, we will provide a direct proof of the sharp Ulyanov inequalities for moduli of smoothness.
Our main result of this section reads as follows.

\begin{theorem}\label{thMainMod}
    Let $f\in L_p(\T^d)$, $d\ge 2$, $0<p<q\le\infty$,  $\a \in \N\cup ((1-1/q)_+,\infty)$, and $\g,m\ge 0$ be such that $\a+\g, \a+m\in \N\cup ((1/p-1)_+,\infty)$.
    Then, for any $\d \in (0,1)$, we have
    \begin{equation}\label{eqlemMM1}
        \w_\a(f,\d)_q\lesssim \frac{\w_{\a+\g}(f,\d)_p}{\d^\g}\s\(\frac 1\d\)+\(\int_0^\d   \(  \frac{\w_{\a+m}(f,t)_p}{t^{d(\frac1p-\frac1q)}} \)^{q_1}\frac {dt}{t}\)^\frac 1{q_1},
    \end{equation}
where % $\s$ is defined by the following equalities

{\rm (1)} if $0<p\le 1$ and $p<q\le\infty$, then
$$
\s(t)
:=\left\{
         \begin{array}{ll}
           t^{d(\frac1p-1)}, & \hbox{$\g> d\(1-\frac1q\)_+$}; \\
           t^{d(\frac1p-1)}, & \hbox{$\g=d\(1-\frac1q\)_+\ge 1$ and $\a+\g\in \N$}; \\
           t^{d(\frac1p-1)}\ln^\frac1{q_1} (t+1), & \hbox{$\g=d\(1-\frac1q\)_+\ge 1$  and $\a+\g\not\in \N$}; \\
           t^{d(\frac1p-1)}\ln^\frac1{q} (t+1), & \hbox{$0<\g=d\(1-\frac1q\)_+<1$}; \\
           t^{d(\frac1p-\frac1q)-\g}, & \hbox{$0< \g<d\(1-\frac1q\)_+$};\\
           t^{d(\frac1p-\frac1q)}, & \hbox{$\g=0$;}
         \end{array}
       \right.
$$

{\rm (2)} if $1<p\le q\le\infty$, then
$$
\s(t):=
\left\{
         \begin{array}{ll}
           1, & \hbox{$\g\ge d(\frac1p-\frac1q),\quad q<\infty$}; \\
           1, & \hbox{$\g> \frac dp,\quad q=\infty$}; \\
           \ln^\frac1{p'} (t+1), & \hbox{$\g=\frac dp,\quad q=\infty$}; \\
           t^{d(\frac1p-\frac1q)-\g}, & \hbox{$0\le \g<d(\frac1p-\frac1q)$}.\\
         \end{array}
       \right.
$$

\end{theorem}

\begin{remark}
\textnormal{Equivalence (\ref{laplas}) and Theorem \ref{th1dK} {(B)} give the sharpness of inequality (\ref{eqlemMM1}) in the case $1<p<\infty$.
}
\end{remark}
\begin{corollary}
Under all conditions of Theorem \ref{thMainMod}, we have
   $$        \w_\a(f,\d)_q\lesssim \(\int_0^\d   \(
   \frac{\w_{\a+\g}(f,t)_p}{t^\g}\s\(\frac 1t\)
 \)^{q_1}\frac {dt}{t}\)^\frac 1{q_1}.
   $$
\end{corollary}

 \begin{proof}[Proof of Theorem \ref{thMainMod}]

By Theorem~\ref{lemMainMod} and Corollary \ref{corNSB}, we have that
\begin{equation*}
  \eta(n)\asymp n^\g \sup_{T_n\in\mathcal{T}_n'}\frac{\sup_{|\xi|=1,\,\xi\in\R^d}\bigg\Vert \(\frac{\partial}{\partial\xi}\)^{\a} T_n\bigg\Vert_{q}}{\sup_{|\xi|=1,\,\xi\in\R^d}\bigg\Vert \(\frac{\partial}{\partial\xi}\)^{\a+\g} T_n\bigg\Vert_{p}}.
\end{equation*}
Hence, it is enough to prove that
\begin{equation}\label{will}
  \sup_{T_n\in\mathcal{T}_n'}\frac{\sup_{|\xi|=1,\,\xi\in\R^d}\bigg\Vert \(\frac{\partial}{\partial\xi}\)^{\a} T_n\bigg\Vert_{q}}{\sup_{|\xi|=1,\,\xi\in\R^d}\bigg\Vert \(\frac{\partial}{\partial\xi}\)^{\a+\g} T_n\bigg\Vert_{p}}\lesssim \s(n).
\end{equation}
%which in some cases (see cases 3.5) and 3.6) below) will prove the theorem.
Note that for $\g=0$ and $0<p<q\le\infty$, this follows immediately from the classical Nikol'skii inequality~\eqref{nik}.
%$$\|T\|_q\lesssim n^{d(1/p-1/q)}\|T\|_p,
%\qquad 0<p<q\le\infty,
%$$
%see, e.g., Theorem \ref{th-hardy-l-n}.

The rest of the proof for $\g>0$ is divided into three cases.

 \vspace{4mm}

 %\vspace{4mm}
\underline{Case 1.} Let $1<p<q< \infty$. In this case,  Theorem~\ref{th-hardy-l-n} yields  \eqref{will}  since,  by
Corollary~\ref{corNSB}, we have that
$$
\sup_{|\xi|=1,\,\xi\in\R^d}
\bigg\Vert \(\frac{\partial}{\partial\xi}\)^{\a} T_n\bigg\Vert_{q}\asymp \Vert (-\D)^{\a/2}T_n\Vert_q
$$
and
$$
\sup_{|\xi|=1,\,\xi\in\R^d}
\bigg\Vert \(\frac{\partial}{\partial\xi}\)^{\a+\g} T_n\bigg\Vert_{p}\asymp \Vert (-\D)^{(\a+\g)/2}T_n\Vert_p.
$$

 \vspace{4mm}
\underline{Case 2.} Let $1<p<q=\infty$. To show \eqref{will}, we  note that by Theorem~\ref{th-hardy-l-n} and Remark~\ref{remark-HLN} with $\vp(y)\equiv 1$ and $\psi(y)=|y|^{-\g}$, we have,
for any $T_n\in \mathcal{T}_n'$ and $\xi\in \R^d$, $|\xi|=1$,
$$
\bigg\Vert \(\frac{\partial}{\partial\xi}\)^{\a} T_n\bigg\Vert_{\infty}\lesssim \s(n)\bigg\Vert (-\D)^{\g/2}\( \frac{\partial}{\partial\xi}\)^{\a} T_n\bigg\Vert_{p}.
$$
Applying twice Corollary~\ref{lemRAZZZ}, we continue estimating as follows
$$
\lesssim \s(n) \Vert (-\D)^{(\a+\g)/2}T_n\Vert_p
\lesssim \s(n)\sup_{\xi\in \R^d,\, |\xi|=1}\bigg\Vert \(\frac{\partial}{\partial\xi}\)^{\a+\g} T_n\bigg\Vert_{p},
$$
which is the desired estimate.

 \vspace{4mm}
\underline{Case 3.} Let $0<p\le 1$, $p<q\le \infty$, and $\g>0$.
In this case, we  use the Marchaud-type inequality (see \eqref{eqthRealKUMMod1} below):
\begin{equation}\label{Marshonetam}
  \w_\a(f,\d)_q\lesssim \d^\a \(\int_\d^1 \(\frac{\w_{\a+\g}(f,t)_p}{t^{\a+d(\frac1p-\frac1q)}}\)^\tau\frac{dt}{t}\)^\frac1\tau+\(\int_0^\d \(\frac{\w_{\a+m}(f,t)_p}{t^{d(\frac1p-\frac1q)}} \)^{q_1}\frac{dt}{t}   \)^\frac1{q_1},
\end{equation}
where
$$
\tau=\left\{
       \begin{array}{ll}
         \min(q,2), & \hbox{$q<\infty$;} \\
         1, & \hbox{$q=\infty$.}
       \end{array}
     \right.
$$
By \eqref{eqMonMod}, it is easy to see that
\begin{equation}\label{nomer3}
  \d^\a \(\int_\d^1 \(\frac{\w_{\a+\g}(f,t)_p}{t^{\a+d(\frac1p-\frac1q)}}\)^\tau\frac{dt}{t}\)^\frac1\tau\lesssim \frac {\w_{\a+\g}(f,\d)_p}{\d^{\g+d(\frac1p-1)}}\left\{
                                          \begin{array}{ll}
                                            1, & \hbox{$\g>d\(1-\frac1q\)$;} \\
                                            \ln^{1/\tau}\(\frac1\d+1\), & \hbox{$\g=d\(1-\frac1q\)$;} \\
                                            \(\frac1\d\)^{d(1-\frac1q)-\g}, & \hbox{$0<\g<d\(1-\frac1q\)$.}
                                          \end{array}
                                        \right.
\end{equation}
Hence, \eqref{Marshonetam} and \eqref{nomer3} imply that \eqref{eqlemMM1} holds in the following four cases:

\begin{itemize}

  \vspace{2mm} \item[(i)] $\g>0$ and $0<q\le 1$;

 \vspace{2mm}
  \item[(ii)] $\g=d(1-1/q)$ and $1<q\le 2$ with $\a+\g\not\in \N$;

 \vspace{2mm}  \item[(iii)] $\g=d(1-1/q)$ and $q=\infty$ with $\a+\g\not\in \N$;

 \vspace{2mm} \item[(iv)]  $0<\g<d\(1-1/q\)_+$.

\end{itemize}

\noindent Thus, it remains to consider:

 \vspace{2mm}
(v) $\g=d\(1-1/q\)\ge 1$, $\a+\g\not\in \N$, and $2<q<\infty$
(note that for such $q$ and $d\ge 2$ we always have $d(1-1/q)\ge 1$);

 \vspace{2mm}
(vi) $\g=d\(1-1/q\)\ge 1$, $1<q\le \infty$, and $\a+\g\in \N$.

\noindent We have that inequality \eqref{will} follows from Lemma~\ref{lemma+} in the case (v) and from Lemma~\ref{lemPolSob} in the case (vi).
\end{proof}

We finish this section with an important corollary.
%Finally, we obtain the sharp Ulyanov inequality between $(L_1, L_\infty)$. % for moduli of smoothness of integer orders.
\begin{corollary}\label{important corollary} Let
 $0<p\le 1<q\le\infty$ and $d\ge 2$. We have
\begin{equation}\label{zzvvdaa}
  \w_\a(f,\d)_q\lesssim \(\int_0^\d   \(
   \frac{\w_{\a+d(1-\frac1q)}(f,t)_p}{t^{d(\frac1p-\frac1q)}}
 \)^{q_1}\frac {dt}{t}\)^\frac 1{q_1}
\end{equation}
   provided that $\a+d(1-1/q)\in\N$.
In particular, for any $d, \a\in\N$,
     we have
    \begin{equation}\label{zzvvdaa2}
        \w_\a(f,\d)_\infty\lesssim
        \int_0^\d   \frac{\w_{\a+d}(f,t)_1}{t^{d}} \frac {dt}{t}.
    \end{equation}
\end{corollary}
Note  that~\eqref{zzvvdaa} does not hold in the one-dimensional case.
 Comparing inequality~\eqref{zzvvdaa} for $p=1$  with (\ref{ul-1+}), % the sharp Ulyanov inequality in the one-dimensional case,
 we
observe new effects in the multivariate case.
For $d=1$, inequality~\eqref{zzvvdaa2} is known (see \cite{Ti}).

\bigskip

\newpage

\section{Sharp Ulyanov inequalities for realizations of $K$-functionals related to Riesz derivatives and corresponding moduli of smoothness%of moduli of smoothness in multidimensional case %$\T^d$
}
\label{sec9}
Let $d\ge 1$, $\psi(\xi)=|\xi|^\a$, and $\vp(\xi)=|\xi|^{\a+\g}$.
By analogy with the one-dimensional case, we define
$$
\mathcal{R}_{\langle\a\rangle}(f,\d)_p=\inf_{T\in \mathcal{T}_{[1/\d]}}\{\Vert
f-T\Vert_p+\d^\a\Vert (-\D)^{\a/2} T\Vert_p\},
$$
where
$$
(-\D)^{\a/2}T(x)=\sum_{|k|_\infty\le n} |k|^\a c_k e^{i(k,x)},\qquad x\in \T^d.
$$

Theorem~\ref{th1dK} for such $\psi$ and $\vp$ immediately implies the following sharp Ulyanov inequality for $\mathcal{R}_{\langle\a\rangle}(f,\d)_p$.

\begin{theorem}\label{th1d}
 Let $f\in L_p(\T^d)$, $d\ge 1$, $0<p<q\le\infty$, $\a>0$, and $\g\ge 0$.
 \\
 \textnormal{(A)} For any $\d\in (0,1)$,  we have
    \begin{equation}\label{eqth1.1Kd2}
        \mathcal{R}_{\langle\a\rangle}(f,\d)_q\lesssim \frac{\mathcal{R}_{\langle\a+\g\rangle}(f,\d)_p}{\d^{\g}}\s_\D\(\frac 1\d\)+
            \left(\int_0^\d
            \bigg(
            \frac{\mathcal{R}_{{\langle\a+\g\rangle}}(f,t)_p}{t^{d(\frac1p-\frac1q)}}
            \bigg)^{q_1}\frac{dt}{t}\right)^{\frac1{q_1}},
    \end{equation}
where % $\s(t)$ is defined by the following equalities:

{\rm (1)} if $0<p\le 1$ and $p<q\le\infty$, then
$$
\s_\D(t)
:=\left\{
         \begin{array}{ll}
           t^{d(\frac1p-1)}, & \hbox{$\g> d\(1-\frac1q\)_+$}; \\
           t^{d(\frac1p-1)}\ln^\frac1{q_1} (t+1), & \hbox{$0<\g=d\(1-\frac1q\)_+$}; \\
           t^{d(\frac1p-\frac1q)-\g}, & \hbox{$0< \g<d\(1-\frac1q\)_+$};\\
           t^{d(\frac1p-\frac1q)}, & \hbox{$\g=0$},
         \end{array}
       \right.
$$

{\rm (2)} if $1<p\le q\le\infty$, then
$$
\s_\D(t):=\left\{
         \begin{array}{ll}
           1, & \hbox{$\g\ge d(\frac1p-\frac1q),\quad q<\infty$}; \\
           1, & \hbox{$\g> \frac dp,\quad q=\infty$}; \\
           \ln^\frac1{p'} (t+1), & \hbox{$\g=\frac dp,\quad q=\infty$}; \\
           t^{d(\frac1p-\frac1q)-\g}, & \hbox{$0\le \g<d(\frac1p-\frac1q)$}.\\
         \end{array}
       \right.
$$
 \textnormal{(B)}
Inequality  \eqref{eqth1.1Kd2} is sharp in the following sense.
Let $\a\in (2\N)\cup (d(1/q-1)_+,\infty)$  and $\a+\g>d(1-1/q)_+$.
%Let $\psi$ be either a polynomial or  $\frac d{d+\a}<q\le \infty$.
%Let also $\psi(x)=C\vp(x)$ with some constant $C\in \C\setminus\{0\}$ in the case $\g=0$ and $0<p<q\le 1$
%or $\g=0$ and $0<p\le 1$, $q=\infty$.
Then there exists a function  $f_0\in L_q(\T^d)$, $f_0\not\equiv \const$,
such that
   $$
\mathcal{R}_{\langle\a\rangle}(f_0,\d)_q \asymp \frac{\mathcal{R}_{\langle\a+\g\rangle}(f_0,\d)_p}{\d^{\g}}\s_\D\(\frac 1\d\)+
            \left(\int_0^\d
            \bigg(
            \frac{\mathcal{R}_{{\langle\a+\g\rangle}}(f_0,t)_p}{t^{d(\frac1p-\frac1q)}}
            \bigg)^{q_1}\frac{dt}{t}\right)^{\frac1{q_1}}
   $$
 as $\d\to 0$.

\end{theorem}

\begin{remark}
\textnormal{We would like to  emphasize  that the main difference in the definitions of $\s$  in Theorems \ref{th1},
\ref{thMainMod}, and
\ref{th1d}  is when $\g$ is the critical parameter, i.e.,  $\g=d\(1-1/q\)_+$.
%In this case  $\s(t)=t^{d(\frac1p-1)}$ in Theorem \ref{thMainMod} and $\s(t)=t^{d(\frac1p-1)}\ln 2t$ in Theorem \ref{th1d}.
}

\end{remark}

%\begin{proof}
%The proof of (A) follows from Theorem \ref{th1dK} with $\psi(x)=|x|^\a$ and $\vp(x)=|x|^{\a+\g}$. We only note that
%in the  case $\g=0$ and $q\le 1$ we have
%$$
%\|\mathcal{D}\(\psi/\vp\) V_n\|_q \asymp n^{d(1-\frac1q)},
%$$
%see Remark \ref{mu+}.
%
%Sharpness of (\ref{eqth1.1Kd2}) follows from part (B) of Theorem \ref{th1dK}.
%\end{proof}

Under certain additional restrictions, similarly to Theorem~\ref{th1dK}$'$, one can prove the following
more accurate inequality.
%, which follows from Theorem~\ref{th1d} and an analogue of (\ref{bsp}) for $\mathcal{R}_{\langle\a\rangle}(f,t)_p$.

\begin{remark}\label{rem51}
\textnormal{Let $f\in L_p(\T^d)$, $d\ge 1$, $0<p<q<\infty$, $\a>0$, and $\g,m\ge 0$ be such that $\a+\g,\a+m\in (2\N) \cup (d(1/p-1)_+,\infty)$.
Then, for any  $\d\in (0,1)$, we have
    \begin{equation*}
    %\label{eqth1.1d}
        \mathcal{R}_{\langle\a\rangle}(f,\d)_q\lesssim
        \frac{\mathcal{R}_{\langle\a+\g\rangle}(f,\d)_p}{\d^{\g}}\s_\D\(\frac 1\d\)
        +
        \left(\int_0^\d\bigg(\frac{\mathcal{R}_{{\langle\a+m\rangle}}(f,t)_p}{t^{d(\frac1p-\frac1q)}}\bigg)^{q}\frac{dt}{t}\right)^{\frac1{q}},
    \end{equation*}
where $\s_\D(\cdot)$ is defined in Theorem~\ref{th1d}.
}
\end{remark}

Note that in some cases the realization $\mathcal{R}_{\langle\a\rangle}$ can be replaced
by special moduli of smoothness (see, for example,~\cite{KT} for the case $q\ge 1$ and~\cite{DR} for the case $0<q<1$; see also \cite{RS3, RS4, RS5}).

\smallskip

%{\sc Example 1.}
\begin{example}
\textnormal{Let $\a>0$, $r\in\N$, and  $1\le q\le\infty$. Denote
\begin{equation*}
%\label{eqModTr}
    \mathrm{w}_{\a}(f,\d)_q:=\bigg\Vert \int_{|u|\ge 1}\frac{\dot{\Delta}_{\d u}^{2r}f(\cdot)}{|u|^{d+\a}} du\bigg\Vert_{q},\qquad
    r=[(d-1+\a)/2]+1,
\end{equation*}
%   {\bf  Change: $ {w}_{\a,r}(f,\d)_q\to  \mathrm{w}_{\a}(f,\d)_q$ and define $r$  as $
%r=[(d-1+\a)/2]+1$}
where
$$
\dot{\Delta}_h^r=\dot{\Delta}_h(\dot{\Delta}_h^{r-1}),\quad
\dot{\Delta}_h^rf(x)=f(x+h)-f(x-h),\quad h\in\R^d.
$$
It follows from \cite{KT}, \cite{RS2}, and Lemma \ref{lemma4.1} that, for $1\le q\le\infty$, one has
\begin{eqnarray*}
\mathrm{w}_{\a}(f,\d)_q &\asymp&
\mathcal{R}_{\langle\a\rangle}\(f,\d\)_q = \inf\limits_{ T \in
\mathcal{T}_{[1/\d]} }\left(\|f-T\|_q+\d^\a\|(-\D)^{\a/2}
T\|_q\right)
\\
&\asymp&
{K}_{\langle\a\rangle}\(f,\d\)_q=
\inf\limits_{
(-\D)^{\a/2} g \in L_q(\mathbb{T}^d)}\left(\|f-g\|_q+\d^\a\|(-\D)^{\a/2} g\|_q\right). %, \qquad 1\le q\le\infty.
\end{eqnarray*}
}
\end{example}

\smallskip

%{\sc Example 2.}
\begin{example}
 \textnormal{Now, we consider the following moduli of smoothness, which were studied in \cite{RS4} for any  $0<q\le \infty$  (see also \cite{DR} for the case $m=1$):
%{\bf  Change: $     \widetilde{w}_{m,d}(f,\d)_q\to \widetilde{w}_2(f,\d)_q:=\widetilde{w}^{(m)}_2(f,\d)_q:=$%  \mathrm{w}_{\a}(f,\d)_q$
%}
\begin{equation*}
%\label{eqModDR}
    \widetilde{w}^{(m)}_2(f,\d)_q:=\sup_{|h|\le\d}\bigg\Vert\frac{\xi_m}{d} \sum_{j=1}^d \mathop{{\sum}'}_{\nu=-m}^m\frac{(-1)^\nu}{\nu^2}\binom{2m}{m-|\nu|} (f(\cdot+\nu h e_j)-f(\cdot))\bigg\Vert_{q},
\end{equation*}
where $m,d\in \N$ and
$$
\xi_m=\(2\sum_{\nu=1}^m \frac{(-1)^\nu}{\nu^2}\binom{2m}{m-\nu}\)^{-1}.
$$
Set
$$
q_{m,d}=\left\{
          \begin{array}{ll}
            0, & \hbox{$d=1$;} \\
            \frac{d}{d+2(m+1)}, & \hbox{$d=2, m=2k, k\in \N$;} \\
            \frac{d}{d+2m}, & \hbox{otherwise.}
          \end{array}
        \right.
$$
It is known~\cite{RS4} (see also~\cite[Theorem~3.2]{DR}) that
$$
\widetilde{w}^{(m)}_2(f,\d)_q %\asymp\widetilde{K}_2\(f,\frac1n\)
%\asymp E_n(f)_p+n^{-2}\Vert\D T_n^*\Vert_p
\asymp \mathcal{R}_{\langle 2 \rangle}(f,\d)_q,\qquad
q_{m,d}<q\le\infty.
$$
%where$$\Vertf-T_n^*\Vert_p=E_n(f)_p.$$
Thus, for all $1\le q\le\infty$ and $m\in \N$, we have
$$
\mathrm{w}_{2}(f,\d)_q\asymp \widetilde{w}^{(m)}_2(f,\d)_q \asymp
\mathcal{R}_{\langle 2 \rangle}\(f,\d\)_q.
$$
}
\end{example}

Using the above equivalences and Remark \ref{rem51}, we obtain the following result.

\begin{corollary}\label{th1d}
Let $f\in L_p(\T^d)$, $d\ge 1$. For any $\d\in (0,1)$,  one has the following three inequalities:

{\rm (i)} if $1\le p<q \le\infty$ and $\a,\g>0$, then
 \begin{equation*}
 %\label{eqth1.1d}
        \mathrm{w}_{\a}(f,\d)_q\lesssim
        \frac{\mathrm{w}_{\a+\g}(f,\d)_p}{\d^{\g}}\s_\D\(\frac 1\d\)
        +
        \left(\int_0^\d\bigg(\frac{\mathrm{w}_{\a+\g}(f,t)_p}{t^{d(\frac1p-\frac1q)}}\bigg)^{q_1}\frac{dt}{t}\right)^{\frac1{q_1}};
    \end{equation*}

{\rm (ii)}  if $p_{m,d}<p<q$, $1\le q\le\infty$, and $0<\g<2$, then
    \begin{equation*}
    %\label{eqth1.1d}
        \mathrm{w}_{2-\g}(f,\d)_q\lesssim
        \frac{\widetilde{w}^{(m)}_2(f,\d)_p}{\d^{\g}}\s_\D\(\frac 1\d\)
        +
        \left(\int_0^\d\bigg(\frac{\widetilde{w}^{(m)}_2(f,t)_p}{t^{d(\frac1p-\frac1q)}}\bigg)^{q_1}\frac{dt}{t}\right)^{\frac1{q_1}};
    \end{equation*}

{\rm (iii)}  if $p_{m,d}<p< q\le\infty$, then
    \begin{equation*}
    %\label{eqth1.1d}
        \widetilde{w}^{(m)}_2(f,\d)_q\lesssim
        \left(\int_0^\d\bigg(\frac{\widetilde{w}^{(m)}_2(f,t)_p}{t^{d(\frac1p-\frac1q)}}\bigg)^{q_1}\frac{dt}{t}\right)^{\frac1{q_1}},
    \end{equation*}

where $\s_\D(\cdot)$ is defined in Theorem \ref{th1d}.
\end{corollary}

\smallskip

%Note also if $d=1$, and $1/(\b+1)<p\le\infty$,
%
%
%then the following modulus of smoothness (see~\cite{RS5}) can be used instead of corresponding fractional realization  of $K$-functional
%$$
%\w_{\langle\b\rangle}(f,\d)_p=\sup_{0<h\le \d}\bigg\Vert \sum_{\nu=-\infty,\, \nu\neq 0}^\infty \frac{f(\cdot+\nu h)-f(\cdot)}{|\nu|^{\b+1}}\bigg\Vert_p.
%$$
%
%It is known that (see~\cite{RS5}) that for $1/(\b+1)<p\le\infty$
%$$
%\w_{\langle\b\rangle}(f,\d)_p\asymp \widetilde{K}_\b(f,\d),\quad \d>0.
%$$

\smallskip

%{\sc Example 3.}
\begin{example}
\textnormal{A more complete picture can be obtained when $d=1$. We will deal with  the following special modulus of smoothness, which
has been recently introduced by Runovski and Schmeisser~\cite{RS3} (see also~\cite{RS5}). For $\a>0$ and $f\in L_p(\T)$,  we define
\begin{equation*}
%\label{eqModulRiesz}
    \w_{\langle\a\rangle}(f,\d)_p:=\sup_{0<h\le \d}\bigg\Vert \sum_{\nu\in \Z\setminus \{0\}} \(-\frac{\b_{|\nu|}(\a)}{\b_0(\a)}\)
f(\cdot+\nu h)-f(\cdot) \bigg\Vert_p,
\end{equation*}
where
$$
\b_m(\a)=\sum_{j=m}^\infty (-1)^{j+1}2^{-2j}\binom{\a/2}{j}\binom{2j}{j-m},\quad m\in \Z_+.
$$
It has been proved in~\cite{RS3}  that, for $\a>0$ and $1/(\a+1)<p\le\infty$, one has
\begin{equation}\label{snejinka}
  \w_{\langle\a\rangle}(f,\d)_p\asymp \mathcal{R}_{\langle\a\rangle}(f,\d)_p,\quad \d>0.
\end{equation}
Note also that, for all $\a\in 2\N$, the modulus $\w_{\langle\a\rangle}(f,\d)_p$ coincides with the classical modulus of smoothness $\w_{\a}(f,\d)_p$.
}
\end{example}

From equivalence~\eqref{snejinka} and Theorem~\ref{th1d}, we have

\begin{corollary}\label{Corth1dMod}
  {\it Let $f\in L_p(\T)$, $0<p<q\le\infty$, $\g\ge 0$,
$\a\in (2\N) \cup ((1/q-1)_+,\infty)$, and
$\a+\g\in (2\N) \cup ((1/p-1)_+,\infty)$.
Then, for any $\d\in (0,1)$,  we have
    \begin{equation*}
    %\label{eqth1.1Md2}
        \w_{\langle \a\rangle}(f,\d)_q\lesssim \frac{\w_{\langle \a+\g\rangle}(f,\d)_p}{\d^{\g}}\s_\D\(\frac 1\d\)+
            \left(\int_0^\d
            \bigg(
            \frac{\w_{\langle \a+\g\rangle}(f,t)_p}{t^{\frac1p-\frac1q}}
            \bigg)^{q_1}\frac{dt}{t}\right)^{\frac1{q_1}},
    \end{equation*}
where
$\s(\cdot)$ is defined  in Theorem \ref{th1d} in the case $d=1$.
}
\end{corollary}

\bigskip

\newpage

\section{Sharp Ulyanov inequality via Marchaud inequality}\label{sec10}

This section is devoted to the study of new types of Ulyanov inequalities with the use of the Marchaud inequalities.
Such inequalities are of great importance when investigating the embedding theorems in Section~\ref{sec13}.

Throughout this section, we assume that
$$
\tau=\tau(q)=\left\{
       \begin{array}{ll}
         \min(q,2), & \hbox{$q<\infty$;} \\
         1, & \hbox{$q=\infty$.}
       \end{array}
     \right.
$$

\subsection{Inequalities for realizations of $K$-functionals}

First, we  obtain an analogue of the  Marchaud inequality (cf. Theorem~\ref{lemMarchaudMod}) for the realizations of the $K$-functionals. %It turns out that in this case we do not need any restrictions for positive parameters $\a$ and $\g$.

\begin{lemma}\label{lemMarchaud}
Let $f\in L_q(\T^d)$, $0<q\le\infty$, $\a>0$, $\g\ge 0$, $\psi\in \mathcal{H}_\a$, and $\vp\in \mathcal{H}_{\a+\g}$. Let also
either $\psi$ be a polynomial
or $d/{(d+\a)}< q\le\infty$. Then, for any $\d\in (0,1)$, one has
\begin{equation}\label{eq.lemMarchaud}
  \mathcal{R}_{\psi}(f,\d)_q\lesssim \d^\a \(\int_\d^1 \(\frac{\mathcal{R}_{\vp}(f,t)_q}{t^{\a}}\)^\tau\frac{dt}{t}\)^\frac1\tau.
\end{equation}
\end{lemma}

\begin{proof}
We start with the inverse estimate given by
\begin{equation}\label{eqzv10}
  \mathcal{R}_{\psi}(f,\d)_q\lesssim \d^\a \(\int_\d^1 \(\frac{E_{[1/t]}(f)_q}{t^{\a}}\)^\tau\frac{dt}{t}\)^\frac1\tau.
\end{equation}
In the case $0<q\le 1$ and $q=\infty$, the proof of~\eqref{eqzv10}  can be found in~\cite[Theorem~4.26, p.~203]{run}.
In the case $1<q<\infty$, it  follows from Theorem~\ref{lemMarchaudMod} (see also~\cite[Theorem 2.1]{Treb}) and~\eqref{laplas}.

To finish the proof of \eqref{eq.lemMarchaud}, %Finally, inequality~\eqref{eqzv10} and
we apply   Jackson's inequality, which follows from Lemma~\ref{real}.
\end{proof}

One of the main results in this section is the following sharp Ulyanov inequality via Marchaud inequality.
\begin{theorem}\label{thRealKUM}
Let  $0<p<q\le\infty$, $\a>0$, $\g\ge 0$, $\psi\in \mathcal{H}_\a$, and $\vp\in \mathcal{H}_{\a+\g}$. Let also
either $\psi$ be a polynomial
or $d/{(d+\a)}< q\le\infty$.

  {\rm (A)}  Let $f\in L_p(\T^d)$. Then, for any $\d\in (0,1)$, we have
\begin{equation}\label{eqthRealKUM}
  \mathcal{R}_{\psi}(f,\d)_q\lesssim \d^\a \(\int_\d^1 \(\frac{\mathcal{R}_{\vp}(f,t)_p}{t^{\a+d(\frac1p-\frac1q)}}\)^\tau\frac{dt}{t}\)^\frac1\tau+\(\int_0^\d \(\frac{\mathcal{R}_{\vp}(f,t)_p}{t^{d(\frac1p-\frac1q)}} \)^{q_1}\frac{dt}{t}   \)^\frac1{q_1}.
\end{equation}

{\rm (B)}
Inequality \eqref{eqthRealKUM} % при $0<p\le 1$ и $p<q\le \infty$ или $1<p<q<\infty$
 is sharp in the following sense.
 Let % $\psi(x)=0$ iff $x=0$ and
$\frac \vp\psi$ and $\frac \psi\vp \in C^\infty(\R^d \backslash \{0\})$,
$\a+\g>d(1-1/q)_+$, $\g>0$, and
\begin{enumerate}
  \item $\g\neq d(1-1/q)>0$ if $2<q<\infty$ and $0<p\le 1$\textnormal{;}
  \item  $\g\neq 1$ if $d=1$, $q=\infty$, $0<p\le 1$, {and} $\frac{\psi(\xi)}{\vp(\xi)}= A{|\xi|^{-\g}}\sign \xi$ \; for any \; $A\in \C \setminus \{0\}$\textnormal{;}
  \item  $\g\neq d(1/p-1/q)$ if $1<p<q\le \infty$.
\end{enumerate}
Then there exists a function $f_0\in L_q(\T^d)$, $f_0\not\equiv \const$,
such that
\begin{equation}\label{eqthRealKUMToch}
  \mathcal{R}_{\psi}(f_0,\d)_q\asymp \d^\a \(\int_\d^1 \(\frac{\mathcal{R}_{\vp}(f_0,t)_p}{t^{\a+d(\frac1p-\frac1q)}}\)^\tau\frac{dt}{t}\)^\frac1\tau+\(\int_0^\d \(\frac{\mathcal{R}_{\vp}(f_0,t)_p}{t^{d(\frac1p-\frac1q)}} \)^{q_1}\frac{dt}{t}   \)^\frac1{q_1}
\end{equation}
as $\d\to 0$.
\end{theorem}

Similarly to  Theorem~\ref{mainlemma}$'$, one can obtain the following  analogue of Theorem~\ref{thRealKUM}.

\medskip

{\sc Theorem~\ref{thRealKUM}$'$} {\it Under all conditions of Theorem~\ref{thRealKUM}, suppose that the function $\vp$ is either a polynomial or $d/(d+\a+\g)<p\le\infty$. Let also $m>0$ and the function $\phi\in \mathcal{H}_{\a+m}$ be either a polynomial or $d/(d+\a+m)<p\le\infty$. Then, for any $\d\in (0,1)$, we have
\begin{equation}\label{eqthRealKUM+++++++++++++++++++}
  \mathcal{R}_{\psi}(f,\d)_q\lesssim \d^\a \(\int_\d^1 \(\frac{\mathcal{R}_{\vp}(f,t)_p}{t^{\a+d(\frac1p-\frac1q)}}\)^\tau\frac{dt}{t}\)^\frac1\tau+ \(\int_0^\d \(\frac{\mathcal{R}_{\phi}(f,t)_p}{t^{d(\frac1p-\frac1q)}} \)^{q_1}\frac{dt}{t}   \)^\frac1{q_1}.
\end{equation}
}

It is worth  mentioning the following important corollary for the case $\g=0$.
\begin{corollary}\label{corKg0}
  Let $f\in L_p(\T^d)$,  $0<p<q\le1$, $\psi, \vp\in \mathcal{H}_\a$, $\a>0$.  Let
either $\psi$ be a polynomial
or $d/{(d+\a)}< q\le 1$ and let either
 $\vp$ be a polynomial or $d/(d+\a)<p< 1$. Let also $m>0$ and the function $\phi\in \mathcal{H}_{\g+m}$ be either a polynomial or $d/(d+\a+m)<p<1$. Then, for any $\d\in (0,1)$, we have
\begin{equation}\label{eqthRealKUM+++++++++++++++}
  \mathcal{R}_{\psi}(f,\d)_q\lesssim \(\int_0^\d \(\frac{\mathcal{R}_{\phi}(f,t)_p}{t^{d(\frac1p-\frac1q)}} \)^{q}\frac{dt}{t}   \)^\frac1{q}+\frac{\mathcal{R}_{\vp}(f,\d)_p}{\d^{d(\frac1p-1)}}\left\{
               \begin{array}{ll}
                 1, & \hbox{$0<q<1$;} \\
                 \ln\(\frac1\d+1\), & \hbox{$q=1$.}
               \end{array}
             \right.
\end{equation}
\end{corollary}

\begin{proof}
The proof of~\eqref{eqthRealKUM+++++++++++++++} follows from~\eqref{eqthRealKUM+++++++++++++++++++} and Lemma~\ref{monoto}.
\end{proof}

\medskip

\begin{remark}
\textnormal{
{\rm (i)} We note that all three cases (1)--(3) in part (B) of Theorem \ref{thRealKUM},
  where the optimality of inequality (\ref{eqthRealKUM}) cannot be proved,
are those cases where % there exist functions $f$ such that
the  Ulyanov inequality, given by Theorem~\ref{th1dK}, provides a sharper estimate than inequality (\ref{eqthRealKUM}).
Moreover, one can construct
%the following
  examples showing that equivalence  (\ref{eqth1.2Kd}) holds in these cases.
 Indeed,
in  cases (1) and (2), taking $f=f_0$ such that $\mathcal{R}_{\vp}(f_0,t)_p\asymp t^{d(1/p-1)+\a+\g}$, we see that the equivalence in~\eqref{eqth1.2Kd} is true, while in~\eqref{eqthRealKUMToch} the equivalence for $f_0$ is not valid. In the case (3), we obtain the same conclusion by taking any $f_0\in C^\infty(\T^d)$.}

\textnormal{{\rm (ii)} At the same time, taking $f_0\in C^\infty(\R^d)$ and letting $0<p<1$ and $\g=d(1-1/q)_+$, we have that the sharp Ulyanov inequality given by~\eqref{eqthRealKUM+++++++++++++++++++} becomes the equivalence
$$
  \mathcal{R}_{\psi}(f_0,\d)_q\asymp \d^\a\asymp \d^\a \(\int_\d^1 \(\frac{\mathcal{R}_{\vp}(f_0,t)_p}{t^{\a+d(\frac1p-\frac1q)}}\)^\tau\frac{dt}{t}\)^\frac1\tau+ \(\int_0^\d \(\frac{\mathcal{R}_{\phi}(f_0,t)_p}{t^{d(\frac1p-\frac1q)}} \)^{q_1}\frac{dt}{t}   \)^\frac1{q_1}.
$$
On the other hand, the sharp Ulyanov inequality given by~\eqref{eqth1.1Kd0+++} implies only}
$$
\mathcal{R}_{\psi}(f_0,\d)_q\asymp \d^\a \lesssim \d^\a\ln^{\frac1q}\frac1\d\asymp \frac{\mathcal{R}_{\vp}(f_0,\d)_p}{\d^{\g}}\s\(\frac
        1\d\)+
        \left(\int_0^\d
        \bigg(
        \frac{\mathcal{R}_{\phi}(f_0,t)_p}{t^{d(\frac1p-\frac1q)}}
        \bigg)^{q_1}\frac{dt}{t}
        \right)^{\frac1{q_1}}.
$$

\end{remark}
\begin{proof}
(A)
Using the  Ulyanov type inequality given by \eqref{eqth1.1Kd0CORCORCOR22} and then the Marchaud inequality given by~\eqref{eq.lemMarchaud}, we get
\begin{equation}\label{eqUlRealKforM2}
\begin{split}
  \mathcal{R}_{\psi}(f,\d)_q
&\lesssim \d^\a
\(\int_\d^1
u^{-\a \tau}
\(
\int_0^u \( \frac{\mathcal{R}_{\vp}(f,t)_p}{t^{d(\frac1p-\frac1q)}} \)^{q_1}\frac{dt}{t}
\)^{\frac\tau{q_1}}
\frac{du}u
\)^\frac1\tau\\
&\lesssim \d^\a \(\int_\d^1
u^{-\a \tau}
\(
\int_0^\d \( \frac{\mathcal{R}_{\vp}(f,t)_p}{t^{d(\frac1p-\frac1q)}} \)^{q_1}\frac{dt}{t}
\)^{\frac\tau{q_1}}
\frac{du}u
\)^\frac1\tau\\
&+\d^\a
\(\int_\d^1
u^{-\a \tau}
\(
\int_\d^u \( \frac{\mathcal{R}_{\vp}(f,t)_p}{t^{d(\frac1p-\frac1q)}} \)^{q_1}\frac{dt}{t}
\)^{\frac\tau{q_1}}
\frac{du}u
\)^\frac1\tau:=I_1+I_2.
\end{split}
\end{equation}
We start with the estimate of the first integral:
\begin{equation}\label{I1}
\begin{split}
  I_1&= \d^\a \(\int_\d^1\frac{du}{u^{1+\a \tau}}\)^\frac1\tau \(\int_0^\d \(\frac{\mathcal{R}_{\vp}(f,t)_q}{t^{d(\frac1p-\frac1q)}} \)^{q_1}\frac{dt}{t} \)^\frac1{q_1}\\
&\lesssim \(\int_0^\d \(\frac{\mathcal{R}_{\vp}(f,t)_q}{t^{d(\frac1p-\frac1q)}} \)^{q_1}\frac{dt}{t} \)^\frac1{q_1}.
\end{split}
\end{equation}
To estimate  $I_2$, we  consider two cases: $\tau>q_1$ and $\tau\le q_1$.
In the first case,  Hardy's inequality (see Lemma~\ref{har}) yields
\begin{equation}\label{I2.1}
  I_2\lesssim \d^\a \(\int_\d^1 \(\frac{\mathcal{R}_{\vp}(f,t)_p}{t^{\a+d(\frac1p-\frac1q)}}\)^\tau\frac{dt}{t}\)^\frac1\tau.
\end{equation}
In the second case, we use  the discretization of integrals.
Let $n, N\in \Z_+$ be such that $2^{-N}<\d\le u\le 2^{-n}$. Since by Lemma \ref{monoto}  we have
$$\mathcal{R}_{\vp}(f,t)_p\asymp \mathcal{R}_{\vp}(f,2t)_p,$$
then
\begin{equation*}
\begin{split}
  \(\int_\d^u \( \frac{\mathcal{R}_{\vp}(f,t)_p}{t^{d(\frac1p-\frac1q)}} \)^{q_1}\frac{dt}{t}
\)^{\frac\tau{q_1}}&\lesssim \(\int_{2^{-N}}^{2^{-n}} \( \frac{\mathcal{R}_{\vp}(f,t)_p}{t^{d(\frac1p-\frac1q)}} \)^{q_1}\frac{dt}{t}
\)^{\frac\tau{q_1}}\\
&\lesssim \(\sum_{k=n}^{N}\(  2^{kd(\frac1p-\frac1q)} \mathcal{R}_{\vp}(f,2^{-k})_p   \)^{q_1}\)^\frac\tau{q_1}\\
&\lesssim \sum_{k=n}^{N}\(  2^{kd(\frac1p-\frac1q)} \mathcal{R}_{\vp}(f,2^{-k})_p\)^\tau\\
&\lesssim \int_{2^{-N}}^{2^{-n}} \( \frac{\mathcal{R}_{\vp}(f,t)_p}{t^{d(\frac1p-\frac1q)}} \)^\tau\frac{dt}{t}
\lesssim \int_{\d}^u \( \frac{\mathcal{R}_{\vp}(f,t)_p}{t^{d(\frac1p-\frac1q)}} \)^\tau\frac{dt}{t}.
\end{split}
\end{equation*}
Thus,
\begin{equation}\label{I2.2}
\begin{split}
  I_2&\lesssim \d^\a \(\int_\d^1 u^{-\a \tau} \int_\d^u   \(\frac{\mathcal{R}_{\vp}(f,t)_p}{t^{d(\frac1p-\frac1q)}}\)^\tau\frac{dt}{t}\frac{du}{u}\)^\frac1\tau
\\
&=\d^\a
\(
\int_\d^1
\(
\frac{\mathcal{R}_{\vp}(f,t)_p}{t^{d(\frac1p-\frac1q)}}
\)^\tau
\int_t^1 \frac{du}{u^{1+\a\tau}}\frac{dt}{t}
\)^\frac1\tau \\
&\lesssim\d^\a
\(
\int_\d^1
\(
\frac{\mathcal{R}_{\vp}(f,t)_p}{t^{\a+d(\frac1p-\frac1q)}}
\)^\tau
\frac{dt}{t}
\)^\frac1\tau .
\end{split}
\end{equation}
Finally, combining \eqref{eqUlRealKforM2}--\eqref{I2.2}, we get \eqref{eqthRealKUM}.

 \vspace{4mm}
(B) The proof of this part
 goes along the same lines as the proof of part (B) of Theorem~\ref{th1dK}.
In what follows, we denote for the sake of simplicity
$$
UM(f,\d)_p:=\d^\a \(\int_\d^1 \(\frac{\mathcal{R}_{\vp}(f,t)_p}{t^{\a+d(\frac1p-\frac1q)}}\)^\tau\frac{dt}{t}\)^\frac1\tau+\(\int_0^\d \(\frac{\mathcal{R}_{\vp}(f,t)_p}{t^{d(\frac1p-\frac1q)}} \)^{q_1}\frac{dt}{t}   \)^\frac1{q_1}.
$$

\medskip

\noindent
\underline{Part (B1): $0<p\le 1$ and $p<q\le\infty$.}
Let $f_0$ be given by~\eqref{eqfun0ples1}, i.e.,
\begin{equation*}
%\label{eqfun0ples1UM}
    f_0(x)=F_{\vp}(x)-\frac1{N^{\a+\g}}F_{\vp}(Nx),\quad F_{\vp}(x)=\sum_{k\neq 0}\frac{e^{i(k,x)}}{\vp(k)},
\end{equation*}
%\begin{equation}\label{eqhvpdUM}
%    F_{\vp}(x)=\sum_{k\neq 0}\frac{e^{i(k,x)}}{\vp(k)}
%\end{equation}
and $N$ is a fixed sufficiently large integer which will be chosen later.

Let $\d\asymp [1/n]$. By (\ref{eqth1.8Kd}) and~\eqref{eqth1.9Kd}, we have, respectively,
\begin{equation}\label{eqth1.8KdUM}
    \mathcal{R}_{\vp}(f_0,\d)_p\lesssim \d^{d(\frac1p-1)+\a+\g}
\end{equation}
and
\begin{equation}\label{eqth1.9KdUM}
\mathcal{R}_{\psi}(f_0,1/n)_q\gtrsim  n^{-\a}
\Vert
\mathcal{D}(\psi/\vp) V_{ n}\Vert_q. %,\qquad 0<q\le\infty.
\end{equation}
Note that we have already obtained the following
two-sided bounds of  $\Vert
\mathcal{D}(\psi/\vp) V_{ n}\Vert_q$:

\medskip

\noindent {\rm (a)} if $0<q\le 1$, Lemma~\ref{v+} gives
$$
\Vert
\mathcal{D}(\psi/\vp) V_{ n}\Vert_q\asymp 1,
$$

\noindent {\rm (b)} if $1<q<\infty$,  Lemmas~\ref{v-} and~\ref{v} imply
\begin{equation*}
%\label{delta-calcUM}
\Vert
\mathcal{D}(\psi/\vp) V_{ n}\Vert_q\asymp
 %$$
%\s_{q,\g,d}(t):=
\left\{
         \begin{array}{ll}
           1, & \hbox{$\g>d\(1-\frac1q\)$}; \\
           \ln^\frac1{q} (n+1), & \hbox{$\g=d\(1-\frac1q\)$}; \\
           n^{d(1-\frac1q)-\g}, & \hbox{$ 0<\g< d\(1-\frac1q\)$,}
%           \\n^{d(1-\frac1q)}, & \hbox{$\g=0$,}
         \end{array}
       \right.
\end{equation*}

\noindent {\rm (c)} if $q=\infty$, by Lemma~\ref{v++},
$$
\Vert
\mathcal{D}(\psi/\vp) V_{ n}\Vert_q\asymp \left\{
         \begin{array}{ll}
           1, & \hbox{$\g>d$;} \\
           1, & \hbox{$\g=d=1$ {and}  $\frac{\psi(\xi)}{\vp(\xi)}=A{|\xi|^{-\g}}\sign \xi$ for some $A\in \C \setminus \{0\}$;} \\
           \ln (n+1), & \hbox{$\g=d=1$ {and}  $\frac{\psi(\xi)}{\vp(\xi)}\ne A{|\xi|^{-\g}}\sign \xi$ for any $A\in \C \setminus \{0\}$;} \\
           \ln (n+1), & \hbox{$\g=d\ge 2$;} \\
           n^{d-\g}, & \hbox{$0<  \g<d$.}
                    \end{array}
       \right.
$$

Now,  using \eqref{eqth1.8KdUM} and taking into account that $\a+\g> d(1-1/q)_+$, we obtain
\begin{equation}\label{UM1}
\begin{split}
  UM(f_0,\d)_p&\lesssim \d^{\a+\g-d(1-\frac1q)}+\d^\a\(\int_\d^1 \(t^{\g-d(1-\frac1q)}\)^\tau\frac{dt}{t}\)^\frac1\tau\\
  &\lesssim \d^\a\left\{
         \begin{array}{ll}
           1, & \hbox{$\g>d\(1-\frac1q\)$ and $0<q\le \infty$}; \\
           \ln^\frac1{\tau} \(\frac 1\d+1\), & \hbox{$\g=d\(1-\frac1q\)$ and $1<q\le \infty$}; \\
           \d^{\g-d(1-\frac1q)}, & \hbox{$ 0<\g< d\(1-\frac1q\)$ and $1<q\le \infty$.}
%           \\n^{d(1-\frac1q)}, & \hbox{$\g=0$,}
         \end{array}
       \right.
  \end{split}
\end{equation}
Thus, combining \eqref{eqth1.9KdUM} and \eqref{UM1}, we derive that for $\g>d(1-1/q)_+$ the required inequality
$$
UM(f_0,\d)_p\lesssim \mathcal{R}_{\psi}(f_0,\d)_q
$$
holds
in all cases except

\begin{itemize}
 \item[(i)]$2<q<\infty$ and $\g=d(1-1/q)$,
 \item[(ii)]$\g=d=1$, $q=\infty$, and $\frac{\psi(\xi)}{\vp(\xi)}=A{|\xi|^{-\g}}\sign \xi$ \;for some $A\in \C \setminus \{0\}$.
\end{itemize}

\bigskip

 \noindent
\underline{Part (B2): $1<p<q<\infty$ and $\g\neq d(1/p-1/q)$.}
In this case, if $\g> d(1/p-1/q)$, then we can take as $f_0$ any function from $C^\infty(\T^d)$.
Then \eqref{kkk} implies
\begin{equation}\label{kkkUM1}
    \begin{split}
\mathcal{R}_{\psi}(f_0,\delta)_q\asymp \d^\a\quad\text{and}\quad \mathcal{R}_{\vp}(f_0,\delta)_q\asymp \d^{\a+\g}.
\end{split}
\end{equation}
Hence,
\begin{equation}\label{kkkUM1+}
UM(f_0,\d)_p\asymp \mathcal{R}_{\psi}(f_0,\delta)_q\asymp \d^\a.
\end{equation}

If
$0\le \g<d(1/p-1/q)$, then we take
$$
f_0(x)=F_{\vp}(x)=\sum_{k\ne 0} \frac{e^{i(k,x)}}{\varphi(k)}.
$$
By (\ref{kkkk+}), we have
$$\mathcal{R}_{\vp}(f_0,\delta)_p\lesssim\delta^{\a+\g +d(\frac{1}{p}-1)}
$$%\end{equation}
and by (\ref{kkkk++})
$$%\begin{equation}\label{kkkk++UM}
%    \begin{split}
\mathcal{R}_{\psi}(f_0,\d)_q\gtrsim
\d^{\a+\g+d(\frac{1}{q}-1)}.
%\end{split}
$$%\end{equation}
Therefore, since $\g<d(1/p-1/q)<d(1-1/q)$, we have
$$
UM(f_0,\d)_p\lesssim \d^{\a+\g+d(\frac{1}{q}-1)}\lesssim \mathcal{R}_{\psi}(f_0,\d)_q.
$$

\bigskip
 \noindent
\underline{Part (B3): $1<p<\infty$, $q=\infty$, and $\g\neq d(1/p-1/q)$.}
Let first $\g>{d}/{p}$. Then we consider $f_0\in C^\infty(\T^d)$. In this case, we have the same equivalences as in (\ref{kkkUM1}) and
(\ref{kkkUM1+}).

Finally, assume that $0<\g<{d}/{p}$.
Here we deal with the function
$$
f_0(x)={\sum_{k\neq 0}} \frac{e^{i(k,x)}}{|k|^{\g}\psi(k)}.
$$
Then, (\ref{kkkk++--})  yields
$$
\mathcal{R}_{\psi}(f_0,\d)_\infty
\gtrsim
\d^{\a+\g-d}.
$$%\end{eqnarray}
On the other hand, by \eqref{kkkk+--},
$$
\mathcal{R}_{\vp}(f_0, \d)_p\lesssim
\d^{\a+\g +d(\frac{1}{p}-1)}.
$$
Thus, applying the last two estimates, we have
$$
UM(f_0,\d)_p\lesssim \d^{\a+\g-d}\lesssim \mathcal{R}_{\psi}(f_0,\d)_q.
$$
\end{proof}

\subsection{Inequalities for  moduli of smoothness}

Using Theorem~\ref{lemMarchaudMod}, the Ulyanov inequality (Theorems~\ref{th1} and~\ref{thMainMod} for $\g=0$),
and repeating arguments that were used in the proof of Theorem~\ref{thRealKUM}~(A), we obtain the following result.

\begin{theorem}\label{th1UMMod1}
 Let $f\in L_p(\T^d)$, $d\ge 1$, $0<p<q\le \infty$, $\a\in\N\cup
    ((1/q-1)_+,\infty)$, and $\g, m> 0$ be  such that $\a+\g, \a+m\in \N\cup
    ((1/p-1)_+,\infty)$.
We have
\begin{equation}\label{eqthRealKUMMod1}
  \w_\a(f,\d)_q\lesssim \d^\a \(\int_\d^1 \(\frac{\w_{\a+\g}(f,t)_p}{t^{\a+d(\frac1p-\frac1q)}}\)^\tau\frac{dt}{t}\)^\frac1\tau+\(\int_0^\d \(\frac{\w_{\a+m}(f,t)_p}{t^{d(\frac1p-\frac1q)}} \)^{q_1}\frac{dt}{t}   \)^\frac1{q_1}.
\end{equation}
\end{theorem}

\begin{remark}
\textnormal{
{\rm (i)} Using Theorem~\ref{thRealKUM} (B) and realization result~\eqref{eq.th6.0}, we obtain that
 inequality~\eqref{eqthRealKUMMod1} % при $0<p\le 1$ и $p<q\le \infty$ или $1<p<q<\infty$
  is sharp in the following sense.
Let $d=1$ and $\a+\g>(1-1/q)_+$. Let also}
\begin{itemize}
\item[(1)]$\g\neq 1-1/q>0$ if $2<q\le \infty$ and $0<p\le 1$,

\item[(2)]
 $\g\neq 1/p-1/q$ if $1<p<q\le \infty$.
  \end{itemize}

\noindent\textnormal{Then there exists a function $f_0\in L_q(\T)$, $f_0\not\equiv \const$,
such that we have
\begin{equation*}
  \w_\a(f_0,\d)_q\asymp \d^\a \(\int_\d^1 \(\frac{\w_{\a+\g}(f_0,t)_p}{t^{\a+\frac1p-\frac1q}}\)^\tau\frac{dt}{t}\)^\frac1\tau+\(\int_0^\d \(\frac{\w_{\a+\g}(f_0,t)_p}{t^{\frac1p-\frac1q}} \)^{q_1}\frac{dt}{t}\)^\frac1{q_1}
\end{equation*}
 as $\d\to 0$.}

\textnormal{{\rm (ii)} Note that two cases (1) and (2) where the sharpness of inequality (\ref{eqthRealKUM}) cannot be proved
are those cases where the Ulyanov inequality, given by Theorem~\ref{thMainMod}, provides a sharper bound than the one given by inequality (\ref{eqthRealKUMMod1}).
}
\end{remark}

%\begin{proof}
%
%(A) To proof the theorem we use Lemma~\ref{lemMarchaudMod}, Ulyanov inequality (Theorem~\ref{thMainMod} for $\g=0$),
%and repeat arguments from the proof of Theorem~\ref{thRealKUM} (A).
%
%
%%(B) The proof of part B can be obtained as a corollary from Theorem~\ref{thRealKUM} (B) and realization result~\eqref{eq.th6.0}.
%
%\end{proof}

Let us formulate an important  corollary which will be used to obtain embedding theorems in Section \ref{sec13}.

\begin{corollary}\label{cor13.18}
  Under the conditions of Theorem~\ref{th1UMMod1}, we have
\begin{equation}\label{qqqqqqqqqqqqqqwwwwwwwwwwwwwwweeeeeeeeeeeeerrrrrrrrrrr}
\begin{split}
      \w_{\a}(f,\d)_q\lesssim &\left(\int_0^\d\bigg(\frac{\w_{\a+m}(f,t)_p}{t^{d(\frac1p-\frac1q)}}\bigg)^{q_1}\frac{dt}{t}\right)^{\frac1{q_1}}\\
&\qquad\qquad+
\frac{\w_{\a+\g}(f,\d)_p}{\d^{\g+d(\frac1p-1)}}\left\{
                                          \begin{array}{ll}
                                            1, & \hbox{$\g>d\(1-\frac1q\)_+$}; \\
                                            \ln^{1/\tau}\(\frac1\d+1\), & \hbox{$\g=d\(1-\frac1q\)_+$}; \\
                                            \(\frac1\d\)^{d(1-\frac1q)-\g}, & \hbox{$\g<d\(1-\frac1q\)_+$}.
                                          \end{array}
                                        \right.
\end{split}
\end{equation}

\end{corollary}
The proof of the corollary follows from  inequality~\eqref{eqthRealKUMMod1} and property (e) of moduli of smoothness from Section~\ref{sec4}.

\medskip

We conclude this section with the proof of Remark~\ref{voobshedopremark}.

\medskip

\begin{proof}[Proof of Remark~\ref{voobshedopremark}]

Let $\phi(\xi)=\vp(\xi)|\xi|^m$, $m\in 2\N$, and $m>d(1/p-1)$.
First, by Corollary~\ref{corKg0}, we obtain the following analogue of~\eqref{qqqqqqqqqqqqqqwwwwwwwwwwwwwwweeeeeeeeeeeeerrrrrrrrrrr}:
\begin{equation}\label{qqq1}
    \mathcal{R}_{\psi}(f,\d)_q\lesssim \left(\int_0^\d\bigg(\frac{\mathcal{R}_{\phi}(f,t)_p}{t^{d(\frac1p-\frac1q)}}\bigg)^{q}\frac{dt}{t}\right)^{\frac1{q}}+
\frac{\mathcal{R}_{\vp}(f,\d)_p}{\d^{d(\frac1p-1)}}\left\{
                                          \begin{array}{ll}
                                            1, & \hbox{$0<q<1$;} \\
                                            \ln(\frac1\d+1), & \hbox{$q=1$.}
                                          \end{array}
                                        \right.
\end{equation}
Next, let $f=T_n\in\mathcal{T}_n$, $n=[1/\d]$. Then by the definition of the realization, we have
\begin{equation}\label{qqq2}
  \mathcal{R}_{\psi}(T_n,1/n)_q\asymp n^{-\a}\Vert \mathcal{D}(\psi)T_n\Vert_q,
\end{equation}
\begin{equation}\label{qqq3}
  \mathcal{R}_{\vp}(T_n,1/n)_p\asymp n^{-\a}\Vert \mathcal{D}(\vp)T_n\Vert_p,
\end{equation}
and
\begin{equation}\label{qqq4}
  \mathcal{R}_{\phi}(T_n,t)_q\lesssim t^{\a+m}\Vert \mathcal{D}(\phi)T_n\Vert_p, \quad 0<t<\d.
\end{equation}
Combining~\eqref{qqq1}--\eqref{qqq4}, we derive
\begin{equation}\label{qqq5}
\begin{split}
    n^{-\a}\Vert \mathcal{D}(\psi)T_n\Vert_q\lesssim n^{-\a-m+d(\frac1p-\frac1q)}&\Vert \mathcal{D}(\phi)T_n\Vert_p\\
&+n^{-\a+d(\frac1p-1)}\Vert \mathcal{D}(\vp)T_n\Vert_p\left\{
                            \begin{array}{ll}
                              1, & \hbox{$0<q<1$;} \\
                              \ln n, & \hbox{$q=1$.}
                            \end{array}
                          \right.
\end{split}
\end{equation}
Finally, taking into account that by the classical Bernstein inequality
$$
\Vert \mathcal{D}(\phi)T_n\Vert_p\lesssim n^m\Vert \mathcal{D}(\vp)T_n\Vert_p,
$$
we get from~\eqref{qqq5} that
\begin{equation*}
  \Vert \mathcal{D}(\psi)T_n\Vert_q\lesssim\Vert \mathcal{D}(\vp)T_n\Vert_p\left\{
                            \begin{array}{ll}
                              n^{d(\frac1p-1)}, & \hbox{$0<q<1$;} \\
                              n^{d(\frac1p-1)}\ln n, & \hbox{$q=1$,}
                            \end{array}
                          \right.
\end{equation*}
which gives~\eqref{dobav0}.
\end{proof}

\bigskip

\newpage

\section{Sharp Ulyanov and Kolyada inequalities in Hardy spaces}
\label{sec11}
In this section, we obtain  sharp Ulyanov and Kolyada type inequalities in Hardy spaces. We start with the real Hardy spaces on $\R^d$ and continue by Subsection~\ref{sec11.2} dealing with analytic Hardy spaces.

\subsection{Sharp Ulyanov and Kolyada inequalities in $H_p(\R^d)$}
Let us recall that the real Hardy spaces $H_p(\R^d)$,
$0<p<\infty$, is the class of tempered distributions  $f\in
\mathscr{S}'(\R^d)$ such that
$$
\Vert f\Vert_{H_p}=\Vert f\Vert_{H_p(\R^d)}=\big\Vert
\sup_{t>0}|\vp_t*f(x)|\big\Vert_{L_p(\R^d)}<\infty,
$$
where $\vp\in \mathscr{S}(\R^d)$, $\widehat{\vp}(0)\neq 0$, and
$\vp_t(x)=t^{-d}\vp(x/t)$ (see~\cite[Ch. III]{Stein93}).

The $K$-functionals and moduli of smoothness in $H_p(\R^d)$ are defined, respectively, as follows (see~\cite{ZS}):
$$
{K}_\a(f,\d)_{H_p}:=\inf_{g}\(\Vert f-g\Vert_{H_p}+\d^\a\Vert
(-\D)^{\a/2} g\Vert_{H_p}\)
$$
and
$$
\w_\a(f,\d)_{H_p}=\sup_{|h|<\d}\Vert \D_h^\a f\Vert_{H_p},
$$
where the fractional difference $\D_h^\a f$ is given by~\eqref{def-mod++}.

Note that for any $f\in H_p(\R^d)$, $0<p<\infty$, $\a\in \N \cup ((1/p-1)_+,\infty)$, and $\d>0$ we have
\begin{equation}\label{HpwK}
  \w_\a(f,\d)_{H_p}\asymp K_\a(f,\d)_{H_p}.
\end{equation}
In the case $\a\in \N$, this equivalence was proved in~\cite[p. 175]{Lu}.
The case $\a>(1/p-1)_+$ can be obtained similarly by using the Fourier multiplier technique (see~\cite{Wil}).

In this section, we restrict ourselves to consideration of $K$-functionals defined by the Laplacian operators since the general case $K_\psi(f,\d)_{H_p}$ with $\psi\in \mathcal{H}_\a$ can be reduced to $K_{(-\Delta)^{\a/2}}(f,\d)_{H_p}$. This follows from an analogue of Lemma~\ref{v-} in $H_p(\R^d)$, $0<p<\infty$ (see, e.g.,~\cite[Ch. III, \S~7]{GF}).

Using boundedness properties  of the  fractional integrals in the Hardy spaces (see \cite{Kr}) and following the proof of Theorem~\ref{mainlemma}, we obtain the sharp Ulyanov inequality in~$H_p(\R^d)$.

\begin{theorem}\label{realHpsharpURd}  Let $f\in H_p(\R^d)$, $0<p<q<\infty$, $\a>0$, and $\theta=d(1/p-1/q)$. Then, for any  $\d\in (0,1)$, we have
    \begin{equation*}
    %\label{eqth2.1d}
        K_\a(f,\d)_{H_q}\lesssim
        \left(\int_0^\d\bigg(\frac{K_{\a+\theta}(f,t)_{H_p}}{t^{\theta}}\bigg)^q\frac{dt}{t}\right)^{\frac1q}.
    \end{equation*}
In addition, if  $\a\in \N \cup ((1/q-1)_+,\infty)$ and $\a+\t \in \N \cup ((1/p-1)_+,\infty)$, then
    \begin{equation*}
    %\label{eqth2.1dMod}
        \w_\a(f,\d)_{H_q}\lesssim
        \left(\int_0^\d\bigg(\frac{\w_{\a+\theta}(f,t)_{H_p}}{t^{\theta}}\bigg)^q\frac{dt}{t}\right)^{\frac1q}.
    \end{equation*}
\end{theorem}

Hence, unlike the case of Lebesgue spaces (cf.~\cite{Ti}), the sharp Ulyanov inequality in Hardy spaces has the same form both in quasi-Banach spaces ($0<p<1$) and Banach spaces ($p\ge 1$).

%Thus, unlike the case $(L_p,L_q)$, $0<p\le 1$, (see~\cite{Ti})
%in the case $(H_p,H_q)$ the sharp Ul'yanov inequality holds.

Now, we concern  with the Kolyada-type inequality.
We recall that in the Lebesgue spaces $L_p(\T^d)$, $1<p<q<\infty$, inequality (\ref{ul-1}) was improved by Kolyada~\cite{kol} as follows:
\begin{equation}\label{ul-kol}
\delta^{1-\theta} \left( \int_{\delta}^1 \bigl(t^{\theta-1} \omega_1 (f,t)_q \bigl)^p \frac{dt}{t} \right)^{
\frac{1}{p} } \lesssim \left( \int_0^{\delta}
\bigl(t^{-\theta} \omega_1 (f,t)_p \bigl)^q
\frac{dt}{t} \right)^{ \frac{1}{q} },\quad \t=d\(\frac1p-\frac1q\).
\end{equation}
This estimate is sharp over the classes
$$
{\textnormal{Lip}}\, (\omega(\cdot),1,p)
=\left\{f\in L_p(\T^d):\omega_1(f,\delta)_p=\mathcal{O}\(\omega(\delta)\)\right\},
$$
that is, for any $\omega\in \Big\{\omega(0)=0,$ $\omega \uparrow$,
$\omega(\delta_1+\delta_2)\le \omega(\delta_1)+\omega(\delta_2)$ $\Big\}$,
there exists a function $f_0\in {\textnormal{Lip}}\, (\omega(\cdot),1,p)
$ such that, for any $\delta>0$,
$$
\delta^{1-\theta} \left( \int_{\delta}^1 \bigl(t^{\theta-1} \omega_1(f_0,t)_q \bigl)^p \frac{dt}{t} \right)^{
\frac{1}{p} } \gtrsim \left( \int_0^{\delta}
\bigl(t^{-\theta} \omega(t)  \bigl)^q
\frac{dt}{t}
\right)^{ \frac{1}{q} }.
$$
It is worth mentioning that (\ref{ul-kol}) is not valid for $p=1$, $d=1$ but is true for $p=1$, $d\ge 2$. % (e.g., take $f(x)= \textnormal{sign} \sin x $).
Note that recently, Trebels \cite{Treb} (see also \cite{gol1}) obtained an analogue of inequality (\ref{ul-kol}) for moduli of smoothness of fractional order in $1\le p<\infty$.

Kolyada \cite{kol} also proved~\eqref{ul-kol} for functions belonging to the analytic Hardy spaces on the disc, that is for $f\in H_p(D)$ and $0<p<q<\infty$. Below, we extended this results to the real Hardy spaces $H_p(\R^d)$.

To study Kolyada's inequalities in $H_p(\R^d)$, we need the following straightforward modification of the Holmstedt formulas (see~\cite{BeSh} and~\cite{Hol})
\begin{equation}\label{eqKolya3}
K(f,t^\t;X,(X,Y)_{\t,q})\asymp t^\t \(\int_t^\infty [s^{-\t} K(f,s)]\frac{ds}{s}\)^\frac1q
\end{equation}
and
\begin{eqnarray}\label{eqKolya4}
  \begin{split}
\widetilde{K}(f,t^{1-\t};(X,Y)_{\t,q},Y)&:=&\inf_{g\in Y}\{ |f-g|_{\t,q}+t^{1-\t}|g|_Y\}\\&\asymp&  \(\int_0^t [s^{-\t} K(f,s)]\frac{ds}{s}\)^\frac1q,
  \end{split}
\end{eqnarray}
where $(X,\Vert \cdot\Vert_X)$ is a quasi-Banach space, $Y\subset X$ is a complete subspace with seminorm $|\cdot|_Y$ and $\Vert \cdot\Vert_Y=\Vert \cdot\Vert_X+|\cdot|_Y$,
$$
K(f,t)\equiv K(f,t;X,Y):=\inf_{g\in Y}(\Vert f-g\Vert_X+t|g|_Y)
$$
is Peetre's $K$-functional, and
$$
(X,Y)_{\t,q}:=\left\{f\in X\,:\, |f|_{\t,q}=\(\int_0^\infty[t^{-\t}K(f,t)]\frac{dt}{t}\)^\frac 1q<\infty,\quad 0<\t<1\right\}
$$
is the interpolation spaces.

We will use the following facts:
\begin{eqnarray}\label{eqKolya5}
(H_p(\R^d),F_{p,2}^\a(\R^d))_{\t,q}&=&B_{p,q}^{\t \a}(\R^d),\quad \a>0,\quad 0<\t<1,
\\
\nonumber
|f|_{B_{p,q}^\a(\R^d)}&\asymp& \Vert f\Vert_{\dot B_{p,q}^\a(\R^d)},%\quad |f|_{F_{p,q}^s}\asymp \Vert f\Vert_{\dot F_{p,q}^s},
\\
\nonumber
 |f|_{F_{p,q}^\a(\R^d)}&\asymp& \Vert f\Vert_{\dot F_{p,q}^\a(\R^d)},
\end{eqnarray}
and
\begin{equation}\label{eqKolya7}
|f|_{F_{p,2}^\a(\R^d)}\asymp \Vert (-\D)^{\a/2} f\Vert_{H^p(\R^d)}
\end{equation}
(see, for example,~\cite{TribF}).

The following theorem can be proved with the help of the same ideas as the ones used in the proof of Theorem~2.6 in~\cite{Treb}.
%, using the embeddings (\ref{eqKolya1}) and (\ref{eqKolya2}), the Holmstedt formulas (\ref{eqKolya3}) and (\ref{eqKolya4}), and  (\ref{eqKolya5})--(\ref{eqKolya7}).

\begin{theorem}\label{thKolyaHp}
Let $f\in H_p(\R^d)$, $0<p< q<\infty$,
$\theta=d\(1/p-1/q\)$, and $\a>\t$. Then
\begin{equation}\label{eqth3.1K}
    \d^{\a-\theta}\(\int_\d^\infty
\(\frac{K_\a(f,t)_{H_q}}{t^{\a-\theta}}\)^p\frac{dt}{t}\)^\frac1p \lesssim\(\int_0^\d \(\frac{K_\a(f,t)_{H_p}}{t^\t}\)^q\frac{dt}{t}\)^\frac1q.
\end{equation}
In addition, if $\a\in \N\cup ((1/p-1)_+,\infty)$, then
\begin{equation}\label{eqth3.1Kmod}
    \d^{\a-\theta}\(\int_\d^\infty
\(\frac{\w_\a(f,t)_{H_q}}{t^{\a-\theta}}\)^p\frac{dt}{t}\)^\frac1p \lesssim\(\int_0^\d \(\frac{\w_\a(f,t)_{H_p}}{t^\t}\)^q\frac{dt}{t}\)^\frac1q.
\end{equation}
\end{theorem}

\begin{proof}
By using \eqref{eqKolya7} and \eqref{eqKolya3}, we obtain
\begin{equation*}
  \begin{split}
  I:&=\d^{1-\frac \t\a}\(\int_{\d^{1/\a}}^\infty\(t^{\t-\a}K_\a(f,t)_{H_q}\)^p\frac{dt}{t}\)^{\frac1p}\\
    &\lesssim \d^{1-\frac \t\a} \( \int_{\d}^\infty \(s^{\frac\t\a-1}K(f,s; H_q,F_{q,2}^\a)\)^p\frac{ds}{s} \)^\frac1p\\
    &\asymp K\(f,\d^{1-\frac\t\a}; H_q, (H_q,F_{q,2}^\a)_{1-\frac\t\a,p}\).
\end{split}
\end{equation*}   Using now
$\Vert f\Vert_{H_q(\R^d)} \lesssim \Vert f\Vert_{\dot B_{p,q}^\t(\R^d)}$
and
$$\Vert f\Vert_{\dot B_{q,p}^{\a-\t}(\R^d)}
\lesssim \Vert f\Vert_{\dot F_{p,2}^\a(\R^d)}\asymp \Vert (-\D)^{\a/2}
f\Vert_{H_p(\R^d)},
$$
(see Lemma~\ref{lemmapluss}), we get
\begin{equation*}
  \begin{split}
I &\le \Vert f-g\Vert_{H_q}+\d^{1-\frac\t\a}|g|_{B_{q,p}^{\a-\t}}\\
    &\lesssim \Vert f-g\Vert_{\dot B_{p,q}^\t}+\d^{1-\frac\t\a} \Vert g\Vert_{\dot F_{p,2}^\a}\\
    &\lesssim |f-g|_{B_{p,q}^\t}+\d^{1-\frac\t\a} |g|_{F_{p,2}^\a}
   \end{split}
\end{equation*}
for all $g\in F_{p,2}^\a$.

Next, taking infimum over all $g\in F_{p,2}^\a$ and applying~\eqref{eqKolya4} and~\eqref{eqKolya5} as well as~\eqref{eqKolya7}, we have
\begin{equation*}
  \begin{split}
I&\lesssim \widetilde{K}\(f,\d^{1-\frac\t\a}; (H_p,F_{p,2}^\a)_{\frac\t\a,q}, F_{p,2}^\a\)\\
&\asymp \(\int_0^\d \(s^{-\frac\t\a} K(f,s;H_p, F_{p,2}^\a)\)^q\frac{ds}{s}\)^\frac1q\\
&\asymp \(\int_0^{\d^{1/\a}} \(t^{-\t} K_\a(f,t)_{H_p}\)^q\frac{dt}{t}\)^\frac1q.
   \end{split}
\end{equation*}
The last inequality implies \eqref{eqth3.1K}.
Finally, \eqref{HpwK} yields~\eqref{eqth3.1Kmod}.
\end{proof}

\begin{remark}\label{remarkMarchaudHp}
\textnormal{
Note that inequality (\ref{eqth3.1Kmod}) implies the sharp Ulyanov
inequality for $f\in H_p(\R^d)$, $0<p<q\le 2$, given by
\begin{equation*}
%\label{eqth3.1-1}
\w_{\a-\t}(f,\d)_{H_q} \lesssim\(\int_0^\d
\(\frac{\w_{\a}(f,t)_{H_p}}{t^\t}\)^q\frac{dt}{t}\)^\frac1q, \qquad \a>\t.
\end{equation*}
This follows from the Marchaud inequality in $H_q(\R^d)$:
\begin{equation*}
%\label{eqth3.1-2}
\w_{\a-\t}(f,\d)_{H_q} \lesssim
    \d^{\a-\theta}\(\int_\d^\infty
\(\frac{\w_\a(f,t)_{H_q}}{t^{\a-\theta}}\)^q\frac{dt}{t}\)^\frac1q,\quad 0<q\le 2.
\end{equation*}
The latter can be proved by using the standard technique with the help of Bernstein's and Jackson's inequalities (see, e.g.,~\cite{ZS} and~\cite[Ch.~1]{TribF}).
%the embedding $B_{q,q}^\s(\R^d) \hookrightarrow H_q^\s(\R^d)$ and the scheme of the proof of Theorem 2.1 from~\cite{Treb}.
}
\end{remark}

The above results in this section remain true in the periodic real Hardy classes $H_p(\T^d)$.
Recall that the real periodic Hardy spaces  $H_p(\T^d)$,
$0<p<\infty$, are defined as a class of the tempered distributions $f\in
\mathscr{S}'(\T^d)$ such that
$$
\Vert f\Vert_{H_p(\T^d)}=\bigg\Vert \sup_{t>0}\bigg|\sum_{k\in\Z^d}
\widehat{\vp}(tk)a_k(f)e^{i(k,x)}\bigg|\,\bigg\Vert_{L_p(\T^d)}<\infty,
$$
where $a_k(f)$ are Fourier coefficients of the distribution $f$
(see, e.g.,~\cite[Ch.~9]{TribF} and \cite[p.~156]{ET}).

In particular, we have the following analogue of Theorem~\ref{thKolyaHp}.

\begin{theorem}\label{thKolyaHpT}
Let $f\in H_p(\T^d)$, $0<p< q<\infty$,
$\theta=d\(1/p-1/q\)$, and $\a>\t$. Then, for any $\d\in (0,1)$, we have
\begin{equation}\label{eqth3.1KT}
    \d^{\a-\theta}\(\int_\d^1
\(\frac{K_\a(f,t)_{H_q(\T^d)}}{t^{\a-\theta}}\)^p\frac{dt}{t}\)^\frac1p \lesssim\(\int_0^\d \(\frac{K_\a(f,t)_{H_p(\T^d)}}{t^\t}\)^q\frac{dt}{t}\)^\frac1q.
\end{equation}
In addition, if $\a\in \N\cup ((1/p-1)_+,\infty)$, then, for any $\d\in (0,1)$, we have
\begin{equation}\label{eqth3.1KmodT}
    \d^{\a-\theta}\(\int_\d^1
\(\frac{\w_\a(f,t)_{H_q(\T^d)}}{t^{\a-\theta}}\)^p\frac{dt}{t}\)^\frac1p \lesssim\(\int_0^\d \(\frac{\w_\a(f,t)_{H_p(\T^d)}}{t^\t}\)^q\frac{dt}{t}\)^\frac1q.
\end{equation}
\end{theorem}

Since $H_p(\T^d)=L_p(\T^d)$ for $1<p<\infty$, both~\eqref{eqth3.1KT} and~\eqref{eqth3.1KmodT} remain true for the Lebesgue spaces. In fact, inequality~\eqref{eqth3.1KmodT} also holds for $L_p(\T^d)$, $p\ge 1$, when $d\ge 2$. This was proved by Kolyada~\cite{kol} in the case $\a=1$. We extend this result to moduli of smoothness of arbitrary integer order. For the case $p>1$ see also~\cite{gol, netr, Treb}.

%It is well-know that the Hardy space $H_p$ coincide with  $L_p$ when $1<p<\infty$. Thus, it is clear that~\eqref{eqth3.1KmodT} as well as~\eqref{eqth3.1KT} are true in the spaces $L_p-L_q$ for $1<p<q<\infty$. It turns out that~\eqref{eqth3.1KmodT} is also true if we replace the space $H_1$ by $L_1$ in the case $d\ge 2$. This was proved by Kolyada in~\cite{kol} in the case $\a=1$. Below we extend this result to moduli of smoothness of arbitrary integer order. For the case $p>1$, see also~\cite{gol} and~\cite{netr}.

\begin{theorem}\label{thKolyaL1T}
Let $f\in L_p(\T^d)$, $d\ge 2$, $1\le p< q<\infty$,
$\theta=d\(1/p-1/q\)$, and $r\in \N$, $r>\t$. Then, for any $\d\in (0,1)$, we have
\begin{equation*}
%\label{eqth3.1KTMMMMOOD}
    \d^{r-\theta}\(\int_\d^1
\(\frac{\w_r(f,t)_{L_q(\T^d)}}{t^{r-\theta}}\)^p\frac{dt}{t}\)^\frac1p \lesssim\(\int_0^\d \(\frac{\w_r(f,t)_{L_p(\T^d)}}{t^\t}\)^q\frac{dt}{t}\)^\frac1q.
\end{equation*}
\end{theorem}

\begin{proof}
As it was mentioned above, the case $p>1$ follows from Theorem~\ref{thKolyaHpT}. Let us consider the case $p=1$.

It follows from~\cite[Theorem 4]{Kol93} that for any $f\in W_1^r(\T^d)$, $r\in \N$, and $1<q<\infty$ the following embedding holds
\begin{equation}\label{embW1Kol}
  \Vert f \Vert_{\dot B_{q,1}^{r-\t}}\lesssim \Vert f\Vert_{\dot W_1^r}.
\end{equation}
Note also that
\begin{equation}\label{embW1Kol2}
  \Vert f\Vert_q\lesssim \Vert f \Vert_{\dot B_{1,q}^{\t}}
\end{equation}
for $f\not\equiv\const$
(see~\eqref{eqKolya1}). At the same time, we have
\begin{equation}\label{embW1Kol3}
  (L_q, W_q^r)_{1-\frac\t r,1}=B_{q,1}^{r-\t}\quad\text{and}\quad (L_1, W_1^r)_{\frac\t r,q}=B_{1,q}^{\t}
\end{equation}
(see, e.g.,~\cite[Sec.~2.5]{TribF}). Thus, using \eqref{embW1Kol}--\eqref{embW1Kol3} and repeating the proof of Theorem~\ref{thKolyaHp}, we obtain~\eqref{eqth3.1KT}.

\end{proof}

\medskip

\subsection{Sharp Ulyanov and Kolyada inequalities in analytic Hardy spaces}\label{sec11.2}

%Рассмотрим аналогичные вопросы в пространствах Харди $H_p(D)$,
%$0<p<\infty$, где $D=\{z\in\C\,:\, |z|<1\}$.

Let us recall that an analytic function $f$ on the unit disc $D=\{z\in\C\,:\, |z|<1\}$
belongs to the space $H_p=H_p(D)$, if
\begin{equation*}
\Vert f\Vert_{H_p}=\sup_{{0<\rho<1}}\left(\int_{0}^{2\pi}|f(\rho
e^{it})|^p dt\right)^{\frac1p}<\infty.
\end{equation*}
%где $\r e^{i\f}=(\r_1e^{i\f_1},\dots,\r_ne^{i\f_n})$,
%$d\f=d\f_1\dots d\f_n$.

We define the modulus of smoothness, and realization of $K$-functional in $H_p(D)$ by analogy with
(\ref{def-mod+})  and \eqref{eqKf2}, correspondingly.

By the Burkholder-Gundy-Silverstein theorem on interrelation between analytic and real Hardy spaces~\cite{BGS} (see also~\cite[p.~111]{Garnett}), using Theorems~\ref{realHpsharpURd}--\ref{thKolyaHpT}, we obtain both the sharp Ulyanov and Kolyada inequalities in $H_p(D)$.

%By the well-known Burkholder-Gundy-Silverstein theorem~\cite{BGS} (see also~\cite[p. 111]{Garnett}), we obtain that Theorem~\ref{thKolyaHpT} and the corresponding analogue of Theorem~\ref{eqth2.1d} imply both the Sharp Ulyanov inequality and Kolyada's inequality in the analytic Hardy spaces on $D$.

\begin{theorem}\label{th2}
   Let $f\in H_p(D)$, $0<p<q<\infty$, $\a\in\N\cup
    ((1/q-1)_+,\infty)$, $\theta=1/p-1/q$, and $\a+\theta\in \N\cup
    (1/p-1,\infty)$. Then, for any $\d\in (0,1)$, we have
    \begin{equation}\label{eqth2.1}
        \w_\a(f,\d)_{H_q}\lesssim
        \left(\int_0^\d\bigg(\frac{\w_{\a+\theta}(f,t)_{H_p}}{t^{\theta}}\bigg)^q\frac{dt}{t}\right)^{\frac1q}.
    \end{equation}
\end{theorem}

\begin{theorem}\label{th3}
Let $f\in H_p(D)$, $0<p< q<\infty$, $\theta=1/p-1/q$, and $\a\in \N\cup (\t,\infty)$.
Then, for any  $\d\in (0,1)$, we have
\begin{equation}\label{eqth3.1}
    \d^{\a-\theta}\(\int_\d^1
\(\frac{\w_\a(f,t)_{H_q}}{t^{\a-\theta}}\)^p\frac{dt}{t}\)^\frac1p\lesssim\(\int_0^\d \(\frac{\w_\a(f,t)_{H_p}}{t^\t}\)^q\frac{dt}{t}\)^\frac1q.
\end{equation}
%где $C$ -- некоторая
%константа, не зависящая от $f$ и $\d$.
\end{theorem}

\begin{remark}\label{rem1-1}
\textnormal{
Note that inequalities (\ref{eqth2.1}) and (\ref{eqth3.1}) do not hold in $L_p$, $p\le 1$ (see~\cite{kol}).}

%2. The condition $q\ge 2$ is essential for the proof based on Lemma \ref{lemH}.

\end{remark}

Let us give an alternative proof of Theorem~\ref{th3} in the case $q\ge 2$.
This proof does not rely on the interpolation properties of function
spaces and uses only summability properties of Taylor coefficients.

To prove the theorem we will need the following refinement of the Hardy-Littlewood inequality (see~\cite{MP}).
\begin{lemma}\label{lemH}
    Let $f(z)=\sum_{k=1}^\infty a_k z^k\in H_p(D)$, $0<p\le 1$, and $0<q<\infty$. Then
\begin{equation}\label{eqlemH1}
    \sum_{n=0}^\infty 2^{np(1-(1/q+1/p))}\(\sum_{2^n\le
|k|<2^{n+1}}|a_k|^q\)^\frac pq\lesssim \Vert f\Vert_{H_p}.
\end{equation}
\end{lemma}

\begin{remark}\label{rem1}
\textnormal{Note that since  $1-(1/q+1/p)<0$, by Hardy's inequality, estimate
(\ref{eqlemH1}) can be equivalently written as
\begin{equation}\label{eqlemH2}
    \sum_{n=0}^\infty 2^{np(1-(1/q+1/p))}\(\sum_{k=1}^{2^{n+1}}|a_k|^q\)^\frac pq\lesssim \Vert f\Vert_{H_p}.
\end{equation}
}
\end{remark}

%In what follows for simplicity we denote $\Vert \cdot\Vert_p=\Vert\cdot\Vert_{H_p}$.

\begin{proof}[Proof of Theorem \ref{th3}]

We follow the scheme proposed in~\cite{PST}. Set
$$
V_n(f)(x):=\sum_{k=0}^{2n} v\(\frac kn\)c_k e^{ikx},
$$
where $v$  is defined by (\ref{vallee})
 %-- функция из доказательства леммы~\ref{lemHL}
(in addition, we suppose that $v$ is a monotonic function on $(0,\infty)$),
$c_k=c_k(f)$ are the Taylor coefficients of the analytic function  $f$.

First, let us note that the realization result can be formulated as follows:
\begin{equation}\label{KamodHardy}
  \Vert f-V_{2^n}(f)\Vert_{H_p}+2^{-\a n}\Vert V_{2^n}^{(\a)}(f)\Vert_{H_p}\asymp \w_\a(f,2^{-n})_{H_p}.
\end{equation}
The estimate $"\gtrsim"$ in~\eqref{KamodHardy} is clear, see the definition of the $K$-functional. The proof of the estimate $"\lesssim"$ follows from the inequalities
\begin{equation}\label{storJackH}
  \Vert f-V_{2^n}(f)\Vert_{{H_p}}\lesssim E_{2^n}(f)_{H_p}\lesssim \w_\a(f,2^{-n})_{H_p}.
\end{equation}
Here, the first inequality can be obtained by using the fact that the function $v$ is a Fourier multiplier in $H_p$ (see~\cite[\S~7.3]{TB}) and the second one is the Jackson inequality in $H_p$  (see~\cite{Stor78}).

Now, let $n\in \N$ be such that $2^{-n}\le \d<2^{-n+1}$. By standard calculations, using properties of moduli of smoothness and~\eqref{KamodHardy},
we derive that
\begin{equation}\label{eqth3.2}
\begin{split}
    &\d^{(\a-\theta)p}\int_\d^1
\(\frac{\w_\a(f,t)_{H_q}}{t^{\a-\theta}}\)^p\frac{dt}{t}
\lesssim
2^{-p(\a-\t)n} \sum_{m=0}^{n} \w_\a\(f,2^{-m}\)_{H_q}^p 2^{(\a-\t)mp}
\\
&\lesssim 2^{-p(\a-\t)n}\(\sum_{m=0}^n 2^{-m\t p}\Vert
V_{2^m}^{(\a)}(f)\Vert_{H_q}^p +\sum_{m=0}^n 2^{m(\a-\t)p}\Vert
f-V_{2^m}(f)\Vert_{H_q}^p\).
\end{split}
\end{equation}
Using the Hausdorff-Young inequality with $q\ge 2$  and (\ref{eqlemH2}),
we obtain
\begin{equation}\label{eqth3.3}
    \begin{split}
 \sum_{m=0}^n 2^{-m\t p}\Vert V_{2^m}^{(\a)}(f)\Vert_{H_q}^p
 &\lesssim
\sum_{m=0}^{n} 2^{-m\t p}\(\sum_{\nu=1}^{2^{m+1}}
|v\(\frac \nu{2^m}\)c_\nu|^{q'} \nu^{\a q'}\)^\frac p{q'}
\\
&= \sum_{m=0}^n 2^{mp(1-(1/q'+1/p))}\(\sum_{\nu=1}^{2^{m+1}}
|v\(\frac\nu{2^n}\)c_\nu|^{q'} \nu^{\a q'}\)^\frac p{q'}
\\&\lesssim\Vert
V_{2^n}^{(\a)}(f)\Vert_{H_p}^p.
     \end{split}
\end{equation}
%Дальнейшие рассуждения аналогичны доказательству теоремы~2
%из~\cite{PST} с использованием леммы~\ref{lemH}.
Next, applying the inequality
$$
\Vert f-V_{2^m}(f)\Vert_q \lesssim \Vert f-V_{2^n}(f)\Vert_q+\Vert
V_{2^m}(f)-V_{2^n}(f)\Vert_q,
$$
we get
\begin{equation}\label{eqth3.4}
\begin{split}
&\sum_{m=0}^n 2^{m(\a-\t)p}\Vert f-V_{2^m}(f)\Vert_{H_q}^p
\\
&\lesssim
\Vert f-V_{2^n}(f)\Vert_{H_q}^p
\sum_{m=0}^{n}2^{m(\a-\t)p}+\sum_{m=0}^{n-1} 2^{m(\a-\t)p}\Vert
V_{2^m}(f)-V_{2^n}(f)\Vert_{H_q}^p
\\
&\lesssim 2^{n(\a-\t)p}\Vert f-V_{2^n}(f)\Vert_{H_q}^p+\sum_{m=0}^{n-1}
2^{m(\a-\t)p}\(\sum_{s=m+1}^n \Vert
V_{2^s}(f)-V_{2^{s-1}}(f)\Vert_{H_q}\)^p.
\end{split}
\end{equation}

%Обозначая $a_\nu=a_{\nu,n}=\vp(\frac\nu{2^n})c_\nu$, оценим сумму в
%(\ref{eqth3.4}). Находим
%\begin{equation}\label{eqth3.5}
%    \begin{split}
% &\sum_{m=0}^n
%2^{m(\a-\t)p}\(\sum_{\nu=2^m}^{2^{n+1}}\bigg|(\vp(\frac\nu{2^n})-\vp(\frac\nu{2^m}))c_\nu\bigg|^{q'}\)^\frac{p}{q'}\le\\
%&\le\sum_{m=0}^n
%2^{m(\a-\t)p}\(\sum_{\nu=2^m}^{2^{n+1}}\bigg|\vp(\frac\nu{2^n})c_\nu\bigg|^{q'}\)^\frac{p}{q'}=\\
%&=\sum_{m=0}^n 2^{-\t
%pm}\(\sum_{l=m}^n\sum_{\nu=2^l}^{2^{l+1}}\bigg|\(\frac{2^m}{\nu}\)^\a
%\nu^\a a_\nu\bigg|^{q'}\)^\frac{p}{q'}\le\\
%&\le\sum_{m=0}^n 2^{-\t pm}\(\sum_{l=m}^n\frac{1}{2^{(l-m)\a
%q'}}\sum_{\nu=2^l}^{2^{l+1}}|\nu^\a a_\nu|^{q'}\)^\frac{p}{q'}\le\\
%&\le\sum_{m=0}^n 2^{-\t pm} \sum_{l=m}^n\frac{1}{2^{(l-m)\a
%p}}\(\sum_{\nu=2^l}^{2^{l+1}}|\nu^\a a_\nu|^{q'}\)^\frac{p}{q'}=\\
%&=\sum_{m=0}^n 2^{-\t pm} \sum_{l=0}^{n-m}\frac{1}{2^{l\a
%p}}\(\sum_{\nu=2^{l+m}}^{2^{l+m+1}}|\nu^\a
%a_\nu|^{q'}\)^\frac{p}{q'}=\\
%&=\sum_{m=0}^n  \sum_{l=0}^{n-m}\frac{1}{2^{(l+m)\t
%p}}\cdot\frac1{2^{(\a-\t)lp}}\(\sum_{\nu=2^{l+m}}^{2^{l+m+1}}|\nu^\a
%a_\nu|^{q'}\)^\frac{p}{q'}= \\
%&=\sum_{m=0}^n
%\(\sum_{l=0}^{m}\frac1{2^{(\a-\t)lp}}\)\frac{1}{2^{m\t
%p}}\(\sum_{\nu=2^{m}}^{2^{m+1}}|\nu^\a
%a_\nu|^{q'}\)^\frac{p}{q'}\le\\
%&\le C\sum_{m=0}^n\frac{1}{2^{m\t
%p}}\(\sum_{\nu=2^{m}}^{2^{m+1}}|\nu^\a a_\nu|^{q'}\)^\frac{p}{q'}\le
%C\Vert V_{2^n}^{(\a)}(f)\Vert_{p}^p.
%    \end{split}
%\end{equation}

Let us consider the second summand in the right-hand side of (\ref{eqth3.4}).  Applying the Hardy inequality for sums, the Hausdorff-Young inequality, and Lemma~\ref{lemH},
we derive
\begin{equation}
    \begin{split}
&\sum_{m=0}^{n-1} 2^{m(\a-\t)p}\(\sum_{s=m+1}^n \Vert
V_{2^s}(f)-V_{2^{s-1}}(f)\Vert_{H_q}\)^p
\\
&
\lesssim
 \sum_{m=0}^{n-1} 2^{m(\a-\t)p} \Vert
V_{2^{m+1}}(f)-V_{2^m}(f)\Vert_{H_q}^p
\\
&\lesssim
  \sum_{m=0}^{n-1} 2^{m(\a-\t)p}
\(\sum_{\nu=2^m}^{2^{m+2}}\left|\(v\(\frac\nu{2^{m+1}}\)-v\(\frac\nu{2^{m}}\)\)c_\nu\right|^{q'}\)^\frac
p{q'}
\\
&\lesssim
 \sum_{m=0}^{n-1} 2^{m(\a-\t)p}
\(\sum_{\nu=2^m}^{2^{m+2}}\left|v\(\frac\nu{2^{m+1}}\)c_\nu\right|^{q'}\)^\frac
p{q'}
\\
&\lesssim
 \sum_{m=0}^{n-1} 2^{-m\t p}
\(\sum_{\nu=2^m}^{2^{m+1}}\left|v\(\frac\nu{2^{n}}\)\nu^\a
c_{\nu}\right|^{q'}\)^\frac p{q'}
\lesssim
\Vert V_n^{(\a)}(f)\Vert_{H_p}.
     \end{split}
\end{equation}
Now, we estimate  the first summand in the right-hand side of~(\ref{eqth3.4}). Taking into account the first inequality in~\eqref{storJackH} and applying Lemma~4.2 from \cite{diti}, we obtain
\begin{equation}\label{eqth3.6}
\Vert f-V_{2^n}(f)\Vert_{H_q}^q\lesssim \sum_{\nu=n}^\infty 2^{\nu
q(\frac1p-\frac1q)}E_{2^\nu}(f)_{H_p}^q.
\end{equation}

Thus, combining inequalities (\ref{KamodHardy}) and (\ref{eqth3.2}) and
(\ref{eqth3.3})--(\ref{eqth3.6}), we finally arrive at
\begin{equation*}
\begin{split}
 &\d^{(\a-\theta)p}\int_\d^1
\(\frac{\w_\a(f,t)_{H_q}}{t^{\a-\theta}}\)^p\frac{dt}{t}\\
&\lesssim 2^{-n(\a-\t)p}\(\Vert
V^{(\a)}_{2^n}(f)\Vert_{H_p}^p+2^{n(\a-\t)p}\Vert f-V_{2^n}(f)\Vert_{H_q}^p\)
\\
&\lesssim2^{n\t p}\w_\a(f,2^{-n})_{H_p}^p+\(\sum_{\nu=n}^\infty 2^{\nu
q\t}E_{2^\nu}(f)_{H_p}^q\)^\frac pq
\\
&\lesssim
\(\int_0^\d
\(\frac{\w_\a(f,t)_{H_p}}{t^\t}\)^q\frac{dt}{t}\)^\frac pq.
\end{split}
\end{equation*}
\end{proof}

\begin{corollary}
  Let $f\in H_p(D)$, $0<p< q<\infty$,
$\theta=1/p-1/q$, and $\a\in \N\cup (\t,\infty)$. Suppose also that the $ \alpha$-th derivative of $f$ in the sense of Weyl belongs to $%f^{(\a)}\in
H_p(D)$. Then $f(e^{it})\in B_{q,p}^{\a-\t}(\T)$.
\end{corollary}

\begin{proof}
By Theorem 2.4 in~\cite{K11}, one has
$$
\w_\a(f,\d)_{H_p}\lesssim \d^\a \Vert f^{(\a)}\Vert_{H_p}.
$$
Thus, the right-hand side of (\ref{eqth3.1}) is bounded and, therefore, $f(e^{it})\in B_{q, p}^{\a-\t}(\T)$.
\end{proof}

%Что касается многомерных пространств Харди $H_p$, то можно получать
%различные обобщения, зависящие от дифференциального оператора и
%области.

\bigskip

\newpage

\section{$(L_p,L_q)$ inequalities of Ulyanov-type involving  derivatives}
\label{sec12}
We  start this section with the following result due to  Marcinkiewicz \cite{marc}: % proved
$$
\|f'\|_p \lesssim
\left(\int_0^1
\bigg(\frac{\omega_2(f,t)_{p}}{t}\bigg)^\tau\frac{dt}{t}\right)^{\frac1\tau},
$$
where $f\in L_p(\T)$, $1<p<\infty$, and $\tau=\min(2,p)$.
Related inequalities were also studied by Besov  \cite{besov}.
A generalization of the previous estimate to the case of the $k$-th derivative is given by
$$
\|f^{(k)}\|_p \lesssim
\left(\int_0^1
\bigg(\frac{\omega_r(f,t)_{p}}{t^k}\bigg)^\tau\frac{dt}{t}\right)^{\frac1\tau},\quad 0<k<r,
$$
which, in particular, follows from the known embedding between Sobolev and Besov spaces.
%In the above inequality $f\in L_p(\T)$, $1\le p <\infty$, and $0<k<r$

Related inequalities for the moduli of smoothness are given by

\begin{theorem}\label{th-deri} (See \cite{diti07, simonov-sb}.)
Let $f \in L_p(\T)$, $1 < p < \infty$, $\varkappa = \max(2,p)$, $\tau = \min(2,p)$,
$\alpha, r>0$, and $\delta \in (0,1)$. Then
\begin{eqnarray*}
  \begin{split}
  \( \int_0^{\delta} \( t^{- r}
\omega_{r + \alpha}(f,t)_p \)^{\varkappa} \frac{dt}{t}
\)^{\frac{1}{\varkappa}} &\lesssim\quad \omega_{\alpha}(f^{(r)},\delta)_p\\ &\lesssim\;
 \( \int_0^{\delta} \( t^{- r} \omega_{r +
\alpha}(f,t)_p \)^{\tau} \frac{dt}{t}
\)^{\frac{1}{\tau}}.
  \end{split}
\end{eqnarray*}
\end{theorem}
Note that for $p=1$ or $p=\infty$ and $\a>0$ one has only weaker estimates
$$
 t^{- r}
\omega_{r + \alpha}(f,t)_p \lesssim \omega_{\alpha}(f^{(r)},t)_p
$$
and
\begin{eqnarray}
\label{th-deri1}
 \omega_{\alpha}(f^{(r)},\delta)_p\lesssim
 \int_0^{\delta}  t^{- r} \omega_{r +
\alpha}(f,t)_p  \frac{dt}{t},
\end{eqnarray}
see \cite[p. 46, 178]{DL} for $\a\in \N$ and \cite{simonov-sb} for $\a>0$. Both inequalities clearly hold also for $1<p<\infty$.

In the case $0<p\le 1$, it is shown in \cite{diti07} that for $\a\in \N$ one has
\begin{eqnarray}\label{th-deri2}
  \begin{split}
 \omega_{\alpha}(f^{(r)},\delta)_p
 \lesssim
 \( \int_0^{\delta} \( t^{- r} \omega_{r +
\alpha}(f,t)_p \)^{p} \frac{dt}{t}
\)^{\frac{1}{p}}, %, \qquad \tau = \min(2,p)
  \end{split}
\end{eqnarray}
where $f^{(r)}$ is the derivative in the sense of $L_p$, which is defined as follows.
A function $f\in L_p(\mathbb{T})$, $0<p<\infty$, has the derivative
$f^{(r)}$ of order $r\in \N$ in the sense $L_p$ if
\begin{equation*}
%\label{senseDerT}
  \bigg\Vert \frac{\D_h^r f}{h^r}-f^{(r)}\bigg\Vert_{p}\to0 \quad \textrm{as} \quad h\rightarrow 0\,.
\end{equation*}

Below we study sharp Ulyanov $(L_p, L_q)$ inequalities for moduli of smoothness  involving derivatives.
Our goal is to improve the known estimate (\cite{diti07})
\begin{eqnarray}
\label{th-deri3}\qquad\qquad
 \omega_{\alpha}(f^{(r)},\delta)_q
 \lesssim
 \( \int_0^{\delta} \( t^{- r- (\frac1p-\frac1q)} \omega_{r +
\alpha}(f,t)_p \)^{q_1} \frac{dt}{t}
\)^{\frac{1}{q_1}}, \qquad  0<p<q\le \infty, %\tau = \min(2,p)
\end{eqnarray}
where $1/p-1/q<\alpha$, $d=1$, and $\a,r\in \N$.

\begin{theorem}\label{th-deri4}
Let $f\in L_p(\T)$, $0<p<q\le \infty$,  $r\in\N$, $\a\in\N\cup
    ((1/q-1)_+,\infty)$, and $\g\ge 0$ be such that
    $\a+\g\in \N\cup ((1/p-1)_+,\infty)$.
% if $0<p<1$ and
%    $\a+\g>0$ if $1\le p\le \infty$.
Then, for any  $\d\in (0,1)$, we have
\begin{eqnarray}
\label{eq.th-deri4}\qquad\qquad
        \w_\a(f^{(r)},\d)_q
 \lesssim \left(\int_0^\d\bigg(
                \frac{\w_{r+\a+\g}(f,t)_p}{t^{\g+r}} \s\Big(\frac1t\Big)
\bigg)^{q_1}\frac{dt}{t}\right)^{\frac1{q_1}},
\end{eqnarray}
where

{\rm (1)} if $0<p\le 1$ and $p<q\le\infty$, then
$$
\s%_{q,\g}
(t)
:=\left\{
         \begin{array}{ll}
           t^{\frac1p-1}, & \hbox{$\g>\(1-\frac1q\)_+$}; \\
           t^{\frac1p-1}\ln^\frac1q (t+1), & \hbox{$0<\g=\(1-\frac1q\)_+$}; \\
           t^{\frac1p-\frac1q-\g}, & \hbox{$0< \g<\(1-\frac1q\)_+$};\\
           t^{\frac1p-\frac1q}, & \hbox{$\g=0$},
         \end{array}
       \right.
$$

{\rm (2)} if $1<p\le q\le\infty$, then
$$
\s(t):=\left\{
         \begin{array}{ll}
           1, & \hbox{$\g\ge \frac1p-\frac1q,\quad q<\infty$}; \\
           1, & \hbox{$\g> \frac1p,\quad q=\infty$}; \\
           \ln^\frac1{p'} (t+1), & \hbox{$\g=\frac1p,\quad q=\infty$}; \\
           t^{(\frac1p-\frac1q)-\g}, & \hbox{$0\le \g<\frac1p-\frac1q$}.\\
         \end{array}
       \right.
$$
\end{theorem}

%\medskip

\begin{remark}
\textnormal{{\rm (i)}
For $q=\infty$,
the condition $\a+\g\in \N\cup ((1/p-1)_+,\infty)$
can be weaken as follows:
 $\a+\g+r\in \N\cup ((1/p-1)_+,\infty)$.}

\textnormal{{\rm (ii)} Inequality (\ref{eq.th-deri4}) means that if the integral on the right-hand side is finite, then $f^{(r)}$ exists in the sense of $L_q$ and
 estimate (\ref{eq.th-deri4}) holds (see also~\cite{diti07}).}

\textnormal{{\rm (iii)} Inequality (\ref{th-deri3}) follows from Theorem \ref{th-deri2} when $\g=0$.}

\end{remark}

\begin{proof}[Proof of Theorem \ref{th-deri2}]
Let first $0<p<q<\infty$. Theorem \ref{th1} yields
$$
        \w_\a(f^{(r)},\d)_q
 \lesssim
\left(\int_0^\d\bigg(
                \frac{\w_{\a+\g}(f^{(r)},t)_p}{t^{\g}} \s\(\frac1t\)
\bigg)^{q}\frac{dt}{t}\right)^{\frac1{q}}.
$$
By Theorem \ref{th-deri} and inequalities (\ref{th-deri1})--(\ref{th-deri2}), we obtain that
\begin{eqnarray*}
  \begin{split}
 \omega_{\a+\g}(f^{(r)},\delta)_p
 \lesssim
 \( \int_0^{\delta} \( t^{- r} \omega_{\a+\g+r}(f,t)_p \)^{\widetilde{p}} \frac{dt}{t}
\)^{{1}/{\widetilde{p}}}, \qquad \widetilde{p} = \min(1,p).
  \end{split}
\end{eqnarray*}
Here we take into account that the proof of inequality (\ref{th-deri2})
  for fractional moduli of smoothness repeats the proof given
  in \cite{diti07}
 for moduli of smoothness of integer order using the realization result~\eqref{eq.th6.0}.

Applying Hardy's inequality given by Lemma \ref{har}, we derive
\begin{eqnarray*}
  \begin{split}
        \w_\a(f^{(r)},\d)_q
 &
 \lesssim
 \left(\int_0^\d
u(t)
 \Bigg\{ \int_0^{t} \Big[ s^{- r} \omega_{\a+\g+r}(f,s)_p \Big]^{\widetilde{p}} \frac{ds}{s}
\Bigg\}^{{q}/{\widetilde{p}}}
{dt}\right)^{\frac1{q}}
\\
& \lesssim
 \left(\int_0^\d%\bigg(
v(t) %                \frac{1}{t^{\g}} \s\(\frac1t\)\bigg)^{q}
 %\Bigg\{ \int_0^{\delta}
 \Big[ t^{- r- 1/\widetilde{p}} \omega_{\a+\g+r}(f,t)_p \Big]^{q} %\frac{dt}{t}
%\Bigg\}^{{q}/{\widetilde{p}}}
{dt}\right)^{\frac1{q}}
\\
& =
 \left(\int_0^\d%
 \Big[
               \frac{1}{t^{\g+r}} \s\(\frac1t\)
   \omega_{\a+\g+r}(f,t)_p \Big]^{q}
   \frac{dt}{t}
\right)^{\frac1{q}},
  \end{split}
\end{eqnarray*}
where
$$ u(t)=\bigg(
                \frac{1}{t^{\g}} \s\(\frac1t\)\bigg)^{q}\frac1t,\qquad
        v(t)=
        \frac{t^{q/\widetilde{p}}}t
        \bigg(
                \frac{1}{t^{\g}} \s\(\frac1t\)\bigg)^{q}.
$$
It is easy to verify that
(\ref{har0co}) holds for the couple $(u,v)$ with $\l=q/\widetilde{p}>1$ and $\l'=q/(q-\widetilde{p})$,
which gives the desired result.

Suppose that  $0<p<q=\infty$. Applying first
inequality (\ref{th-deri1}) and then Theorem~\ref{th1}, Fubini's theorem gives
\begin{eqnarray*}
  \begin{split}
 \omega_{\a}(f^{(r)},\delta)_q
 &\lesssim
 \int_0^{\delta}  t^{- r} \omega_{\a+r}(f,t)_q\frac{dt}{t}
\\
&\lesssim
 \int_0^{\delta}  t^{- r}
  \int_0^t
                \frac{\w_{\a+r+\g}(f,u)_p}{u^{\g}} \s\(\frac1u\)
\frac{du}{u}
 \frac{dt}{t}
\\
&\lesssim
 \int_0^{\delta}
   \frac{\w_{\a+r+\g}(f,u)_p}{u^{\g}} \s\(\frac1u\)
  \int_u^\infty
  t^{- r}
 \frac{dt}{t}
              \frac{du}{u},
  \end{split}
\end{eqnarray*}
completing the proof.
\end{proof}

Now, we consider the multi-dimensional case.
It is well known that
for $f\in L_q(\T^d)$, $1\le q\le\infty$, $d\ge 2$, one has
\begin{eqnarray}\label{eqMultDerevMod}
  \begin{split}
 \omega_{\a}(D^\b f,\delta)_q
 \lesssim
 \( \int_0^{\delta} \( t^{- {|\beta|_1}} \omega_{\a+|\beta|_1}(f,t)_q \)^{\varrho } \frac{dt}{t}
\)^{{1}/{\varrho }},
  \end{split}
\end{eqnarray}
where $\a\in \N$, $\b\in \Z_+^d$,
% $$|\beta|_1=\sum_{j=1}^d \b_j$$
and
$$
\varrho = \left\{
                                   \begin{array}{ll}
                                     \min(2,q), & q<\infty, \\
                                     1, & q=\infty
                                   \end{array}
                                 \right.
$$
(see \cite[Ch.~4]{BeSh}, \cite{johnen} and \cite{Treb}).

\begin{remark}\label{remarkxixixi}
\textnormal{If $1<q<\infty$, the proof of~\eqref{eqMultDerevMod} given in    \cite[Theorem 2.3]{Treb}
can be extended to the case $\a>0$.
}
\end{remark}

Note also that the reverse inequality to~\eqref{eqMultDerevMod} looks as follows (see~\cite{johnen}):
$$
\w_{\a+r}(f,\d)_q\lesssim \d^r \sup_{|\b|_1=r}\w_{\a} (D^\b f,\d)_q,\quad \a,r\in \N.
$$

By using inequality~\eqref{eqMultDerevMod} and Theorem~\ref{thMainMod}, we can prove the following multidimensional analog of Theorem~\ref{th-deri4}.

\begin{theorem}\label{th-deri4Mult}
Let $f\in L_p(\T^d)$, $d\ge 2$, $0<p<q$, $1\le q\le \infty$, $\a\in \N$, $\g\ge 0$,
 $\b\in \Z_+^d$, and
$\a+|\beta|_1+\g\in \N\cup ((1/p-1)_+,\infty)$. Then, for any $\d\in (0,1)$, we have

\begin{equation}\label{eq.th-deri4Mult.1}
    \begin{split}
 \w_{\a}(D^\b f,\d)_q
 \lesssim
\( \int_0^{\d} \( \frac{\omega_{\a+|\beta|_1+\g}(f,t)_p}{t^{{|\beta|_1+\g}}} \s\(\frac1\d\) \)^{\varrho } \frac{dt}{t}
\)^{{1}/{\varrho }},
  \end{split}
\end{equation}
where
%$$
%\varrho= \left\{
%                                   \begin{array}{ll}
%                                     \min(2,q), & q<\infty, \\
%                                     1, & q=\infty,
%                                   \end{array}
%                                 \right.
%$$
%and
%$\s(\cdot)$ is defined by the following formulas

{\rm (1)} if $0<p\le 1$ and $1\le q\le\infty$, then
$$
\s(t)
:=\left\{
         \begin{array}{ll}
           t^{d(\frac1p-1)}, & \hbox{$\g> d\(1-\frac1q\)$}; \\
           t^{d(\frac1p-1)}, & \hbox{$\g=d\(1-\frac1q\)\ge 1$ and $\a+\g\in \N$}; \\
           t^{d(\frac1p-1)}\ln^\frac1{q_1} (t+1), & \hbox{$\g=d\(1-\frac1q\)\ge 1$  and $\a+\g\not\in \N$}; \\
           t^{d(\frac1p-1)}\ln^\frac1{q} (t+1), & \hbox{$0<\g=d\(1-\frac1q\)<1$}; \\
           t^{d(\frac1p-\frac1q)-\g}, & \hbox{$0< \g<d\(1-\frac1q\)$};\\
           t^{d(\frac1p-\frac1q)}, & \hbox{$\g=0$;}
         \end{array}
       \right.
$$

{\rm (2)} if $1<p\le q\le\infty$, then
$$
\s(t):=
\left\{
         \begin{array}{ll}
           1, & \hbox{$\g\ge d(\frac1p-\frac1q),\quad q<\infty$}; \\
           1, & \hbox{$\g> \frac dp,\quad q=\infty$}; \\
           \ln^\frac1{p'} (t+1), & \hbox{$\g=\frac dp,\quad q=\infty$}; \\
           t^{d(\frac1p-\frac1q)-\g}, & \hbox{$0\le \g<d(\frac1p-\frac1q)$}.\\
         \end{array}
       \right.
$$

\end{theorem}

\begin{proof}
The proof is slightly different from the proof of Theorem~\ref{th-deri4}. We  use first~\eqref{eqMultDerevMod} and then we apply Theorem~\ref{thMainMod} to get
\begin{equation*}
    \begin{split}
 \w_{\a}(D^\b f,\d)_q
 \lesssim
 \Bigg\{ \int_0^{\delta} %\Big[
 t^{- {|\beta|_1\varrho}}
\left(\int_0^t\bigg(
                \frac{\w_{\a+|\beta|_1+\g}(f,u)_p}{u^{\g}} \s\(\frac1u\)
\bigg)^{q_1}\frac{du}{u}\right)^{\varrho/{q_1}} \frac{dt}{t}
\Bigg\}^{{1}/{\varrho }}.
% \Bigg\{ \int_0^{\d} \( \frac{\omega_{\a+|\b|+\g}(f,t)_p}{t^{{|\b|+\g}}} \s\(\frac1\d\) \)^{\varrho } \frac{dt}{t}
%\Bigg\}^{{1}/{\varrho}}.
  \end{split}
\end{equation*}
To complete the proof, %in place of Lemma~\ref{har}
we only need the Hardy inequality for monotone functions (see, e.g.,~\cite[Theorem 3.5, p.~28]{DL}).
Let $\xi>0$, $0<\l<\infty$, and $\phi$ be a non-negative monotone function on $\R_+$. Then the following inequality holds
$$
\int_0^\infty \(t^{-\xi} \int_0^t \phi(s)\frac{ds}{s}\)^\l\frac{dt}{t}\lesssim \int_0^\infty \(t^{-\xi}\phi(t)\)^\l\frac{dt}{t}.
$$
Noting that this inequality also holds for weakly  decreasing functions $\phi$, that is, satisfying $\phi(x)\le C \phi(y)$ for $x\ge y\ge 2x$ (see, e.g.,~\cite{TiZe}), we arrive at~(\ref{eq.th-deri4Mult.1}).

%The proof of the second part is similar to the above proof using
%Theorem~\ref{thMainMod2} in place of Theorem~\ref{thMainMod}.

\end{proof}

\begin{remark}
\textnormal{{\rm (i)} Note that in Theorem~\ref{th-deri4Mult}, $q\ge 1$ and we deal with the usual partial derivatives $D^\b f$.}

\textnormal{{\rm (ii)} Assuming in Theorem~\ref{th-deri4Mult} that $p\ge 1$, inequality \eqref{eq.th-deri4Mult.1} holds if we replace  $\varrho$ by $q_1$, which gives a stronger result. This can be shown as in the proof of Theorem~\ref{th-deri4}. Indeed, first, we use Theorem~\ref{thMainMod} for the derivative $D^\b f$ and then we apply inequality~\eqref{eqMultDerevMod}.}

\textnormal{{\rm (iii)} If $1<q<\infty$, then~\eqref{eq.th-deri4Mult.1} holds for $\a>0$ (see Remark~\ref{remarkxixixi}).
}
\end{remark}

We finish this section with the following result
 related to the discussion in Subsection~\ref{sec1.2p<1}. We show  that while dealing with
absolutely continuously functions the pathological behavior of smoothness properties in $L_p$, $0<p<1$,
disappears.

\begin{proposition}\label{proposition}
Let $0<p<1$, $\a\in \N%\cup (\max\{\frac1p-1,1\},\infty)
$, $f^{(\a-1)}\in AC(\T)$, and
$$
\w_\a(f,\d)_p=o(\d^\a),\quad \d\to 0,
$$
then $f\equiv \const$.
  \end{proposition}

\begin{proof}
Let
$0<p<1$, $f\in AC(\T)$ and
$$
\w_1(f,\d)_p=o(\d),\quad \d\to 0,
$$
then $f\equiv \const$ (see~\cite{SKO}).
Using this and the estimate
$$
\w_1(f^{(\a-1)},\d)_p \lesssim\left(\int _0^\d\(\frac{\w_\a(f,p)_p}{t^{\a}}\)^p t^{p-1}{dt}\right)^{1/p}=o(\d),\quad \d\to 0,
$$
(see~\eqref{th-deri2}),
we arrive at $f^{(\a-1)}\equiv \const$ and hence the desired result follows.  %~\ref{thModFrD} (см. ниже)получаем
\end{proof}

\bigskip

\newpage

\section{Embedding theorems for function spaces}
\label{sec13}

As mentioned earlier, both sharp Ulyanov's and Kolyda's inequalities are closely related to embedding theorems for smooth function spaces {\rm(}Lipschitz, Nikol'skii--Besov, etc.{\rm)}.
%In particular, these inequalities imply the following embedding for the generalized Lipschitz spaces, sometimes called Nikol'skii spaces.

In what follows, we restrict ourselves to the classical Lipschitz spaces
$$
{\textnormal {Lip}} (\alpha, k, p)
= \Big\{f\in L_p(\T^d)\,:\, \omega_k(f,\delta)_p= \mathcal{O}( \delta^\alpha )\Big\}
$$
and their generalizations given by
$$
{\textnormal {Lip}}^{\beta} (\alpha, k, p)
= \Big\{f\in L_p(\T^d)\,:\, \omega_k(f,\delta)_p= \mathcal{O}( \delta^\alpha \,\,\ln^\beta 1/\delta)\Big\}.
$$
The latter spaces are the well-known  % Log-Lipschitz
 logarithmic Lipschitz spaces, widely used in functional analysis (see, e.g.,
\cite{bure, robert})
and  differential equations (see, e.g.,
\cite{Colombini, zuazua}).

The theory of embedding theorems for function spaces has been studied for a long time, beginning  with the work of Hardy and Littlewood (see~\cite{hardy-l, HL}). They proved that
$$
{\textnormal {Lip}} (\alpha,1,p)\hookrightarrow {\textnormal {Lip}} (\alpha-\theta,1,q),\qquad
$$
where
$$1
\le p<q<\infty, \quad\theta ={1}/{p}- {1}/{q},\quad \theta< \alpha\le 1,\quad d=1.
$$
Later, this result was extended by many authors mostly in the case $1\le p<q\le \infty$ (see~\cite{besov1, gol1, hatr, kol, KOLYADA1, netr, ST, simonov-sb, S, Ti, Treb, uly}). In particular, Kolyada's inequality (see~\eqref{ul-kol}) implies that
$$
\Lip(\a,\a,p)\hookrightarrow B_{q,p}^{\a-\t},\quad 1<p<q<\infty.
$$

Recall that for an important limit case $\alpha=k$, we have
${\textnormal {Lip}} (\alpha, \alpha, p)\equiv {W^\alpha_p}(\T)$ for $1<p\le \infty$ and
$$f\in{\textnormal {Lip}} (\alpha, \alpha, 1)\quad\mbox{iff}\quad
\left\{
  \begin{array}{ll}
    \mathcal{D}^{\alpha-1}f\in BV(\T)%\cap L_1[0,2\pi]
    , &\qquad \hbox{$\alpha>1;$} \\
    f\in BV(\T), &\qquad \hbox{$\alpha=1;$} \\
    I^{1-\alpha}f\in BV(\T), &\qquad \hbox{$\alpha<1,$}
  \end{array}
\right.
$$
where $BV$ is the space  of  all  functions  which  are  of
bounded  variation  on  every  finite  interval and $I^{1-\alpha}$ is the fractional integral (see~\cite[XII, \S~8, (8.3), p.~134]{Z}).
By
$\mathcal{D}^\alpha f$
we denote the
Liouville-Gr\"{u}nwald-Letnikov derivative of order $\alpha>0$ of a
function $f$ in the $L_p$-norm:
if for $f\in L_p(\T)$ there exists $g\in L_p(\T)$ such that
$$
\lim\limits_{h\to 0+}\left\|h^{-\alpha}\triangle_h^\alpha
f(\cdot)-g(\cdot)\right\|_p=0,
$$
 then $g=\mathcal{D}^\alpha f$ (see
\cite{But},  \cite[\S~20, (20.7)]{samko}).
%Set ?????!!!!!!
%$$
%\mathcal{W}^\alpha_p:=\left\{f\in L_p(\T)\,:\,
%\mathcal{D}^\alpha f \,\,\,\mbox{exists as an element in}\,\,\, L_p(\T)\right\}.
%$$
Note that  for any
$f\in L_p(\T)$, $p\ge 1$, we have  $f^{(\alpha)}(x)  = \mathcal{D}^\alpha f(x)$ a.e., where
$f^{(\alpha)}$ is the Weyl derivative of~$f$.

In the case $0<p<1$, the space ${\textnormal {Lip}} (\alpha, k, p)$ with limiting smoothness (that is $\a=1/p+k-1$) consists only of "trivial"\, functions.
 %in the following sense.
 Let us state it  more precisely.
 We have
$$
f\in {\textnormal {Lip}} (1/p+k-1, k, p)
$$
if and only if after correction of $f$ on a set of measure zero its $(k-1)$-st derivative is of the form
$$
f^{(k-1)}(x)=d_0+\sum_{x_l<x} d_l,
$$
where $\{x_l\}$ is a sequence of different points from $[0,2\pi)$ and $\sum_l|d_l|^p<\infty$ (see~\cite{K03} and~\cite[4.8.26]{TB}, see also~\cite{Krot} for the case $k=1$).

\subsection{New embedding theorems for Lipschitz spaces}
Our main goal in this section is to apply Ulyanov inequalities to study optimal embeddings of the type
$$
\Lip(\eta,\a,p) \hookrightarrow X_q, \quad 0<p<q\le \infty,
$$
where $X_q$ is a Lipschitz-type or Besov space with respect to $L_q$-norm.

It is worth mentioning again that the space $\Lip (\eta,\a,p)$ for $0<p\le 1$ and $\eta>\a+d(1/p-1)$ consists only of a.e. constant functions. This follows from Theorem~\ref{lemBasicMod}. Note also that the embedding $\Lip(\eta,\a,p)\hookrightarrow L_q$ holds for $0<p<q\le \infty$ if and only if $\eta>\t=d(1/p-1/q)$. Thus, it is natural to assume below that $\eta>\t$ and $\eta\le \a+d(1/p-1)$ if $0<p<1$.

\begin{theorem}\label{thEmbedd}
  Let $d\ge 1$, $0<p<q\le \infty$, $\a\in \N\cup ((1/p-1)_+,\infty)$, $\t=d(1/p-1/q)$.

\textnormal{(A)} If $0<p\le 1$ and $p<q\le \infty$, then, for any $\t<\eta\le \a+d(1/p-1)$, one has

\begin{equation*}\label{eqthEmbed1.2d}
\begin{split}
      &\Lip(\eta,\a,p) \hookrightarrow  \\
&\left\{
                                         \begin{array}{ll}
                                           \Lip(\eta-\t,\a-d(1-\frac1q),q), & \hbox{$\frac{d}{d-1}\le q\le \infty$ and $\a\in\N$};\\
                                           \Lip\(\eta-\t,\a,q\), & \hbox{$\frac{d}{d-1}\le q\le 2$ and $\a\not\in \N$}; \\
                                           \Lip\(\eta-\t,\a,q\), & \hbox{$q<\frac{d}{d-1}$ and $d\ge 2$}; \\
                                           \Lip\(\eta-\t,\a,q\), & \hbox{$q\le 2$ and $d=1$}; \\
                                           \Lip(\eta-\t,\a,q)\cap \Lip^{\frac1{q}}\(\eta-\t,\a-d(1-\frac1q),q\), & \hbox{$2<q< \infty$ and $\a\not\in\N$};\\
                                           \Lip(\eta-\t,\a,q)\cap \Lip^{\frac1q}\(\eta-\t,\a-d(1-\frac1q),q\), & \hbox{$2<q<\infty$ and $d=1$};\\
                                           \Lip\(\eta-\t,\a,q\), & \hbox{$q=\infty$, $d\ge 2$, and $\a\not\in\N$};\\
                                           \Lip(\eta-\t,\a-d(1-\frac1q),q), & \hbox{$q=\infty$, $d=1$, and $\a\not\in\N$},\\
                                         \end{array}
                                       \right.
\end{split}
\end{equation*}

\textnormal{(B)} If $1<p<q\le \infty$, then, for $\t<\eta=\a$, one has

\begin{equation}\label{eqthEmbed1.4-}
  \Lip(\eta, \a, p)\hookrightarrow
\left\{
  \begin{array}{ll}
       B^{\a-\theta}_{q\, p}, & \hbox{$1 < p\le \min(2,q)$}; \\
    \Lip (\eta-\t, \a-\t, q)\,\,\cap \,\,B^{\a-\t}_{q\, p}, &\hbox{$\min(2,q)<p<q<\infty$};\\
        \Lip (\eta-\t, \a, q)\,\cap \Lip^{1/p'} (\eta-\t, \a-\t, q), & \hbox{$ 1< p<q=\infty$},
  \end{array}
\right.
\end{equation}
while, for any $\t<\eta<\a$, one has
  \begin{equation}\label{eqthEmbed1.4}
    \Lip(\eta,\a,p) \hookrightarrow \Lip\(\eta-\t,\a,q\).
%\left\{
%                                      \begin{array}{ll}
%                                        \Lip\(\eta-\t,\a-\t,q\), & \hbox{$q<\infty$;} \\
%                                        \Lip\(\eta-\t,\a,q\)\cap \Lip^{1/p'} (\eta-\t, \a-\t, q), & \hbox{$q=\infty$.}
%                                      \end{array}
%                                    \right.
  \end{equation}
\end{theorem}

\begin{remark}
\textnormal{In the special case $1=p<q<\infty$, $d\ge 2$, and $\eta=\a\in\N$, some embeddings from Theorem~\ref{thEmbedd} (A) can be improved as follows (see Theorem~\ref{thKolyaL1T})
$$
\Lip(\eta,\a,p)\hookrightarrow B_{q,1}^{\a-\t},\qquad \t=d(1-\frac1q).
$$
Note that by the Marchaud inequality~\eqref{eq.lemMarchaudMod}, we have
$$
B_{q,1}^{\a-\t} \hookrightarrow \Lip\(\eta-\t,\a-d(1-\frac1q),q\),
$$
see also the proof of part (B) below.}
\end{remark}

%See also {\rm Fig. 1a} and {\rm Fig. 1b} after the proof of the theorem.
\begin{remark}
\textnormal{In the scale of Lipschitz spaces, embeddings in Theorem~\ref{thEmbedd} are sharp. This follows from Theorem~\ref{emb-th-lip}.}
\end{remark}

Let us illustrate the embeddings, which follow from Theorem~\ref{thEmbedd}, for the most important case $\eta=\a$ and $0<p<q\le \infty$. In Fig.~1 and Fig.~2, we use the following notation:
\begin{equation*}
  \begin{split}
     &X_1=\Lip\(\a-\t,\a-d\(1-1/q\),q\),\\
      &X_2=\Lip(\a-\t,\a,q),\\
&X_3=B_{q,p}^{\a-\t},\\
&X_4=\Lip(\a-\t,\a-\t,q)\cap X_3,\\
&X_5=\Lip(\a-\t,\a-\t,q)\cap \Lip^{\frac1{p'}}(\a-\t,\a-\t,q).
   \end{split}
\end{equation*}
%$$
%X_1=\Lip\(\a-\t,\a-d\(1-\frac1q\),q\),
%$$
%$$
%X_2=\Lip(\a-\t,\a,q),
%$$
%$$
%X_3=B_{q,p}^{\a-\t},
%$$
%$$
%X_4=\Lip(\a-\t,\a-\t,q)\cap X_3,
%$$
%and
%$$
%X_5=\Lip(\a-\t,\a-\t,q)\cap \Lip^{\frac1{p'}}(\a-\t,\a-\t,q).
%$$
We use the solid line \textbf{------} (or the dashed line \textbf{- - - -}) to emphasize that the corresponding boundary is included (or excluded).

%\begin{tikzpicture}
%\node at (0,0) {We use}; \node at (5,0) {to emphasize that the border of $X_j$ is closed};
%\draw[thick] (1,-0.05)--(2,-0.05);
%\end{tikzpicture}

\bigskip

\begin{center}
\begin{tikzpicture}[line join = round, line cap = round]
%\pgfmathsetmacro{\factor}{1/sqrt(2)};
%\coordinate [label=right:] (A) at (1,0*\factor);
%\coordinate [label=left:] (B) at (2,3*\factor);

\draw[->] (-0.3,0) -- (6.3,0) node[right] {$\frac1p$};
\draw[->] (0,-0.3) -- (0,6.0) node[left] {$\frac1q$};

\draw[dashed,thick] (3+0.08,3+0.08)--(6+0.08,6+0.08);
\draw[dashed,thick] (3+0.08,2+0.04)--(6,2+0.04);
\draw[dashed, thick] (3+0.08,3)--(3+0.08,2+0.04);

\draw[thick] (3+0.08,2-0.04)--(6,2-0.04);
\draw[dashed, thick] (3+0.08,2-0.04)--(3+0.08,0);
\draw[thick] (3+0.08,0)--(6,0);

%\fill (0,3) circle (1pt);
\draw (-0.05,3) -- (0.05,3);
\node[left] at (0,3) {$1$};

%\fill (3,0) circle (1pt);
\draw (3,-0.06) -- (3,0.06);
\node[below] at (3,0) {$1$};

\draw (-0.05,2) -- (0.05,2);
\node[left] at (0,2) {$\small{1-1/d}$};

\draw (-0.05,1.5) -- (0.05,1.5);
\node[left] at (0,1.5) {$1/2$};

\draw (1.5,-0.06) -- (1.5,0.06);
\node[below] at (1.5,0) {$1/2$};

\node at (5,4) {$X_2$};
\node at (5,1) {$X_1$};

\draw[thick] (3,0)--(3,3);
\draw[dashed,thick] (1.5+0.08,1.5+0.08)--(3-0.06,3-0.08);
\draw[thick] (1.5,1.5)--(1.5,+0.08);
\draw[ultra thick] (0,0)--(3-0.08,0);
\draw[dashed,thick] (1.5+0.04,0.08)--(3-0.08,0.08);
\draw[dashed,thick] (0.08,0.08)--(1.5-0.04,0.08);
\draw[dashed,thick] (0.08,0.08)--(1.5-0.08,1.5-0.08);
\draw[dashed,thick] (1.5-0.08,0.08)--(1.5-0.08,1.5-0.08);

\node at (1,0.5) {$X_4$};
\node at (2.25,1) {$X_3$};

\node at (0.7,-0.7) {$X_5$};
\draw[->] (0.7,-0.5) -- (1.3,-0.05);

\end{tikzpicture}
\end{center}
\begin{center}
{\small\textbf{Fig. 1:} The sets of pairs $(1/p,1/q)$ for which the embeddings $\Lip(\a,\a,p)\hookrightarrow X_i$, $i=\overline{1,5}$, from Theorem~\ref{thEmbedd} are fulfilled in the case $d\ge 2$ and $\eta=\a\in \N$.}
\end{center}

\begin{center}
\begin{tikzpicture}[line join = round, line cap = round]
%\pgfmathsetmacro{\factor}{1/sqrt(2)};
%\coordinate [label=right:] (A) at (1,0*\factor);
%\coordinate [label=left:] (B) at (2,3*\factor);

\draw[->] (-0.3,0) -- (6.3,0) node[right] {$\frac1p$};
\draw[->] (0,-0.3) -- (0,6.0) node[left] {$\frac1q$};

\draw[dashed,thick] (3+0.08,3+0.08)--(6+0.08,6+0.08);
\draw[thick] (3,1.5+0.04)--(6,1.5+0.04);
\draw[thick] (3,3)--(3,1.5+0.04);

\draw[dashed,thick] (3,1.5-0.04)--(6,1.5-0.04);
\draw[thick] (3,1.5-0.04)--(3,0.08);
\draw[ultra thick] (3,0)--(6,0);
\draw[dashed,thick] (3,0.08)--(6,0.08);

%\fill (0,3) circle (1pt);
\draw (-0.05,3) -- (0.05,3);
\node[left] at (0,3) {$1$};

%\fill (3,0) circle (1pt);
\draw (3,-0.06) -- (3,0.06);
\node[below] at (3,0) {$1$};

%%\fill (0,2) circle (1pt);
%\draw (-0.05,2) -- (0.05,2);
%\node[left] at (0,2) {$\small{1-1/d}$};

%\fill (0,0) circle (1pt);
%\node[fill=white] at (0,-0.4) {$0$};

%\fill (0,1.5) circle (1pt);
\draw (-0.05,1.5) -- (0.05,1.5);
\node[left] at (0,1.5) {$1/2$};

%\fill (1.5,0) circle (1pt);
\draw (1.5,-0.06) -- (1.5,0.06);
\node[below] at (1.5,0) {$1/2$};

\node at (5,4) {$X_2$};
\node at (5,1) {$X_5$};

%\filldraw[fill=blue!20!white, draw=blue] (0,0)  -- (0.7,1.7) -- (2,2.5)-- (2.5,2.5) -- (4,2) -- (2.5,0.5) -- (0,0);

\draw[dashed,thick] (3-0.08,0)--(3-0.08,3-0.08);
\draw[dashed,thick] (1.5+0.04,1.5+0.08)--(3-0.08,3-0.08);
%\draw[dashed,thick] (1.5+0.08,1.5+0.08)--(3-0.06,3-0.08);
\draw[thick] (1.5+0.04,1.5)--(1.5+0.04,+0.08);
\draw[ultra thick] (0,0)--(3-0.08,0);
\draw[dashed,thick] (1.5+0.04,0.08)--(3-0.06,0.08);
\draw[dashed,thick] (0.08,0.08)--(1.5-0.04,0.08);
\draw[dashed,thick] (0.08,0.08)--(1.5-0.04,1.5-0.04);
\draw[dashed,thick] (1.5-0.04,0.08)--(1.5-0.04,1.5-0.08);

\node at (1,0.5) {$X_4$};
\node at (2.25,1) {$X_3$};

\node at (0.7,-0.7) {$X_5$};
\draw[->] (0.7,-0.5) -- (1.3,-0.05);

\node at (4.7,-0.7) {$X_2$};
\draw[->] (4.7,-0.5) -- (3.7,-0.05);

\end{tikzpicture}
\end{center}
\begin{center}
{\small\textbf{Fig. 2:} The sets of pairs $(1/p,1/q)$ for which the embeddings $\Lip(\a,\a,p)\hookrightarrow X_i$, $i=\overline{1,5}$, from Theorem~\ref{thEmbedd} are fulfilled in the case $d\ge 2$ and $\eta=\a\not\in \N$.}
\end{center}

\begin{proof}
(A) We need to compare the embeddings, which three main inequalities (the classical Ulyanov inequality, the sharp Ulyanov inequality, and the Ulyanov--Marchaud inequality) provide.

First of all, let us compare the sharp Ulyanov inequality and the Ulyanov--Marchaud inequality.
%(in what follows, it will be convenient to compare the Ulyanov-Marchaud and the Sharp Ulyanov inequality).
In the case $0<p\le 1$, we have from the sharp Ulyanov inequality given by Theorems~\ref{th1} and~\ref{thMainMod} that
\begin{equation}\label{proof1}
\begin{split}
    \w_{\a-\g}(f,\d)_q&\lesssim \left(\int_0^\d\bigg(\frac{\w_{\a}(f,t)_p}{t^{d(\frac1p-\frac1q)}}\bigg)^{q_1}\frac{dt}{t}\right)^{\frac1{q_1}}\\
&+
\frac{\w_\a(f,\d)_p}{\d^{\g+d(\frac1p-1)}}\left\{
                                          \begin{array}{ll}
                                            1, & \hbox{$\g>d\(1-\frac1q\)_+$;} \\
                                            1, & \hbox{$\g=d\(1-\frac1q\)_+\ge 1$, $d\ge 2$, and $\a\in\N$;} \\
                                            \ln^{1/{q_1}}\(\frac1\d+1\), & \hbox{$\g=d\(1-\frac1q\)_+\ge 1$, $d\ge 2$, and $\a\not\in\N$;} \\
                                            1, & \hbox{$\g=d=1$ and $q=\infty$;} \\
                                            \ln^{1/{q}}\(\frac1\d+1\), & \hbox{$0<\g=d\(1-\frac1q\)_+< 1$;} \\
                                            \(\frac1\d\)^{d(1-\frac1q)-\g}, & \hbox{$0<\g<d\(1-\frac1q\)_+$;}\\
                                            \(\frac1\d\)^{d(1-\frac1q)}, & \hbox{$\g=0$.}
                                          \end{array}
                                        \right.
\end{split}
\end{equation}
At the same time, from the Ulyanov--Marchaud inequality \eqref{eqthRealKUMMod1} and property $(e)$ given in Section~\ref{sec4}, we obtain
\begin{equation}\label{proof2}
\begin{split}
    \w_{\a-\g}(f,\d)_q\lesssim &\left(\int_0^\d\bigg(\frac{\w_{\a}(f,t)_p}{t^{d(\frac1p-\frac1q)}}\bigg)^{q_1}\frac{dt}{t}\right)^{\frac1{q_1}}\\
&
\quad\quad+\frac{\w_\a(f,\d)_p}{\d^{\g+d(\frac1p-1)}}\left\{
                                          \begin{array}{ll}
                                            1, & \hbox{$\g>d\(1-\frac1q\)$;} \\
                                            \ln^{1/\tau}\(\frac1\d+1\), & \hbox{$\g=d\(1-\frac1q\)$;} \\
                                            \(\frac1\d\)^{d(1-\frac1q)-\g}, & \hbox{$\g<d\(1-\frac1q\)$,}
                                          \end{array}
                                        \right.
\end{split}
\end{equation}
where
$$
\tau=\tau(q)=\left\{
       \begin{array}{ll}
         \min(q,2), & \hbox{$q<\infty$;} \\
         1, & \hbox{$q=\infty$.}
       \end{array}
     \right.
$$

It is easy to see that~\eqref{proof2} gives the same estimates as inequality~\eqref{proof1} in all cases except the  following three cases:

\medskip

1) $2<q<\infty$ and $\g=d(1-1/q)$. In this case, by~\eqref{proof1}, for $\a\not\in \N$, we have  that
$$
\Lip(\eta,\a,p) \hookrightarrow \Lip^{\frac1{q}}\(\eta-\t,\a-d\(1-\frac1q\),q\)
$$
and, for $\a\in \N$, we have
$$
\Lip(\eta,\a,p) \hookrightarrow \left\{
                                  \begin{array}{ll}
                                    \Lip\(\eta-\t,\a-d\(1-\frac1q\),q\), & \hbox{$d\(1-\frac1q\)\ge 1$;} \\
                                    \Lip^{\frac1{q}}\(\eta-\t,\a-d\(1-\frac1q\),q\), & \hbox{$d\(1-\frac1q\)< 1$,}
                                  \end{array}
                                \right.
$$

2) $0<q\le 2$, $\g=d(1-1/q)_+\ge 1$, and $\a\in \N$,

3) $q=\infty$, $\g=d=1$ or $\g=d\ge2$, $\a\in \N$.

\medskip

\noindent In the last two cases, \eqref{proof1} implies that
$$
\Lip(\eta,\a,p) \hookrightarrow \Lip\(\eta-\t,\a-d\(1-\frac1q\),q\).
$$

\medskip

Let us note that the Ulyanov--Marchaud inequality is a direct consequence of the classical Ulyanov inequality and the Marchaud inequality (see the proof of Theorem~\ref{thRealKUM}). Thus, the classical Ulyanov inequality always gives more optimal embeddings than the Ulyanov--Marchaud inequality. Taking into account this remark, to prove (A), it is sufficient to compare the embedding
\begin{equation*}
  \Lip(\eta,\a,p)\hookrightarrow \Lip(\eta-\t,\a,q),
\end{equation*}
which follows from the classical Ulyanov (see~\eqref{eqth1.1--} and~\eqref{eqlemMM1} with $\g=0$), and the corresponding embeddings, which follow from the sharp Ulyanov inequality in the above exceptional cases.

1) If $2<q<\infty$, $\g=d(1-1/q)$, by Theorem~\ref{emb-th-lip}, $\Lip(\eta-\t,\a,q)$ and $\Lip^{1/q}(\eta-\t,\a-d(1-1/q),q)$ are not comparable.  Therefore, for $\a\not\in\N$, we have
$$
\Lip(\eta,\a,p)\hookrightarrow \Lip(\eta-\t,\a,q)\cap \Lip^{\frac1{q}}\(\eta-\t,\a-d\(1-\frac1q\),q\).
$$
Let us consider the case $\a\in \N$. First, let $d(1-1/q)\ge 1$. In this case, it is clear that
$$
\Lip\(\eta-\t, \a-d\(1-\frac1q\),q\) \hookrightarrow \Lip(\eta-\t,\a,q).
$$
Second, let $d(1-1/q)<1$. Then, by Theorem~\ref{emb-th-lip}, $\Lip(\eta-\t,\a,q)$ and $\Lip^{1/q}\(\eta, \a-d(1-1/q),q\)$ are not comparable.

2) In the second case, $0<q<2$, $\g=d(1-1/q)_+\ge 1$, $\a\in \N$, that is, $d/(d-1)\le q<2$, we have that
\begin{equation}\label{zvezda+++++}
  \Lip(\eta,\a,q)  \hookrightarrow \Lip\(\eta-\t, \a-d\(1-\frac1q\),q\) \hookrightarrow \Lip(\eta-\t,\a,q),
\end{equation}
which implies the desired embedding.

3) If $q=\infty$, $\g=d=1$ or $\g=d\ge2$, $\a\in \N$, then  as in the case 2), embeddings~\eqref{zvezda+++++} hold.

%then
%\begin{equation*}
%  \Lip(\eta,\a,p)\hookrightarrow \Lip^\frac1q\(\eta, \a-d(1-\frac1q),q\);
%\end{equation*}

%2) if $\a\in \N$, then
%\begin{equation*}
%  \Lip(\eta,\a,p)\hookrightarrow \left\{
%                                   \begin{array}{ll}
%                                     \Lip(\eta, \a-d(1-\frac1q),q), & \hbox{$d(1-\frac1q)\ge 1$;} \\
%                                     \Lip^\frac1q(\eta, \a-d(1-\frac1q),q), & \hbox{$d(1-\frac1q)< 1$.}
%                                   \end{array}
%                                 \right.
%\end{equation*}
%
%%3) if $q=\infty$ and $d(1-\frac1q)<1$, then
%%\begin{equation*}
%%  \Lip(\eta,\a,p)\hookrightarrow \Lip\(\eta, \a-d(1-\frac1q),q\).
%%\end{equation*}
%
%In the  case $2<q< \infty$, we have by Theorem~\ref{emb-th-lip} that
%$
%  \Lip(\eta,\a,p)
%$ and $\Lip^\frac1q\(\eta, \a-d(1-\frac1q),q\)$ are not comparable.
%Therefore, when either $2<q<\infty$ and $d(1-\frac1q)<1$ or $2<q\le \infty$ and $d(1-\frac1q)\ge 1$, $\a\not\in\N$, we have
%$$
%\Lip(\eta,\a,p)\hookrightarrow \Lip(\eta-\t,\a,q)\cap \Lip^{\frac1{q}}\(\eta-\t,\a-d(1-\frac1q),q\).
%$$
%
%
%Let us consider the case $\a\in \N$. Let first $d(1-\frac1q)\ge 1$. In this case, it is obvious that
%$$
%\Lip(\eta-\t, \a-d(1-\frac1q),q) \hookrightarrow \Lip(\eta-\t,\a,q).
%$$
%
%In the case $d(1-\frac1q)<1$, again by using Theorem~\ref{emb-th-lip}, we have that $
%  \Lip(\eta,\a,p)
%$ and $\Lip^\frac1q\(\eta, \a-d(1-\frac1q),q\)$ are not comparable.

%Finally, in the case $q=\infty$ and $d(1-\frac1q)<1$, it is obvious that
%$$
%\Lip(\eta-\t,\a-d(1-\frac1q),q)\hookrightarrow \Lip(\eta-\t,\a,q).
%$$

Combining the above embeddings, we complete the proof of~(A).

\smallskip
%\noindent the classical Ulyanov inequality gives more optimal embedding than Sharp Ulyanov inequality (\ref{eqlemMM1}).

(B) Let $1<p<q\le\infty$.
%The case $d=1$, $1<p<q\le \infty$ and $\a=\eta$ see in~\cite{Ti}, while the case $\eta<\a$ follows from Marchaud inequality and inequality~(\ref{eqth1.1--}). By analogy with the case $d=1$, using Marchaud inequality~\eqref{eq.lemMarchaudMod} and Sharp Ulyanov inequality~\eqref{eqlemMM1}, one can prove the corresponding embedding in the case of several variables.
%
%
%-------------------------
%
%-------------------------
As above, we see that from the Ulyanov--Marchaud inequality \eqref{eqthRealKUMMod1} and property $(e)$ we obtain
\begin{equation*}
%\label{proof3}
\begin{split}
    \w_{\a-\g}(f,\d)_q\lesssim &\left(\int_0^\d\bigg(\frac{\w_{\a}(f,t)_p}{t^{d(\frac1p-\frac1q)}}\bigg)^{q_1}\frac{dt}{t}\right)^{\frac1{q_1}}\\
&\quad\quad\quad\quad\quad\quad+
\frac{\w_\a(f,\d)_p}{\d^{\g}}\left\{
                                          \begin{array}{ll}
                                            1, & \hbox{$\g>\t$;} \\
                                            \ln^{1/\tau}\(\frac1\d+1\), & \hbox{$\g=\t$;} \\
                                            \(\frac1\d\)^{\t-\g}, & \hbox{$\g<\t$.}
                                          \end{array}
                                        \right.
\end{split}
\end{equation*}
This implies that in all cases except the case $\g=\t$ the classical Ulyanov inequality gives the same estimates as the sharp Ulyanov inequality (see~\eqref{eqth1.1--} for $d=1$ and~\eqref{eqlemMM1} for $d\ge 2$).
Thus, to prove (B), we compare the three following  embeddings:

$$\Lip(\eta,\a,p)\hookrightarrow \Lip(\eta-\t,\a,q),\leqno{(1)}$$

which follows from the classical Ulyanov inequality (see \eqref{eqth1.1--} and~\eqref{eqlemMM1} in the case~$\g=0$);

$$
\Lip(\eta,\a,p)\hookrightarrow \left\{
                                 \begin{array}{ll}
                                   \Lip(\eta-\t,\a-\t,q), & \hbox{$p<q<\infty$;} \\
                                   \Lip^{\frac1{p'}}(\eta-\t,\a-\t,q), & \hbox{$q=\infty$,}
                                 \end{array}
                               \right.\leqno{(2)}
$$

which follows  from the sharp Ulyanov inequalities \eqref{eqth1.1--} and \eqref{eqlemMM1} in the case $\g=\t$;

$$
\Lip(\eta,\a,p)\hookrightarrow B_{q,p}^{\eta-\t},\leqno{(3)}
$$

which follows from Kolyada's inequality in the case $1<p<q<\infty$ and $\eta=\a$.

\medskip

First, if $\eta<\a$, comparing embeddings (1) and (2) and taking into account the Marchaud inequality~\eqref{eq.lemMarchaudMod}, we obtain
$$
\Lip(\eta-\t,\a-\t,q)\equiv \Lip(\eta-\t,\a,q)\hookrightarrow \Lip^{\frac1{p'}}(\eta-\t,\a-\t,q),
$$
which implies~\eqref{eqthEmbed1.4}.
Second, if $\eta=\a$ and $q=\infty$,  the same argument and Theorem~\ref{emb-th-lip} yield~\eqref{eqthEmbed1.4-}. Finally, if $\eta=\a$ and $q<\infty$, we use the following embedding
$$
W_p^\a \hookrightarrow B_{q,p}^{\a-\t},\quad 1<p<q<\infty,
$$
which easily follows from Kolyada's inequality (see Theorem~\ref{thKolyaHpT}) and the fact that $W_p^\a$ coincides with the Lipschitz-type space  $\{f\in L_p\,:\,\w_\a(f,\d)_p=\mathcal{O}(\d^\a)\}$.
Thus, combining the above embedding with
the well-known interrelation between the fractional Riesz and Besov spaces $W_p^{\a-\t}$ and $B_{q,p}^{\a-\t}$ (see~\cite[p. 155]{Stein}), we easily obtain~\eqref{eqthEmbed1.4-} in the case $q<\infty$.
\end{proof}

We conclude this subsection with a remark on more general function spaces.
Define
$$
{\textnormal{Lip}} \,\big(\omega(\cdot),l,X\big)\,:=
\, \Big\{f\in X(\T^d)\,:\,\omega_{l}(f,\delta)_X= \mathcal{O} \(\omega(\delta)\),\quad \delta\to 0\Big\},$$
where
 $\omega(\cdot)$ is a  non-decreasing function on $[0,1]$  such that
 $\omega(\delta)\to 0$ as $\delta\to 0$ and
 $\delta^{-l}\omega(\delta)$ is non-increasing {\rm(}note that this class is natural class of majorant for fractional moduli of smoothness, see \cite{tik2}{\rm)} and $X(\T^d)$ is an appropriate function space.

For the Lebesgue spaces, {\rm Theorems \ref{th1} and \ref{thMainMod}} imply that
\begin{equation*}
%\label{emb1}
{\textnormal{Lip}} \,\big(\omega(\cdot),\a+\g,
L_{p}({\T^d})
\big)
\subset
%\hookrightarrow
{\textnormal{Lip}} \,\big(\widetilde{\omega}(\cdot),\a,
L_{q}({\T^d})
\big),\qquad0<p<q\le \infty,
\end{equation*}
provided that
\begin{equation*}
%\label{emb2}
 \left(\int_0^\d\bigg(
 \frac{\omega(t)}{t^{\g}} \s\(\frac1t\)
\bigg)^{q_1}\frac{dt}{t}\right)^{\frac1{q_1}}
=O\Big(\widetilde{\omega}(t)\Big),
\end{equation*}
where $\s(\cdot)$ is defined in Theorems \ref{th1} and \ref{thMainMod}.
%and
%$$
%\s%_{q,\g}
%(t)
%:=\left\{
%         \begin{array}{ll}
%           t^{\frac1p-1}, & \hbox{$\g>\(1-\frac1q\)_+$;} \\
%           t^{\frac1p-1}\ln^\frac1q (t+1), & \hbox{$0<\g=\(1-\frac1q\)_+$;} \\
%           t^{\frac1p-\frac1q-\g}, & \hbox{$0< \g<\(1-\frac1q\)_+$;}\\
%           t^{\frac1p-\frac1q}, & \hbox{$\g=0$.}
%         \end{array}
%       \right.
%$$
%Let us consider the case of  one periodic functions in detail.

This approach  can   also be used  to obtain embedding theorems for  general Calder\'{o}n-type spaces
$$\Lambda^{l}( G,E) =
\, \Big\{f\in G:\;\; \|f\|_G +\|\omega_{l}(f,\cdot)_G\|_E  <\infty
\Big\},$$ introduced by  Calder\'{o}n \cite{calderon}. Note that the
classical Besov spaces $B_{p,q}^{\alpha}$ are a particular case of
the Calder\'{o}n spaces.

\subsection{Embedding properties of Besov and Lipschitz-type spaces }
Here, we study sharp embedding  theorems between
Lipschitz and logarithmic Lipschitz spaces (Theorem \ref{emb-th-lip})
 and between Lipschitz and  Besov spaces (Theorem \ref{emb-th-lip+}).
  This problem is of interest in its own right.

\begin{theorem}\label{emb-th-lip}
  Let $\a,\b>0$, $1\le q\le \infty$, $\tau=\left\{
                                     \begin{array}{ll}
                                       \min (2,q), & \hbox{$1\le q< \infty$;} \\
                                       1, & \hbox{$q=\infty$,}
                                     \end{array}
                                   \right.
  $
and $\varepsilon \in (0,1/\tau)$. We have %Then we have the following assertions:

\medskip

%$$$$
%\begin{picture}(60, 40)
%\put(100,30){\oval(100,44)}
%\put(100,30){\circle{35}}
%\put(115,30){\circle{35}}
%\end{picture}

%PICTURE

{\rm 1)} $ \Lip(\a,\a+\b,q)$ and  $\Lip^\varepsilon(\a,\a,q)$ are subclasses of  $\Lip^{1/\tau}(\a,\a,q)$;

{\rm 2)} $ \Lip(\a,\a+\b,q)$ and  $\Lip^\varepsilon(\a,\a,q)$ are not comparable.

%See also Fig. 2 below.

\end{theorem}

\medskip

Fig.~3 illustrates the embeddings in Theorem~\ref{emb-th-lip}.

\bigskip

\begin{center}
\begin{tikzpicture}[fill opacity=0.05, ,xscale=0.8,yscale=0.8]

\fill[blue] (0,0) ellipse (4.2 and 2.5);
\draw[fill opacity=2] (0,0) ellipse (4.2 and 2.5);

\fill[red] (-1.4,0) ellipse (2.5 and 1.5);
\draw[fill opacity=2] (-1.4,0) ellipse (2.5 and 1.5);
\fill[green] (1.4,0) circle (2.5 and 1.5);
\draw[fill opacity=2] (1.4,0) circle (2.5 and 1.5);

\node[fill opacity=2,xscale=0.8,yscale=0.8] at (0,-2) {$\Lip^{1/\tau}(\a,\a,q)$};

\node[fill opacity=2,xscale=0.8,yscale=0.8] at (-2.5,0) {${\small \Lip(\a,\a+\b,q)}$};

\node[fill opacity=2,xscale=0.8,yscale=0.8] at (2.5,0) {$\Lip^\varepsilon(\a,\a,q)$};

\end{tikzpicture}
\end{center}

\phantom{qqq}

\begin{center}
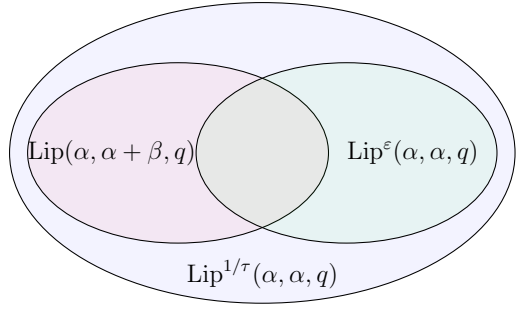
{\small \textbf{Fig. 3:} The relationship among the spaces from Theorem~\ref{emb-th-lip}.}
\end{center}

\bigskip
\bigskip

The proof of the first part follows from
Marchaud inequality~\eqref{eq.lemMarchaudMod}. The proof of the second part is based on
the following three results.

\begin{proposition}\label{emb-th-lip1}
  Let $1\le q,r\le \infty$, $\a,\b>0$, and $\varepsilon >0$.
 Then there exists a function $f_1\in  \Lip^\varepsilon(\a,\a,q) \diagdown \Lip(\a,\a+\b,r)$.
\end{proposition}

\begin{proof}Let us define the lacunary series
$$
f_1(x)
=\sum_{n=1}^\infty c_n \cos \lambda_n x, \qquad \lambda_{n+1}/\lambda_n\ge \lambda>1.
$$
Zygmund's theorem (see~\cite[Ch. 8]{Z}) implies that $\|f_1\|_q\asymp \|f_1\|_2$ for $1\le q < \infty$ and $\|f_1\|_\infty\asymp \|\{ c_n\}_{n=1}^\infty\|_{\ell_1}$ (see~\cite{St56_3}).
Taking
$$
\lambda_n=2^{2^n},\qquad c_n=%\frac
{2^{n\varepsilon}}{\lambda_n^{-\a}}
$$
and choosing $s$ for any natural $N$ such that $\lambda_s\le N<\lambda_{s+1}$,
we get by the realization results~\eqref{eq.th6.0} that

\begin{equation*}
  \begin{split}
     \omega_{\a}(f_1,1/N)_q &\lesssim
N^{-\a}
\left\| S_N^{({\a})}(f_1)
\right\|_q
+
\left\| f_1-S_N(f_1)
\right\|_q \\
&\lesssim
N^{-\a}
\left(
\sum_{\lambda_\nu\le N } c_\nu^2 \lambda_\nu^{2\a}\right)^{1/2}
+
\left(
\sum_{\lambda_\nu\ge N }
c_\nu^2 \right)^{1/2}\\
&=
N^{-\a}
\left(
\sum_{\nu=0}^s 2^{2\nu\varepsilon}\right)^{1/2}
+
\left(
\sum_{\nu=s+1}^\infty \frac{2^{2\nu\varepsilon}}{
\lambda_\nu^{2\a}
}\right)^{1/2}\\
&\lesssim
N^{-\a}
{\ln^{\varepsilon} (N+1)},
   \end{split}
\end{equation*}
i.e., $f_1\in  \Lip^\varepsilon(\a,\a,q)$ for $q<\infty$. If $q=\infty$, similar calculation implies that
\begin{equation*}
  \begin{split}
     \omega_{\a}(f_1,1/N)_q &\lesssim
N^{-\a}
\left\| S_N^{({\a})}(f_1)
\right\|_q
+
\left\| f_1-S_N(f_1)
\right\|_q\\
&\lesssim
N^{-\a}
\sum_{\lambda_\nu\le N } c_\nu \lambda_\nu^{\a}
+
\sum_{\lambda_\nu\ge N }
c_\nu \lesssim
N^{-\a}
{\ln^{\varepsilon} (N+1)}.
   \end{split}
\end{equation*}
On the other hand, Jackson's inequality~\eqref{JacksonSO} for  $N=\lambda_s$ gives
 $$
\omega_{\a+\b}(f_1,1/N)_r \gtrsim
E_{N-1}(f_1)_r
\gtrsim
\|f_1-S_{N-1}(f_1)\|_r
\gtrsim
\left(
\sum_{\nu=s}^\infty \frac{2^{2\nu\varepsilon}}{
\lambda_\nu^{2\a}
}\right)^{1/2}
\ge
 \frac{2^{s\varepsilon}}{
\lambda_{s}^{\a}}.
$$
%for any $\lambda_s\le N<\lambda_{s+1}$.
Assuming that
$f_1\in  \Lip(\a,\a+\b,r)$, we arrive at the contradiction

 $$
\frac{1}{
\lambda_{s}^{\a}}\gtrsim\omega_{\a+\b}(f_1,1/N)_r \gtrsim
 \frac{2^{s\varepsilon}}{
\lambda_{s}^{\a}}.
$$
Thus, $f_1\notin  \Lip(\a,\a+\b,r)$ for $r< \infty$.
If $r=\infty$
$$
\omega_{\a+\b}(f_1,1/\lambda_s)_r \gtrsim
E_{\lambda_s-1}(f_1)_r
\gtrsim
c_{\lambda_s}
= \frac{2^{s\varepsilon}}{
\lambda_{s}^{\a}}.
$$
\end{proof}

\begin{proposition}\label{emb-th-lip2}
  Let $1\le q,r\le \infty$, $\a,\b>0$, and $\varepsilon \in (0,1/2)$ for $r<\infty$ and $\varepsilon \in (0,1)$ for $r=\infty$.
 Then there exists a function $f_2\in  \Lip(\a,\a+\b,q)\diagdown \Lip^\varepsilon(\a,\a,r)$.

\end{proposition}
\begin{proof}
We consider $$
f_2(x)
=\sum_{n=1}^\infty c_n \cos 2^nx, \qquad c_n= \frac{1}{2^{\a n} }.
$$
It is clear that $f_2\in L_q(\T)$, $1\le q\le \infty$. Moreover, since $\|f_2\|_q\asymp \|f_2\|_2$ for $1\le q < \infty$, we get by~\eqref{eq.th6.0} that
for integer $N\in [2^n, 2^{n+1})$
\begin{equation*}
  \begin{split}
    \omega_{\a+\b}(f_2,1/N)_q &\lesssim
N^{-\a-\b}
\left\| S_N^{({\a+\b})}(f_2)
\right\|_q
+
\left\| f-S_N(f_2)
\right\|_q\\
       &\lesssim
2^{-n(\a+\b)}
\left(
\sum_{k=1}^n c_k^2 2^{2k(\a+\b)}\right)^{1/2}
+
\left(
\sum_{k=n+1}^\infty c_k^2 \right)^{1/2}\lesssim
2^{-\a n}\lesssim
N^{-\a},
   \end{split}
\end{equation*}
%$$
%\omega_{\a+\b}(f_2,1/N)_q \lesssim
%N^{-\a-\b}
%\left\| S_N^{({\a+\b})}(f_2)
%\right\|_q
%+
%\left\| f-S_N(f_2)
%\right\|_q
%$$
%
%$$
%\lesssim
%2^{-n(\a+\b)}
%\left(
%\sum_{k=1}^n c_k^2 2^{2k(\a+\b)}\right)^{1/2}
%+
%\left(
%\sum_{k=n+1}^\infty c_k^2 \right)^{1/2}\lesssim
%2^{-\a n}\lesssim
%N^{-\a},
%$$
i.e., $f_2\in  \Lip(\a,\a+\b,q)$ if $q<\infty$. For $q=\infty$ the same holds using
$\|f_2\|_\infty\asymp \|\{ c_n\}_{n=1}^\infty\|_{\ell_1}$.

Moreover,  for $1\le r<\infty$,
$$\omega_{\a}(f_2,1/N)_r \gtrsim
N^{-\a}
\left\| S_N^{({\a})}(f_2)
\right\|_r
\gtrsim
2^{-n \a}
\left(
\sum_{k=1}^N 1 \right)^{1/2}
\gtrsim
N^{-\a}
{\ln^{1/2} (N+1)}
$$
and
$$\omega_{\a}(f_2,1/N)_\infty \gtrsim
N^{-\a}
\left\| S_N^{({\a})}(f_2)
\right\|_\infty
\gtrsim
N^{-\a}
{\ln (N+1)}.
$$
Then
$f_2\notin \Lip^\varepsilon(\a,\a,r)$ for any $\varepsilon \in (0,1/2)$ when $r<\infty$ and for any $\varepsilon \in (0,1)$ when $r=\infty$.
\end{proof}

%\begin{picture}(200,80)
%\put(50,40){\oval(100,30)}
%\put(30,40){\oval(100,30)}
%\end{picture}

\begin{proposition}\label{emb-th-lip3}
  Let $1< q\le 2$, $\a,\b>0$, and $\varepsilon \in (0,1/\tau)$. Then there exists a function $f_3\in  \Lip(\a,\a+\b,q)\diagdown \Lip^\varepsilon(\a,\a,q)$.

\end{proposition}

\begin{proof}
Define $$
f_3(x)=\sum_{n=1}^\infty a_n \cos nx, \qquad a_n^q= \frac{1}{n^{\a q +q-1}}.
$$
Since
$$
\sum_{n=1}^\infty a_n^q n^{q-2}= \sum_{n=1}^\infty \frac{1}{n^{\a q +1}}<\infty,
$$
by the Hardy--Littlewood theorem~\eqref{monot}
we get that $f_3\in L_q(\T)$. Moreover, realization~\eqref{eq.th6.0} and Theorem 6.1 from~\cite{GT}  gives
$$
\omega_{\a+\b}(f_3,1/n)_q \asymp
n^{-\a-\b}
\left(
\sum_{k=1}^n a_k^p k^{(\a+\b)q+q-2}\right)^{1/q}
+
\left(
\sum_{k=n+1}^\infty a_k^q k^{q-2}\right)^{1/q}
\lesssim
n^{-\a},
$$
i.e., $f_3\in  \Lip(\a,\a+\b,q)$.

On the other hand,
\begin{equation*}
  \begin{split}
     \omega_{\a}(f_3,1/n)_q
&\gtrsim
n^{-\a}
\left\| S_n(f_3)^{(\a)}
\right\|_q \\
&\gtrsim
n^{-\a}
\left(
\sum_{k=1}^n a_k^q k^{\a q+q-2}\right)^{1/q}
\gtrsim
n^{-\a}
\left(
\sum_{k=1}^n  \frac1k \right)^{1/q}\gtrsim \frac{\ln^{1/q} (n+1)}{n^{\a}},
   \end{split}
\end{equation*}
which yields that
$f_3\notin \Lip^\varepsilon(\a,\a,q)$ for any $\varepsilon \in (0,1/\tau)$.
\end{proof}

The next theorem shows the relationship for the spaces $\Lip(\a-\t,\a-\t,q)$ and $B_{q,p}^{\a-\t}$.

\begin{theorem}\label{emb-th-lip+}
Let $1<p<q<\infty$, $\t=d(1/p-1/q)$, and $\a>\t$. We have

{\rm 1)} if $p\le \min(2,q)$, then $B_{q,p}^{\a-\t} \hookrightarrow \Lip(\a-\t,\a-\t,q)$;

{\rm 2)} if $\min(2,q)<p$, then the spaces $\Lip(\a-\t,\a-\t,q)$ and $B_{q,p}^{\a-\t}$ are not comparable.

\end{theorem}

The first part is a simple consequence of the Marchaud inequality (see Theorem~\ref{lemMarchaudMod}).

To prove the second part,  we consider the following functions:
$$
f_1(x)=\sum_{\nu=1}^\infty a_\nu \cos \nu x,\quad a_\nu=\frac1{\nu^{1-\frac1q+\a-\t}(\ln(\nu+1))^\frac1q}
$$
and
$$
f_2(x)=\sum_{\nu=1}^\infty a_{2^\nu} \cos 2^\nu x,\quad a_{2^\nu}=\frac1{2^{\nu(\a-\t)}\nu^\frac12},
$$
for which we have,

$f_1\in \Lip(\a-\t,\a-\t,q)\setminus B_{q,p}^{\a-\t}$ if $p\le \min(2,q)$,

$f_1\in \Lip(\a-\t,\a-\t,q)\setminus B_{q,p}^{\a-\t}$ and $f_2\in B_{q,p}^{\a-\t} \setminus \Lip(\a-\t,\a-\t,q)$ if $\min(2,q)<p$

\noindent (see~\cite{ST}).

\end{document}